\documentclass{amsart}
\usepackage{a4wide,amsmath,amsthm,amssymb,amsfonts}
\usepackage{mathdots}
\usepackage{latexsym}
\usepackage[all]{xy}
\usepackage{bm}
\usepackage{graphicx}
\usepackage{subfigure}
\usepackage{mathrsfs}
\usepackage{float}
\usepackage{comment}
\usepackage{makecell}
\usepackage{xcolor}
\usepackage{cite}

\usepackage{imakeidx}
\makeindex

\usepackage[pagebackref]{hyperref}
\usepackage{hyperref}
\usepackage{verbatim}
\usepackage{amsrefs}
\usepackage{bbm}
\usepackage{cleveref}
\usepackage{tikz}
\usetikzlibrary{calc}
\usepackage{adjustbox}
\numberwithin{equation}{section}
\newtheorem{theorem}{\bf{Theorem}}[section]
\newtheorem{proposition}[theorem]{\bf{Proposition}}
\newtheorem{definition}[theorem]{\bf{Definition}}
\newtheorem{corollary}[theorem]{\bf{Corollary}}
\newtheorem{lemma}[theorem]{\bf{Lemma}}
\newtheorem{remark}[theorem]{\bf{Remark}}
\newtheorem{example}[theorem]{\bf{Example}}
\newtheorem{conjecture}[theorem]{\bf{Conjecture}}
\newtheorem{assumption}[theorem]{\bf{Assumption}}

\newcommand{\abs}[1]{\left|#1\right|}

\newcommand{\car}[1]{\left|#1\right|}
\newcommand{\pairangone}[1]{\langle#1\rangle}
\newcommand{\pairang}[2]{\langle#1,#2\rangle}
\newcommand{\rest}{\lvert}

\newcommand{\mrgl}{\operatorname{GL}}
\newcommand{\mro}{\operatorname{O}}
\newcommand{\mru}{\operatorname{U}}
\newcommand{\mrsu}{\operatorname{SU}}
\newcommand{\mrso}{\operatorname{SO}}
\newcommand{\mrsl}{\operatorname{SL}}
\newcommand{\mrsp}{\operatorname{Sp}}

\newcommand{\mrm}{\operatorname{M}}

\newcommand{\mrdet}{\operatorname{det}}
\newcommand{\mrdim}{\operatorname{dim}}
\newcommand{\mrhom}{\operatorname{Hom}}

\newcommand{\mrdiag}{\operatorname{diag}}
\newcommand{\mrgcd}{\operatorname{gcd}}

\newcommand{\mraut}{\operatorname{Aut}}

\newcommand{\mrtemp}{\operatorname{Temp}}

\newcommand{\mrInd}{\operatorname{Ind}}

\newcommand{\mrres}{\operatorname{Res}}

\newcommand{\mrstab}{\operatorname{Stab}}

\newcommand{\mrim}{\operatorname{Im}}
\newcommand{\mrder}{\operatorname{der}}

\newcommand{\mrst}{\mathrm{St}}
\newcommand{\mrAff}{\mathrm{Aff}}
\newcommand{\mraff}{\mathrm{aff}}
\newcommand{\mrid}{\mathrm{Id}}
\newcommand{\mrext}{\mathrm{ext}}

\newcommand{\mrgal}{\mathrm{Gal}}

\newcommand{\mrch}{\mathrm{Ch}}
\newcommand{\mrfix}{\mathrm{Fix}}

\newcommand{\mrur}{\mathrm{ur}}
\newcommand{\mrsc}{\mathrm{sc}}

\newcommand{\mbz}{\mathbb{Z}}
\newcommand{\mbc}{\mathbb{C}}
\newcommand{\mbr}{\mathbb{R}}

\newcommand{\mbg}{\mathbb{G}}
\newcommand{\mbp}{\mathbb{P}}

\newcommand{\mfp}{\mathfrak{p}}

\newcommand{\mfo}{\mathfrak{o}}
\newcommand{\mfa}{\mathfrak{a}}

\newcommand{\mfh}{\mathfrak{h}}

\newcommand{\mfg}{\mathfrak{g}}

\newcommand{\mfH}{\mathfrak{H}}
\newcommand{\mfS}{\mathfrak{S}}

\newcommand{\mcs}{\mathcal{S}}
\newcommand{\mco}{\mathcal{O}}
\newcommand{\mcc}{\mathcal{C}}
\newcommand{\mcb}{\mathcal{B}}
\newcommand{\mca}{\mathcal{A}}

\newcommand{\mcf}{\mathcal{F}}
\newcommand{\mcg}{\mathcal{G}}
\newcommand{\mch}{\mathcal{H}}

\newcommand{\mcv}{\mathcal{V}}

\newcommand{\ch}[2]{\mathrm{Ch}_{#1}^{#2}}   

\newcommand{\Ft}{\mathcal{F}^{\theta}}

\newcommand{\Ftmax}{\mathcal{F}_{\mathrm{max}}^{\theta}}

\newcommand{\Fteff}{\mathcal{F}_{\mathrm{eff}}^{\theta}}

\newcommand{\trank}[1]{r_{\theta}(#1)}

\newcommand{\tdist}[1]{d_{\theta}(#1)}

\newcommand{\pa}[2]{\mathrm{Panel}_{#1}^{#2}}

\newcommand{\parah}[1]{P_{#1}}

\newcommand{\para}[2]{P_{#1,#2}}

\newcommand{\bs}[1]{\boldsymbol{#1}}

\newcommand{\ol}[1]{\overline{#1}}

\newcommand{\Hstar}{H_{\star}^{\mathrm{sc}}}

\newcommand{\Hbarstar}{\ul{H}_{\star}^{\mathrm{sc}}}

\newcommand{\ul}[1]{\underline{#1}}

\begin{document}

\title[Distinction of the Steinberg representation]{Distinction of the Steinberg representation with respect to a symmetric pair}
	
\author{Chuijia Wang}
\address{Department of Mathematics, The University of Hong Kong, Pokfulam, Hong Kong, China}
\email{wangchuijia@u.nus.edu}
	
\author{Jiandi Zou}
\address{Institute for Advanced Study in Mathematics of Harbin Institute of Technology, Harbin, China}	
\email{idealzjd@gmail.com}
	
	\begin{abstract}
		
		Let $K$ be a non-archimedean local field of residual characteristic $p\neq 2$. Let $G$ be a connected reductive group over $K$,  let $\theta$ be an involution of $G$ over $K$, and let $H$ be the connected component of $\theta$-fixed subgroup of $G$ over $K$. By realizing the Steinberg representation of $G$ as the $G$-space of complex smooth harmonic cochains following the idea of Broussous--Court\`es, we study its space of distinction by $H$ as a finite dimensional complex vector space. We give an upper bound of the dimension, and under certain conditions, we show that the upper bound is sharp by explicitly constructing a basis using the technique of Poincar\'e series. Finally, we apply our general theory to the case where $G$ is a general linear group and $H$ a special orthogonal subgroup, which leads to a complete classification result.
		
	\end{abstract}

	\maketitle
	\tableofcontents
	
	\section{Introduction}
	
	\subsection{Background}
	
	Let $K$ be a non-archimedean local field of residual characteristic $p$, let $\mfo_K$ be the ring of integers of $K$, let $\mfp_K$ be the maximal ideal of $\mfo_K$, and let $\bs{k}$ be the residue field of $K$. Let $G$ be a connected reductive group over $K$, let $\theta$ be an involution of $G$ over $K$, and let $H$ be the connected component of the $\theta$-fixed subgroup $G^{\theta}$ of $G$ over $K$. In this article, by abuse of notation, for an algebraic group $G$ over $K$, we also use $G$ to represent its group of $K$ rational points if there is no ambiguity.
	
	Let $\mrst_G$ be the complex Steinberg representation of $G$, and let $\chi$ be a character of $H$ with some possible technical conditions. The main goal of this article is to study the space of distinction of the Steinberg representation
	$$\mrhom_{H}(\mrst_G,\chi)$$
	and its dimension as a $\mbc$-vector space. When the above space is non-zero, $\mrst_G$ is called $(H,\chi)$-distinguished. We remark that the above dimension is known to be finite in even more general settings \cite{delorme2010constant}*{Theorem 4.5}. If $\chi$ is unitary, the dimension is exactly the multiplicity of $\mrst_G$ in the decomposition of the unitary $G$-representation of $\chi$-twisted $L^2$-space $L^2(H\backslash G,\chi)$. Here, by abuse of notation our $H\backslash G$ means $H(K)\backslash G(K)$.
	
	One main motivation of considering this particular problem is that, it serves as an excellent test example of the so-called ``relative (local) Langlands program'' in the sense of Sakellaridis--Venkatesh \cite{sakellaridis2017periods}, where the authors considered (split) spherical varieties over local fields. Under our settings and assuming $\chi=1$, one would like to understand the decomposition of the unitary representations $L^2(H\backslash G)$ as a direct integral of irreducible unitary representations\footnote{Indeed in \cite{sakellaridis2017periods} Sakellaridis--Venkatesh considered the set $(H\backslash G)(K)$ of $K$-rational points of the spherical variety $H\backslash G$, and the corresponding $L^2$-space $L^2((H\backslash G)(K))$ instead, which somehow fits their general framework better.}
	\begin{equation}\label{eqL2decomp}
		L^2(H\backslash G)=\int_{\pi\in\mathrm{Unit}(G)}\pi^{\oplus m(\pi)}d\pi,
	\end{equation}
	where $d\pi$ is the Plancherel measure on the unitary dual $\mathrm{Unit}(G)$ of $G$ (\emph{cf.} \cite{sakellaridis2017periods}*{\S 6.1}).
	The decomposition should be understood as a relative version of the so-called ``Plancherel decomposition" due to Harish-Chandra \cite{waldspurger2003formule}. When $H\backslash G$ is strongly tempered, the above measure $d\pi$ is supported on tempered representations. In this case, it is interesting to understand which tempered representations of $G$ occur in the decomposition \eqref{eqL2decomp}, and what the corresponding multiplicity $m(\pi)$ is.
	
    Remarkably, a conjectural solution of describing the decomposition \eqref{eqL2decomp} was proposed in \cite{sakellaridis2017periods}. Motivated by the classical local Langlands correspondence, they introduced, in the spirit of Langlands, the complex dual group $\widehat{H\backslash G}$ of the spherical variety $H\backslash G$ together with a natural morphism $\widehat{\iota}:\widehat{H\backslash G}\times \mrsl_2(\mbc)\rightarrow \widehat{G}$ based on the early work of Gaitsgory--Nadler \cite{GN10}. We remark that in certain cases $\widehat{H\backslash G}$ is not even well formulated (for instance $H\backslash G=\mrso_n\backslash \mrgl_n$). Imposing necessary assumptions, which in particular include that $H\backslash G$ is strongly tempered, the conjecture could be roughly interpreted as follows: 
    \footnote{We warn the readers that our description here is rather coarse. For a rigorous statement, we leave \emph{loc. cit.} for more details.}
	\begin{itemize}
		\item (Galois side) The Arthur parameter $\psi$ of $\pi$ factors through $\widehat{\iota}$ if and only if $m(\pi)\neq 0$.
		\item ($p$-adic group side) There exists a $p$-adic group $G'$ (not even algebraic in certain cases) whose dual group is identified with the dual group of $H\backslash G$, and a functorial lift $\iota_{G'}^G:\mrtemp(G')\rightarrow\mrtemp(G)$, such that $\pi\in \mathrm{Im}(\iota_{G'}^G)$ if and only if $m(\pi)\neq 0$, where $\mrtemp(\cdot)$ denotes the set of equivalence classes of irreducible tempered representations.
	\end{itemize}
	Moreover, one expects to extract the exact value of $m(\pi)$ from the two possible descriptions of $\pi$, which is definitely a more difficult question. Indeed, in \emph{loc. cit.} their conjecture mostly focuses on the cases where $L^2(H\backslash G)$ is multiplicity free. In certain cases, the above conjectural description is fulfilled using local harmonic analysis, for instance \cite{beuzart2021plancherel}, \cite{beuzart2020multiplicities}, \cite{duhamel2019formule}, \cite{lapid2022explicit}, etc. Notably in the $\mru_n\backslash\mrgl_n$-case (\emph{cf.} \cite{beuzart2020multiplicities}), the multiplicity-free result fails. Also, to get enough local inputs, a delicate analysis of global trace formulae is required (\emph{cf.} \cite{feigon2012representations}).
	
	Another general framework is proposed by D. Prasad \cite{prasad2015arelative}. He considered the case where $H\backslash G$ is a quadratic Galois symmetric pair and $\chi$ is a quadratic character of $H$ associated to the quadratic field extension. In this case, the condition of an irreducible representation $\pi$ of $G$ being $(H,\chi)$-distinguished as well as the distinguished dimension are conjecturally described in terms of the L-parameter and L-packet of $\pi$ (\emph{cf.} \cite{prasad2015arelative}*{Conjecture 2}). 
	
	Instead of mentioning recent progress of Prasad's conjecture, we focus on the development of this conjecture for the Steinberg representation. In this case, the Steinberg representation is expected to be always $(H,\chi)$-distinguished with the corresponding multiplicity being one. This conjecture was proved by Broussous--Court\`es and Court\`es \cite{broussous2014distinction}, \cite{courtes2015distinction}, \cite{courtes2017distinction} when $H$ is split over $K$ and $p\neq 2$ by studying the geometry of Bruhat--Tits buildings, and by Beuzart-Plessis \cite{beuzart2018distinguished} in general using harmonic analysis. 
	
	\subsection{Distinction of the Steinberg representation of $\mrgl_n(K)$ by an orthogonal subgroup}
	
	The main motivation (or even ``la raison d'\^etre") of this article is to resolve the problem of distinction of the Steinberg representation of $G=\mrgl_n(K)$ by an orthogonal subgroup. Here in \S 1.2 we state the main results in this special case, and in \S 1.3 we will state more general results that cover this special case and sketch the proofs. 
	
	Before introducing our main results, let us try to persuade the readers why this special case is  interesting. We consider the set of invertible symmetric matrices $X$ in $G$, endowed with a right $G$-action given by $x\cdot g=\,^tgxg$, $g\in G$, $x\in X$. Then we are interested in the Plancherel decomposition of the unitary $G$-representation $L^2(X)$. Given a finite set of representatives $\varepsilon$ of $X/G$, we have\footnote{Here, still by abuse of notation we write $\mro_{n}(\varepsilon)\backslash G=\mro_{n}(\varepsilon)(K)\backslash G(K)$. A more uniform version is that for any orthogonal subgroup $\mro_{n}(\varepsilon)$ of $G$, if we consider the $K$ rational points of the symmetric space $\mro_{n}(\varepsilon)\backslash G$, then we have $X\cong(\mro_{n}(\varepsilon)\backslash G)(K)$ as $G(K)$-spaces. This has already been explained in the introduction of \cite{hakim2013tame}.}
	$$X\cong\sqcup_{\varepsilon}\mro_{n}(\varepsilon)\backslash G\quad\text{and}\quad L^2(X)\cong\oplus_{\varepsilon}L^{2}(\mro_{n}(\varepsilon)\backslash G)$$
	as $G$-spaces and $G$-representations,
	where $\mro_{n}(\varepsilon)$ (resp. $\mrso_{n}(\varepsilon)$) denotes the orthogonal (resp. special orthogonal) subgroup of $G$ with respect to the orthogonal involution $\tau_{\varepsilon}(g)=\varepsilon^{-1}\,^tg^{-1}\varepsilon$, $g\in G$. Since each $\mro_{n}(\varepsilon)\backslash G$ is strongly tempered \cite{gurevich2016criterion}*{\S 5.3}, we have the Plancherel decomposition
	$$L^2(X)=\int_{\mrtemp(G)}\pi^{\oplus m(\pi)}d\pi.$$
	This case exceeds the framework of Sakellaridis--Venkatesh we mentioned before, since neither $L^{2}(X)$ nor $L^{2}(\mro_{n}(\varepsilon)\backslash G)$ is multiplicity free, and until now there is no good definition of the dual group of $X$.
	
	We briefly recall the so-called metaplectic correspondence (\emph{cf.} \cite{flicker1986metaplectic}, \cite{zou2022metaplectic}). Let $G'$ be a two-fold Kazhdan--Patterson covering group of $G$. Given a genuine irreducible tempered representation $\pi'$ of $G'$ and an irreducible tempered representation $\pi$ of $G$, we say that $\pi$ is the \emph{metaplectic lift} of $\pi'$ if their Harish-Chandra characters satisfy a related trace formula (\emph{cf.} \cite{zou2022metaplectic}*{(3.2)}). Then, for a fixed $\pi$, we let $(\iota_{G'}^{G})^{-1}(\pi)$ be the set of genuine irreducible tempered representations $\pi'$ of $G'$ that lift to $\pi$. It is a finite set.
	
	We give the following conjecture describing the above decomposition.
	
	\begin{conjecture}\label{conjmetaplecticdistinction}
		
		Under the above settings,
		
		\begin{enumerate}
			\item $\pi\in\mrtemp(G)$ occurs in the Plancherel decomposition if and only if $\pi$ is the metaplectic lift for some $\pi'$.
			\item $m(\pi)=\car{(\iota_{G'}^{G})^{-1}(\pi)}\cdot d_{\text{Wh}}(\pi')^2$, where $\pi'$ is any genuine irreducible tempered representation of $G'$ lifting to $\pi$ and $d_{\text{Wh}}(\pi')$ denotes the Whittaker dimension of $\pi'$.
		\end{enumerate}
		
	\end{conjecture}
	
	Using \cite{zou2022metaplectic}*{Theorem 4.7} (based on Conjecture 4.6 in \emph{ibid.}), the Whittaker dimension of a genuine irreducible representation $\pi'$ of $G'$ could be calculated. We denote by $\mcc^{\infty}(X)$ the space of complex smooth functions on $X$, which is endowed with a $G$-action. Then, numerically, we may verify the above conjecture in the following cases:
	
	For $\pi$ supercuspidal, we have (\emph{cf.} \cite{zou2022supercuspidal}*{Theorem 1.4})
	
	\begin{center}
		
		\begin{tabular}{|c|c|c|c|}
			
			\hline	
			& $d_{\text{Wh}}(\pi')$ & $|(\iota_{G'}^{G})^{-1}(\pi)|$ & $\mathrm{dim}_{\mathbb{C}}\mrhom_{G}(\pi,\mcc^{\infty}(X))$ \\ \hline
			$n$ odd & 1 & 4 & 4\\   \hline
			$n$ even & 2 & 1 & 4 \\ \hline
			
		\end{tabular}
		
	\end{center}
	
	For $\pi$ ``good" tempered principal series (\emph{cf.} \cite{feigon2012representations}*{arguments in Section 6}), 
	
	\begin{center}
		
		\begin{tabular}{|c|c|c|c|}
			
			\hline	
			& $d_{\text{Wh}}(\pi')$ & $|(\iota_{G'}^{G})^{-1}(\pi)|$ & $\mathrm{dim}_{\mathbb{C}}\mrhom_{G}(\pi,\mcc^{\infty}(X))$ \\ \hline
			$n$ odd & $2^{n-1}$ & 4 & $4^{n}$\\   \hline
			$n$ even & $2^{n}$ & 1 & $4^{n}$ \\ \hline
			
		\end{tabular}
		
	\end{center}
	
	Our main goal here is to verify the conjecture for $\pi$ being the Steinberg representation of $G$. More precisely, we would like to verify the red color part of the following chart:
	
	\begin{center}
		
		\begin{tabular}{|c|c|c|c|}
			
			\hline	
			&  $d_{\text{Wh}}(\pi')$ & $|(\iota_{G'}^{G})^{-1}(\pi)|$ & $\mathrm{dim}_{\mathbb{C}}\mrhom_{G}(\pi,\mcc^{\infty}(X))$\\ \hline
			$n$ odd & $(n+1)/2$ & 4 & $\textcolor{red}{(n+1)^{2}}$\\   \hline
			$n$ even & $n+1$ & 1 & $\textcolor{red}{(n+1)^{2}}$ \\ \hline
			
		\end{tabular}
		
	\end{center}
	
	More generally, we have the following theorem. 
	
	\begin{theorem}\label{thmmainfour}
		
	Assume $p\neq 2$.	Let $G=\mrgl_{n}(K)$, $H'=\mro_{n}(\varepsilon)$ and $H=\mrso_{n}(\varepsilon)$. Then
		the dimension of $$\mrhom_{H'}(\mrst_{G},\mbc)\cong \mch(G)^{H'}$$
		is
		$$\begin{cases} (k+1)(k+2)/2 &\text{if}\ n=2k\ \text{and}\ H\ \text{is split};\\k(k+1)/2	& \text{if}\ n=2k\ \text{and}\ H\ \text{is quasi-split but not split};\\(k-1)k/2  & \text{if}\ n=2k\ \text{and}\ H\ \text{is not quasi-split};\\(k+1)(k+2)/2 & \text{if}\ n=2k+1\ \text{and}\ H\ \text{is split};\\k(k+1)/2 & \text{if}\ n=2k+1\ \text{and}\ H\ \text{is not quasi-split}.\end{cases}$$
		The dimension of
		$$\mrhom_{H}(\mrst_{G},\mbc)\cong \mch(G)^{H}$$ is the same as above, except that when $n=2$ and $H$ is split it is $4$. 
		
	\end{theorem} 

    We will explain its proof in \S 1.4 after introducing the main general results in \S 1.3.
	
	Finally, summing over $H'=\mro_n(\varepsilon)$ with $\varepsilon$ ranging over $X/G$, we get 
	
	\begin{corollary}
		
		Assume $p\neq 2$. Then $\mathrm{dim}_{\mathbb{C}}\mrhom_{G}(\mrst_G,\mcc^{\infty}(X))=(n+1)^2$.
		
	\end{corollary}
	
	\begin{proof}
		
		When $n=1$, we have $4=1+1+1+1$; When $n=2$, we have $9=3+1+1+1+1+1+1$; When $n=2k+1\geq 3$, we have $(2k+2)^2=4(k+1)(k+2)/2+4k(k+1)/2$; When $n=2k\geq 4$, we have $(2k+1)^2=(k+1)(k+2)/2+6k(k+1)/2+(k-1)k/2$. Using the classification of $X/G$ in \S \ref{subsectionsymmetricmatrices} and Theorem \ref{thmmainfour}, the proof is finished.
		
	\end{proof}
	
	However, to fully resolve the Conjecture \ref{conjmetaplecticdistinction}, we expect a local or even global trace formulae comparison to be launched, which is currently beyond our scope.
	
	\subsection{Main general results}
	
	With general motivations and the above special case in mind, we may introduce our main results in general. We keep our notations as above and further assume that $p\neq 2$. In the main part of the article, this assumption is valid from Section \ref{sectioninvolution} until the end.

    Before moving on, we briefly explain the importance of the $p\neq 2$ assumption, which will be essentially used on:

    \begin{itemize}

    \item showing that the first group cohomology of an involution $\theta$ on a pro-$p$-group is trivial;

    \item the existence of a $\theta$-stable apartment containing a fixed chamber (\emph{cf.} Proposition \ref{propAcontainC}). 

    \item the classification result in Section \ref{sectionrankone}.

    \end{itemize}
	
	Our starting point is the series of works  \cite{broussous2014distinction}, \cite{courtes2015distinction}, \cite{courtes2017distinction} mentioned above, where the authors describe the Steinberg representation of $G$ in terms of the Bruhat--Tits theory. More precisely, let $\mcb(G)$ be the reduced Bruhat--Tits building of $G$ as a poly-simplicial complex, and $\ch{}{}(G)$ the set of chambers of $\mcb(G)$. Consider the $\mbc$-vector space of harmonic cochains 
	$$\mch(G):=\{f:\ch{}{}(G)\rightarrow\mbc\mid f\ \text{is harmonic}\},$$
	where by ``harmonic" we mean that for any panel (i.e. a codimension 1 facet) $D$ in $\mcb(G)$, the sum of the value of $f$ on those chambers adjacent to $D$ is zero. The $G$-action on $\mcb(G)$ induces a $G$-action on $\mch(G)$ (with a twist by a quadratic character $\epsilon_G$ associated to the ``orientation" of $\mcb(G)$). Then, the smooth part of $\mch(G)$ is exactly the Steinberg representation $\mrst_G$, which is a known result due to Borel--Serre. See Section \ref{sectionsteinberg} for more details.
	
	Then, we would like to study the space
	$$\mrhom_{H}(\mrst_{G},\chi)\cong \mch(G)^{(H,\chi)}$$
	and its dimension. Notice that the involution $\theta$ induces an involution on $\mcb(G)$. Let $\Ft(G)$ be the set of $\theta$-stable facets in $\mcb(G)$. In particular, a $\theta$-stable facet $F$ is called \emph{maximal} if its $\theta$-invariant part $F^{\theta}$ is a chamber of the $\theta$-invariant affine subspace $\mca^{\theta}$ of a certain $\theta$-stable apartment $\mca$ of $\mcb(G)$. We denote by $\Ftmax(G)$ the subset of $\Ft(G)$ of maximal $\theta$-stable facets. By definition, both $\Ft(G)$ and $\Ftmax(G)$ are invariant under $H$-action. Moreover, we may show that both of them have finitely many $H$-orbits.
	
	Our first result is summarized as follows:
	
	\begin{theorem}[\emph{cf.} Theorem \ref{thmupperbound}]\label{thmmainone}
		
		Let $\chi$ be a character of $H$ satisfying Assumption \ref{assumpchi}. Then the dimension of
		$$\mrhom_{H}(\mrst_{G},\chi)\cong \mch(G)^{(H,\chi)}$$ is bounded by $$\car{\Ftmax(G)/H}$$
		
	\end{theorem}

	Assumption \ref{assumpchi} is a mild condition on the character $\chi$, which is satisfied in most natural cases (including the trivial character and Prasad's character in the Galois case). Also, in Example \ref{examplechi} we see that if this assumption is not satisfied, the original problem of distinction could be ``bad" enough.
	
	We sketch the proof of the above theorem. Given a chamber $C\in\ch{}{}(G)$, we may always find a $\theta$-stable apartment $\mca$ that contains $C$ (\emph{cf.} Proposition \ref{propAcontainC}). We say that
	\begin{itemize}
		\item $C$ is of \emph{$\theta$-rank} $r$ if the dimension of the $\theta$-invariant affine subspace $\mca^{\theta}$ of $\mca$ is $r$.
		\item $C$ is of \emph{$\theta$-distance} $d$ if the combinatorial distance between $C$ and $\mca^{\theta}$ is $d$. In particular, $C$ is of \emph{$\theta$-distance} $0$ if the $\theta$-invariant part of its unique maximal $\theta$-stable facet $F$, denoted by $F^\theta$, is a chamber of $\mca^{\theta}$.
		
	\end{itemize}
	
	Moreover, the above $\mca$ containing $C$ is unique up to $G^{\theta}$-conjugation. Thus the above definition is independent of the choice of $\mca$.
	
	We denote by $\ch{}{0}(G)$ the set of chambers of $\theta$-distance 0. Given a maximal $\theta$-stable facet $F$ in $\mcb(G)$, we denote by $\ch{F}{0}(G)$ the set of chambers of $\theta$-distance 0 that admit $F$ as the maximal $\theta$-stable facet. Then by definition we necessarily have
	$$\ch{}{0}(G)=\bigsqcup_{F\in\Ftmax(G)}\ch{F}{0}(G).$$   
	
	Given $f\in \mch(G)^{(H,\chi)}$. Using the $(H,\chi)$-equivariance as well as the harmonicity, we may show that
	\begin{itemize}
		\item (\emph{cf.} Proposition \ref{propreddist0}) the value of $f$ on $\ch{}{}(G)$ is determined by its value on $\ch{}{0}(G)$.
		\item (\emph{cf.} Proposition \ref{propJWBgallerylocal}) Given a maximal $\theta$-stable facet $F$, the value of $f$ on the whole $\ch{F}{0}(G)$ is determined by its value on any fixed chamber in $\ch{F}{0}(G)$, as well as its value on those chambers of $\theta$-rank greater than the dimension of $F^{\theta}$.
		
	\end{itemize} 
	
	The first claim reduces to studying the $H$-orbits of the set of chambers in $\mcb(G)$ adjacent to any given skew panel $D$ (\emph{cf.} Proposition \ref{propskewpair} and Definition \ref{defclasspanel}). 
	
	For the second claim, we take the integral model $\mcg_F$ of $G$ associated to $F$, and the maximal reductive quotient $\ul{G}_F$ of the special fiber of $\mcg_F$ as a reductive group over $\bs{k}$. Since $F$ is $\theta$-stable, both $\mcg_F$ and $\ul{G}_F$ are endowed with a $\theta$-action as well, and we have a canonical bijection
	$$\ch{F}{0}(G)\leftrightarrow\ch{-}{\theta}(\ul{G}_F),$$
	where $\ch{-}{\theta}(\ul{G}_F)$ denotes the set of Borel subgroups $\ul{B}$ of $\ul{G}_F$ that are $\theta$-split (i.e. $\theta(\ul{B})\cap \ul{B}$ is a maximal torus of $\ul{G}_F$). Let $\ul{H}_{F}$ be the connected component of the $\theta$-fixed subgroup $\ul{G}_F^\theta$ of $\ul{G}_F$. We first assume $\chi=1$ to simplify our discussion. Then, $\ch{-}{\theta}(\ul{G}_F)/\ul{H}_{F}$ has a graph structure denoted by $\gamma(\ul{G}_F,\theta)$, where its vertices are $\ul{H}_{F}$-orbits of $\ch{-}{\theta}(\ul{G}_F)$ and its edges are $\ul{H}_{F}$-orbits of panels between chambers in $\ch{-}{\theta}(\ul{G}_F)$. Indeed, from our rank one calculation in Section \ref{sectionrankone}, there are at most two $\ul{H}_{F}$-orbits of chambers in $\ch{-}{\theta}(\ul{G}_F)$ adjacent to a fixed panel. Then each ``edge" links at most two ``vertices", which justifies the definition of a graph. Moreover, the second claim is somehow reduced to showing that $\gamma(\ul{G}_F,\theta)$ is connected (\emph{cf.} Proposition \ref{propBGHconnnected}), whose proof finally boils down to the rank one case and explicit calculations in Section \ref{sectionrankone}. In general, the above argument still works if we assume Assumption \ref{assumpchi} on $\chi$.
	
	By induction on $\mrdim(F^\theta)$, the above two claims are enough to show Theorem \ref{thmmainone}.
	
	We also have a refined version of Theorem \ref{thmmainone}. We say that two $\theta$-stable apartments are \emph{$\theta$-equivalent} if their associated $\theta$-invariant affine subspaces are $H$-conjugate. Given a $\theta$-stable apartment $\mca$, we denote by $[\mca]_{H,\sim}$ the $\theta$-equivalence class of $\mca$. Indeed, there are only finitely many $H$-conjugacy classes of $\theta$-stable apartments, as well as finitely many $\theta$-equivalence classes.
	
	Given an $H$-conjugacy class of a maximal $\theta$-stable facet $F$, let $\mca$ be a $\theta$-stable apartment such that $F^\theta$ is a chamber of $\mca^\theta$. Assuming Assumption \ref{assumpZHZG}, such $\mca$ is unique up to $\theta$-equivalence. Then, we have the following bijection
	\begin{equation}\label{eqFthetamaxHbij}
		\Ftmax(G)/H\rightarrow \bigsqcup_{[\mca]_{H,\sim}}\ch{}{}(\mca^{\theta})/H,\quad [F]_{H}\mapsto [F^{\theta}]_{H},
	\end{equation}
	where $\ch{}{}(\mca^{\theta})$ denotes the set of chambers of $\mca^{\theta}$, and $[F]_{H}$ and  $[F^{\theta}]_{H}$ denote the $H$-conjugacy class of $F$ and $F^{\theta}$ respectively.
	Now we define a graph $\Gamma(G,\theta)$ as a disjoint union of subgraphs $\Gamma(G,\theta,[\mca]_{H,\sim})$ with $[\mca]_{H,\sim}$ ranging over $\theta$-equivalence classes of $\theta$-stable apartments of $\mcb(G)$. The set of vertices of $\Gamma(G,\theta,[\mca]_{H,\sim})$ is defined to be $\ch{}{}(\mca^{\theta})/H$, and the set of (single or double) edges of $\Gamma(G,\theta,[\mca]_{H,\sim})$ is defined as in \S \ref{subsectiongraphGammaGAtheta}. 
	
	Moreover, a vertex $[F^\theta]_H$ is called \emph{effective} if there exists a non-zero $(H,\chi)$-equivariant harmonic cochain on  the set of chambers having $F$ as a facet denoted by $\ch{F}{}(G)$, that is supported on $\ch{F}{0}(G)$. From the proof of Theorem \ref{thmmainone}, such a function is unique up to a scalar. To provide another perspective,  when $\chi=1$ the effectiveness of $F$ implies the weak  effectiveness of $F$, which somehow means that the graph $\gamma(\ul{G}_F,\theta)$ we discussed above is bipartite. Finally,
	we call a connected component of $\Gamma(G,\theta)$ \emph{effective} if every vertex within is effective. 
	
	Then, we have the following refined version of Theorem \ref{thmmainone}, whose proof is also similar.
	
	\begin{theorem}[\emph{cf.} Theorem \ref{thmupperboundrefine}]\label{thmmaintwo}
		
		Under the Assumption \ref{assumpchi} and Assumption \ref{assumpZHZG}, the dimension of
		$$\mrhom_{H}(\mrst_{G},\chi)\cong \mch(G)^{(H,\chi)}$$ is bounded by the number of effective connected components of $\Gamma(G,\theta)$ that do not have any double edge.
		
	\end{theorem}

    We remark that the only place we need the Assumption \ref{assumpZHZG} is that it guarantees the bijection \eqref{eqFthetamaxHbij}. However, this assumption could possibly be removed in specific cases, see Remark \ref{remassumZHZG}. Also, Assumption \ref{assumpZHZG} itself is mild, which is true in many standard cases.
	
	We expect that the upper bound given in Theorem \ref{thmmaintwo} is sharp, since to our knowledge it is in accordance with all the known examples. Thus for specific examples, it somehow gives an effective strategy of determining a reasonable upper bound of the distinguished dimension, since we may classify all the $\theta$-equivalence classes $[\mca]_{H,\sim}$, as well as depict the related graph $\Gamma(G,\theta)=\sqcup_{[\mca]_{H,\sim}}\Gamma(G,\theta,[\mca]_{H,\sim})$. 
	
	Finally, under certain conditions one would like to show the above upper bound to be sharp. We need to construct a finite set of non-zero $(H,\chi)$-equivariant linear forms on the smooth subspace $\mch^\infty(G)$ of $\mch(G)$ as a basis of $\mrhom_{H}(\mrst_{G},\chi)$. We use the technique of Poincar\'e series as in \cite{courtes2017distinction}. 
	
	In general, we need to consider a character $\chi$ satisfying Assumption \ref{assumpchi} and Assumption \ref{assumpchi'} that are mild conditions  being satisfied in most standard cases. But to simplify our discussion, we assume both $\chi$ and $\epsilon_G\rest_H$ to be trivial characters in the following paragraph. 
	
	Given an effective maximal $\theta$-stable facet $F$, by definition we have a non-zero $(H,\chi)$-equivariant harmonic cochain $\phi_F$ on $\ch{F}{}(G)$ supported on $\ch{F}{0}(G)$. We define a linear form $\lambda_{F}\in\mrhom_{H}(\mrst_{G},\mbc)$ by
	$$\lambda_{F}(f)=\sum_{C\in\ch{F}{0}(G)}\phi_{F}(C)\sum_{C'\in\mco_{C}}f(C'),\quad f\in\mch^\infty(G),$$
	where $\mco_{C}$ denotes the $H$-orbit of $C$. A priori, it is not clear if the above infinite sum is absolutely convergent. So we need Assumption \ref{assumpabsconverg} to guarantee it. Then from our construction, $\lambda_F$ is an $H$-invariant linear form. 
	
    In many standard cases, Assumption \ref{assumpabsconverg} is satisfied and could be verified directly, see \cite{gurevich2016criterion}. On the other hand, since we have good reasons to focus on strongly tempered spherical varieties as in \S 1.1, this assumption is not outrageous at all.
	
	However, it is not clear that $\lambda_F$ is non-zero. To show this, our strategy is to find a harmonic cochain $f_F\in\mch^\infty(G)$ (not necessarily $H$-invariant) that extends $\phi_F$. Then we would like to show that $\lambda_F(f_F)$ is non-zero by showing that it is indeed a Poincar\'e series with variables of absolute value smaller than $1$. To make the above idea work, one key assumption is that any poly-simplicial complex $\mca^\theta$ as above having $F^\theta$ as a chamber must be an affine Coxeter complex realized by the $H$-action (\emph{cf.} Assumption \ref{assumpaffinecoxeter}). Another key assumption is Assumption \ref{assumpnodoubleedge}, which somehow says that the graph $\Gamma(G,\theta)$  has no double edge (loop). 
	
	Then, we have the following theorem.
	
	\begin{theorem}[\emph{cf.} Theorem \ref{thmdistdimcal}]\label{thmmainthree}
		Assume Assumption \ref{assumpZHZG}, \ref{assumpabsconverg}, \ref{assumpchi}, \ref{assumpchi'}. Assume Assumption  \ref{assumpaffinecoxeter}, \ref{assumpnodoubleedge} for each $\theta$-stable apartment $\mca$ with the corresponding subgraph $\Gamma(G,\theta,[\mca]_{H,\sim})$ having an effective connected component, which guarantees that $\Gamma(G,\theta,[\mca]_{H,\sim})$ has a single effective vertex and no double edge. Then the dimension of
		$$\mrhom_{H}(\mrst_{G},\chi)\cong \mch(G)^{(H,\chi)}$$
		equals the number of $\theta$-equivalence classes $[\mca]_{H,\sim}$ that are effective.
		
	\end{theorem}

    Unlike the former assumptions as we have already explained above, a priori  Assumption \ref{assumpaffinecoxeter} is not expected to be satisfied in general. But in practice, many interesting cases happen to be so. 
    
    As for Assumption \ref{assumpnodoubleedge}, since it is a necessary condition of each effective connected component $\Gamma(G,\theta,[\mca_{H,\sim}])$ (consisting of a single vertex in this case) contributing to the distinction (\emph{cf.} Theorem \ref{thmmaintwo}), thus it must be assumed.
	
	We briefly explain our  strategy. For each effective $\theta$-equivalence class $[\mca_i]_{H,\sim}$, $i=1,\dots,k$, we let $F_i$ be an effective maximal $\theta$-stable facet, such that $F_i^\theta$ is a chamber of $\mca_i^\theta$. Then, combining with Theorem \ref{thmmaintwo} we only need to show that $\lambda_{F_1},\dots,\lambda_{F_k}$ are linearly independent. 
	
	We may without loss of generality assume that $\mrdim\mca_1^\theta\geq\mrdim\mca_2^\theta\geq\dots\geq\mrdim\mca_k^\theta$, and we may construct $\phi_{F_{i}}$, $\lambda_{F_{i}}$ and $f_{F_{i}}$ as before. It requires to show that the matrix $(\lambda_{F_{i}}(f_{F_{j}}))_{1\leq i,j\leq k}$ is lower triangular and invertible. To show that it is lower triangular, we need to prove $\lambda_{F_{i}}(f_{F_{j}})=0$ for $i<j$, which follows from Corollary \ref{corlambdaFfF'}.
	To show that $\lambda_{F_{i}}(f_{F_{i}})\neq 0$ for each $i$, we use the technique of Poincar\'e series explained as above.
	
	Since we are trying to consider the most general case, the argument in Section \ref{sectionPoincare} becomes quite technical.
	
	Finally, let us remark that to show that Theorem \ref{thmmaintwo} gives the sharp bound in general, one needs to construct non-zero $(H,\chi)$-equivariant  linear forms without the Assumption \ref{assumpaffinecoxeter}. This is not impossible, since it has been done by Court\`es in a rather ad-hoc way in the type $A_{2n}$ split Galois case (\emph{cf.} \cite{courtes2017distinction}*{\S 7.2}). So we expect that our strategy above serves as an effective model of solving specific problems of distinction of the Steinberg representation with respect to a symmetric pair, even if in some cases the key Assumption \ref{assumpaffinecoxeter} is not satisfied.
	
	\subsection{Proof of Theorem \ref{thmmainfour}}
	
	Following the above notations, we briefly explain the proof of Theorem \ref{thmmainfour}. Our strategy here is to verify the assumptions in Theorem \ref{thmmainthree}, and to calculate the cardinality of the set of $\theta$-equivalence classes as our desired distinguished dimension. More precisely, in this case we classify
	\begin{itemize}
		
		\item all the $G$-orbits of invertible symmetric matrices and the corresponding orthogonal subgroups.
		
		\item all the $H$-orbits of $\theta$-stable apartments as well as $\theta$-equivalence classes. In particular,  $\mca^{\theta}$ is an affine Coxeter complex of type $CB_r$ for a $\theta$-stable apartment $\mca$, where $r=\mrdim(\mca^{\theta})$  and $n\geq 3$.
		
		\item all the $H$-orbits of $\Ft(G)$ and $\Ftmax(G)$. In particular, every facet in $\Ftmax(G)$ is effective.
		
	\end{itemize}
	
	Our calculation is quite explicit, which is interesting in its own right. We remark that the $n=2$ case is a bit different, since Assumption \ref{assumpZHZG} is not satisfied. Still, we may do it by hand using our previous strategy in proving Theorem \ref{thmmainthree}.
	
	Finally, the above strategy works for  other symmetric pairs $G/H$ and suitable characters $\chi$ of $H$. The readers are welcome to take their favorite examples and do a similar calculation as above.
	
	\subsection{Structure}
	
	We sketch the structure of this article. 
	
	Section \ref{apartment}-\ref{sectionsteinberg} are preliminaries. In Section \ref{apartment}, we recall the theory of Euclidean reflection complexes, with an emphasis on the vectorial and affine apartments. In Section \ref{building}, we recall the theory of vectorial and Bruhat--Tits buildings. In Section \ref{sectioninvolution}, we discuss basic results for a Bruhat--Tits building $\mcb(G)$ when $G$ is endowed with an involution $\theta$. In Section \ref{sectionsteinberg}, we recall the definition and properties of the Steinberg representation $\mrst_G$ using harmonic cochains.
	
	Section \ref{sectionrankone} is independent and technical, where we classify all the involutions $\theta$ of a rank one group $\ul{G}$ over a finite field $\bs{k}$, as well as the $\ul{H}$-orbits of the Borel subgroups of $\ul{G}$, where $\ul{H}=(\ul{G}^\theta)^\circ$. The readers are encouraged to skip the proofs and admit the main results during the first read.
	
	In Section \ref{sectiongeometry}, we prove Theorem \ref{thmmainone} following our previous sketch. This part includes our key ideas of studying the combinatorial geometry of $\ch{}{}(G)/H$.
	
	In Section \ref{sectionHorbitofFmax}, we continue our discussion in Section \ref{sectiongeometry} and prove our refined Theorem \ref{thmmaintwo}.
	
	In Section \ref{sectionPoincare}, we prove Theorem \ref{thmmainthree}. As mentioned before, the core is to calculate the $(H,\chi)$-equivariant linear form $\lambda_F$ evaluated on the test function $f_{F}$ for an effective maximal $\theta$-stable facet $F$, which turns out to be a Poincar\'e series with variables of absolute value smaller than 1. Thus $\lambda_F(f_F)$ is non-zero. We pave the way for that by developing various technical results.
	
	In Section \ref{sectionorthogonalgroup}, we consider the distinction of the Steinberg representation of
	$G=\mrgl_{n}(K)$ by an orthogonal (resp. special orthogonal) subgroup $H'$ (resp. $H$). In the end, we prove Theorem \ref{thmmainfour}.

    We have  created a list of symbols for the convenience of the readers.
	
	\subsection{Acknowledgements}
	
    Both authors were partially supported by the Israel Science Foundation (grant No. 737/20). Also, the first author was partially supported by Hong Kong RGC grant GRF 14307521, the Lee Hysan Foundation and the Research Grants Council of the Hong Kong Special Administrative Region, China (Project No. 17305223, 17308324), and the second author was partially supported  by the project PAT4832423 of the Austrian Science Fund (FWF).

    This work was started when both authors were postdocs at Technion hosted by Maxim Gurevich. They would like thank him for his kind support as well as many enlightening discussions. 
    The first author would like to thank Jiu-Kang Yu for helpful discussions. He also wants to express his gratitude to Tasho Kaletha for sharing the preliminary draft of his wonderful book on the Bruhat--Tits theory. The second author would like to thank Angelot Behajaina, Paul Broussous, Alberto Minguez and Vincent S\'echerre for useful discussions and comments. Both authors want to thank Dipendra Prasad for his valuable comments and feedback. They also want to thank the anonymous referee for enlightening suggestions, which greatly improve the exposition of this article.

    Part of the work was done when the first author was visiting Kyoto University in the spring of 2024 and the second author was visiting Zhejiang University  and BIMSA  in the winter of 2023. The first author would like to thank Masao Oi and Atsushi Ichino and the second author would like to thank Fan Gao and Taiwang Deng for their hospitality.

    Finally, the authors would like to dedicate this work to Fran\c cois Court\`es to his memory. Needless to say, our work was deeply influenced  by his work \cite{courtes2017distinction}. It is a pity that we could not have the chance to discuss with him.
	
	\section{Vectorial and affine apartments attached to reductive groups}\label{apartment}
	
	In this section, we will first introduce the general concept of Euclidean reflection groups and their associated Coxeter complex. We will then provide explicit examples and models for the abstract theory, which originates from the theory of reductive groups. In particular, the finite case is available for all fields but the affine case concerns only non-archimedean local fields. 
	
	For an algebraic group $G$ defined over a field $K_0$, we will abuse the notation $G$ such that it both stands for the algebraic group and its group of $K_0$ rational points if there is no ambiguity (if there is ambiguity, we use $G(K_0)$ to stand for its group of $K_0$ rational points).

	\subsection{Euclidean reflection groups}\label{euclideanref}
	
	In this part, we consider Euclidean reflection groups and the related Coxeter complexes. We refer to \cite{bourbaki1968lie}, \cite{Mac72} and \cite[Section 10]{abramenko2008buildings} for the basic settings. 
	
	Let $V$ be an Euclidean space of dimension $n$, i.e., an $n$-dimensional $\mbr$-vector space with an inner product, and $\mca$ an affine space whose associated vector space is $V$, i.e, $\mca$ is a $V$-torsor. For any $x\in \mca,v\in V$, the $v$ action on $x$ is denoted by $x+v$. For any $x,y\in \mca$, we write $y-x$ for the unique $v\in V$ such that $y=x+v$.\index{$\mca$}
	
	A map $f:\mca\rightarrow \mca'$ is called \emph{affine} if there is an $\mbr$-linear map $Df\in \mrhom_\mbr(V,V')$ \index{$Df$} such that 
	$$f(x+v)=Df(v)+f(x),\quad \forall x\in \mca.$$
	We call $Df$ the \emph{linear part} of $f$. In particular, an affine map $\mca\rightarrow \mbr$ is called an \emph{affine functional}. Let $\mca^*$ denote the set of affine functionals on an affine space $\mca$ (not to confuse with the notation of an affine space with the associated vector space $V^*$). Observe that $\mca^*$ has a structure of an $\mbr$-vector space with the pointwise addition and scalar multiplication. For any $f\in \mca^*$, we have $Df\in V^*$. Hence we have an exact sequence of $\mbr$-vector spaces (without a canonical splitting)
	\begin{equation}
		0\rightarrow \mbr\rightarrow \mca^*\rightarrow V^*\rightarrow 0,
	\end{equation}
	where $\mbr$ stands for the space of constant affine functionals.

	Let $\mrAff(\mca)=\mrgl(V)\ltimes V$ \index{$\mrAff(\mca)$} denote the group of \emph{affine transformations} of $\mca$ with a natural embedding $V\hookrightarrow \mrAff(\mca)$ given by translations and a natural projection $\mrAff(\mca)\rightarrow \mrgl(V)$ sending $f$ to $Df$. For any subgroup $W$ of $\mrAff(V)$, the \emph{linear part} $W_v$ of $W$ is the image of $W$ in $\mrgl(V)$ under the projection. A \emph{hyperplane} inside $\mca$ is an affine subspace of $\mca$ of codimension one. A \emph{reflection} of $V$ or $\mca$ is an order two isometry such that the fixed point set is a hyperplane.
	
	\begin{definition}
		A group $W$ of affine isometries of $\mca$ is called an \emph{Euclidean reflection group}, if there exists a locally finite set $\mfH$ of hyperplanes in the affine space $\mca$, such that \begin{itemize}
			\item $\mfH$ is stable under the reflection $s_{\mfh}$ along $\mfh$ for any $\mfh\in\mfH$;
			\item $W$ is generated by $s_{\mfh}, \mfh\in\mfH$.
		\end{itemize} 
		Here, $\mfH$ is called the set of \emph{walls} of $\mca$. \index{$\mfH$}\index{$\mfh$}
	\end{definition}
	
	We remark that Euclidean reflection groups are special cases of Coxeter groups. In particular, we will only be interested in the case where the related Euclidean reflection group is either a finite or an affine Weyl group attached to a reductive group over a particular field.
	
	Depending on whether $W$ is finite or not, those hyperplanes $\mfh$ in $\mfH$ divide $\mca$ into partitions of either open poly-cones or open bounded poly-simplices, that are called \emph{facets} of $\mca$. For each hyperplane $\mfh\in \mfH$, it divides the apartment to a union of $\mfh$ with two \emph{open half-spaces} $\mfh^{>0},\mfh^{<0}$. The corresponding \emph{closed half-spaces} are defined by $\mfh^{\geq0}=\mfh^{>0}\cup \mfh,\mfh^{\leq0}=\mfh^{<0}\cup \mfh$. For a bounded set $\Omega$ of $\mca$, the \emph{closure} $\overline{\Omega}$ of $\Omega$ is defined as the intersection of closed half-spaces that contains $\Omega$. Those facets of dimension $n$ (resp. $n-1$) in $\mca$ are called \emph{chambers} (resp. \emph{panels}) of $\mca$. For a fixed chamber $C$, a hyperplane $\mfh\in\mfH$ is called a \emph{wall} of $C$ if $\mfh\cap \overline{C}$ is of codimension $1$ in $\overline{C}$.
	
	\begin{remark}
		Fix a chamber $C$. Let $S_C$ denote the set of reflections determined by 
		walls of $C$. The locally finite assumption implies the finiteness of $S_C$ and the finiteness of 
		$W_v$.
	\end{remark}

	\begin{definition}\cite[Proposition 1.4.5]{Rou23}
		A \emph{gallery} of length $d$ in $\mca$ is a finite sequence of chambers $C_{0},C_{1},\dots,C_{d}$, such that $C_{i-1}$ and $C_{i}$ share a common panel $D_{i}$ for $i=1,2,\dots,d$.
		
		For two facets $F,F'$ in $\mca$, we  define their \emph{(combinatorial) distance} to be 
		\begin{itemize}
			\item either the length of a minimal gallery connecting $F$ and $F'$;
			\item or the number of walls that strictly separate $F$ and $F'$.  
		\end{itemize}
	\end{definition}
	In general, for two bounded sets $F,F'$ of $\mca$, we denote by $\mfH_{\mca}(F,F')$ \index{$\mfH_{\mca}(F,F')$} the set of hyperplanes in $\mca$ that strictly separate $F$ and $F'$. 
	
	We have the following geometric method to find a minimal gallery connecting a facet $F$ with a chamber $C$.
	
	\begin{lemma}\label{lemmaFCdist}
		
		Let $L$ be a line segment starting from a point in $F$ and ending with a point in $C$ that does not meet any facets in $\mca$ of codimension greater than one other than $F$. Then the gallery that intersects $L$ is a minimal gallery connecting $F$ and $C$.
		
	\end{lemma}
	
	\begin{proof}
		
		Let $C_{0},C_{1},\dots,C_{d}$ be the gallery intersecting $L$, such that $F\subset\overline{C_{0}}$ and $C_{d}=C$, and let $D_{1},\dots,D_{d}$ be the corresponding panels. Let $\mfh_{D_{i}}$ be the wall in $\mca$ that contains $D_{i}$. From our assumption on $L$, it is the disjoint union of $F\cap L$, $C_{i}\cap L$ and $D_{i}\cap L$.
		
		We claim that each $\mfh_{D_{i}}$ strictly separates $F$ and $C$. Indeed, first $F$ cannot be contained in $\mfh_{D_{i}}$, otherwise by convexity the full line segment $L$ must be in $\mfh_{D_{i}}$, contradicting to the fact that the ending point of $L$ is in $C$. Also $F$ and $C$ cannot be on the same side of $\mfh_{D_{i}}$, otherwise by convexity the full line segment $L$ must be on the same side of $\mfh_{D_{i}}$, contradicting to the fact that $L$ intersects $D_{i}$. On the other hands, those walls $\mfh_{D_i}$ must be pairwise different, otherwise by convexity $L$ must lie in one of the wall and cannot intersect $C$.
		
		Thus the gallery above must be a minimal gallery from our second definition of the distance.
		
	\end{proof}
	
	The most common Euclidean reflection groups and their Coxeter complex appears in the theory of root system. 
	
	\vspace{2mm}
	\noindent$\bullet$\ \textbf{The finite case}
	
	Let $V$ be a finite dimensional $\mbr$-vector space such that its dual $V^*$ is equipped with a positive definite bilinear form $(,)$ identifying $V^*$ with $V$. We use the notation $\langle,\rangle$ to denote the natural pairing between $V^*$ and $V$ induced by $(,)$. For any $\alpha \in V^*$, let $\alpha^\vee$ denote  $\frac{2(\cdot,\alpha)}{(\alpha,\alpha)}\in V$, namely, the unique element in $V$ orthogonal to $\ker \alpha$ with $\langle \alpha,\alpha^\vee\rangle=2$.

	An \emph{abstract finite root system} is a finite set of non-zero vectors $\Phi\subset V^*$ \index{$\Phi$} satisfying the usual axioms for root systems.
	\begin{enumerate}
		\item $\Phi$ spans $V^*$ over $\mbr$.
		\item $\Phi$ is closed under reflection $\{r_\alpha\}_{\alpha\in \Phi}$, where $r_\alpha$ denotes the reflection defined by $r_\alpha(x)=x-\langle x,\alpha^\vee\rangle\alpha$ for any $x\in V^*,\alpha\in \Phi$.
		\item For any $\alpha,\beta\in \Phi$, $\langle\beta,\alpha^\vee\rangle\in \mbz$.
	\end{enumerate}
	
	If $\Phi$ further satisfies the condition that the only multiple of $\alpha$ in $\Phi$ is $\pm\alpha$ for any $\alpha\in \Phi$, then $\Phi$ is called a reduced root system. For any $\alpha$ inside a root system $\Phi$, the only possible integral multiples of $\alpha$ in $\Phi$ are $\pm \alpha$ and $\pm2\alpha$ by the third condition. For $\alpha\in \Phi$, we call $\alpha$ \emph{divisible} (resp. \emph{multipliable}) if $\frac{\alpha}{2}\in \Phi$ (resp. $2\alpha\in \Phi$) and \emph{non-divisible} (resp. \emph{non-multipliable}) otherwise. Let $\Phi^\vee$ denote the set of \emph{coroots} $\{\alpha^\vee\}_{\alpha \in \Phi}$. Then $\Phi^\vee\subset V$ defines the \emph{dual root system}.

	For any $\alpha\in \Phi$, we define the corresponding reflection $s_\alpha$ on $V$ by
	$$s_{\alpha}(x)=x-\langle \alpha,x\rangle \alpha^\vee,\quad\forall x\in V.$$\index{$s_{\alpha}$}
	The \emph{finite Weyl group} $W:=W(\Phi)$ is a Coxeter group generated by reflections $\{s_\alpha\}_{\alpha\in \Phi}$. For each $\alpha\in \Phi$, one can define the hyperplane\index{$\mfh_\alpha$} 
	$$\mfh_\alpha:=V^{s_\alpha}=\ker \alpha=\{v\in V\mid\,\langle\alpha,v\rangle=0\}\subset V$$
	Let $\mfH=\{\mfh_\alpha\}_{\alpha\in \Phi}$ \index{$\mfH$} denote the corresponding set of walls, i.e. the set of hyperplanes $\mfh_\alpha$ in $V$.
	
	The \emph{Coxeter complex} associated to $\Phi\subset V^*$ is the poly-conical complex structure on $V$ with the set of facets determined by the division of $\mfH$ on $V$. The Coxeter complex on $V$ enjoys the following nice properties
	
	\begin{itemize}
		\item The chambers are connected components of $V-\cup_{\mfh_\alpha\in\mfH}\mfh_\alpha$.
		\item The boundary $\partial F:=\overline{F}-F$ of a facet $F$ is a union of facets of smaller dimension. 
		\item The vector space $V$ is endowed with an action of the finite Weyl group $W$ which preserves the poly-conical complex structure of $V$. Moreover, the corresponding action on the set of chambers is simply transitive.
	\end{itemize}
	Let $\Delta$\index{$\Delta$} denote a set of simple roots in $\Phi$, i.e., a subset of $\Phi$ consisting of a basis of $V$.
	The related \emph{fundamental chamber} is the poly-cone 
	$$C_0=\cap_{\alpha\in \Delta}\mfh_\alpha^{>0}=\{v\in V\mid\,\langle\alpha,v\rangle>0,\forall \alpha\in \Delta\},$$
	whose closure $\overline{C_0}$ is a fundamental domain for the $W$-action on $V$. There is a bijection between the set of facets contained in $\overline{C_0}$ and the set of non-empty subsets of $\Delta$ by sending $F$ to $\Delta_F:=\{\alpha\in \Delta\mid \langle \alpha,x\rangle=0,x\in F\}$.
	
	Finally, let $\mca$ be an affine space with underlying space $V$. If we fix a point $0\in\mca$ as the origin, then we have a bijection $V\rightarrow \mca,\ v\mapsto v+0$. Using this, the Coxeter complex structure on $V$ are transferred to that on $\mca$, endowing $\mca$ with a Coxeter complex structure with respect to the finite Euclidean reflection group $W$. In particular, the corresponding set of walls $\mfH$ in $\mca$ is finite with each $\mfh\in \mfH$ passing the origin $0$.

	\vspace{2mm}
	\noindent$\bullet$\ \textbf{The affine case}
	
	Let $\mca$ be an affine space whose associated vector space is $V$. An \emph{abstract affine root system} \cite[Definition 1.3.1]{KP23} is a subset $\Psi\subset \mca^*$ \index{$\Psi$} satisfying the following axioms
	
	\begin{enumerate}
		\item $\Psi$ spans the $\mbr$-vector space $\mca^*$ and the set $D\Psi$ of linear parts of $\Psi$ is a finite root system $\Phi\subset V^*$.
		\item The reflection $r_{\psi}$ on $\mca^*$ defined by $r_{\psi}(x):=x-\langle Dx,(D\psi)^\vee\rangle\psi,\ \forall x\in \mca^*$ preserves $\Psi$, where $(D\psi)^\vee \in V$ denotes the coroot of $D\psi$.
		\item For any $\psi_1,\psi_2\in \Psi$, $\langle D\psi_1,(D\psi_2)^\vee\rangle\in \mbz$.
		
		\item For any $\alpha\in \Phi$, the set $\{\psi\in \Psi\mid\,D\psi=\alpha\}$ does not have an accumulation point.
	\end{enumerate}
	
	For any $\psi\in \Psi$, we also define the corresponding reflection $s_\psi$ on $\mca$ by
	\begin{equation}\label{reflectionona}
		s_\psi(x):=x-\psi(x)(D\psi)^\vee,\quad\forall x\in \mca.\index{$s_\psi$}
	\end{equation}
	The \emph{affine Weyl group} $W_\mraff:=W(\Psi)$ \index{$W_\mraff$} is a Coxeter group generated by reflections $\{s_\psi\}_{\psi\in \Psi}$. 
	There is a canonical projection  \begin{equation}
		\begin{aligned}
			W(\Psi)&\rightarrow W(\Phi)\\
			s_\psi&\mapsto s_{D\psi}
		\end{aligned}
	\end{equation}
	whose kernel consists of translations in $W(\Psi)$. 
	
	In contrast with the finite root system, there is a notion of extended Weyl group for an affine root system. More precisely, define the $\mrAff(\mca)$-action on $\mca^{*}$ by $w\cdot\psi:=\psi\circ w^{-1}, \psi\in \mca^*,w\in \mrAff(\mca)$. Then the \emph{extended affine Weyl group} $\widetilde{W_\mraff}:=\widetilde{W(\Psi)}$ \index{$\widetilde{W_\mraff}$} of $\Psi$ is the subgroup of $\mrAff(\mca)$ whose induced action on $\mca^*$ preserves $\Psi$. In particular, the linear part of $\widetilde{W_\mraff}$ lies in $W(\Phi)$. By definition, $W_\mraff$ embeds into $\widetilde{W_\mraff}$ as a normal subgroup and we have 
	\begin{equation}\label{embedaffine}
		1\rightarrow W_\mraff\hookrightarrow \widetilde{W_\mraff}\rightarrow \widetilde{W_\mraff}/W_\mraff\rightarrow 1.
	\end{equation}
	
	For each $\psi\in \Psi$, we define the hyperplane $$\mfh_\psi:=\mca^{s_\psi}=\{x\in \mca\mid\,\psi(x)=0\}\subset \mca.$$\index{$\mfh_\psi$}
	Let $\mfH_\mraff=\{\mfh_\psi\}_{\psi\in \Psi}$\index{$\mfH_\mraff$} denote the corresponding set of walls, i.e. the set of hyperplanes $\mfh_\psi$ in $\mca$. The Coxeter complex associated to $\Psi\subset\mca^*$ is the poly-simplicial complex structure on $\mca$ with the set of facets determined by the division of $\mfH_\mraff$. The Coxeter complex on $\mca$ enjoys the following nice properties
	
	\begin{itemize}
		
		\item The chambers are connected components of $\mca-\cup_{\mfh_\psi\in\mfH_\mraff}\mfh_\psi$.
		\item The boundary $\partial F:=\overline{F}-F$ of a facet $F$ is a union of facets of smaller dimension. 
		\item The affine space $\mca$ is endowed with an action of the affine Weyl group $W_{\mraff}$ which preserves the poly-simplicial complex structure of $\mca$. Moreover, the corresponding action on the set of chambers is simply transitive.
	\end{itemize}

	Let $\Delta^\mraff$ \index{$\Delta^\mraff$} denote a set of simple affine roots in $\Psi$, i.e., a subset of non-divisible elements in $\Psi$ consisting of a basis of $\mca^*$. Similarly, in the affine case, the related fundamental chamber (usually called the fundamental alcove in this case) is the bounded open domain 
	$$C_0=\cap_{\psi\in \Delta^\mraff}\mfh_\psi^{>0}=\{x\in \mca\mid\,\psi(x)>0,\forall \psi\in \Delta^\mraff\},$$
	whose closure is a fundamental domain for the $W_\mraff$-action on $\mca$. There is a bijection between the set of facets contained in $\overline{C_0}$ and the set of non-empty subsets of $\Delta^\mraff$ by sending $F$ to $\Delta^\mraff_F:=\{\psi\in \Delta^\mraff\mid \psi(x)=0,x\in F\}$.

	Since the induced action of $\widetilde{W_\mraff}$ on $\mca^*$ preserves $\Psi$, it also preserves $\mfH_\mraff$ and permutes the set of chambers. The corresponding action on the set of chambers is transitive but not free. For any choice of chamber $C$ in $\mca$, the stabilizer $\mrstab_{\widetilde{W_\mraff}}(C)$ is a finite group that provides a splitting of \eqref{embedaffine}, namely, we have $\widetilde{W_\mraff}=\mrstab_{\widetilde{W_\mraff}}(C)\ltimes W_\mraff $.

	\subsection{The vectorial apartment attached to a reductive group over a field}\label{vectorialapart}
	Let $K_0$ be an arbitrary field and $\overline{K_0}$ the algebraic closure of $K_0$. Let $G$ be a reductive group over $K_0$ and $G^{\mrder}$ the derived subgroup of $G$. Let $S$ be a maximal $K_0$-split torus of $G$ and $S^{\mrder}:=S\cap G^{\mrder}$ the maximal $K_0$-split torus of $G^{\mrder}$. Let $Z(G)$ be the center of $G$ with $A=A(G)$ being the maximal $K_0$-split torus in $Z(G)$. 
	
	Let $X_*(S):=\mrhom_{K_0}(\mbg_m,S)$ denote the co-character lattice associated to $S$. Since $Z(G)(\overline{K_0})\cdot G^{\mrder}(\overline{K_0})=G(\overline{K_0})$, there exists an isogeny $Z(G)(\overline{K_0})\times G^{\mrder}(\overline{K_0})\rightarrow G(\overline{K_0})$ inducing an isogeny $A\times S^{\mrder}\rightarrow S$. Hence we have an isomorphism of $\mbr$-vector spaces $(X_*(S^{\mrder})\otimes _{\mbz}\mbr )\times (X_*(A)\otimes _{\mbz}\mbr) \cong X_*(S)\otimes _{\mbz}\mbr$, i.e., $X_{*}(S^{\mrder})\otimes_{\mbz}\mbr\cong (X_*(S)/X_*(A))\otimes _{\mbz}\mbr\cong X_*(S/A)\otimes _{\mbz}\mbr$.
	
	Let $V$ denote the finite dimensional $\mbr$ vector space $X_*(S/A)\otimes _\mbz\mbr$. We may also realize $X_*(S/A)\otimes _\mbz\mbr$ as an affine space as well with $0$ being the origin, which we denote by $\mca=\mca_v(G,S)$\index{$\mca_v(G,S)$} with the subscript $v$ representing ``vectorial''.
	
	Let $\Phi=\Phi(G,S)$\index{$\Phi(G,S)$} denote the set of non-zero weights appearing in the adjoint action of $S$ on the Lie algebra $\mfg$ of $G$. Then $\Phi\subset V^{*}$ defines the standard finite relative root system associated to $(G,S)$. As we explained in \S \ref{euclideanref}, there exists a Coxeter complex structure on $V$ and $\mca$ defined by the root system $\Phi\subset V^{*}$. 
	
	Let $N$ denote the normalizer $N_{G}(S)$ and $Z$ the centralizer $Z_{G}(S)$, which is known as the reductive anisotropic kernel of $G$, namely, $Z^{\mrder}$ is an anisotropic semisimple group over $K_0$ due to the maximal $K_0$-splitness of $S$. Then, the natural action of $N$ on $S$ induces an action of $N$ on $V$ (resp. $\mca$), which further induces a canonical isomorphism between $W(\Phi)$ and $N/Z$ (\emph{cf.} \cite{BT65}). As a result, the $N$-action on $V$ and $\mca$ is simplicial and transitive on the set of chambers.
	
	For any $\alpha\in \Phi$, we define the \emph{root group} $U_\alpha$ to be the closed subgroup of $G$ such that it is normalized by $S$ and its Lie algebra is a direct sum of weight spaces whose weights are positive integral multiples of $\alpha$. We  write $U^*_\alpha$ for $U_\alpha-\{1\}$. 
	
	For any facet $F$ in $V$ (resp. $\mca$), one attach the following subgroups
	\begin{equation}\label{sphericalparabolic}
		\begin{aligned}
			P_{F}:=\langle S ,\{U_\alpha\}_{\alpha\in \Phi,\langle \alpha,F\rangle\geq 0}\rangle,\quad
			U_{F}:=\langle \{U_\alpha\}_{\alpha\in \Phi,\langle \alpha,F\rangle> 0}\rangle,\quad
			M_{F}:=\langle S, \{U_\alpha\}_{\alpha\in \Phi,\langle \alpha,F\rangle=0 }\rangle.\index{$P_F$}\index{$M_F$}\index{$U_F$}
		\end{aligned}
	\end{equation}
	Indeed, these $P_F$ are parabolic subgroups of $G$ with Levi decomposition $P_F=M_F\ltimes U_F$ such that $M_F$ are Levi subgroups. Indeed, $M_F=Z_G(\mrim(x))$ for any $x\in F\cap X_*(S/A)$.

	\subsection{The affine apartment attached to a reductive group over a non-archimedean local field.}\label{affineapartment}
	
	We adopt the notations in \S \ref{vectorialapart} with $K_0=K$ being a non-archimedean local field. Let $v_K:K\rightarrow \mbr\cup \{\infty\}$ be the non-archimedean valuation on $K$. 
	
	Let $\mca=\mca(G,S)$\index{$\mca(G,S)$} denote the affine apartment with underlying real vector space $V=X_*(S/A)\otimes_\mbz\mbr$ constructed by Bruhat and Tits \cite{bruhat1972groupes}, which we explain as follows.

	\begin{definition}\cite[Definition 6.1.2]{KP23}\label{valuation}
		A \emph{$v_K$-compatible valuation} of the root datum of $(G,S)$ is a family $\phi=\{\phi_\alpha\}_{\alpha\in \Phi}$ of morphisms  $\phi_\alpha:U_\alpha(K)\rightarrow \mbr\cup \{\infty\}$ such that 
		\begin{enumerate}
			\setcounter{enumi}{-1}	
			\item For each $\alpha\in \Phi$, $\phi_\alpha(U_\alpha(K))$ has at least 3 elements.
			\item For each $\alpha$ and $r\in \mbr$, $U_{\alpha,\phi,r}:=\phi_{\alpha}^{-1}([r,\infty))$ is a subgroup of $U_\alpha(K)$ with $U_{\alpha,\phi,\infty}=\{1\}$.
			\item For every $n\in N$ whose image in $W$ is $s_\alpha$, the function $\phi_\alpha(u)-\phi_{-\alpha}(nun^{-1})$ is constant on $U^*_\alpha$.
			\item For any $\alpha,\beta\in \Phi$ with $\beta$ not a negative multiple of $\alpha$ and $i,j\in \mbr$, the commutator subgroup of $U_{\alpha,\phi,i}$ and $U_{\beta,\phi,j}$ is contained in the group generated by $U_{m\alpha+n\beta,\phi,mi+nj}$ for $m,n\in \mbz_{>0}$ with $m\alpha+n\beta\in \Phi$.
			\item If both $\alpha,2\alpha\in \Phi$, then $\phi_{2\alpha}=2\phi_\alpha\rest_{U_{2\alpha}}$. In particular, $U_{\alpha,\phi,r}\cap U_{2\alpha}(K)=U_{2\alpha,\phi,2r}$.
			
			\item For any $\alpha\in\Phi$, $u\in U_\alpha(K)$ and $u',u''\in U_{-\alpha}(K)$ such that $u'uu''$ normalizes $S$, we have $\phi_{-\alpha}(u')=\phi_{\alpha}(u)$.
			
			\item For any $\alpha\in \Phi$ and $z\in Z$, we have $\phi_\alpha(z^{-1}uz)=\phi_\alpha(u)+\langle \alpha,v_Z(z)\rangle$ (\emph{cf.} \eqref{vector}). In particular, for any $s\in S$, we have $\phi_\alpha(s^{-1}us)=\phi_\alpha(u)-v_K(\alpha(s))$.
			
			\item For any $\alpha\in\Phi$, the subsets $U_{\alpha,\phi,r}$, for $r\in\mbr$, are bounded open subsets of $U_{\alpha}(K)$ and form a neighborhood basis of $\{1\}$ in $U_{\alpha}(K)$.
		\end{enumerate}
	\end{definition}
	When $G$ is quasi-split, for any $v\in V$, there exists a family of valuations $\{\phi_{\alpha,v}\}_{\alpha\in \Phi(G,S)}$ such that $\phi_{\alpha,0}:U_{\alpha}(K)\rightarrow \mathbb{R}\cup \{\infty\}$ corresponds to $v_K$ for a fixed choice of Chevalley pinning. Let $K^\mrur$ denote the maximal unramified extension of $K$ inside $\overline{K}$. Constructions of these valuations for general cases follow from the constructions for the quasi-split group $G_{K^\mrur}$ over $K^\mrur$ together with an unramified Galois descent argument. We will not go deep into the general constructions of these valuations and an explicit construction can be found in \cite[\S 6.1; Chapter 9]{KP23}. Instead, we will highlight the following identification
	\[\mca\leftrightarrow \{\text{equipollent } v_K \mbox{-}\text{compatible valuations of the root datum of }(G,S)\},\]
	where two valuations $\phi_x$ and $\phi_y$ of the root datum of $(G, S)$ are called \emph{equipollent} if there exists $v\in V(S)$ such that $\phi_{y,\alpha}(u)=\phi_{x,\alpha}(u)+\langle\alpha,v\rangle$. Here we write $\phi_x=\{\phi_{x,\alpha}\}_{\alpha\in\Phi}$ for the valuation corresponding to $x\in \mca$. Also, we write $U_{\alpha,x,r}=U_{\alpha,\phi_x,r}$ for any $\alpha\in\Phi$.

	For any $\psi\in \mca^*$ with $D\psi=\alpha\in \Phi$, we define the affine root group 
	$U_\psi:=U_{\alpha,x,\psi(x)}$,
	which does not depend on the choice of $x$ due to the equipollence property. We also define $U_{\psi+}:=\cup _{\psi'>\psi}U_{\psi'}$. Then the set $\Psi=\Psi(G,S)\subset \mca^*$ \index{$\Psi(G,S)$} of affine root system of $(G,S)$ is defined by (\emph{cf.} \cite[Definition 6.3.4]{KP23})
	$$\{\psi\in \mca^*\mid\,D(\psi)\in \Phi, \; U_\psi\nsubseteq U_{\psi+}\cdot U_{2\alpha}\}.$$
	By \cite[Proposition 6.3.13]{KP23}, the above construction provides an affine root system in the sense of \S\ref{euclideanref}. As in \S\ref{euclideanref}, let $\mfh_\psi=\{x\in \mca\mid\,\psi(x)=0\}$ denote the hyperplane corresponding to $\psi$. The hyperplane arrangement $\{\mfh_\psi\}_{\psi\in \Psi}$ endows $\mca$ with the structure of a Coxeter complex.
	
	A natural question is whether we have a similar isomorphism to $W(\Phi)\cong N/Z$ in the finite case. It is known that the natural restriction map $X^*(Z)\rightarrow X^*(S)$ is injective \cite[\S1.6]{Pra20} and its image is of finite index. Hence there exists a unique homomorphism 
	\begin{equation}\label{extvector}
		\nu_{Z,\mrext}:Z\rightarrow X_{*}(S)\otimes_{\mbz}\mathbb{R}
	\end{equation}
	such that for any $\chi\in X^{*}(Z)$ and $z\in Z$, we have 
	$$\langle \chi\rest_S, \nu_{Z,\mrext}(z)\rangle=-v_K(\chi (z)).$$
	Composing the projection $X_{*}(S)\otimes_{\mathbb{Z}}\mathbb{R}\rightarrow  X_*(S)/X_*(A)\otimes _{\mbz}\mbr$, we also have a homomorphism
	\begin{equation}\label{vector}
		\nu_{Z}:Z\rightarrow X_{*}(S)/X_*(A)\otimes_{\mathbb{Z}}\mathbb{R}.
	\end{equation} 
	The homomorphism \eqref{vector} defines an action of $Z$ on the apartment by translations
	\begin{equation*}
		\begin{aligned}
			\nu:Z&\rightarrow \mrAff(\mca),\\
			z&\mapsto (a\mapsto a+\nu_{Z}(z)).
		\end{aligned}
	\end{equation*}
	By Tits \cite[\S1.2]{tits1979reductive}, this translation action can be extended to an action of $N:=N_{G}(S)$ on $\mca$, which we still denote by $\nu$
	$$\nu:N\rightarrow \mrAff(\mca),\index{$\nu$}$$
	and the action depends on a choice of base point in $\mca$. Let $W:=N/Z$ denote the finite Weyl group of $S$ and let $Z^b$ denote $\ker \nu_{Z,\mrext}$. Indeed, one has an alternative description of $Z^b$ \cite[Proposition 1.5]{Pra20}, namely, the unique maximal bounded subgroup of $Z$
	\begin{equation}\label{maximalbounded}
		Z^b:=\{z\in Z\mid\,\chi(z)\in \mfo_K^\times,\forall \chi\in X_*(Z)\}.
	\end{equation}
	Hence the map $\nu_{Z,\mrext}$ induces an isomorphism $Z/Z^b\cong \mbz^r$ for $r=\mrdim S$. We also have an exact sequence 
	$$1\rightarrow Z/Z^b\rightarrow N/Z^b\rightarrow W\rightarrow 1.$$
	When $G$ is split, one has $Z= S$ and $Z/Z^b=S(K)/S(\mfo_K)\cong X_*(S)$. In this case, there is an isomorphism from $W\ltimes X_{*}(S)$ to $N/S(\mfo_K)$ sending $(w,x)$ to $\dot{w}x(\varpi_K)$, where $\varpi_K$ is a chosen uniformizer in $K$ and $\dot{w}$ is any lift of $w$ in $N$.
	
	Let $\widetilde{W^{1}}$ denote the image of $N/Z^b$ in $\mrAff(\mca(G,S))$. \index{$\widetilde{W^{1}}$} By \cite[Proposition 6.3.13; \S 6.6]{KP23}, we have inclusions 
	$$W_\mraff\hookrightarrow \widetilde{W^{1}}\hookrightarrow \widetilde{W_\mraff}.$$
	In the case when $G$ is split, the intermediate group $\widetilde{W^1}$ is simply the image of $W\ltimes X_{*}(S)$ in $\mrAff(\mca(G,S))$.

	In general, all these groups $W_\mraff, \widetilde{W^1}$ and $\widetilde{W_{\mraff}}$ act on $\mca^*$ preserving $\Psi$ \cite[Lemma 6.3.12]{KP23}, which means that the corresponding actions on $\mca(G,S)$ are poly-simplicial. All these actions on the set of chambers in $\mca(G,S)$ are transitive, but only the $W_\mraff$-action is simply transitive.
	
	We will also work with the extended apartment $\mca^\mrext(G,S)$\index{$\mca^\mrext(G,S)$}, i.e., the affine space associated to $X_{*}(S)\otimes_{\mbz}\mathbb{R}$. There is an isomorphism of affine spaces $\mca^\mrext(G,S)\cong \mca(G,S)\times (X_*(A)\otimes _{\mbz}\mbr)$, where the second factor is regarded as an affine space associated to the vector space $X_*(A)\otimes _{\mbz}\mbr$ without a preferred base point.

	\subsection{Open bounded subgroups attached to a non-empty subset in an apartment $\mca$}

	We follow the notations in \S\ref{affineapartment}. Let $G$ be a reductive group over a non-archimedean local field $K$ with the residue field $\bs{k}$. Let
	$G^{\mrder}$ denote the derived $K$-subgroup of $G$ and $G^{\mrsc}$ the simply connected cover of $G^{\mrder}$. We have the natural map 
	$$p:G^{\mrsc}\rightarrow G^{\mrder}\rightarrow G.$$
	Let $G^\sharp$\index{$G^\sharp$} denote $p(G^\mrsc(K))$ inside $G(K)$. By \cite[Definition 2.6.23]{KP23} there exists a group $G^0$\index{$G^0$} characterized by the following properties
	\begin{enumerate}
		\item If $G$ is quasi-split, $G^0$ is define to be $G^\sharp \cdot T^0$, where $T=Z_G(S)$ is a maximal torus of $G$ and $T^0$ denotes the Iwahori subgroup of $T$. More precisely, let $L/K^\mrur$ be any finite Galois extension splitting $T$. Then $T^0=T(K)\cap T(K^\mrur)^0$, where $T(K^\mrur)^0$ denotes the image of $T(L)^b$ (\emph{cf.} \eqref{maximalbounded}) under the norm map $T(L)\rightarrow T(K^\mrur)$.
		\item If $G/Z(G)$ is anisotropic, $G^0$ is defined to be $G_{K^\mrur}^0\cap G(K)$, where $G_{K^\mrur}$ is quasi-split over $K^\mrur$.
		\item For $G$ in general, $G^0$ is defined to be $G^\sharp \cdot Z_G(S)^0$, where $Z_G(S)$ is a minimal Levi subgroup of $G$. Notice that $Z_G(S)/Z(Z_G(S))$ is anisotropic since $Z_G(S)^{\mrder}$ is.   
	\end{enumerate}

	Let $\Omega$ be a non-empty subset of an apartment $\mca=\mca(G,S)$ associated to a maximal $K$-split torus $S$. One can attach to $\Omega$ an open subgroup $\parah{\Omega}$ \cite[\S 7.7.0.1]{KP23} of $G$. More precisely, this subgroup has the following explicit descriptions.
	
	For any $\alpha\in \Phi,x \in \mca,r\in \mbr$, we have already defined subgroups of $U_\alpha$ in \Cref{valuation}
	$$U_{\alpha,x,r}=\{u\in U_{\alpha}^*\mid\,\phi_{x,\alpha}(u)\geq r\}\cup\{1\},\quad U_{\alpha,x,r+}=\{u\in U_{\alpha}^*\mid\,\phi_{x,\alpha}(u)>r\}\cup\{1\}.$$
	Alternatively, one has $U_{\alpha,x,r+}=\cup_{s>r} U_{\alpha,x,s}$. For any non-empty subset $\Omega$ of $\mca$, we further define 
	$$U_{\alpha,\Omega,0}:=\bigcap_{x\in \Omega}U_{\alpha,x,0},\quad U_{\alpha,\Omega,0+}:=\bigcap_{x\in \Omega}U_{\alpha,x,0+}.$$ 
	Let $G_{\Omega}^\sharp $ be the subgroup of $G$ generated by $U_{\alpha,\Omega,0}$ for every $\alpha\in \Phi$ and let $G_{\Omega,+}^\sharp $ be the subgroup of $G$ generated by $U_{\alpha,\Omega,0+}$ for every $\alpha\in \Phi$. It is clear that $G_\Omega^\sharp\subset G^\sharp$.

	We define $\parah{\Omega}$ (resp. $\parah{\Omega,+}$)\index{$P_\Omega,P_{\Omega,G}$}\index{$P_{\Omega,+},P_{\Omega,+,G}$} to be the product $G_\Omega^\sharp \cdot Z_G(S)^0$ (resp. $G_{\Omega,+}^\sharp \cdot Z_G(S)^0$), where $Z_G(S)^0$ denotes the subgroup as in (2) in the definition of $G^0$ so that we have $G^0=G^\sharp \cdot Z_G(S)^0$. In particular, $\parah{\Omega,+}$ is the pro-unipotent radical of $\parah{\Omega}$. 
	In particular, if $\Omega$ is a facet $F$ inside $\mca$, then the group $\parah{F}$\index{$P_F$} is a parahoric subgroup associated to $F$. We will also use the notation $\para{\Omega}{G}$ (resp. $\para{\Omega,+}{G}$) to highlight the group $G$ when there are more than one groups involved.
	
	In particular, if $\Omega$ consists of a single point $\bs{p}$, we write $\parah{\bs{p}}:=\parah{\{\bs{p}\}}$.
	
	Let $\Omega',\Omega$ be subsets of $\mca$ such that $\Omega\subset \overline{\Omega'}$. Then one has a chain of root subgroups 
	\[U_{\alpha,\Omega,0+}\subset U_{\alpha,\Omega',0+}\subset U_{\alpha,\Omega',0}\subset U_{\alpha,\Omega,0},\quad \forall\alpha\in \Phi(G,S)\]
	This implies that one has a chain of subgroups
	\[G_{\Omega,+}^\sharp\subset G_{\Omega',+}^\sharp\subset G_{\Omega'}^\sharp\subset G_{\Omega}^\sharp.\]
	Thus, one has
	\begin{equation}\label{parahoricinclu}
		\parah{\Omega,+}\subset \parah{\Omega',+}\subset \parah{\Omega'}\subset \parah{\Omega}.
	\end{equation}
	In \Cref{integralmodel}, we will see that for any facet $F$, $\parah{F}/\parah{F,+}$ is isomorphic to $\underline{G}_F(\bs{k})$, where $\underline{G}_F$ is a reductive group over the residual field $\bs{k}$. Then \eqref{parahoricinclu} suggests that we have an embedding $\underline{G}_{F'}\hookrightarrow \underline{G}_{F}$ for $F,F'$ facets such that $F\subset \overline{F'}$.
	
	\section{Vectorial buildings and Bruhat--Tits buildings}\label{building}

	\subsection{Vectorial buildings}
	As we have seen that there is a close relation between the Coxeter complex associated to a finite Weyl group $W$ and the Tits cone $X_*(S/A)\otimes_\mbz\mbr$ parameterizing all the parabolic subgroups of $G$ containing a given maximal split torus $S$, it is natural to cook up a poly-conical complex, whose facets represent all the parabolic subgroups of $G$. Such a geometric object is known as the vectorial building and denoted by $\mcv(G)$.

	Let $G$ be a reductive algebraic group over a field $K_0$ and $S$ a maximal $K_0$-split torus of $G$. Here $K_0$ can be an arbitrary field, and we will mainly focus on the case when $K_0=\bs{k}$ is a finite field for our later application. Let $V$ denote $X_*(S/A)\otimes \mbr$.  As we have already seen in \eqref{sphericalparabolic}, for a fixed maximal $K_0$-split torus $S$ and each point $x$ in $V$, one can define a corresponding parabolic subgroup $P_x$ generated by $S$ and $\{U_{\alpha}\}_{\alpha\in\Phi(G,S),\langle\alpha ,x\rangle \geq 0}$.

	One defines the vectorial building to be $\mcv(G):=G\times \mca_v(G,S)/\sim$\index{$\mcv(G)$}, where $(g_1,x_1)\sim (g_2,x_2)$ if and only if there exists $n\in N$ such that $x_2=n x_1$ and $g_1^{-1}g_2n\in P_{x}$. Notice that there is a $G$-action on $\mcv(G)$ given by $h(g,x):=(hg,x),\ h,g\in G,\ x\in V$, and an embedding $V\rightarrow \mcv(G)$ given by $x\rightarrow (1,x)$. The definition of a $\mcv(G)$ is independent of the choice of $S$. Indeed, the vectorial apartment $\mca_v(G,S^g)$ is identified with the subset $\mca_v(G,S)^g=\{(g^{-1},x)\mid x\in\mca_v(G,S)\}$ of $\mcv(G)$.
	
	The poly-conical complex structure on $V$ defines a poly-conical complex structure on $\mcv(G)$ by $G$-translation so that the $G$-action preserves the poly-conical complex structure on $\mcv(G)$. Similar to the case of the vectorial apartment, we have the notions of chambers, panels, galleries, etc.
	
	\begin{lemma}
		Let $F$ be a facet of the Tits cone $V\subset \mcv(G)$. The stabilizer $\mrstab_G(F)$ and the pointwise fixator $\mrfix_G(F)$ of $F$ in $G$ coincide with the parabolic subgroup $P_F$ in \eqref{sphericalparabolic}. 
	\end{lemma}
	\begin{proof}
		It's clear from definition that $P_F=P_x\subset \mrfix_G(F)$ for any $x\in F$. By definition, we have $\mrfix_G(F)=\{g\in G|\,(1,x)\sim (g,x),\forall x\in F\subset V\}$, which means that there exists $n\in N_x$ such that $gn\in P_x$ for any $x\in X_*(S/A)\cap F$. The condition $n\in N_x$ is equivalent to such that $nx(t)n^{-1}=x(t)$ for any $t\in \mbg_m$, i.e., $n\in Z_G(\mrim(x))=M_x\subset P_x$. This further implies $g\in P_x=P_F$. Hence we know $\mrfix_G(F)= P_F\subset \mrstab_G(F)$. The only remaining thing to prove is $\mrstab_G(F)\subset P_F$. For any $g\in \mrstab_G(F)$, $g\cdot F=F$ implies that $gP_Fg^{-1}=P_{gF}=P_F$, which further means $g\in N_G(P_F)=P_F$.
		
	\end{proof}

	Hence $P_F$ does not depend on the choice of the vectorial apartment $V$ which $F$ lies in. Let $\mcf(G)$\index{$\mcf(G)$} denote the set of facets in $\mcv(G)$. We have the following $G$-equivariant bijection 
	\begin{equation}\label{parafacet}
		\begin{aligned}
			\mcf(G)&\leftrightarrow \{\text{parabolic subgroups of }G\},\\
			F&\mapsto P_F,
		\end{aligned}
	\end{equation}
	where the $G$-action on the latter set is given by conjugation. In particular, we let $\ch{}{}(G)$\index{$\ch{}{}(G)$} denote the set of chambers of $\mcv(G)$. The $G$-action on $\ch{}{}(G)$ is transitive. The bijection \eqref{parafacet} is refined to
	\begin{equation}
		\ch{}{}(G)\leftrightarrow \{\text{minimal parabolic subgroups of }G\}.
	\end{equation}
	In particular, when $K_0=\bs{k}$ is a finite field, every reductive group $G$ is quasi-split over $\bs{k}$ by Lang's theorem. Hence $\ch{}{}(G) $ parameterizes the finite set of Borel subgroups of $G$.

	\begin{remark}
		
		A closely related object is the notion of a spherical building. In order to construct a spherical building $\mcs(G)$, we simply replace $V$ by the sphere $\mcs\cong (X_*(S/A)\otimes_\mbz \mbr-\{0\})/\mbr^{>0}$ parameterizing rays in $X_*(S/A)\otimes \mbr$ and adopt the same gluing construction. Indeed the spherical building is defined in \cite{CLT80} in a slightly different way but of the same nature.
		
		Moreover, the spherical building has an alternative combinatorial description if $G$ is semisimple. Namely, it is a geometric realization of an abstract simplicial complex $\Delta(G)$ with simplices representing proper parabolic subgroups of $G$. Moreover, the vertices correspond to maximal proper parabolic subgroups and the vertices $\bs{p}_1,\cdots, \bs{p}_i,\ 1\leq i\leq l$ span an $(i-1)$-dimensional simplex $\{\bs{p}_1,\cdots, \bs{p}_i\}$ if $\bs{p}_1\cap\cdots\cap \bs{p}_i$ is also a parabolic subgroup.
		
		It is more convenient for us to work with the Tits cone $X_*(S/A)\otimes _\mbz\mathbb{R}$ and the vectorial building $\mcv(G)$ instead of the associated sphere $\mcs$ and the spherical building $\mcs(G)$. The advantage of considering the cone is that it is an Euclidean building, and has a similar combinatorial structure as the affine Bruhat--Tits building so that one can describe the set of chambers containing a fixed type of facet in an easier way. Unlike the case of spheres when each point in the spherical building corresponds to a proper parabolic subgroup, the vectorial building contains $\{0\}$ as a facet whose stabilizer is the group $G$ itself.

	\end{remark}

	\subsection{Bruhat--Tits buildings}
	
	Similarly, the Bruhat--Tits building also has a realization by gluing the apartments and also has a certain moduli interpretation. Roughly speaking, the Bruhat--Tits building is a geometric realization of the simplicial complex of all parahoric subgroups of a reductive group $G$ over a non-archimedean local field $K_0=K$. The Bruhat--Tits building \cite[Proposition 4.4.4; Remark 7.6.5]{KP23} is obtained by gluing the apartments $\mca(G,S)$ defined in \Cref{affineapartment}, namely, 
	$$\mcb(G):=G\times\mca(G,S))/\sim,\index{$\mcb(G)$}$$
	such that $(g_1,x_1)\sim (g_2,x_2)$ if and only if there exists $n\in N$ such that $x_2=\nu(n)x_1$ and $g_1^{-1}g_2n\in G_{x_1}$, where, for $x\in \mca(G,S)$, $G_x$ denotes the subgroup of $G$ generated by 
	$$N_x:=\mathrm{Stab}_{N}(x)=\{n\in N\mid\,\nu(n)x=x\},$$
	$$U_{\alpha,x,0}=\{u\in U_{\alpha}^*\mid\,\phi_{x,\alpha}(u)\geq 0\}\cup\{1\},\quad \forall \alpha\in \Phi.$$ 
	
	Notice that $\mcb(G)$ is equipped with a $G$-action given by $h\cdot (g,x):=(hg,x)$ for any $g,h\in G,x\in \mca(G,S)$, and there exists an embedding $\mca(G,S)\rightarrow \mcb(G)$ given by $x\mapsto (1,x)$, which is $N$-equivariant. Indeed the group $G_x$ equals the stabilizer of $x\in \mcb(G)$ under this $G$ action by \cite{SS97}. Still, the definition of $\mcb(G)$ is independent of the choice of $S$. Moreover, the affine apartment $\mca(G,S^g)$ is identified with the subset $\mca(G,S)^g:=\{(g^{-1},x)\mid x\in\mca(G,S)\}$ of $\mcb(G)$, where $S^g$ denotes the torus $g^{-1}Sg$.

	Similar to the case of the vectorial building, $\mcb(G)$ has a structure of a poly-simplicial complex given by the $G$-translation of the poly-simplicial structure on the apartment $\mca$ so that the group action of $G$ preserves the poly-simplicial structure. We also use the notation $\ch{}{}(G)$\index{$\ch{}{}(G)$} for the set of chambers in $\mcb(G)$. In particular, both $G$ and $G^0$ act transitively on $\ch{}{}(G)$.
	
	The above gluing construction also works for the extended Bruhat--Tits building if we replace $\mca(G,S)$ by $\mca^\mrext(G,S)$. Moreover, we have
	\begin{equation}\label{extord}
		\mcb^{\mrext}(G)\cong\mcb(G)\times (X_*(A)\otimes_{\mbz}\mbr),\index{$\mcb^{\mrext}(G)$}
	\end{equation}
	where the second factor is regarded as an affine space associated to the vector space $X_*(A)\otimes _{\mbz}\mbr$ without a preferred base point.

	\subsection{The integral model associated to a bounded subset $\Omega\subset \mca\subset \mcb(G)$}\label{integralmodel}
	In this section, we review the concept of a connected parahoric integral model attached to a bounded subset $\Omega\subset\mca$ studied by Bruhat--Tits in \cite{bruhat1984groupes}. The integral model is constructed over the field $K^{\mrur}$ first, where $G_{K^\mrur}$ is always quasi-split. Constructions of general integral models will follow from the unramified descent.

	Let $\Omega$ be a non-empty subset inside an apartment $\mca$. We have already mentioned in the last section that one can attach an open bounded subgroup \cite[7.7.0.1]{KP23} $\parah{\Omega} $ of $G$ to $\Omega$. Actually, the subgroup $\parah{\Omega}$ does not depend on the choice of apartment $\mca$ which $\Omega$ lies in. 
	
	It is direct to check that $\parah{\Omega}$ is a subgroup of the pointwise fixator of $\Omega$ in $G$. On the other hand, $\parah{\Omega}$ equals the pointwise fixator as well as the stabilizer of $\Omega$ in $G^0$ \cite[Proposition 7.7.5]{KP23}. Thus when regarding $\mcb(G)$ as the building associated to a Tits system of $G^0$, various results in \cite{bruhat1972groupes} could still be used for our group $\parah{\Omega}$ here.
	
	If $\bs{p}$ is a point of general position in a facet $F$ of $\mca$, meaning that $\bs{p}$ does not lie in any affine subspace of $\mca$ whose corresponding reflection stabilizes $F$, then an element $g$ stabilizing $F$ and fixing $\bs{p}$ must fix every point in $F$ and we have $\parah{F}=\parah{\bs{p}}$.
	
	Now we assume that $\Omega$ is a bounded subset in $\mca$. In this case, Bruhat and Tits  construct a smooth connected integral model $\mcg_\Omega$ of $\parah{\Omega}$, called the \emph{connected parahoric integral model} attached to $\Omega$. By an \emph{integral model}, we mean a group scheme $\mcg_{\Omega}$\index{$\mcg_{\Omega}$} over $\mfo_K$ such that the generic fiber of $\mcg_{\Omega}$ is $G$, and $\mcg_{\Omega}(\mfo_K)=\parah{\Omega}$. Still, $\mcg_{\Omega}$ does not depend on the choice of $\mca$. 
	\begin{remark} 
		We warn the reader that there are some differences between our notations and those in \cite{bruhat1984groupes} and \cite{KP23}. In \cite{KP23}*{\S 8.3} for instance, $\mcg_\Omega$ means a different integral model, i.e., a smooth integral model for the non-connected parahoric subgroup with respect to $\Omega$, which is usually not affine. Our $\mcg_\Omega$ indeed means the integral model $\mcg_\Omega^0$ in \emph{loc. cit.}.
		
	\end{remark}

	We will not go deep into the explicit construction of the integral model $\mcg_\Omega$. Instead, we may list some useful properties of this integral model, which will be enough for our applications.

	The special fiber $\underline{\mathtt{G}}_{\Omega}$ of the Bruhat--Tits connected integral model $\mcg_{\Omega}$ is a connected smooth algebraic group over $\bs{k}$
	not necessarily reductive. Let $\underline{\mathtt{G}}_\Omega^+$ be the unipotent radical of $\underline{\mathtt{G}}_{\Omega}$. Let $\underline{G}_\Omega$\index{$\underline{G}_\Omega$} denote the maximal reductive quotient of $\underline{\mathtt{G}}_{\Omega}$ over $\bs{k}$, i.e. $\underline{G}_\Omega:=\underline{\mathtt{G}}_\Omega/\underline{\mathtt{G}}_\Omega^+$.
	
	By \cite[Proposition 2.1]{Pra20}, for a maximal $\bs{k}$-split torus $\underline{S}$ of $\underline{\mathtt{G}}_\Omega$, we can find an $\mfo_K$-split integral model $\mcs\subset \mcg_\Omega$ of $\underline{S}$. The generic fiber is also a maximal $K$-split torus $S$ of $G$. 
	
	Consider the surjective reduction map $$\parah{\Omega}=\mcg_{\Omega}(\mfo_K)\rightarrow \underline{\mathtt{G}}_\Omega(\bs{k}).$$
	It is known that $\parah{\Omega,+}$ equals the pre-image of $\underline{\mathtt{G}}_\Omega^+(\bs{k})$ under the reduction map. This means that we have identifications
	$$\parah{\Omega}/\parah{\Omega,+}\cong \underline{\mathtt{G}}_\Omega(\bs{k})/\underline{\mathtt{G}}_\Omega^+(\bs{k})=\underline{G}_\Omega(\bs{k}),$$
	where the second equality follows from the fact that the first Galois cohomology $H^1(\bs{k},\underline{\mathtt{G}}_\Omega^+)=1$.

	Indeed, the reductive quotient of the special fiber of an integral model only depends on the the affine subspace in $\mca$ spanned by $\Omega$. Consider the natural partial order on the set of bounded subsets of $\mcb(G)$ defined by $\Omega'\leq\Omega$ if and only if $\Omega'\subset \overline{\Omega}$. For $\Omega'\leq\Omega$, one has the related $\mfo_K$-morphism $\rho_{\Omega',\Omega}:\mcg_\Omega\rightarrow \mcg_{\Omega'}$ which is an identity on the generic fiber. The induced morphism on the special fiber $\underline{\mathtt{G}}_{\Omega}\rightarrow \underline{\mathtt{G}}_{\Omega'}$ composed with the quotient map $\underline{\mathtt{G}}_{\Omega'}\rightarrow \underline{G}_{\Omega'}$ will further descend to a morphism $\underline{G}_{\Omega}\rightarrow \underline{G}_{\Omega'}$ by \cite[Proposition 8.4.16]{KP23}. Moreover, we have the following lemma.
	\begin{lemma}\cite[Proposition 8.4.16, \S 8.4.17]{KP23}\label{reductivespecial}
		Let $\Omega_1,\Omega_2$ be two non-empty bounded subsets of $\mca$ such that
		the affine subspace  of $\mca$ spanned by $\Omega_1$ is the same as the
		affine subspace of $\mca$ spanned by $\Omega_2$. Then the morphisms $\rho_{\Omega_i,\Omega_1\cup \Omega_2}:\mcg_{\Omega_1\cup \Omega_2}\rightarrow\mcg_{\Omega_i}$ induce isomorphisms
		$f_{i,12}:\underline{G}_{\Omega_1\cup \Omega_2}\rightarrow \underline{G}_{\Omega_i}$, and the composition $f_{1,2}:=f_{2,12}\circ f_{1,12}^{-1}:\underline{G}_{\Omega_1}\rightarrow \underline{G}_{\Omega_2}$ is an isomorphism.
	\end{lemma}

	\begin{corollary}
		Let $F$ be a facet in $\mca$ and $\mca'$ the affine subspace of $\mca$ spanned by $F$, then $\parah{F}=\parah{\mca'}\parah{F,+}$.
	\end{corollary}
	\begin{proof}
    Let $\mathcal{I}$ be the set of bounded subsets $\Omega\subset \mca'$ such that
	$$F\subset \Omega,
	\quad
	\mathrm{span}(\Omega)=\mca'.$$
	This family is directed by finite unions. Moreover, it is cofinal. 
	Hence
	\[
	P_{\mca'}=\bigcap_{\Omega\in\mathcal I}P_\Omega .
	\]
    Fix $\Omega\in\mathcal I$. Since $F\subset\Omega$, we have
	$P_\Omega\subset P_F.$ By \Cref{reductivespecial}, applied to the bounded subsets $F$ and $\Omega$, which have
	the same affine span $\mca'$, the induced morphism on maximal reductive quotients
	is an isomorphism $\parah{\Omega}/\parah{\Omega,+}\cong \parah{F}/\parah{F,+}$. Equivalently, the image of $P_\Omega$ in
	$P_F/P_{F,+}$ is all of $P_F/P_{F,+}$. Thus
	\[
	P_\Omega P_{F,+}=P_F
	\]
	for every $\Omega\in\mathcal I$.
    
	It remains to pass to the intersection over $\Omega$. Let
	$\pi:P_F\to P_F/P_{F,+}$ be the quotient map, and let $c$ be a coset of
	$P_{F,+}$ in $P_F$. For every $\Omega\in\mathcal I$, the set
	$P_\Omega\cap c$ is nonempty by the equality $P_\Omega P_{F,+}=P_F$. It is also closed in the
	compact group $P_F$. If $\Omega_1,\ldots,\Omega_m\in\mathcal I$, then
	$\Omega_1\cup\cdots\cup\Omega_m$ lies in $\mathcal I$, and
	$$\emptyset\neq P_{\Omega_1\cup\cdots\cup\Omega_m}\cap c
	\subset
	\bigcap_{i=1}^m (P_{\Omega_i}\cap c).$$
	Therefore the closed sets $P_\Omega\cap c$ have the finite intersection
	property. Then the compactness of $P_F$ forces
	$$\bigcap_{\Omega\in\mathcal{I}}(P_{\Omega}\cap c)\neq\emptyset.$$
	Thus every coset of $P_{F,+}$ in $P_F$ has a representative in
    $\bigcap_{\Omega\in\mathcal I}P_\Omega=P_{\mca'},$ implying that
	 $P_F=P_{\mca'}P_{F,+}.$
		
	\end{proof}
	In particular, if $\Omega=D$ is a panel lying in a given hyperplane $\mfh\subset \mca$, the reductive quotient $\underline{G}_D$ only depends on the hyperplane $\mfh$.

	Moreover, one has the following description of the root datum of the reductive  special fiber $\underline{G}_\Omega$ for any bounded $\Omega\subset \mca$.
	
	\begin{proposition}\cite[Proposition 9.4.23]{KP23}\label{proppanelDroot}
		The  root datum of $\underline{G}_\Omega$ with respect to $\underline{S}$ is given by $(X^*(\underline{S}),\Phi_\Omega,X_*(\underline{S}),\Phi_\Omega^\vee)$, where $\Phi_\Omega$ denotes the set of roots $\{D\psi\mid\, \psi\rest_\Omega=0\}\subset \Phi$ and $\Phi_\Omega^\vee$ denotes the set of coroots $\{\alpha^\vee\in \Phi^\vee\mid \alpha\in \Phi_\Omega\}$.
		
	\end{proposition}
	
	Hence for a panel $D$ lying in a hyperplane $\mfh$, the reductive quotient $\underline{G}_D$ is a reductive group over $\bs{k}$ of semi-simple rank $1$. In particular, $\underline{G}_D^\mrsc$ is a semi-simple, simply connected algebraic group over $\bs{k}$, which must be a certain Weil restriction of scalars of $\mrsl_2$ or $\mrsu_3$ depending on whether $\Phi(\underline{G}_D^\mrsc,\underline{S})$ is reduced or not. 
	
	Let $F$ be a fixed facet contained in $\mca$. For any facet $F\leq F'$, the image of the induced map $\underline{\mathtt{G}}_{F'}\rightarrow \underline{\mathtt{G}}_{F}\rightarrow {\underline{G}_{F}}$ is a parabolic subgroup $\underline{P}_{F'}$ of ${\underline{G}_{F}}$. Hence the map $F'\mapsto \underline{P}_{F'}$ establishes a bijection between the set $$\mcf_{F}(G):=\{F'\text{ is a facet in   }\mcb(G)\mid\,F\leq F'\}\index{$\mcf_{F}(G)$}$$ and the set of $\bs{k}$-parabolic subgroups of the reductive special fiber ${\underline{G}_{F}}$.
	
	Let $\mcf({\underline{G}_{F}})$ denote the set of facets of $\mcv({\underline{G}_{F}})$.  Since the set of parabolic subgroups of ${\underline{G}_{F}}$ is parameterized by the vectorial building of ${\underline{G}_{F}}$, we indeed get a bijection 
	\begin{equation}\label{bijectionfacet}
		\mcf_{F}(G)\leftrightarrow \mcf({\underline{G}_{F}}).
	\end{equation}
	
	The bijection \eqref{bijectionfacet} relates the combinatorics of the Bruhat--Tits building with a fixed adjacent facet to the combinatorics of the vectorial building over the residue field $\bs{k}$. 
	For instance, let  $\pa{}{}(\underline{G}_F)$ denote the set of panels of $\mcv(\underline{G}_F)$. We also write
	$$\ch{F}{}(G)=\{C\text{ is a chamber in   }\mcb(G)\mid\,F\leq C\},\quad\pa{F}{}(G)=\{D\text{ is a panel in   }\mcb(G)\mid\,F\leq D\}.\index{$\ch{F}{}(G)$}$$ 
	When the codimension of $F$ is at least 1, the bijection \eqref{bijectionfacet} restricts to bijections
	\begin{equation}\label{bijection}
	\ch{F}{}(G)\leftrightarrow \ch{}{}(\underline{G}_F)\quad\text{and}\quad\pa{F}{}(G)\leftrightarrow \pa{}{}(\underline{G}_F),
	\end{equation}
	so that a gallery inside $\mcb(G)$ with each chamber and panel having $F$ as a facet corresponds to a gallery in $\mcv(\underline{G}_F)$.
	
	The bijection \eqref{bijectionfacet} also induces a bijection on the related apartments. Let $\mcs$ be a maximal $\mfo_K$-split torus of $\mcg_F$ whose special fiber $\ul{S}$ is a maximal $\bs{k}$-split torus of $\ul{G}_F$ and generic fiber $S$ is  a maximal $K$-split torus of $G$. Denote by $\mcf(\mca(G,S))$ (resp. $\mcf(\mca_v(\ul{G}_F,\ul{S}))$) the set of facets of $\mca(G,S)$ (resp. $\mca_v(\ul{G}_F,\ul{S})$). Then $\eqref{bijectionfacet}$ induces a bijection
	$$	\mcf_{F}(G)\cap\mcf(\mca(G,S))\leftrightarrow \mcf({\underline{G}_{F}})\cap\mcf(\mca_v(\ul{G}_F,\ul{S})).$$
	
	\section{Involutions on apartments and buildings}\label{sectioninvolution}
	
	\subsection{$\theta$-stable apartments}
	
	Let $G$ be a reductive group over $K_0$, where $K_0$ is either the finite field $\bs{k}$ or the non-archimedean local field $K$. Let $\theta$ be an involution on $G$ over $K_0$, i.e., $\theta\in \mraut_{K_0}(G)$ such that $\theta^2=\mrid$. Let $H'=G^{\theta}$ and $H=(G^{\theta})^{\circ}$. 
	
	When $K_0=\bs{k}$ (resp. $K_0=K$), we consider the related vectorial (resp. Bruhat--Tits) building of $G$. 
	
	A torus $S$ of $G$ over $K_0$ is called \emph{$\theta$-stable} if $\theta(S)=S$. A $\theta$-stable torus $S$ is called \emph{$\theta$-invariant} (resp. \emph{$\theta$-split}) if $\theta(s)=s$ (resp. $\theta(s)=s^{-1}$) for any $s\in S$. If $S$ is $\theta$-stable, then the $\theta$-action decomposes the torus into the $\theta$-invariant part and the $\theta$-split part. More precisely, one defines subtori
	$$S^+:=\{s\in S\mid\,\theta(s)=s\}^\circ, \quad S^-:=\{s\in S\mid\,\theta(s)=s^{-1}\}^\circ,\index{$S^+,S^-$}$$
	and an induced isogeny 
	$$S^+\times  S^-\rightarrow S$$
	whose kernel is an abelian $2$-group. Since $S$ is $K_0$-split, both $S^+$ and $S^-$ are $K_0$-split.
	
	Given a maximal $K_0$-split torus $S$ of $G$, we define the related apartment $\mca=\mca_v(G,S)$ (resp. $\mca=\mca(G,S)$). We call $\mca$ \emph{$\theta$-stable} if $\mca$ is stable under the $\theta$-action, or equivalently, $S$ is $\theta$-stable. Indeed if $S$ is $\theta$-stable, we may extend $\theta$ to an affine involution on $\mca$ \footnote{This could be seen by either \cite{bruhat1984groupes}*{\S 4.2.12} or the following more direct argument. Notice that $\theta$ stabilizes the set of roots $\Phi=\Phi(G,S)$ and maps $U_\alpha$ to $U_{\theta(\alpha)}$. Thus for a given $v_K$-compatible valuation $(\phi_{x,\alpha})_{\alpha\in\Phi}$ of $(G,S)$ corresponding to $x\in \mca$, we may verify that $(\phi_{x,\alpha}\circ\theta)_{\alpha\in\Phi}$ is the $v_K$-compatible valuation of $(G,S)$ corresponding to $\theta(x)\in \theta(\mca)$.}. Remark that the existence of a maximal $K_0$-split $\theta$-stable torus $S$ is guaranteed by \cite{helminck1993rationality}*{Proposition 2.3}, and from now on we fix such an $S$. Let $\mca^\theta$\index{$\mca^\theta$} denote the $\theta$-invariant part as an affine subspace of $\mca$.

	For any facet $F$ of $\mca$, its $\theta$-invariant part
	$F^\theta$\index{$F^\theta$} is not empty if and only if $F\cap \theta(F)$ is not empty, if and only if $F$ is $\theta$-stable. 
	Let $\Ft(\mca)$ denote the set of $\theta$-stable facets inside the apartment $\mca$. Then, we have a decomposition
	$$\mca^\theta=\bigsqcup_{F\in\Ft(\mca)}F^\theta,$$ 
	which endows $\mca^\theta$ with a poly-conical or poly-simplicial complex structure with related $F^\theta$ being its facets. In particular, $F^\theta$ is a chamber of $\mca^\theta$ if it is of  dimension  $\mrdim(\mca^\theta)$, or equivalently $F$ is a maximal $\theta$-stable facet in $\mca$ with respect to the partial order of inclusion. 

    \begin{figure}[htbp]
	\begin{center}
		\tikzstyle{every node}=[scale=1]
		\begin{tikzpicture}[line width=0.4pt,scale=0.4][>=latex]
			\pgfmathsetmacro\ax{1}
			\pgfmathsetmacro\bx{sin(60)}
			\pgfmathsetmacro\by{cos(60)}
			\pgfmathsetmacro\cy{cos(120)}
			\pgfmathsetmacro\cx{sin(120)}	
			
			\draw[blue,-] (-6*\ax,2*\bx) -- (7*\ax,2*\bx);  	
			\draw[blue,-] (-6*\ax,0) -- (7*\ax,0);
			\draw[blue,-] (-6*\ax,-2*\bx) -- (7*\ax,-2*\bx);
			\draw[blue,-] (-6*\ax,-4*\bx) -- (7*\ax,-4*\bx);

			\draw[blue,-] (-7*\by,-7*\bx) -- (5*\by,5*\bx);
			\draw[blue,-] (-2-7*\by,-7*\bx) -- (-2+5*\by,5*\bx);
			\draw[blue,-] (2-7*\by,-7*\bx) -- (2+5*\by,5*\bx);
			\draw[blue,-] (4-7*\by,-7*\bx) -- (4+5*\by,5*\bx);

			\draw[blue,-] (-7*\cy,-7*\cx) -- (5*\cy,5*\cx);
			\draw[blue,-] (-2-7*\cy,-7*\cx) -- (-2+5*\cy,5*\cx);
			\draw[blue,-] (2-7*\cy,-7*\cx) -- (2+5*\cy,5*\cx);
			\draw[blue,-] (4-7*\cy,-7*\cx) -- (4+5*\cy,5*\cx);
			\draw[blue,-] (-4-7*\cy,-7*\cx) -- (-4+5*\cy,5*\cx);
			\draw[black,<->] (8+2.5*\cy,+2.5*\cx) -- (8+3.5*\cy,3.5*\cx);
			
			\draw[red,-] (-7*\ax,-6*\bx) -- (8,4*\bx);
			\node at (0:8) [text=blue]{\(\mca\)};
			\node at (26:9) [text=red]{\(\mca^\theta\)};
			\node at (20:7.5) [text=black]{\(\theta\)};
		\end{tikzpicture}
	\end{center}
	\caption{$\mca$ and $\mca^\theta$ for $G=\mrgl_3(K)$, $S=\{\mrdiag(t_1,t_2,t_3)\mid t_i\in K^\times\}$,\\ and $\theta(g)=J_3^{-1}\,^{t}g^{-1}J_3$ for $J_3$ as \eqref{Jn}. }
\end{figure}
	
	Now we discuss the $\theta$-rank of a chamber. We need the following proposition.
	
	\begin{proposition}\label{propAcontainC} 
		
		Given $C\in\ch{}{}(G)$. There exists a $\theta$-stable apartment $\mca$, unique up to $H'$-conjugacy, that contains $C$.	
		
	\end{proposition}
	
	\begin{proof}
		
		When $K_0=\bs{k}$, it is shown by Helminck--Wang \cite{helminck1993rationality}*{Lemma 2.4}. When $K_0=K$, it is shown by Court\`es \cite{courtes2017distinction}*{Proposition 4.8} in the split Galois case and Broussous \cite{broussous2024orbits}*{Section 3} in general.
		
	\end{proof}
	
	\begin{definition}\label{defthetarank}
		
		The \emph{$\theta$-rank} of a $\theta$-stable apartment $\mca$ associated to a $\theta$-stable maximal $K_0$-split torus $S$ is defined to be the $K_0$-rank of $S^+/S^+\cap A(G)$ or $S^+\cap G^{\mrder}$, or equivalently, the dimension of the affine subspace $\mca^{\theta}$. 
		For $C\in\ch{}{}(G)$, its \emph{$\theta$-rank} is defined as the $\theta$-rank of any $\theta$-stable $\mca$ containing $C$. 
		
	\end{definition}

	Because of Proposition \ref{propAcontainC}, our definition makes sense. In particular, if $K_0=\bs{k}$, then the $\theta$-rank of a Borel subgroup $B$ of $G$ is defined as the $\theta$-rank of its related chamber $C_B$. Recall that in this case, the set of chambers $\ch{}{}(G)$ is indeed in bijection with the set of Borel subgroups of $G$, and the chambers in $\mca$ relate to the Borel subgroups containing $S$.
	
	Finally, we discuss $H$-conjugacy classes (resp. $H'$-conjugacy class) of $\theta$-stable apartments. Given a $\theta$-stable apartment $\mca=\mca(G,S)$ and $g\in G$, if the apartment $\mca^g:=g^{-1}\cdot \mca$ is also $\theta$-stable, then equivalently we get $\theta(S^g)=S^g$, or $\theta(g)g^{-1}\in N_{G}(S)$, where $S^g$ denotes the torus $g^{-1}Sg$. As a result, the set of $H$-conjugacy (resp. $H'$-conjugacy) classes of $\theta$-stable apartments is in bijection with the set of double cosets
	$$\{N_{G}(S)gH\mid g\in G,\ \theta(g)g^{-1}\in N_{G}(S)\}\quad(\text{resp.}\ \{N_{G}(S)gH'\mid g\in G,\ \theta(g)g^{-1}\in N_{G}(S)\}).$$
	We remark that the above set of double cosets is finite, bounded by the cardinality of $H\backslash G/ B$ (resp. $H'\backslash G/ B$), where $B$ is a minimal parabolic subgroup of $G$ over $K_0$ (\emph{cf.} \cite{helminck1993rationality}*{Proposition 6.10, Corollary 6.16}). 
	
	\subsection{The embedding $\iota:\mcb(H)\hookrightarrow\mcb(G)$}
	
	From now on until the end of this section, we assume $K_0=K$ and we consider the Bruhat--Tits building of $G$. We emphasize that $p\neq 2$.
	
	By functoriality, the involution $\theta$ on $G$ induces an involutive action on the reduced Bruhat--Tits building $\mcb(G)$, which is also denoted by $\theta$, such that for any $g\in G,x\in \mcb(G)$, one has 
	\begin{equation}\label{equivariant}
		\theta (g\cdot x)=\theta (g)\cdot \theta(x).
	\end{equation}
	Moreover, $\theta$ on $\mcb(G)$ is simplicial and maps apartments to apartments (\emph{cf.} \cite{bruhat1984groupes}*{\S 4.2.12}).
	
	Based on the explicit realization of $\mcb (G)$, we can describe the action of $\theta$ on $\mcb (G)$ in the following way. We pick a $\theta$-stable maximal $K$-split torus $S$ in $G$. Then, $\mcb(G)$ could be defined by $G\times\mca(G,S))/\sim$, which does not depend on the choice of $S$. Since $S$ is $\theta$-stable, $\mca$ is naturally equipped with a $\theta$-action. 
	Then $\theta$ acts on $\mcb(G)$ via the diagonal action, namely,
	$$\theta((g,x)):=(\theta(g),\theta(x)),\quad g\in G,\ x\in \mca.$$
	It is easy to check that the above action of $\theta$ satisfies \eqref{equivariant} since $\theta\in \mraut_F(G)$. A similar discussion works for the extended building $\mcb^\mrext(G)$ as well.
	
	By \cite{PY02}, one has a canonical $H$-equivariant embedding of extended building
	$$\iota^\mrext:\mcb^\mrext(H)\hookrightarrow \mcb^\mrext(G),$$ such that $\mrim(\iota^\mrext)=  \mcb^\mrext(G)^{\theta}\index{$\iota^\mrext$}$ as a set. It is also known that $\iota^\mrext$ is toral (\emph{cf.} \cite[\S 1.9]{PY02} and \cite[\S 1.33]{Lan00}). It means that for any maximal $K$-split torus $S_H$ of $H$, there exists a maximal $K$-split torus $S$ of $G$ containing $S_H$, such that the restriction map $\iota^\mrext\rest_{\mca^\mrext(H,S_H)}$ has image in $\mca^\mrext(G,S)$ and induces an affine map $\iota^\mrext\rest_{\mca^\mrext(H,S_H)}:\mca^\mrext(H,S_H)\rightarrow \mca^\mrext(G,S)$. 
	
	However, this embedding of extended building does not always descend to an embedding of reduced building, namely, there does not always exist an embedding $\iota:\mcb(H)\hookrightarrow \mcb(G)$\index{$\iota$} such that the following diagram commutes
	\begin{equation}\label{eqdiagramextred}
		\xymatrix{
			\mcb^{\mathrm{ext}}(H) \ar@{^{(}->}[r]^{\iota^\mathrm{ext}} \ar@{->>}[d]_{\pi} & \mcb^{\mathrm{ext}}(G) \ar@{->>}[d]^{\pi} \\
			\mcb(H) \ar@{^{(}->}[r]^{\iota} & \mcb(G)
		},
	\end{equation}
	where $\pi:\mcb^\mrext(G)\rightarrow \mcb(G)$ (resp. $\pi:\mcb^\mrext(H)\rightarrow \mcb(H)$) denotes the canonical projection. 
	
	Here is a typical example.
	
	\begin{example}
		Let $G$ be $\mrgl_2$, $\theta$ the involution $\theta(g)=J_2{}^tg^{-1}J_2$, where $J_2$ denotes the element $\left(\begin{matrix}
			0&1\\
			1&0
		\end{matrix}\right)$. In this case, one has $H=(G^{\theta})^\circ=\mrso_2(J_2)\cong \mbg_m$ being the split orthogonal group with the embedding of $H\hookrightarrow G$ being $t\mapsto\mrdiag(t,t^{-1})$. Notice that the standard apartment of $\mcb^{\mrext}(G)$ with respect to the maximal torus $S=\mrdiag(t_1,t_2)$ is the plane $\mbr^2=\mbr e_1^{\vee}+\mbr e_2^{\vee}$, where $e_1^\vee(t)=\mrdiag(t,1)$ and $e_2^\vee(t)=\mrdiag(1,t)$. In this case, the $\theta$-action on the standard apartment is given by $\theta(e_1^{\vee})=-e_2^{\vee}, \theta(e_2^{\vee})=-e_1^{\vee}$, and one has $\mca^{\mrext}(H,S^\theta)\cong\mca^{\mrext}(G,S)^\theta=\mbr(e_1^{\vee}-e_2^{\vee})$. In this case, the projection map $\pi:\mca^{\mrext}(G,S)\rightarrow\mca(G,S)$ is the canonical quotient $\mbr e_1^{\vee}+\mbr e_2^{\vee}\rightarrow \mbr e_1^{\vee}+\mbr e_2^{\vee}/\mbr(e_1^{\vee}+e_2^{\vee})$ and the composition $\pi\circ \iota^{\mrext} :\mca^{\mrext}(H,S^\theta)\rightarrow \mca(G,S)$ is surjective. However, the fact that $\mcb(H)$ is a point means that there does not exist $\iota$ such that the diagram \eqref{eqdiagramextred} is commutative.
	\end{example}
	
	To avoid this inconvenience, we may pose the following assumption. 
	\begin{assumption}\label{assumpZHZG}
		$Z(H)\subset Z(G)\cap H$. 
	\end{assumption}
	\begin{lemma}
		If $H,G$ satisfy Assumption \ref{assumpZHZG}, then there exists an embedding $\iota:\mcb(H)\hookrightarrow\mcb(G)$ such that the diagram \eqref{eqdiagramextred} commutes.
	\end{lemma}
	\begin{proof}
		Fix a $\theta$-stable maximal $K$-split torus $S$, we only need to check the commutativity of the diagram on the standard apartment $\mca(H,S^\theta)$.  Notice that $Z(H)\subset Z(G)\cap H$ implies $A(H)\subset A(G)\cap H$. The canonical map $\iota_v:X_*(S^{\theta})\rightarrow X_*(S)/X_*(A(G))$ factors through $X_*(S^\theta)/X_*(S^\theta\cap A(G))$. The condition $A(H)\subset A(G)\cap H$ implies $A(H)\subset A(G)\cap S^\theta$, hence the map $\iota_v$ factors through $X_*(S^\theta)/X_*(A(H))\rightarrow X_*(S)/X_*(A(G))$, which extends to an embedding $\iota_v:X_*(S^\theta)/X_*(A(H))\otimes_{\mbz}\mbr\rightarrow X_*(S)/X_*(A(G))\otimes_{\mbz}\mbr$.
		
		Now we define the corresponding embedding $\iota:\mca(H,S^\theta)\hookrightarrow \mca(G,S)$ as follows. For any $x\in\mca(H,S^\theta)$ and $x'\in \mca^{\mrext}(H,S^\theta)$ such that $\pi(x')=x$, we define $\iota(x)=\pi(\iota^{\mrext}(x'))$. This definition is independent of the choice of $x'$, since for another possible choice $x''$ with $x''-x'\in X_*(A(H))\otimes_{\mbz}\mbr$, we have 
		$\iota^{\mrext}(x'')=\iota^{\mrext}(x')+\iota_v(x''-x')$ with $\iota_v(x''-x')\in X_*(A(G))\otimes_{\mbz}\mbr$. Thus $\pi(\iota^{\mrext}(x'))=\pi(\iota^{\mrext}(x''))$. So $\iota$ is well defined such that $\pi\circ\iota^{\mrext}=\iota\circ\pi$.
		
	\end{proof}
	
	Notice that $\theta$ preserves the center $Z(G)$ as well as $A(G)$. Using $\mrim(\iota^\mrext)=  \mcb^\mrext(G)^{\theta}$ and  \eqref{eqdiagramextred}, we have an identification 
	$\iota(\mcb(H))=\mcb(G)^\theta$. We remark that $\iota$ is toral as well, thus its restriction to any apartment in $\mcb(H)$ is an affine map.

	\begin{remark}
		
		We do not assume $G$ to be split over $K$. Typical examples come from the Weil restriction of scalars. More precisely, let $L/K$ be a quadratic extension of $p$-adic fields, $H$ a reductive group over $K$ with $G$ being $\mrres_{L/K}H_{L}$ and $\theta$ the non-trivial element in $\mrgal(L/K)$ inducing a $K$-involution on $G$. In this case, $G$ is not split over $K$ even if one assumes that $H$ is split over $K$. However, we have a canonical identification 
		$$\mcb(G)=\mcb(\mrres_{L/K}H_{L})= \mcb(H_{L}),$$
		where $\mcb(H_{L})$ denotes the reduced Bruhat--Tits building of $H_L$ as a reductive group over $L$.
	\end{remark}

	\subsection{The poly-simplicial structure on $\mcb(G)^\theta$}

	Fix our notations as in the previous subsection. As before, a facet $F$ of $\mcb(G)$ is $\theta$-stable if and only if its $\theta$-invariant part $F^\theta$ is not empty. Let $\Ft(G)$\index{$\Ft(G)$} denote the set of $\theta$-stable facets of $\mcb(G)$.

	Let $\mcb(G)^{\theta}$ be the $\theta$-invariant part of $\mcb(G)$. Then, we have a decomposition
	$$\mcb(G)^\theta=\bigsqcup_{F\in\Ft(G)}F^\theta,$$ 
	which endows $\mcb(G)^\theta$ with a poly-simplicial complex structure on $\mcb(G)^{\theta}$ with $F^\theta$ being its facets. We have the following bijection
	\begin{equation}\label{thetainter}
		\begin{aligned}
			\Ft(G)&\leftrightarrow \mcf(\mcb(G)^{\theta}):=\{\text{facets of }\mcb(G)^{\theta}\}\\
			F&\mapsto F^\theta=F\cap \mcb(G)^{\theta}.
		\end{aligned}
	\end{equation}
	The $G$-action on $\mcb(G)$ induces $H$-actions on $\Ft(G)$ and $\mcf(\mcb(G)^{\theta})$, which induces a simplicial $H$-action on the complex $\mcb(G)^\theta$. 
	
	Given a $\theta$-stable facet $F$, we may find a $\theta$-stable apartment that contains $F$. Then, $F^\theta$ is a facet of $\mca^\theta$. In this way, we have $$\mcb(G)^\theta=\bigcup_{\theta\text{-stable}\ \mca}\mca^\theta$$
	as poly-simplicial complexes.

	Assume Assumption \ref{assumpZHZG}. Identifying $\mcb(H)$ with $\iota(\mcb(H))$, the poly-simplicial structure $\mcb(G)^\theta$ is a refinement of that of $\mcb(H)$ \cite{prasad2020finite}*{Proposition 3.8}. 
	Here, by a refinement we mean that any $F_H\in \mcf(H)$ is a union of $F^\theta$ with $F\in\Ft(G)$. Since $\overline{F_H}$ is compact, it is the union of finitely many facets in $\mcb(G)^\theta$.
	As a result, since there are finitely many $H$-orbits of $\mcf(H)$ (for instance, using the fact that $H$ acts transitively on the set of chambers of $\mcb(H)$), there are finitely many $H$-orbits of $\Ft(G)$. 
	
	In general, even if Assumption \ref{assumpZHZG} is not satisfied, using a similar argument to $\mcb^{\mrext}(G)^\theta$ and $\mcb^{\mrext}(H)$, we may show that there are finitely many $H$-orbits of $\theta$-stable facets in the extended building $\mcb^{\mrext}(G)$. Since every $\theta$-stable facet of $\mcb(G)$ is the projection of a $\theta$-stable facet of $\mcb^{\mrext}(G)$, there are finitely many $H$-orbits of $\Ft(G)$.
	
	It is interesting to give an explicit upper bound of $\car{\Ft(G)/H}$.

	\section{The Steinberg representation $\mrst_G$ of $G$}\label{sectionsteinberg}

	Let $G$ be a connected reductive group over a non-archimedean local field $K$. 
	We fix a minimal parabolic subgroup $B$ of $G$ over $K$. By the Steinberg representation of $G$, we mean the complex irreducible representation $$\mrst_G:=\mrInd_{B}^G\mbc/\sum_{P,B\subsetneq P}\mrInd_P^G \mbc,\index{$\mrst_G$}$$ where $P$ in the sum ranges over the finite set of parabolic subgroups of $G$ strictly containing $B$ and $\mrInd_P^G$ denotes the parabolic induction functor. Also, one uses the natural embedding $\mrInd_P^G \mbc\rightarrow \mrInd_{B}^G\mbc$ corresponding to the projection $G/B\rightarrow G/P$ to identify $\mrInd_P^G \mbc$ with a subrepresentation of $\mrInd_{B}^G\mbc$.

	Despite the above description via parabolic inductions, we need an alternative way to define the Steinberg representation, which is closely related to the geometry of the Bruhat--Tits building.

	\subsection{The sign character from the Bruhat--Tits building}
	
	In order to get a Bruhat--Tits theoretical construction of the Steinberg representation, we first need to define the ``orientation sign character" $\epsilon_G$ of $G$ related to the Bruhat--Tits building $\mcb(G)$.
	
	First assume that $G$ is $K$-simple, then $\mcb(G)$ is simplicial. Fix a labelling of $\mcb(G)$, namely, a simplicial map  $\eta:\mcb(G)\rightarrow  \Delta_l$ from $\mcb(G)$ to the standard $l$-dimensional simplex $\Delta_l$ preserving the dimension, where $l$ denotes the dimension of a chamber of $\mcb(G)$. The existence of $\eta$ follows from the non-trivial fact that the Bruhat--Tits building $\mcb(G)$ is a labellable simplicial complex.  Fix a chamber $C$ of $\mcb(G)$ and let $\{\bs{p}_0,\cdots,\bs{p}_l\}$ denote the set of vertices of $C$. 
	
	We define a sign character $\epsilon_G:G\rightarrow\{\pm1\}$\index{$\epsilon_G$}, such that $\epsilon_G(g)$ is defined as the sign of the permutation 
		$$\begin{pmatrix}
			\eta(\bs{p}_0) & \eta(\bs{p}_1) & \cdots & \eta(\bs{p}_{l})\\ \eta(g\cdot \bs{p}_0)& \eta(g\cdot \bs{p}_1) & \cdots & \eta(g\cdot \bs{p}_{l})
		\end{pmatrix}$$
		in the symmetric group $\mfS_{l+1}$ on $l+1$ letters for $g\in G$. 
		
		We need to check that this sign character is well defined, namely, it does not depend on the choice of the chamber $C$, and it is a group homomorphism. When the underlying reductive group is split, this has already been done in \cite{broussous2014distinction}*{Lemma 2.1}. As in \emph{loc. cit.},  we only need to consider the subgroup $G^0$ of $G$, and show that the $G^0$-action preserves the labelling map $\eta$. So, if we replace $C$ by $g\cdot C$ with $g\in G^0$, then the value of $\epsilon_G$ remains unchanged. Since the $G^0$-action on $\ch{}{}(G)$ is transitive, $\epsilon_G$ is independent of $C$. Also, for $g_1,g_2\in G$, the sign of the permutation 
		$$\begin{pmatrix}
			\eta(\bs{p}_0) & \eta(\bs{p}_1) & \cdots & \eta(\bs{p}_{l})\\ \eta(g_1g_2\cdot \bs{p}_0)& \eta(g_1g_2\cdot \bs{p}_1) & \cdots & \eta(g_1g_2\cdot \bs{p}_{l})
		\end{pmatrix}$$
		equals the product of the sign of the permutation  
		$$\begin{pmatrix}
			\eta(\bs{p}_0) & \eta(\bs{p}_1) & \cdots & \eta(\bs{p}_{l})\\ \eta(g_2\cdot \bs{p}_0)& \eta(g_2\cdot \bs{p}_1) & \cdots & \eta(g_2\cdot \bs{p}_{l})
		\end{pmatrix}$$
		and the sign of the permutation $$\begin{pmatrix}
			\eta(g_2\cdot \bs{p}_0)& \eta(g_2\cdot \bs{p}_1) & \cdots & \eta(g_2\cdot \bs{p}_{l})\\
			\eta(g_1g_2\cdot \bs{p}_0)& \eta(g_1g_2\cdot \bs{p}_1) & \cdots & \eta(g_1g_2\cdot \bs{p}_{l})
		\end{pmatrix}.$$
		From the definition of $\epsilon_G$ and its independence of $C$, we have $\epsilon_G(g_1g_2)=\epsilon_G(g_1)\epsilon_G(g_2)$.

		To prove that $G^0$ preserves the labelling of $\mcb(G)$, we only need to show that $\mcb(G)$ is the building associated to some BN-pair related to the group $G^0$. This is indeed the case in \cite[Theorem 7.5.3]{KP23}.
		More precisely, let $S$ be a maximal $K$-split torus of $G$. Let $N=N_G(S)$ be the normalizer of $S$ in $G$. For a chamber $C$ in the apartment $\mca(G,S)$, let $I=\parah{C}$ be the corresponding Iwahori subgroup. Then,  $(G^0,N^0,I,\Delta^\mraff)$ is a so-called Iwahori--Tits system, where $N^0=G^0\cap N$, and $\Delta^\mraff=\Delta^\mraff(G,S,I)$ denotes the corresponding set of affine simple roots. Moreover, $\mcb(G)$ is the building associated to the Tits system $(G^0,N^0,I,\Delta^\mraff)$. Hence $G^0$ preserves the labelling of $\mcb(G)$.
		
		As a corollary, we have

		\begin{proposition}\label{resoforient} $\epsilon_G$ is trivial on $G^0$. In particular, let $F$ be a facet and $\parah{F}$ the parahoric subgroup associated to $F$, then $\epsilon_G$ is trivial on $\parah{F}$.
			
		\end{proposition}
		
		\begin{proof}
			
			Since $G^0$ preserves the labelling, $\epsilon_G$ is trivial on $G^0$. In particular, $\epsilon_G$ is trivial on the parahoric subgroup $P_{F} $ as the stabilizer of $F$ in $G^0$.
			
		\end{proof}
		
		In general, $G$ is isogenous to a product $G_1\times\dots\times G_k$ of adjoint $K$-simple groups $G_i,\forall i=1,\dots k$, i.e. there exists an surjective $K$-morphism $G\rightarrow G_1\times\dots\times G_k$ with a finite kernel. Then the tensor product  $\epsilon_{G_1}\boxtimes\dots\boxtimes\epsilon_{G_k}$, as a character of $G_1(K)\times\dots\times G_k(K)$, defines a character $\epsilon_G$ of $G(K)$ along the natural map $G(K)\rightarrow G_1(K)\times\dots\times G_k(K)$. Since $G^0$ is functorial by \cite[Remark 2.6.26]{KP23}, Proposition \ref{resoforient} is still satisfied.
		
		\begin{example}\label{exampleepsilonGL}
			
			In the case where $G=\mrgl_n(K)$, we have $G^0=\{g\in G\mid \mrdet(g)\in\mfo_K^\times\}$ and $G$ is generated by $G^0$ and the element	
			$$\Pi_K:=	\left(\begin{matrix}
				&1&&&\\
				&&1&&\\
				&&&\ddots&\\
				&&&&1\\
				\varpi_K&&&&
			\end{matrix}\right)\in \mrgl_n(K),$$
			where $\varpi_K$ is a uniformizer of $K$. Moreover, let $C$ be the chamber in $\mcb(G)$ such that $\parah{C}$ is the standard Iwahori subgroup with upper triangular elements and diagonal elements in $\mfo_K$ and lower triangular elements in $\mfp_K$. Then we have $\parah{C}^{\Pi_K}=\parah{C}$ and $\Pi_K\cdot C=C$. Moreover, let $\bs{p}$ be the vertex of $C$ such that the corresponding parahoric subgroup is $\mrgl_n(\mfo_F)$, then  $\{\bs{p},\Pi_K\cdot\bs{p},\dots,\Pi_K^{n-1}\cdot\bs{p}\}$ is the set of vertices of $C$, on which the $\Pi_K$-action corresponds to the permutation $1\mapsto 2\mapsto\dots\mapsto n\mapsto 1$. 
			Thus, $$\epsilon_G(\Pi_K)=(-1)^{n+1}\quad\text{and thus}\quad\epsilon_G(g)=(-1)^{(n+1)v_K(\det g)},$$
			where $v_K$ denotes the valuation of $K$. In particular for $H=\mro_n(K)$ or $\mrso_n(K)$, one has $\epsilon_G\rest_H=1$.
          \end{example}

		\subsection{The Steinberg representation as the space of harmonic cochains on $\mcb(G)$}
		
		Let $\ch{}{}(G)$ denote the set of chambers of $\mcb(G)$ and $\mbc[\mrch(G)]$ the space of $\mbc$-valued functions on $\ch{}{}(G)$. Consider the $G$-action on $\mbc[\mrch(G)]$ by
		$$g\cdot f(C)=\epsilon_G(g)f(g^{-1}\cdot C),\quad g\in G.$$ 
		
		\begin{definition}
			
			An element $f$ of $\mbc[\mrch(G)]$ is called \emph{harmonic} if $\sum _{D\subseteq \ol{C}}f(C)=0$ for any panel $D$ in $\mcb(G)$, where the sum ranges over the chambers $C$ having $D$ as a facet.
			
		\end{definition}
		
		Let $\mch(G)$\index{$\mch(G)$} be the $G$-subspace of $\mbc[\ch{}{}(G)]$ of harmonic cochains and $\mch^{\infty}(G)$\index{$\mch^{\infty}(G)$} the $G$-subspace $\mch(G)$ of smooth harmonic cochains. Here, a function $f\in\mbc[\ch{}{}(G)]$ is smooth if its stabilizer in $G$ is open.

		We have the following characterization of Steinberg representation.
		
		\begin{proposition}[Borel--Serre]\cite{broussous2014distinction}*{Proposition 3.2} The representation induced by the $G$-action on $\mch(G)$ is equivalent 
			to the algebraic dual of the Steinberg representation $\mrst_G$. The representation induced by the $G$-action on $\mch^{\infty}(G)$ is equivalent 
			to the Steinberg representation $\mrst_G$.
			
		\end{proposition}
		
		\section{Calculation in the rank one case}\label{sectionrankone}
		
		Let $\bs{k}$ be a finite field of odd characteristic and $\ul{G}$ a connected reductive group over $\bs{k}$, which is a quasi-split group due to a well-known result of Steinberg. Let $\ul{S}$ be a maximal $\bs{k}$-split torus of $\ul{G}$. Let $\ul{T}$ be the centralizer of $\ul{S}$ in $\ul{G}$, which is known to be a maximal torus of $\ul{G}$. We refer to \cite{bruhat1984groupes}*{\S 4.1} for various basic statements that are going to be tacitly used later on.
		
		We first discuss the case where $\ul{G}$ is a rank one pinned semi-simple group, i.e. a Weil restriction of scalars of $\mrsl(2)$ or $\mrsu(3)$. Our goal is two-fold:
		\begin{itemize}
			
			\item We classify all the involutions $\theta$ of $\ul{G}$ over $\bs{k}$ such that $\ul{S}$ is $\theta$-stable.
			
			\item Let $\ul{H}=(\ul{G}^{\theta})^{\circ}$. We classify the $\ul{H}$-conjugacy classes of Borel subgroups of $\ul{G}$. In other words, we classify the $\ul{H}$-conjugacy classes of chambers $\ch{}{}(\ul{G})$.
			
		\end{itemize}
		For general $\ul{G}$, we also explain how the above rank one calculation could be used.

		\subsection{The $\mrsl_{2}$-case}
		
		In this subsection, let $\bs{l}$ be a finite extension over $\bs{k}$. We consider the case where $\ul{G}$ is the Weil restriction of scalars $\mrres_{\bs{l}/\bs{k}}\mrsl_{2}$, which is a semi-simple group of rank one over $\bs{k}$.
		
		Let $\ul{S}=\{\mrdiag(x,x^{-1})\mid x\in\bs{k}^{\times}\}$. Then by definition $\ul{T}=\{\mrdiag(x,x^{-1})\mid x\in\bs{l}^{\times}\}$. Let $\Phi(\ul{G},\ul{S})=\{\alpha,-\alpha\}$ be the set of roots, and $\ul{U}_{\alpha}$ and $\ul{U}_{-\alpha}$ the corresponding unipotent subgroups consisting of unipotent upper triangular and lower triangular matrices respectively. Then $\ul{T}$, $\ul{U}_{\alpha}$ and $\ul{U}_{-\alpha}$ generate the full group $\ul{G}$.
		
		Let $\theta$ be an involution of $\ul{G}$ that stabilizes both $\ul{S}$ and $\ul{T}$. Then the restriction of $\theta$ to $\ul{S}$ is either the identity or the inversion. 
		
		We need the following lemma, which is interesting in its own right.
		
		\begin{lemma}\label{lemmathetaTrest}
			
			Let $\bs{l}/\bs{k}$ be a finite cyclic extension, $\ul{S}=\mbg_{m}$ the multiplicative group over $\bs{k}$, and  $\ul{T}=\mrres_{\bs{l}/\bs{k}}\mbg_{m}$ and $\theta$ an automorphism of finite order of $\ul{T}$ over $\bs{k
			}$ whose restriction to $\ul{S}$ is trivial. Then $\theta$ is induced by an element $\sigma$ in the Galois group $\mrgal(\bs{l}/\bs{k})$.
			
		\end{lemma}
		
		\begin{proof}
			
			Let $d=[\bs{l}:\bs{k}]$. The $\bs{k}$-structure of the torus $\ul{T}$ is determined by the character lattice $\mrhom(\ul{T}_{\bs{l}},(\mbg_{m})_{\bs
				{l}})\cong\mbz^{\oplus d}$ together with the action of $\mrgal(\bs{l}/\bs{k})$. In particular, the action of the generator $\sigma_{0}\in\mrgal(\bs{l}/\bs{k})$ is given by $$(x_{1},x_{2},\dots,x_{d})\mapsto(x_{2},x_{3}\dots,x_{d},x_{1}).$$ 
			Then, to define an automorphism $\theta$ of $\ul{T}$ over $\bs{k}$ is equivalent to defining an element $A\in\mrgl_{d}(\mbz)$ of finite order that is commutative with the permutation matrix $w$ related to $1\mapsto 2\mapsto\dots \mapsto d\mapsto 1$. This commutativity condition implies that $A$ is the linear combination $a_{0}I_{d}+a_{1}w+\dots+a_{d-1}w^{d-1}$ with $a_{0},\dots,a_{d-1}\in\mbz$. Write $P(X)=a_{0}+a_{1}X+\dots+a_{d-1}X^{d-1}$ as a polynomial in $\mbz[X]$. Since $X^{d}-1$ is the minimal polynomial of $w$, the condition $A^{N}=I_{d}$ implies that $X^{d}-1$ divides $P(X)^{N}-1$. We have the following characterization of $P(X)$.
			
			\begin{lemma}\label{polynomial}
				
				Let $P(X)=a_{0}+a_{1}X+\dots+a_{d-1}X^{d-1}$ be a polynomial in $\mbz[X]$, such that $X^{d}-1$ divides $P(X)^{N}-1$, then $P(X)=\pm X^{kd/d'}$ with $d'=\mrgcd(d,N)$ and $k=0,1,\dots,d'-1$.
				
			\end{lemma}
			
			\begin{proof}
				
				Let $\zeta_{d}$ be a $d$-th primitive root of unity. The condition means that there exist $N$-th roots of unity $\xi_{0},\dots,\xi_{d-1}$, such that $P(\zeta_{d}^{i})=\xi_{i}$ for $i=0,1,\dots,d-1$. By solving linear equations, this simply means that \begin{equation}\label{eqvandemonde}
					(\xi_{0}+\zeta_{d}^{-i}\xi_{1}+\dots+\zeta_{d}^{-i(d-1)}\xi_{d-1})/d=a_{i},\quad i=0,1,\dots,d-1.
				\end{equation}
				Taking the absolute value, the fact $a_{i}\in\mbz$ implies that $a_{i}$ equals $0$, $1$ or $-1$. Moreover, $a_{i}\neq 0$ if and only if vectors $(\xi_{0},\dots,\xi_{d-1})$ and $(1,\zeta_{d}^{i},\dots,\zeta_{d}^{i(d-1)})$ are in the same complex line. This could happen for at most one $i$, and only for those $i$ as a multiple of $d/d'$ with $d'=\mrgcd(d,N)$. As a result, there exists some $k\in\{0,1,\dots,d'-1\}$ such that $a_{kd/d'}=1$ or $-1$, and $a_{i}=0$ with $i\neq kd/d'$, which finishes the proof. 
				
			\end{proof}
			
			Using \Cref{polynomial} and considering the fact that the restriction of $\theta$ to $\ul{S}$ is the identity, there exists some $k\in\{0,1,\dots,d'-1\}$ such that $A=w^{kd/d'}$. It means that $\theta$ is induced by the Galois action $\sigma=\sigma_{0}^{kd/d'}$. 
			
		\end{proof}
		
		Now we classify all the possible involutions $\theta$ of $\underline{G}$.
		
		First we consider \textbf{Case (I)}, saying that $\theta\rest_{\ul{S}}=1$. Then $\theta$ stabilizes $\ul{U}_{\alpha}$ and $\ul{U}_{-\alpha}$. Using Lemma \ref{lemmathetaTrest}, $\theta\rest_{\ul{T}}$ is either trivial or the Galois action of order 2.
		
		\textbf{Case (I).(\romannumeral1):} Assume that $\theta\rest_{\ul{T}}$ is trivial on $\ul{T}$. There exists $x\in \bs{l}^{\times}$ such that $$\theta\big(\begin{pmatrix}
			1 & 1\\ 0 & 1
		\end{pmatrix}\big)=
		\begin{pmatrix}
			1 & x\\ 0 & 1
		\end{pmatrix}.$$
		Taking the $\ul{T}$-conjugation, we have that $$\theta\big(\begin{pmatrix}
			1 & u\\ 0 & 1
		\end{pmatrix}\big)=\begin{pmatrix}
			1 & xu\\ 0 & 1
		\end{pmatrix}$$ 
		for any $u\in\bs{l}$. Since $\theta$ is an involution, $x$ equals $1$ or $-1$. Also, since $\theta$ stabilizes the normalizer of $\ul{T}$, which is generated by $\ul{T}$ and $w_{0}=\begin{pmatrix}
			0 & 1\\ -1 & 0
		\end{pmatrix}$, it is easy to see that $\theta(w_{0})=\pm w_{0}$. Thus taking the $w_{0}$-conjugation, we similarly have $$\theta\big(\begin{pmatrix}
			1 & 0\\ u & 1
		\end{pmatrix}\big)=\begin{pmatrix}
			1 & 0\\ xu & 1
		\end{pmatrix}$$ 
		for any $u\in\bs{l}$. Since $\ul{T},\ul{U}_{\alpha}, \ul{U}_{-\alpha}$ generate $\ul{G}$, the involution $\theta$ is either the identity or the conjugation by $\begin{pmatrix}
			1 & 0\\ 0 & x
		\end{pmatrix}$ with $x=1$ or $-1$.  These two different cases will be referred to as \textbf{Case (I).(\romannumeral1).(1)} and \textbf{Case (I).(\romannumeral1).(2)} respectively later on.
		
		\textbf{Case (I).(\romannumeral2):} Assume that there exists a quadratic extension $\bs{l}/\bs{l_{0}}$ and an order 2 Galois action $\sigma_{0}\in\mrgal(\bs{l}/\bs{l}_{0})$, such that $\theta\rest_{\ul{T}}$ is given by the $\sigma_{0}$-action. Then similar to \textbf{Case (I).(\romannumeral1)}, we may show that there exists $x\in\bs{l}^{\times}$ satisfying $\sigma_{0}(x)x=1$, such that
		$$\theta(g)=\begin{pmatrix}
			1 & 0\\ 0 & x^{-1}
		\end{pmatrix}\sigma_{0}(g)\begin{pmatrix}
			1 & 0\\ 0 & x
		\end{pmatrix},\quad g\in\ul{G}.$$
		
		Now we consider \textbf{Case (II)}, saying that $\theta\rest_{\ul{S}}$ is the inversion. Then $\theta$ maps $\ul{U}_{\alpha}$ to $\ul{U}_{-\alpha}$. Using Lemma \ref{lemmathetaTrest}, $\theta\rest_{\ul{T}}$ is either the inversion or its composition with the Galois action of order 2.
		
		\textbf{Case (II).(\romannumeral1):} Assume that $\theta\rest_{\ul{T}}$ is the inversion on $\ul{T}$. Since $\theta$ maps $\ul{U}_{\alpha}$ to $\ul{U}_{-\alpha}$, there exists $x\in \bs{l}^{\times}$ such that $$\theta\big(\begin{pmatrix}
			1 & 1\\ 0 & 1
		\end{pmatrix}\big)=
		\begin{pmatrix}
			1 & 0\\ -x^{-1} & 1
		\end{pmatrix}.$$
		Taking the $\ul{T}$-conjugation and using the fact that $\theta$ is an involution, we have that $$\theta\big(\begin{pmatrix}
			1 & u\\ 0 & 1
		\end{pmatrix}\big)=\begin{pmatrix}
			1 & 0\\ -x^{-1}u & 1
		\end{pmatrix}\quad\text{and thus}\quad\theta\big(\begin{pmatrix}
			1 & 0\\ u & 1
		\end{pmatrix}\big)=\begin{pmatrix}
			1 & -xu\\ 0 & 1
		\end{pmatrix}$$ 
		for any $u\in\bs{l}$. Since $\ul{T},\ul{U}_{\alpha}, \ul{U}_{-\alpha}$ generate $\ul{G}$, we have 
		$$\theta(g)=\begin{pmatrix}
			1 & 0\\ 0 & x^{-1}
		\end{pmatrix}\,^{t}g^{-1}\begin{pmatrix}
			1 & 0\\ 0 & x
		\end{pmatrix},\quad g\in\ul{G}.$$
		We remark that $-x\in\bs{l}^{\times 2}$ and $-x\in\bs{l}^{\times}-\bs{l}^{\times 2}$ are essentially two different cases, which will be referred to as \textbf{Case (II).(\romannumeral1).(1)} and \textbf{Case (II).(\romannumeral1).(2)} respectively later on.
		
		\textbf{Case (II).(\romannumeral2):} Assume that there exists a quadratic extension $\bs{l}/\bs{l_{0}}$ and an order 2 Galois action $\sigma_{0}\in\mrgal(\bs{l}/\bs{l}_{0})$, such that $\theta\rest_{\ul{T}}$ is given by the the inversion on $\ul{T}$ composing with $\sigma_{0}$-action. Then similar to the above case, we may show that there exists $x\in\bs{l}_{0}^{\times}$ such that
		$$\theta(g)=\begin{pmatrix}
			1 & 0\\ 0 & x^{-1}
		\end{pmatrix}\sigma_{0}(\,^{t}g^{-1})\begin{pmatrix}
			1 & 0\\ 0 & x
		\end{pmatrix},\quad g\in\ul{G}.$$
		
		Then we study the $\ul{H}$-orbits of Borel subgroups in different cases. We let $q$ be the cardinality of $\bs{l}$, then there exist $q+1$'s different Borel subgroups of $\ul{G}$.
		
		\textbf{Case (I).(\romannumeral1).(1):}
		In this case, $\ul{H}$ equals the full group $\ul{G}$, thus $\ul{H}$ acts transitively on the set of Borel subgroups of $\ul{G}$.
		
		\textbf{Case (I).(\romannumeral1).(2):} In this case, $\ul{H}=\ul{T}$. Then $\ul{H}$ fixes the Borel group
		$$\ul{B}=\big\{\begin{pmatrix}
			\ast & \ast \\ 0 & \ast 
		\end{pmatrix}\big\}\quad\text{and}\quad \ul{B}^{-}=\big\{\begin{pmatrix}
			\ast & 0 \\ \ast & \ast 
		\end{pmatrix}\big\}.$$
		These two Borel subgroups are of $\theta$-rank 1, since the $\theta$-stable torus $\ul{S}$ is $\theta$-invariant.
		
		Now we consider $\ul{B}^{u_{x}}$ ranging over the rest of Borel subgroups, where $u_{x}:=\begin{pmatrix}
			1 & 0 \\ x & 1 
		\end{pmatrix}$ with $x$ ranging over $\bs{l}^{\times}$. Also, by direct calculation, it is easy to see that  $\ul{S}^{u_{x}}$ is $\theta$-stable and $(\ul{S}^{u_{x}})^{\theta}=\{\pm I_{2}\}$. As a result, $\ul{B}^{u_{x}}$ is of $\theta$-rank 0. For $h=\begin{pmatrix}
			y & 0 \\ 0 & y^{-1} 
		\end{pmatrix}$, we have
		$$\ul{B}^{u_{x}h}=\ul{B}^{h^{-1}u_{x}h}=\ul{B}^{u_{xy^{2}}}.$$
		So the $\ul{H}$-orbit of $\ul{B}^{u_{x}}$ is $\{\ul{B}^{u_{xy^{2}}}\mid y\in\bs{l}^{\times}\}$. There are exactly two $\ul{H}$-orbits $$\{\ul{B}^{u_{y^{2}}}\mid y\in\bs{l}^{\times}\}\quad \text{and} \quad\{\ul{B}^{u_{\epsilon y^{2}}}\mid y\in\bs{l}^{\times}\}\ \text{with}\ \epsilon\in\bs{l}^{\times}-\bs{l}^{\times 2}$$ of $\theta$-rank 0, of cardinality $(q-1)/2$ and $(q-1)/2$ respectively. 
		
		To sum up, in this case there are four $\ul{H}$-orbits of Borel subgroups of $\theta$-rank $1,1,0,0$ and of cardinality $1,1,(q-1)/2,(q-1)/2$ respectively.
		
		\textbf{Case (I).(\romannumeral2):} Recall that 
		$$\theta(g)=\begin{pmatrix}
			1 & 0\\ 0 & x^{-1}
		\end{pmatrix}\sigma_{0}(g)\begin{pmatrix}
			1 & 0\\ 0 & x
		\end{pmatrix},\quad g\in\ul{G}$$
		for some $x\in\bs{l}_{0}^{\times}$. Since $x=\sigma_{0}(y)y$ for some $y\in\bs{l}^{\times}$, up to conjugate by $\begin{pmatrix}
			1 & 0 \\ 0 & y
		\end{pmatrix}$, we may essentially assume that $\theta=\sigma_{0}$. In this case, $\ul{H}=\mrsl_{2}(\bs{l}_{0})$. 
		
		The set of Borel subgroups in this case is in bijection with the projective line $\mbp^{1}(\bs{l})$, and the corresponding $\ul{G}$-action is given by the usual matrix product on the homogeneous coordinates. 
		
		\begin{lemma}
			
			The $\mbp^{1}(\bs{l})$ is stratified into two $\mrsl_{2}(\bs{l}_{0})$-orbits $\mbp^{1}(\bs{l}_{0})$ and $\mbp^{1}(\bs{l})-\mbp^{1}(\bs{l}_{0})$.
			
		\end{lemma}
		
		\begin{proof}
			The proof can be obtained by direct calculation and is omitted.
		\end{proof} 
		Thus the first orbit is of cardinality $\sqrt{q}+1$ and contains the Borel subgroup consisting of upper triangular matrices. It is of $\theta$-rank 1. 
		
		The second orbit is of cardinality $q-\sqrt{q}$. Moreover, we consider $a,b\in\bs{l}$ such that $\sigma_{0}(a)b-\sigma_{0}(b)a=1$ and $g=\begin{pmatrix} a & b \\ \sigma_{0}(a) & \sigma_{0}(b) \end{pmatrix}$, then by definition $\sigma_{0}(g)g^{-1}=\begin{pmatrix} 0 & 1 \\ 1 & 0 \end{pmatrix}$. Thus it is easy to see that $\ul{S}^{g}$ is $\theta$-stable and of $\theta$-rank 0. Hence the second orbit consists of Borel subgroups of $\theta$-rank 0.
		
		To sum up, there are exactly two orbits of Borel subgroups of $\theta$-rank $1$, $0$ and of cardinality $\sqrt{q}+1$ and $q-\sqrt{q}$ respectively. 
		
		\textbf{Case (II).(\romannumeral1).(1)}: Recall that $$\theta(g)=\begin{pmatrix}
			1 & 0\\ 0 & x^{-1}
		\end{pmatrix}\,^{t}g^{-1}\begin{pmatrix}
			1 & 0\\ 0 & x
		\end{pmatrix},\quad g\in\ul{G}$$
		with $-x\in\bs{l}^{\times 2}$, so $\ul{H}$ is a special orthogonal subgroup of $\ul{G}$. It is easily seen that we may pick an element $g'\in \mrgl_{2}(\bs{l})$ such that $$\,^{t}g'\begin{pmatrix}
			1 & 0 \\ 0 & x
		\end{pmatrix}g'=\begin{pmatrix}
			0 & 1 \\ 1 & 0
		\end{pmatrix},$$ thus taking the $g'$-conjugation, the involution $\theta$ becomes
		$$\theta'(g):=\begin{pmatrix}
			0 & 1 \\ 1 & 0
		\end{pmatrix}\,^{t}g^{-1}\begin{pmatrix}
			0 & 1 \\ 1 & 0
		\end{pmatrix}=\begin{pmatrix}
			1 & 0 \\ 0 & -1
		\end{pmatrix}g\begin{pmatrix}
			1 & 0 \\ 0 & -1
		\end{pmatrix},\quad g\in\ul{G},$$
		which is indeed the involution in \textbf{Case (I).(\romannumeral1).(2)}. 
		
		Using the result there, we know that there are four $\ul{H}$-orbits of Borel subgroups of $\theta$-rank $0,0,1,1$ and of cardinality $(q-1)/2,(q-1)/2,1,1$ respectively. In particular, if we write $-x=a^2$, then by direct calculation there exists some element $$w_{z}=\begin{pmatrix}
			0 & -z \\ z^{-1} & 0
		\end{pmatrix},\ z\in\bs{l}^{\times}$$
		in $\ul{H}$ if and only if $-1=a^{2}z^{-2}$ for some $z$, or if and only if $q\equiv 1\ (\text{mod}\ 4)$. Thus  the standard Borel subgroup $\ul{B}$ and its opposite $\ul{B}^{-}=\ul{B}^{w_{z}}$ are in the same $\ul{H}$-orbit if and only if $q\equiv 1\ (\text{mod}\ 4)$.
		
		\textbf{Case (II).(\romannumeral1).(2)}: Recall that $$\theta(g)=\begin{pmatrix}
			1 & 0\\ 0 & x^{-1}
		\end{pmatrix}\,^{t}g^{-1}\begin{pmatrix}
			1 & 0\\ 0 & x
		\end{pmatrix},\quad g\in\ul{G}$$
		with $-x\in\bs{l}^{\times}-\bs{l}^{\times 2}$, so $\ul{H}$ is a special orthogonal subgroup of $\ul{G}$. Since $\ul{S}$ is $\theta$-split, the Borel subgroup $\ul{B}=\big\{\begin{pmatrix}
			\ast & \ast \\ 0 & \ast 
		\end{pmatrix}\big\}$ is of $\theta$-rank $0$. 
		
		Let $u_{0}:=\begin{pmatrix}
			1 & 0 \\ x_{0} & 1 
		\end{pmatrix}$ for some $x_{0}\in\bs{l}$. We first study $u=\begin{pmatrix}
			1 & 0 \\ y_{0} & 1 
		\end{pmatrix}$, $y_{0}\in\bs{l}$ such that $\ul{B}^{u_{0}}$ and $\ul{B}^{u_{0}u}$ are in the same $\ul{H}$-orbit. It requires to find $b\in\ul{B}$ such that $h=u_{0}^{-1}bu_{0}u\in\ul{H}$, and thus $\ul{B}^{u_{0}u}=\ul{B}^{u_{0}h}$. So we have
		\begin{equation}\label{eqcaseIIi2eq1}
			\theta(u_{0}^{-1}bu_{0}u)(u_{0}^{-1}bu_{0}u)^{-1}=\varepsilon^{-1}\,^{t}u_{0}\,^{t}b^{-1}\,^{t}u_{0}^{-1}\,^{t}u^{-1}\varepsilon u^{-1}u_{0}^{-1}b^{-1}u_{0}=1,
		\end{equation}
		where $\varepsilon=\mrdiag(1,x)$. Solving \eqref{eqcaseIIi2eq1}, it requires to find $b\in\ul{B}$ such that
		\begin{equation}\label{eqcaseIIi2eq2}
			\,^{t}b^{-1}\begin{pmatrix}1+x(x_{0}+y_{0})^{2} & -x(x_{0}+y_{0})\\ -x(x_{0}+y_{0}) & x\end{pmatrix}b^{-1}=\begin{pmatrix}1+xx_{0}^{2} & -xx_{0}\\ -xx_{0} & x\end{pmatrix}.
		\end{equation}
		In particular, $1+x(x_{0}+y_{0})^{2}$ and $1+xx_{0}^{2}$ are not 0 from our assumption on $x$.
		
		If $1+x(x_{0}+y_{0})^{2}/(1+xx_{0}^{2})\in\bs{l}^{\times 2}$, then by taking elementary transformations it is clear that there exists such $b\in\ul{B}$. Otherwise if $1+x(x_{0}+y_{0})^{2}/(1+xx_{0}^{2})\in\bs{l}^{\times}-\bs{l}^{\times 2}$ such $b$ does not exist.
		
		Moreover, $\ul{B}$ and its opposite $\ul{B}^{-}$ are in the same $\ul{H}$-orbit if and only if there exists $z\in\bs{l}^{\times}$ such that $\begin{pmatrix}
			0 & z \\ -z^{-1} & 0	
		\end{pmatrix}\in \ul{H}$, and if and only if $x$ is a square in $\bs{l}^{\times}$.
		
		Now we divide our discussion into two cases. 
		
		\begin{itemize}
			\item First we assume $q\equiv 1\ (\text{mod}\ 4)$. In this case $-x$ and $x$ are not squares in $\bs{l}$. Then by direct calculation $1+xy_{0}^{2}\in\bs{l}^{\times 2}$ has $(q+1)/2$ solutions, and $1+xy_{0}^{2}\in\bs{l}^{\times}-\bs{l}^{\times2}$ has $(q-1)/2$ solutions with $y_{0}$ ranging over $\bs{l}$. Moreover, $\ul{B}$ and $\ul{B}^{-}$ are not in the same $\ul{H}$-orbit. So there are exactly two $\ul{H}$-orbits of cardinality $(q+1)/2$ and $(q+1)/2$ respectively. The one containing $\ul{B}$ is of $\theta$-rank 0, and we show that the other one is also of $\theta$-rank 0. Indeed, we want to find $b\in\ul{B}$ such that $\ul{S}^{bu_{0}}$ is a $\theta$-split torus, and thus $\ul{B}^{u_{0}}$ is of $\theta$-rank 0. This is equivalent to finding $b$ such that $\theta(bu_{0})(bu_{0})^{-1}\in\ul{T}$, or equivalently $\,^{t}b^{-1}\begin{pmatrix}1+xx_{0}^{2} & -xx_{0}\\ -xx_{0} & x\end{pmatrix}b\in \ul{T}$. Since $1+xx_{0}^{2} \neq 0$, such $b$ can always be found.
			
			\item Secondly we assume $q\equiv 3\ (\text{mod}\ 4)$. Then $-x$ is not a square whereas $x$ is a square in $\bs{l}$. Then by direct calculation $1+xy_{0}^{2}\in\bs{l}^{\times 2}$ has $(q-1)/2$ solutions, and $1+xy_{0}^{2}\in\bs{l}^{\times}-\bs{l}^{\times2}$ has $(q+1)/2$ solutions with $y_{0}$ ranging over $\bs{l}$. Moreover, $\ul{B}$ and $\ul{B}^{-}$ are in the same $\ul{H}$-orbit. So there are exactly two $\ul{H}$-orbits of cardinality $(q+1)/2$ and $(q+1)/2$ respectively, and we prove as above that they are of $\theta$-rank 0.
		\end{itemize}

		To sum up, in this case there are two $\ul{H}$-orbits of Borel subgroups. They are of $\theta$-rank 0,0 and of cardinality $(q+1)/2$ and $(q+1)/2$ respectively. 
		
		\textbf{Case (II).(\romannumeral2)}: Recall that 
		$$\theta(g)=\begin{pmatrix}
			1 & 0\\ 0 & x^{-1}
		\end{pmatrix}\sigma_{0}(\,^{t}g^{-1})\begin{pmatrix}
			1 & 0\\ 0 & x
		\end{pmatrix},\quad g\in\ul{G}$$ 
		for a certain $x\in\bs{l}_{0}^{\times}$. In this case $\ul{H}$ is a special unitary subgroup of $\ul{G}$. It is easily seen that we may pick an element $g'\in \mrgl_{2}(\bs{l})$ such that $$\sigma_{0}(\,^{t}g')\begin{pmatrix}
			1 & 0 \\ 0 & x
		\end{pmatrix}g'=\begin{pmatrix}
			0 & -1 \\ 1 & 0
		\end{pmatrix},$$
		then taking the $g'$-conjugation the involution $\theta$ becomes
		$$\theta'(g):=\begin{pmatrix}
			0 & 1 \\ -1 & 0
		\end{pmatrix}\sigma_{0}(\,^{t}g^{-1})\begin{pmatrix}
			0 & -1 \\ 1 & 0
		\end{pmatrix}=\sigma_{0}(g)\in\ul{G},$$
		which is indeed the involution in \textbf{Case (I).(\romannumeral2)}. 
		
		Using the result there, we know that there are two $\ul{H}$-orbits of Borel subgroups of $\theta$-rank $0,1$ and of cardinality $q-\sqrt{q},\sqrt{q}+1$ respectively.
		
		We sum up our result in the following proposition.
		
		\begin{proposition}\label{proprank1classSL2}
			
			Let $\ul{G}$ be the Weil restriction of scalars $\mrres_{\bs{l}/\bs{k}}\mrsl_{2}$ and let $\theta$, $\ul{H}$, $\ul{S}$, $\ul{T}$ be defined as above, then we exhaust the possible involution $\theta$ into the following cases: \footnote{In our convention, in each case the first $\ul{H}$-orbit of Borel subgroups listed below is exactly the one containing the upper triangular Borel subgroup.}
			
			\begin{itemize}
				\item \textbf{Case (I).(\romannumeral1).(1):} $\theta$ is the identity, $\ul{H}$ is the whole group $\ul{G}$, and there is exactly one $\ul{H}$-orbit.
				\item \textbf{Case (I).(\romannumeral1).(2):} $\theta$ is the conjugation given by $\mrdiag(1,-1)$, $\ul{H}$ is the diagonal torus $\ul{T}$, and there are  four $\ul{H}$-orbits of Borel subgroups,  of $\theta$-rank $1,1,0,0$ and of cardinality $1,1,(q-1)/2,(q-1)/2$ respectively.
				\item \textbf{Case (I).(\romannumeral2):} The involution $\theta$ is given by 
				$$\theta(g)=\begin{pmatrix}
					1 & 0\\ 0 & x^{-1}
				\end{pmatrix}\sigma_{0}(g)\begin{pmatrix}
					1 & 0\\ 0 & x
				\end{pmatrix},\quad g\in\ul{G}$$
				with $x\in\bs{l}_{0}^{\times}$, and $\ul{H}$ is a conjugation of $\mrsl_{2}(\bs{l}_{0})$. There are exactly two orbits of Borel subgroups of $\theta$-rank $1$, $0$ and of cardinality $\sqrt{q}+1$ and $q-\sqrt{q}$ respectively. 
				\item \textbf{Case (II).(\romannumeral1).(1)}: The involution $\theta$ is given by $$\theta(g)=\begin{pmatrix}
					1 & 0\\ 0 & x^{-1}
				\end{pmatrix}\,^{t}g^{-1}\begin{pmatrix}
					1 & 0\\ 0 & x
				\end{pmatrix},\quad g\in\ul{G}$$
				with $-x\in\bs{l}^{\times 2}$, and $\ul{H}$ is a special orthogonal subgroup of $\ul{G}$. There are four $\ul{H}$-orbits of Borel subgroups of $\theta$-rank $0,0,1,1$ and of cardinality $(q-1)/2,(q-1)/2,1,1$ respectively.
				
				\item \textbf{Case (II).(\romannumeral1).(2)}: The involution $\theta$ is given by
				$$\theta(g)=\begin{pmatrix}
					1 & 0\\ 0 & x^{-1}
				\end{pmatrix}\,^{t}g^{-1}\begin{pmatrix}
					1 & 0\\ 0 & x
				\end{pmatrix},\quad g\in\ul{G}$$
				with $-x\in\bs{l}^{\times}-\bs{l}^{\times 2}$, and $\ul{H}$ is a special orthogonal subgroup of $\ul{G}$. There are two $\ul{H}$-orbits of Borel subgroups of $\theta$-rank $0,0$ and of cardinality $(q+1)/2$ and $(q+1)/2$ respectively. 
				
				\item \textbf{Case (II).(\romannumeral2)}: The involution $\theta$ is given by
				$$\theta(g)=\begin{pmatrix}
					1 & 0\\ 0 & x^{-1}
				\end{pmatrix}\sigma_{0}(\,^{t}g^{-1})\begin{pmatrix}
					1 & 0\\ 0 & x
				\end{pmatrix},\quad g\in\ul{G}$$ 
				with $x\in\bs{l}_{0}^{\times}$, and $\ul{H}$ is a special unitary subgroup of $\ul{G}$. There are two $\ul{H}$-orbits of Borel subgroups of $\theta$-rank $0,1$ and of cardinality $q-\sqrt{q},\sqrt{q}+1$ respectively.
				
			\end{itemize}
			
		\end{proposition}
		
		\subsection{The $\mrsu_{3}$-case} 
		
		In this subsection, let $\bs{l}$ be a finite extension over $\bs{k}$ and $\bs{l}_{2}/\bs{l}$ a quadratic extension. Let $\sigma\in\mrgal(\bs{l}_{2}/\bs{l})$ be the order two involution. We consider the case where $\ul{G}$ is the Weil restriction of scalars $\mrres_{\bs{l}/\bs{k}}\mrsu_{3}$, more precisely $$\mrsu_{3}(\bs{l})=\{g\in\mrgl_{3}(\bs{l}_{2})\mid w_{0}\sigma(\,^{t}g^{-1})w_{0}=g\}\quad\text{with}\quad w_{0}=\begin{pmatrix}0 &0 &1\\ 0 & 1 & 0\\ 1 & 0 & 0\end{pmatrix}.$$
		Thus $\ul{G}$ is a semi-simple group of rank one over $\bs{k}$. Let $\ul{S}=\{\mrdiag(x,1,x^{-1})\mid x\in\bs{k}^{\times}\}$. Then by definition $\ul{T}=\{\mrdiag(y,\sigma(y)y^{-1},\sigma(y)^{-1})\mid y\in\bs{l}_{2}^{\times}\}$. Let $\Phi(\ul{G},\ul{S})=\{\alpha,-\alpha,2\alpha,-2\alpha\}$ be the set of roots, and $\ul{U}_{\alpha}$, $\ul{U}_{-\alpha}$, $\ul{U}_{2\alpha}$, $\ul{U}_{-2\alpha}$ the corresponding unipotent subgroups. More precisely,
		\begin{align*}
			&\ul{U}_{2\alpha}=\bigg\{\begin{pmatrix} 1 & 0 & -v_2\\ 0 & 1 & 0\\ 0 & 0 & 1 
			\end{pmatrix}\mid v_2\in\bs{l}_{2},\ \sigma(v_2)+v_2=0\bigg\}\\
			\subset\quad & \ul{U}_{\alpha}=\bigg\{\begin{pmatrix} 1 & -\sigma(v_1) & -v_2\\ 0 & 1 & v_1\\ 0 & 0 & 1 
			\end{pmatrix}\mid v_1, v_2\in\bs{l}_{2},\ \sigma(v_1)v_1=\sigma(v_2)+v_2\bigg\},
		\end{align*}
		and $\ul{U}_{-2\alpha}\subset\ul{U}_{-\alpha}$ are defined by taking the transpose. Still $\ul{T}$, $\ul{U}_{\alpha}$ and $\ul{U}_{-\alpha}$ generate the full group $\ul{G}$.
		
		Let $\theta$ be an involution of $\ul{G}$ that stabilizes both $\ul{S}$ and $\ul{T}$. Then the restriction of $\theta$ to $\ul{S}$ is either the identity of the inversion. 
		
		First we consider \textbf{Case (I)}, saying that $\theta\rest_{\ul{S}}=1$. Then $\theta$ stabilizes $\ul{U}_{\alpha}$, $\ul{U}_{-\alpha}$, $\ul{U}_{2\alpha}$, $\ul{U}_{-2\alpha}$. Using Lemma \ref{lemmathetaTrest}, $\theta\rest_{\ul{T}}$ is either trivial or the Galois action of order 2.
		
		\textbf{Case (I).(\romannumeral1):} Assume that $\theta\rest_{\ul{T}}$ is trivial on $\ul{T}$. Considering the $\ul{T}$-conjugation and the fact that $\theta$ is an involution, we may find $\epsilon,\delta\in\bs{l}_{2}^{\times}$, such that
		$$\theta\bigg(\begin{pmatrix} 1 & -\sigma(v_1) & -v_2\\ 0 & 1 & v_1\\ 0 & 0 & 1 
		\end{pmatrix}\bigg)=\begin{pmatrix} 1 & -\sigma(\epsilon v_1) & -\delta v_2\\ 0 & 1 & \epsilon v_1\\ 0 & 0 & 1 
		\end{pmatrix},\quad v_1, v_2\in\bs{l}_{2}.$$
		We omit the routine detail. Since $$\sigma(v_1)v_1=\sigma(v_2)+v_2\quad\text{and}\quad\sigma(\epsilon v_1)\epsilon v_1=\sigma(\delta v_2)+\delta v_2$$ for any $v_1,v_2$, we necessarily have
		$$\sigma(\delta)=\delta=\sigma(\epsilon)\epsilon\in\bs{l}^{\times}.$$
		Since $\theta$ is an involution, we further have $\epsilon^{2}=\delta^{2}=1$, implying that $\delta=1$ and $\epsilon=\pm1$. Also, $\theta$ stabilizes the normalizer of $\ul{T}$, which is generalized by $\ul{T}$ and $w_{0}':=-w_{0}$. Thus $\theta(w_0')=tw_0'$ for some $t\in\ul{T}$, $t^2=1$.
		Taking the $w_{0}'$-conjugation, we have
		$$\theta\bigg(\begin{pmatrix} 1 & 0 & 0 \\ -\sigma(v_1) & 1 & 0\\ -v_2 & v_1 & 1 
		\end{pmatrix}\bigg)=\begin{pmatrix} 1 & 0 & 0\\ -\sigma(\epsilon' v_1) & 1 & 0\\ -\delta' v_2 & \epsilon' v_1 & 1 
		\end{pmatrix},\quad v_1, v_2\in\bs{l}_{2}$$
		for some $\delta'=1$ and $\epsilon'=\pm 1$. Moreover, fix $v\in\bs{l}_2^{\times}$ such that $\sigma(v)+v=0$. Taking the involution $\theta$ on the equation
		$$\begin{pmatrix} 1 & 0 & 0\\ 0 & 1 & 0\\ -v^{-1} & 0 & 1 
		\end{pmatrix}\cdot\begin{pmatrix} 1 & 0 & v\\ 0 & 1 & 0\\ 0 & 0 & 1 
		\end{pmatrix}\cdot\begin{pmatrix} 1 & 0 & 0\\ 0 & 1 & 0\\ -v^{-1} & 0 & 1 
		\end{pmatrix}=\begin{pmatrix} 0 & 0 & v\\ 0 & 1 & 0\\ -v^{-1} & 0 & 0
		\end{pmatrix},$$
		we get $\theta(w_0')=w_0'$. Thus $\epsilon=\epsilon'$. 
		
		Since $\ul{T}$, $\ul{U}_{\alpha}$ and $\ul{U}_{-\alpha}$ generate $\ul{G}$, the involution $\theta$ is either the identity or the conjugation by $\mathrm{diag}(1,-1,1)$, depending on whether $\epsilon=1$ or $-1$. These two different cases will be referred to as \textbf{Case (I).(\romannumeral1).(1)} and \textbf{Case (I).(\romannumeral1).(2)} respectively later on.
		
		\textbf{Case (I).(\romannumeral2):} Assume that $\theta\rest_{\ul{T}}$ is given by the $\sigma$-action. Then similar to the above case, we may show that there exists $x\in\bs{l}_{2}^{\times}$, $\sigma(x)x=1$, such that
		$$\theta(g)=\mathrm{diag}(1,x^{-1},1)\sigma(g)\mathrm{diag}(1,x,1),\quad g\in \ul{G}.$$
		
		Now we consider \textbf{Case (II)}, saying that $\theta\rest_{\ul{S}}$ is the inversion. Then $\theta$ maps $\ul{U}_{\alpha}$ to $\ul{U}_{-\alpha}$ and $\ul{U}_{2\alpha}$ to $\ul{U}_{-2\alpha}$. Using Lemma \ref{lemmathetaTrest}, $\theta\rest_{\ul{T}}$ is either the inversion or its composition with the Galois action $\sigma$.
		
		\textbf{Case (II).(\romannumeral1):} Assume that $\theta\rest_{\ul{T}}$ is the inversion on $\ul{T}$ composing with the Galois action $\sigma$. Considering the $\ul{T}$-conjugation and the fact that $\theta$ is an involution, we have that
		$$\theta\bigg(\begin{pmatrix} 1 & -\sigma(v_1) & -v_2\\ 0 & 1 & v_1\\ 0 & 0 & 1 
		\end{pmatrix}\bigg)=\begin{pmatrix} 1 & 0 & 0\\ \epsilon v_1 & 1 & 0\\ -\delta v_2 & -\sigma(\epsilon v_1) & 1 
		\end{pmatrix},\quad v_1, v_2\in\bs{l}_{2},$$
		for certain $\epsilon,\delta\in\bs{l}^{\times}$ with $\delta=\epsilon^{2}$. We omit the detail which is routine. Since $\ul{T}$, $\ul{U}_{\alpha}$ and $\ul{U}_{-\alpha}$ generate $\ul{G}$, we have
		$$\theta(g)=\begin{pmatrix} 0 & 0 & \delta^{-1}\\ 0 & \epsilon\delta^{-1} & 0\\ 1 & 0 & 0
		\end{pmatrix}g\begin{pmatrix} 0 & 0 & 1\\ 0 & \epsilon^{-1}\delta & 0\\ \delta  & 0 & 0
		\end{pmatrix},\quad g\in\ul{G}.$$
		Notice that $\begin{pmatrix} 0 & 0 & \delta^{-1}\\ 0 & \epsilon\delta^{-1} & 0\\ 1 & 0 & 0 
		\end{pmatrix}\begin{pmatrix} 0 & 0 & \delta^{-1}\\ 0 & \epsilon\delta^{-1} & 0\\ 1 & 0 &  0
		\end{pmatrix}=\delta^{-1}I_{3}$, such $\theta$ is indeed an involution.
		
		\textbf{Case (II).(\romannumeral2):} Assume that $\theta\rest_{\ul{T}}$ is the inversion on $\ul{T}$. Considering the $\ul{T}$-conjugation and the fact that $\theta$ is an involution, we have that
		$$\theta\bigg(\begin{pmatrix} 1 & -\sigma(v_1) & -v_2\\ 0 & 1 & v_1\\ 0 & 0 & 1 
		\end{pmatrix}\bigg)=\begin{pmatrix} 1 & 0 & 0\\ \epsilon \sigma(v_1) & 1 & 0\\ -\delta \sigma(v_2) & -\sigma(\epsilon) v_1 & 1 
		\end{pmatrix},\quad v_1\in\bs{l}_{2}^{\times},\ v_2\in\bs{l}_{2},$$
		for certain $\epsilon,\delta\in\bs{l}_{2}^{\times}$ with $\sigma(\delta)=\delta=\sigma(\epsilon)\epsilon\in\bs{l}$. Still we omit the detail which is routine. Since $\ul{T}$, $\ul{U}_{\alpha}$ and $\ul{U}_{-\alpha}$ generate $\ul{G}$, we have
		$$\theta(g)=\begin{pmatrix} 0 & 0 & \delta^{-1}\\ 0 & \epsilon\delta^{-1} & 0\\ 1 & 0 & 0
		\end{pmatrix}\sigma(g)\begin{pmatrix} 0 & 0 & 1\\ 0 & \epsilon^{-1}\delta & 0\\ \delta  & 0 & 0
		\end{pmatrix},\quad g\in\ul{G}.$$
		Notice that $\begin{pmatrix} 0 & 0 & \delta^{-1}\\ 0 & \epsilon\delta^{-1} & 0\\ 1 & 0 & 0 
		\end{pmatrix}\sigma\bigg(\begin{pmatrix} 0 & 0 & \delta^{-1}\\ 0 & \epsilon\delta^{-1} & 0\\ 1 & 0 &  0
		\end{pmatrix}\bigg)=\delta^{-1}I_{3}$, such $\theta$ is indeed an involution.
		
		Then we study the $\ul{H}$-orbits of Borel subgroups in different cases. We let $q$ be the cardinality of $\bs{l}$, then there exist $q^{3}+1$'s different Borel subgroups of $\ul{G}$.
		
		\textbf{Case (I).(\romannumeral1).(1):} In this case, $\ul{H}$ equals the full group $\ul{G}$, thus $\ul{H}$ acts transitively on the set of Borel subgroups of $\ul{G}$.
		
		\textbf{Case (I).(\romannumeral1).(2):} In this case, $\theta$ is given by conjugation by $t:=\mrdiag(1,-1,1)$, and indeed $\ul{H}$ is isomorphic to the unitary group $\mru_{2}(\bs{l})$. Let $\ul{B}$ be the upper triangular Borel subgroup of $\ul{G}$ and $\ul{B}^{-}$ its opposite. Since $\begin{pmatrix}
			0 & 0 & 1\\ 0 & -1 & 0 \\ 1 & 0 & 0
		\end{pmatrix}\in \ul{H}$, the two Borel subgroups $\ul{B}$ and $\ul{B}^{-}$ are in the same $\ul{H}$-orbit. Write 
		$$w_{0}=\begin{pmatrix}
			0 & 0 & 1\\ 0 & 1 & 0\\ 1 & 0 & 0
		\end{pmatrix},\ u_{0}=\begin{pmatrix}1 & 0 & 0 \\ \sigma(x_{0}) & 1 & 0 \\ 
			-y_{0} & -x_{0} & 1
		\end{pmatrix}\quad\text{and}\quad u=\begin{pmatrix}1 & 0 & 0 \\ \sigma(x_{1}) & 1 & 0 \\ 
			-y_{1} & -x_{1} & 1
		\end{pmatrix},$$
		where $\sigma(x_{i})x_{i}=\sigma(y_{i})+y_{i},\ i=0,1.$
		
		We study the condition such that $\ul{B}^{u_{0}}$ and $\ul{B}^{u_{0}u}$ are in the same $\ul{H}$-orbit, equivalently for some $b\in\ul{B}$,
		$$\theta(u_{0}^{-1}bu_{0}u)=u_{0}^{-1}bu_{0}u,$$ 
		or equivalently $$b^{-1}u_{0}tu_{0}^{-1}b=u_{0}utu^{-1}u_{0}^{-1}.$$
		Using the formula
		$b^{-1}=w_{0}\sigma(\,^{t}b)w_{0}$, we further have
		$$\sigma(\,^{t}b)w_{0}u_{0}tu_{0}^{-1}b=w_{0}u_{0}utu^{-1}u_{0}^{-1},$$
		or equivalently
		\begin{equation}\label{eqIi2eq1}
			\sigma(\,^{t}b)\begin{pmatrix}
				-2\sigma(x_{0})x_{0} & 2x_{0} & 1\\
				2\sigma(x_{0}) & -1 & 0\\
				1 & 0 & 0
			\end{pmatrix}b=\begin{pmatrix}
				-2\sigma(x_{0}+x_{1})(x_{0}+x_{1}) & 2(x_{0}+x_{1}) & 1\\
				2\sigma(x_{0}+x_{1}) & -1 & 0\\
				1 & 0 & 0
			\end{pmatrix}.
		\end{equation}
		
		If $x_{0}=0$, then to find such $b$ the only possibility is $x_{1}=0$. In other words, the $\ul{H}$-orbit of $\ul{B}$ consists of $\ul{B}^{-}$ and $\ul{B}^{u}$ with $u$ such that $x_{1}=0$ and $y_{1}\in\bs{l}_{2}$ with $\sigma(y_{1})+y_{1}=0$. It is of cardinality $q+1$ and $\theta$-rank $1$.
		
		If $x_{0}\neq 0$, then to find a solution $b\in\ul{B}$ we must have $x_{0}+x_{1}\neq 0$. Using elementary row and column transformations, we may show that such $b$ always exists. Thus those $\ul{B}^{u}$ with $x_{1}\neq 0$ lie in the same $\ul{H}$-orbit, which is of cardinality $q^{3}-q$. Moreover, we have the following lemma.
		
		\begin{lemma}\label{lemmadiagunitary}
			
			For any $t'=\mrdiag(\delta,\epsilon,1)$ with  $\delta=\epsilon^{2}\in\bs{l}^{\times}$ and any $u$ as above such that $y_{1}+\epsilon\neq 0$, We may find $b\in\ul{B}$ such that
			$$t'^{-1}w_{0}buw_{0}t'(bu)^{-1}=t'^{-1}\sigma(\,^{t}b^{-1})\sigma(\,^{t}u^{-1})t'u^{-1}b^{-1}=t.$$
			
		\end{lemma}  
		
		\begin{proof}
			
			By direct calculation, we need to find $b$ such that
			$$\sigma(\,^{t}b^{-1})\begin{pmatrix} (\sigma(y_{1})+\epsilon)(y_{1}+\epsilon) & -x_{1}(y_{1}+\epsilon) & -y_{1} \\ -\sigma(x_{1})(\sigma(y_{1})+\epsilon) & \epsilon+\sigma(x_{1})x_{1} & 0 \\ -\sigma(y_{1}) & 0 & 0 \end{pmatrix}b^{-1}=\begin{pmatrix}\delta & 0 & 0\\ 0 & \epsilon & 0 \\ 0 & 0 & 1\end{pmatrix}$$
			It follows from elementary row and column transformations.	
			
		\end{proof}
		
		Using the above lemma, the involution $\theta$ is conjugate to
		$$\theta'(g)=t'^{-1}w_{0}gw_{0}t',\quad g\in\ul{G}$$
		with $t'=\mrdiag(\delta,\epsilon,1)$. Since the diagonal split torus $\ul{S}$ is $\theta'$-split, this in particular shows that the second $\ul{H}$-orbit we discussed above must be of $\theta$-rank $0$. 
		
		To sum up, there are two $\ul{H}$-orbits of Borel subgroups. They are of $\theta$-rank $1,0$ and of cardinality $q+1,q^{3}-q$ respectively.
		
		\textbf{Case (I).(\romannumeral2):} Write $t:=\mrdiag(1,x,1)$ with $x\in\bs{l}_{2}^{\times}$, $\sigma(x)x=1$. Then in this case, we have
		$$\theta(g)=t^{-1}\sigma(g)t,\quad g\in \ul{G}.$$
		Moreover, $\ul{H}$ is a special orthogonal subgroup $\mrso_{3}(\bs{l})$ of $\ul{G}$, and there are two possible conjugacy classes of such $\ul{H}$. Let $\ul{B}$, $u_{0}$, $u$ be defined as in \textbf{Case (I).(\romannumeral1).(2)}. Similarly, $\ul{B}$ and $\ul{B}^{-}$ are in the same $\ul{H}$-orbit. Using a similar calculation, $\ul{B}^{u_{0}}$ and $\ul{B}^{u_{0}u}$ are in the same $\ul{H}$-orbit if and only if for some $b\in\ul{B}$
		$$\theta(u_{0}^{-1}bu_{0}u)=u_{0}^{-1}bu_{0}u,$$ 
		or equivalently $$\sigma(b^{-1})\sigma(u_{0})tu_{0}^{-1}b=\sigma(u_{0}u)tu^{-1}u_{0}^{-1}.$$
		Using the formula
		$b^{-1}=w_{0}\sigma(\,^{t}b)w_{0}$, we further have
		$$\,^{t}bw_{0}\sigma(u_{0})tu_{0}^{-1}b=w_{0}\sigma(u_{0})\sigma(u)tu^{-1}u_{0}^{-1},$$
		or equivalently
		\begin{equation}\label{eqIii2eq1}
			\begin{aligned}
				&	\,^{t}b\begin{pmatrix}
					-2\sigma(y_{0})+\sigma(x_{0})^{2}x & x_{0}-x\sigma(x_{0}) & 1\\
					x_{0}-x\sigma(x_{0}) & x & 0\\
					1 & 0 & 0
				\end{pmatrix}b\\
				=&\begin{pmatrix}
					-2\sigma(y_{0}+y_{1})+\sigma(x_{0}+x_{1})^{2}x & x_{0}+x_{1}-x\sigma(x_{0}+x_{1}) & 1\\
					x_{0}+x_{1}-x\sigma(x_{0}+x_{1}) & x & 0\\
					1 & 0 & 0
				\end{pmatrix}.
			\end{aligned}
		\end{equation}
		
		If $x\sigma(x_{0})^{2}=2\sigma(y_{0})$, to find such $b$ we must have $x\sigma(x_{0}+x_{1})^{2}=2\sigma(y_{0}+y_{1})$. Thus we have $$x_{0}\sigma(x_{0})=y_{0}+\sigma(y_{0})=x\sigma(x_{0})^{2}/2+\sigma(x)x_{0}^{2}/2$$
		Using the fact that $\sigma(x)=x^{-1}$ and solving the above equation, we have
		$x\sigma(x_{0})=x_{0}$ and $y_{0}=-x_{0}^{2}/2$, and such $x_{0}$ is of cardinality $q$. So the $\ul{H}$-orbit of $\ul{B}$ is of cardinality $q+1$ and of $\theta$-rank $1$.
		
		If $x\sigma(x_{0})^{2}\neq2\sigma(y_{0})$, we have the following lemma.
		
		\begin{lemma}
			There exists $b\in\ul{B}$ satisfying \eqref{eqIii2eq1} if and only if $x\sigma(x_{0}+x_{1})^{2}\neq2\sigma(y_{0}+y_{1})$ and 
			$$(x\sigma(x_{0})^{2}-2\sigma(y_{0}))/(x\sigma(x_{0}+x_{1})^{2}-2\sigma(y_{0}+y_{1}))\in\bs{l}_{2}^{\times 2}.$$
		\end{lemma} 
		
		\begin{proof}
			
			The proof follows from elementary row and column transformations.
			
		\end{proof} 
		Thus, depending on the condition whether $$x\sigma(x_{0})^{2}-2\sigma(y_{0})\in\bs{l}_{2}^{\times2}\quad\text{or}\quad x\sigma(x_{0})^{2}-2\sigma(y_{0})\in\bs{l}_{2}^{\times}-\bs{l}_{2}^{\times2},$$
		there are two related $\ul{H}$-orbits of cardinality $(q^{3}-q)/2$ and $(q^{3}-q)/2$ respectively, whose proof is elementary and is omitted. 
		
		We claim that these two $\ul{H}$-orbits are of $\theta$-rank $0$. To that end, for $u_{0}$ such that $x\sigma(x_{0})^{2}-2\sigma(y_{0})\neq 0$ we need to find $b\in\ul{B}$ such that $\ul{S}^{bu_{0}}$ is $\theta$-split. We have the following lemma.
		
		\begin{lemma}\label{lemmadiagortho}
			
			Let $\epsilon \in\bs{l}_{2}^{\times}$ and $\delta=\sigma(\epsilon)\epsilon$ such that $\epsilon/(x\sigma(x_{0})^{2}-2\sigma(y_{0}))\in \bs{l}_{2}^{\times2}$ , then there exists $b\in\ul{B}$ such that
			$$\theta(bu_{0})(bu_{0})^{-1}=w_{0}\mrdiag(\epsilon\delta,\delta,\epsilon).$$
			
		\end{lemma}
		
		\begin{proof}
			
			We first remark that $\delta\in\bs{l}_{2}^{\times 2}$, since $\sigma$ is given by the $q$-th power map. By a direct calculation similar to that of \eqref{eqIii2eq1}, equivalently we need to find $b$ such that
			$$\,^{t}b^{-1}\begin{pmatrix}
				-2\sigma(y_{0})+x\sigma(x_{0})^{2} & x_{0}-x\sigma(x_{0}) & 1\\
				x_{0}-x\sigma(x_{0}) & x & 0\\
				1 & 0 & 0
			\end{pmatrix}b^{-1}=\begin{pmatrix}
				\epsilon\delta & 0 & 0\\
				0 & x\delta & 0\\
				0 & 0 & \epsilon
			\end{pmatrix}.$$
			This could be realized using elementary row and column transformations.
			
		\end{proof}
		
		Using this lemma, it is direct to verify that $\ul{S}^{bu_{0}}$ is $\theta$-stable and $(\ul{S}^{bu_{0}})^{\theta}=\{I_{3},\mrdiag(-1,1,-1)\}$. Thus for any $u_{0}$ such that $x\sigma(x_{0})^{2}\neq2\sigma(y_{0})$, the related $\ul{B}^{u_{0}}$ is of $\theta$-rank $0$.
		
		To sum up, there are three $\ul{H}$-orbits of Borel subgroups of $\theta$-rank $1,0,0$ and of cardinality $q+1,(q^{3}-q)/2,(q^{3}-q)/2$ respectively.
		
		\textbf{Case (II).(\romannumeral1):} In this case we have  
		$$\theta(g)=\begin{pmatrix} 0 & 0 & \delta^{-1}\\ 0 & \epsilon\delta^{-1} & 0\\ 1 & 0 & 0
		\end{pmatrix}g\begin{pmatrix} 0 & 0 & 1\\ 0 & \epsilon^{-1}\delta & 0\\ \delta  & 0 & 0
		\end{pmatrix},\quad g\in\ul{G}$$
		for certain $\epsilon,\delta\in\bs{l}^{\times}$ with $\delta=\epsilon^{2}$. Using Lemma \ref{lemmadiagunitary}, such an involution is conjugate to the involution
		$$\theta'(g)=\mrdiag(1,-1,1)g\mrdiag(1,-1,1),\quad g\in\ul{G},$$
		which was studied in \textbf{Case (I).(\romannumeral1).(2)}.
		So $\ul{H}$ is conjugate to $\mru_{2}(\bs{l})$, and there are two $\ul{H}$-orbits of Borel subgroups. They are of $\theta$-rank $0,1$ and of cardinality $q^{3}-q,q+1$ respectively.  
		
		\textbf{Case (II).(\romannumeral2):} In this case we have
		$$\theta(g)=\begin{pmatrix} 0 & 0 & \delta^{-1}\\ 0 & \epsilon\delta^{-1} & 0\\ 1 & 0 & 0
		\end{pmatrix}\sigma(g)\begin{pmatrix} 0 & 0 & 1\\ 0 & \epsilon^{-1}\delta & 0\\ \delta  & 0 & 0
		\end{pmatrix},\quad g\in\ul{G}$$ 
		for certain $\epsilon,\delta\in\bs{l}_{2}^{\times}$ with $\sigma(\delta)=\delta=\sigma(\epsilon)\epsilon\in\bs{l}$. Using Lemma
		\ref{lemmadiagortho}, such an involution is conjugate to the involution 
		$$\theta'(g)=\sigma(g),\quad g\in\ul{G},$$
		which was studied in \textbf{Case (I).(\romannumeral2)}. So $\ul{H}$ is conjugate to $\mrso_{3}(\bs{l})$, and there are three $\ul{H}$-orbits of Borel subgroups. They are of $\theta$-rank $0,0,1$ and of cardinality $(q^{3}-q)/2,(q^{3}-q)/2,q+1$ respectively. In particular by direct calculation $\theta(-w_{0})=-w_{0}$, thus the standard Borel subgroup $\ul{B}$ and its opposite $\ul{B}^{-}=\ul{B}^{w_{0}}$ are in the same $\ul{H}$-orbit.
		
		We sum up our result in the following proposition.
		
		\begin{proposition}\label{proprank1classSU3}
			
			Let $\ul{G}$ be the Weil restriction of scalars $\mrres_{\bs{l}/\bs{k}}\mrsu_{3}$ and $\theta$, $\ul{H}$, $\ul{S}$, $\ul{T}$ as above, then we exhaust the possible involution $\theta$ into the following cases:  \footnote{Still, in each case the first $\ul{H}$-orbit of Borel subgroups listed below is exactly the one containing the upper triangular Borel subgroup.}
			
			\begin{itemize}
				\item \textbf{Case (I).(\romannumeral1).(1):} $\theta$ is the identity, $\ul{H}$ is the whole group $\ul{G}$, and there is exactly one $\ul{H}$-orbit.
				\item \textbf{Case (I).(\romannumeral1).(2):} $\theta$ is the conjugation given by $\mrdiag(1,-1,1)$, and $\ul{H}$ is the unitary group $\mru_{2}(\bs{l})$. There are two $\ul{H}$-orbits of Borel subgroups of $\theta$-rank $1,0$ and of cardinality $q+1,q^{3}-q$ respectively.
				\item \textbf{Case (I).(\romannumeral2):} 
				The involution $\theta$ is given by 
				$$\theta(g)=t^{-1}\sigma(g)t,\quad g\in \ul{G}$$
				for some
				$t:=\mrdiag(1,x,1)$ with $x\in\bs{l}^{\times}$, $\sigma(x)x=1$, and $\ul{H}$ is a special orthogonal subgroup of $\ul{G}$. There are three $\ul{H}$-orbits of Borel subgroups of $\theta$-rank $1,0,0$ and of cardinality $q+1,(q^{3}-q)/2,(q^{3}-q)/2$ respectively.
				
				\item \textbf{Case (II).(\romannumeral1)}: The involution $\theta$ is given by $$\theta(g)=\begin{pmatrix} 0 & 0 & \delta^{-1}\\ 0 & \epsilon\delta^{-1} & 0\\ 1 & 0 & 0
				\end{pmatrix}g\begin{pmatrix} 0 & 0 & 1\\ 0 & \epsilon^{-1}\delta & 0\\ \delta  & 0 & 0
				\end{pmatrix},\quad g\in\ul{G}$$
				for certain $\epsilon,\delta\in\bs{l}^{\times}$ with $\delta=\epsilon^{2}$, and $\ul{H}$ is a conjugation of $\mru_{2}(\bs{l})$ of $\ul{G}$. There are two $\ul{H}$-orbits of Borel subgroups of $\theta$-rank $0,1$ and of cardinality $q^{3}-q,q+1$ respectively.
				
				\item \textbf{Case (II).(\romannumeral2)}: The involution $\theta$ is given by $$\theta(g)=\begin{pmatrix} 0 & 0 & \delta^{-1}\\ 0 & \epsilon\delta^{-1} & 0\\ 1 & 0 & 0
				\end{pmatrix}\sigma(g)\begin{pmatrix} 0 & 0 & 1\\ 0 & \epsilon^{-1}\delta & 0\\ \delta  & 0 & 0
				\end{pmatrix},\quad g\in\ul{G}$$ 
				for certain $\epsilon,\delta\in\bs{l}_{2}^{\times}$ with $\sigma(\delta)=\delta=\sigma(\epsilon)\epsilon\in\bs{l}^{\times}$, and $\ul{H}$ is a special orthogonal subgroup of $\ul{G}$. There are three $\ul{H}$-orbits of Borel subgroups of $\theta$-rank $0,0,1$ and of cardinality $(q^{3}-q)/2,(q^{3}-q)/2,q+1$ respectively.
				
			\end{itemize}
			
		\end{proposition}
		
		\subsection{The general case}
		
		Now we consider a general reductive group $\ul{G}$ over $\bs{k}$. Let $\ul{S}$ be a maximal split torus over $\bs{k}$, let $\ul{T}$ be the centralizer of $\ul{S}$ in $\ul{G}$, and let $\theta$ be an involution of $\ul{G}$ over $\bs{k}$ that stabilizes $\ul{S}$ and $\ul{T}$. Let $\Phi(\ul{G},\ul{S})$ be the corresponding set of roots.
		
		We first state a general lemma, which will be quite useful later on.
		
		\begin{lemma}\label{lemmascinvolutionext}
			
			Let $\bs{k}$ be a general field of characteristic different from 2. Let $\ul{G}^{\mathrm{sc}}$ be the simply connected cover of the derived subgroup of $\ul{G}$ endowed with the natural morphism $p:\ul{G}^{\mathrm{sc}}\rightarrow\ul{G}^{\mathrm{der}}\rightarrow\ul{G}$. Then there is a unique involution $\theta'$ of $\ul{G}^{\mathrm{sc}}$ over $\bs{k}$ such that $p\circ\theta'=\theta\circ p$, and  $(\ul{G}^{\mathrm{sc}})^{\theta'}$ is a connected reductive group over $\bs{k}$. 
			
		\end{lemma}
		
		\begin{proof}
			
			The first part and the connectness of $(\ul{G}^{\mathrm{sc}})^{\theta'}$ follows from \cite{steinberg1968endomorphisms}*{\S 9.16 and Theorem 8.2}, while the $\bs{k}$-rationality of $\theta'$ follows from the uniqueness. 
			
		\end{proof}
		
		\begin{remark}
			
			In general, it is not true that $(G^{\mathrm{sc}})^{\theta'}=H^{\mathrm{sc}}$, since the former group is not necessarily simply connected. Consider for instance the case that $\ul{G}=\mrgl(n)$ and $\ul{H}=\mrso(n)$.
			
		\end{remark}
		
		We write $p_{*}$ for the induced bijection between $\ch{}{}(\ul{G}^{\mathrm{sc}})$ and $\ch{}{}(\ul{G})$, then the set of $(\ul{G}^{\mathrm{sc}})^{\theta'}$-orbits  of $\ch{}{}(\ul{G}^{\mathrm{sc}})$ forms a partition of the set of $\ul{H}$-orbits of $\ch{}{}(\ul{G})$ via $p_{*}$. 
		Indeed, let $\Hbarstar=p((\ul{G}^{\mathrm{sc}})^{\theta'}(\bs{k}))$\index{$\Hstar$} be a subgroup of $\ul{G}(\bs{k})$. Then the set of $\Hbarstar$-orbits of $\ch{}{}(\ul{G})$ forms a partition of the set of $\ul{H}$-orbits of $\ch{}{}(\ul{G})$.
		
		Take a non-divisible root $\alpha\in\Phi(\ul{G},\ul{S})$. Depending on whether $\alpha$ is multipliable or not, the union $\mbr_{>0}\alpha\cap \Phi(\ul{G},\ul{S})$ is either $\{\alpha\}$ or $\{\alpha,2\alpha\}$. Let $\ul{U}_{\alpha}$ and $\ul{U}_{-\alpha}$ be the unipotent subgroups of $\ul{G}$ with respect to $\alpha$ and $-\alpha$. Let $\ul{M}_{\alpha}$ be the Levi subgroup of $\ul{G}$ related to $\alpha$, i.e. generated by $\ul{T}$, $\ul{U}_{\alpha}$ and $\ul{U}_{\alpha}$. Then the derived subgroup of $\ul{M}_{\alpha}$ is the semi-simple subgroup $\pairang{\ul{U}_{-\alpha}}{\ul{U}_{\alpha}}$ generated by $\ul{U}_{-\alpha}$ and $\ul{U}_{\alpha}$.
		
		Let $\ul{G}_{\alpha}$ be the simply connected cover of  $\pairang{\ul{U}_{-\alpha}}{\ul{U}_{\alpha}}$, which is a semi-simple simply connected rank one group over $\bs{k}$. Then we have  \begin{equation}\label{eqGalpha}\ul{G}_{\alpha}\cong\begin{cases}
				\mrres_{\bs{l}_{\alpha}/\bs{k}}\mrsl_{2},\quad&\text{if}\ 2\alpha\notin\Phi(\ul{G},\ul{S});\\
				\mrres_{\bs{l}_{\alpha}/\bs{k}}\mrsu_{3},\quad&\text{if}\ 2\alpha\in\Phi(\ul{G},\ul{S}).
		\end{cases}\end{equation} 
		Here $\bs{l}_{\alpha}/\bs{k}$ is a finite extension depending on $\alpha$, and let $q_{\alpha}$ be the cardinality of $\bs{l}_{\alpha}$. By abuse of notation we also regard $\ul{U}_{\alpha}$ and $\ul{U}_{-\alpha}$ as the corresponding unipotent subgroups of $\ul{G}_{\alpha}$. Then, the above construction gives a composition of   morphisms:
		\begin{equation}\label{eqparpha}
			p_{\alpha}:\ul{G}_{\alpha}\rightarrow\pairang{\ul{U}_{-\alpha}}{\ul{U}_{\alpha}}\rightarrow\ul{M}_{\alpha}
		\end{equation}
		whose restriction to $\ul{U}_{\alpha}$ and $\ul{U}_{-\alpha}$ is the identity map. Then $(p_{\alpha})_{*}$ induces a bijection between the set of Borel subgroups of $\ul{G}_{\alpha}$ to that of $\ul{M}_{\alpha}$. 
		
		Let $\ul{S}_{\alpha}=p_{\alpha}^{-1}(\ul{S})$ and $\ul{T}_{\alpha}=p_{\alpha}^{-1}(\ul{T})$ which is the centralizer of $\ul{S}_{\alpha}$ in $\ul{G}_{\alpha}$. Up to a suitable choice of the isomorphism in $\eqref{eqGalpha}$, we may identify $\ul{U}_{\alpha}$ (resp. $\ul{U}_{-\alpha}$) with the subgroup of upper (resp. lower) triangular matrices, $\ul{S}_{\alpha}$ with the diagonal $\bs{k}$-split torus and $\ul{T}_{\alpha}$ with the maximal diagonal torus of $\mrres_{\bs{l}_{\alpha}/\bs{k}}\mrsl_{2}$ or $\mrres_{\bs{l}_{\alpha}/\bs{k}}\mrsu_{3}$.
		
		From now on we assume that $\theta$ fixes the line $\mbr\alpha$. Then by definition either $\theta(\alpha)=\alpha$ or $\theta(\alpha)=-\alpha$, and the involution $\theta$ stabilizes the subgroup $\pairang{\ul{U}_{-\alpha}}{\ul{U}_{\alpha}}$. Using Lemma \ref{lemmascinvolutionext}, we may define a unique involution $\theta_{\alpha}$ on $\ul{G}_{\alpha}$ over $\bs{k}$, such that $\theta_{\alpha}\circ p_{\alpha}=p_{\alpha}\circ \theta$. Moreover if we define $\ul{H}_{\alpha}=\ul{G}_{\alpha}^{\theta_{\alpha}}$ as a reductive subgroup of $\ul{G}_{\alpha}$, then the set of $\ul{H}_{\alpha}$-orbits of Borel subgroups of $\ul{G}_{\alpha}$ forms a partition of the set of $\ul{H}\cap \pairang{\ul{U}_{-\alpha}}{\ul{U}_{\alpha}}$-orbits (resp.  $\Hbarstar\cap \pairang{\ul{U}_{-\alpha}}{\ul{U}_{\alpha}}$-orbits) of Borel subgroups of $\pairang{\ul{U}_{-\alpha}}{\ul{U}_{\alpha}}$. 
		
		Moreover, since $\pairang{\ul{U}_{-\alpha}}{\ul{U}_{\alpha}}$ is the derived subgroup of $\ul{M}_{\alpha}$, the set of Borel subgroups of $\pairang{\ul{U}_{-\alpha}}{\ul{U}_{\alpha}}$ is in bijection with that of $\ul{M}_{\alpha}$, and moreover the set of $\ul{H}\cap \pairang{\ul{U}_{-\alpha}}{\ul{U}_{\alpha}}$-orbits (resp.  $\Hbarstar\cap \pairang{\ul{U}_{-\alpha}}{\ul{U}_{\alpha}}$-orbits) of Borel subgroups of $\pairang{\ul{U}_{-\alpha}}{\ul{U}_{\alpha}}$ forms a partition of the set of $\ul{H}\cap \ul{M}_{\alpha}$-orbits (resp.  $\Hbarstar\cap \ul{M}_{\alpha}$-orbits) of Borel subgroups of $\ul{M}_{\alpha}$ with respect to the above bijection.
		
		Combining with Proposition \ref{proprank1classSL2} and Proposition \ref{proprank1classSU3}, we have the following proposition describing the $\ul{H}\cap \ul{M}_{\alpha}$-orbits (resp.  $\Hbarstar\cap \ul{M}_{\alpha}$-orbits) of Borel subgroups of $\ul{M}_{\alpha}$. 
		
		\begin{proposition}\label{propgeneralclass}
			
			Let $\ul{G}$, $\ul{S}$, $\ul{T}$, $\theta$, $\ul{H}$ be defined as above,  $\alpha\in\Phi(\ul{G},\ul{S})$ such that $\mbr\alpha$ is $\theta$-stable, and $\ul{G}_{\alpha}$, $\ul{M}_{\alpha}$ $\theta_{\alpha}$, $\ul{S}_{\alpha}$, $\ul{T}_{\alpha}$ be defined as above. Consider Borel subgroups of $\ul{M}_{\alpha}$.
			\begin{enumerate}
				\item \textbf{Case (SL):} Assume that $\ul{G}_{\alpha}\cong \mrres_{\bs{l}_{\alpha}/\bs{k}}\mrsl_{2}$. Then
				\begin{itemize}
					\item \textbf{Case (SL).(\romannumeral1):} if the restriction of $\theta_{\alpha}$ to $\ul{T}_{\alpha}$ is either the identity or the inversion, then one of the following cases happens:
					\begin{itemize}
						\item there is one $\ul{H}\cap\ul{M}_{\alpha}$-orbit of Borel subgroups;
						\item there are two $\ul{H}\cap\ul{M}_{\alpha}$-orbits of Borel subgroups of $\theta$-rank $1,0$ and of cardinality $2,q_{\alpha}-1$;
						\item there are two $\ul{H}\cap\ul{M}_{\alpha}$-orbits of Borel subgroups of $\theta$-rank $0,0$ and of cardinality $(q_{\alpha}+1)/2,(q_{\alpha}+1)/2$;
						\item there are three $\ul{H}\cap\ul{M}_{\alpha}$-orbits of Borel subgroups of $\theta$-rank $1,0,0$ and of cardinality $2,(q_{\alpha}-1)/2,(q_{\alpha}-1)/2$;
						\item there are three $\ul{H}\cap\ul{M}_{\alpha}$-orbits of Borel subgroups of $\theta$-rank $1,1,0$ and of cardinality $1,1,q_{\alpha}-1$;
						\item there are four $\ul{H}\cap\ul{M}_{\alpha}$-orbits of Borel subgroups of $\theta$-rank $1,1,0,0$ and of cardinality $1,1,(q_{\alpha}-1)/2,(q_{\alpha}-1)/2$.
					\end{itemize}  
					\item \textbf{Case (SL).(\romannumeral2):} if the restriction of $\theta_{\alpha}$ to $\ul{T}_{\alpha}$ is either the order 2 Galois action or its composition with the inversion, then there are two $\ul{H}\cap\ul{M}_{\alpha}$-orbits of Borel subgroups of $\theta$-rank $1$,$0$ and of cardinality $\sqrt{q_{\alpha}}+1$, $q_{\alpha}-\sqrt{q_{\alpha}}$.
					
				\end{itemize}
				
				\item \textbf{Case (SU):} Assume that $\ul{G}_{\alpha}\cong \mrres_{\bs{l}_{\alpha}/\bs{k}}\mrsu_{3}$. Then 
				\begin{itemize}
					\item \textbf{Case (SU).(\romannumeral1):} if the restriction of $\theta_{\alpha}$ to $\ul{T}_{\alpha}$ is either the identity or the order 2 Galois action composing with the inversion, then one of the following cases happens:
					\begin{itemize}
						\item there is one $\ul{H}\cap\ul{M}_{\alpha}$-orbit of Borel subgroups;
						\item there are two $\ul{H}\cap\ul{M}_{\alpha}$-orbits of Borel subgroups of $\theta$-rank $1,0$ and of cardinality $q_{\alpha}+1,q_{\alpha}^{3}-q_{\alpha}$.
					\end{itemize}  
					\item \textbf{Case (SU).(\romannumeral2):} if the restriction of $\theta_{\alpha}$ to $\ul{T}_{\alpha}$ is either the inversion or the order 2 Galois action, then one of the following cases happens:
					\begin{itemize}
						\item there are two $\ul{H}\cap\ul{M}_{\alpha}$-orbits of Borel subgroups of $\theta$-rank $1,0$ and of cardinality $q_{\alpha}+1,q_{\alpha}^{3}-q_{\alpha}$;
						\item there are three $\ul{H}\cap\ul{M}_{\alpha}$-orbits of Borel subgroups of $\theta$-rank $1,0,0$ and of cardinality $q_{\alpha}+1,(q_{\alpha}^{3}-q_{\alpha})/2,(q_{\alpha}^{3}-q_{\alpha})/2$.
					\end{itemize}
					
				\end{itemize}
				
			\end{enumerate}
			
			Moreover, all the $\ul{H}\cap\ul{M}_{\alpha}$ could be changed into $\Hbarstar\cap\ul{M}_{\alpha}$ as well with the results unchanged.
			
		\end{proposition}

		\begin{remark}
			
			In Proposition \ref{propgeneralclass}, we used the fact that two Borel subgroups of $\ul{M}_{\alpha}$ of different $\theta$-rank must be in different $\ul{H}\cap\ul{M}_{\alpha}$-orbits. This helps us to eliminate several possible cases. Moreover, we also notice that some cases only ``theoretically" exist \footnote{This includes the second case in \textbf{Case (SL).(\romannumeral1)} and the first case in \textbf{Case (SU).(\romannumeral2)}. However for other new cases not included in Proposition \ref{proprank1classSL2} and Proposition \ref{proprank1classSU3}, i.e. the fourth and fifth cases in \textbf{Case (SL).(\romannumeral1)}, we may consider $\ul{G}=\mrres_{\bs{l}/\bs{k}}\mrgl_{3}$ and $\theta$ being the split orthogonal involution, or  $\ul{G}=\mrres_{\bs{l}/\bs{k}}\mrgl_{2}$ and $\theta$ being the conjugation by $\mrdiag(1,-1)$ to realize them respectively.}. It means that in practice we may not find the related group $\ul{M}_{\alpha}$ and the involution $\theta$, such that the $\ul{H}\cap\ul{M}_{\alpha}$-orbits of the Borel subgroups of $\ul{M}_{\alpha}$ are as required, although currently we cannot find examples to justify them or proofs to eliminate them.
			
		\end{remark}

		\section{Combinatorial geometry of chambers}\label{sectiongeometry}
		
		In this part, we study the combinatorial geometry of chambers in a building or an apartment endowed with an involution. As a result, we give an upper bound of the distinguished dimension of the Steinberg representation.
		
		\subsection{Involutions on Euclidean reflection groups}
		
		In this part, we consider Euclidean reflection groups and the related Coxeter complex equipped with an involution. 
		
		Let $W$ be an Euclidean reflection group. By definition, there exists a locally finite set $\mfH$ of hyperplanes in the affine space $\mca=\mbr^{n}$, such that \begin{itemize}
			\item $\mfH$ is stable under the reflection $s_{\mfh}$ along $\mfh$ for any $\mfh\in\mfH$;
			\item $W$ is generated by $s_{\mfh}, \mfh\in\mfH$.
		\end{itemize}

		Assume that there exists an affine involution $\theta$ of $\mca$ that stabilizes $\mfH$. Then $\mca^{\theta}$ is an affine subspace of $\mca$ of dimension $r$, called the $\theta$-rank of $\mca$. 
		
		A facet $F$ of $\mca$ is called \emph{$\theta$-stable} if $\theta(F)=F$. In particular, the set $F^{\theta}$ of $\theta$-fixed points of a $\theta$-stable facet $F$ is either a poly-cone or a poly-simplex in $\mca^{\theta}$, depending on whether $W$ is finite or not. Let $\Ft(\mca)$ be the set of $\theta$-stable facets in $\mca$. Then the set 
		$$\mcf(\mca^{\theta})=\{F^{\theta}\mid F\in\Ft(\mca)\}$$
		gives a partition of $\mca^{\theta}$, realizing it as a poly-conical or poly-simplicial complex. Those $F^{\theta}$ are called the facets of $\mca^{\theta}$. Indeed, the map
		$$\Ft(\mca)\rightarrow\mcf(\mca^{\theta}),\quad F\mapsto F^{\theta}$$
		is a bijection. In particular, $F^{\theta}$ is called a \emph{chamber} (resp. \emph{panel}) of $\mca^{\theta}$ if it is of dimension $r$ (resp. $r-1$). Indeed, $F^{\theta}$ is a chamber of $\mca^{\theta}$ if and only if $F$ is  a $\theta$-stable facet of maximal dimension in $\mca$. 
		
		We may define the corresponding set of walls $$\mfH_{\theta}=\{\mfh\cap\mca^{\theta}\mid \mfh\in\mfH,\ \mca^{\theta}\nsubseteq\mfh\}\index{$\mfH_{\theta}$}$$ in $\mca^{\theta}$. We say that a hyperplane $\mfh_{\theta}\in\mfH_{\theta}$\index{$\mfh_{\theta}$} is \emph{$\theta$-adapted}, if there exists a hyperplane $\mfh\in\mfH$ perpendicular to\footnote{Here, we call a hyperplane $\mfh$ perpendicular to an affine subspace $\mca'$ of $\mca$ if the normal vector of $\mfh\in\mfH$ is parallel to some non-zero vector in $\mca'$.} $\mca^{\theta}$ such that $\mfh_{\theta}=\mfh\cap\mca^{\theta}$. We notice that such $\mfh$ is unique once it exists. 
		
		For two facets $F^{\theta},F'^{\theta}$ of $\mca^{\theta}$, we define $\mfH_{\mca^{\theta}}(F^{\theta},F'^{\theta})$\index{$\mfH_{\mca^{\theta}}(F^{\theta},F'^{\theta})$} as the set of walls in $\mfH_{\theta}$ that strictly separate $F^{\theta}$ and $F'^{\theta}$.

		\begin{remark} 
			
			If by chance every hyperplane in $\mfH_{\theta}$ is $\theta$-adapted, 
			then the underlying complex structure of $\mca^{\theta}$ given by the set of facets $\mcf(\mca^{\theta})$ is realized by an Euclidean reflection group. But in general it is NOT true that $\mfH_{\theta}$ is stable under the reflections given by $s_{\mfh_{\theta}}$ with $\mfh_{\theta}\in\mfH_{\theta}$. In other words, in general $\mca^{\theta}$, as a poly-simplicial complex with the set of facets being $\mcf(\mca^{\theta})$, is NOT a complex realized by an Euclidean reflection group. 
			
		\end{remark}

		We define the \emph{$\theta$-distance} of a chamber $C$, denoted by $\tdist{C}$.  A chamber $C$ in $\mca$ is of \emph{$\theta$-distance 0}, if $C$ has a $\theta$-stable facet $F$ of maximal dimension, i.e. $F^{\theta}$ is a chamber of $\mca^{\theta}$. A chamber $C$ is of \emph{$\theta$-distance $d$}, if there exists a gallery $C_{0},C_{1},\dots,C_{d}$ of minimal length, such that $\tdist{C_{0}}=0$ and $C_{d}=C$. Equivalently, $\tdist{C}$ denotes the minimal possible distance between $C$ and some maximal $\theta$-stable facet $F$, or the distance between $C$ and $\mca^{\theta}$. 
		
		We have the following first description of a minimal gallery, whose proof is identical to that of \cite{courtes2017distinction}*{Corollary 5.6}:
		
		\begin{lemma}
			
			Let $C_{0},C_{1},\dots,C_{d}=C$ be a minimal gallery in $\mca$ between $\mca^{\theta}$ and $C$, then $\tdist{C_{i}}=i$ and for each panel $D_{i}$ between $C_{i-1}$ and $C_{i}$, the hyperplane $\mfh_{D_{i}}$ is not $\theta$-stable.
			
		\end{lemma}
		
		We also have the following practical way of describing a possible minimal gallery.
		
		\begin{proposition}\label{propdistCAtheta}
			
			We may find a point $\bs{p}$ in $C$ such that its perpendicular towards $\mca^{\theta}$ intersects a chamber $F^{\theta}$ on a point denoted by $\bs{p}'$, and does not pass any facet of $\mca$ of codimension greater than one other than $F$. Moreover, the gallery passing $L=[\bs{p},\bs{p}']$ is a minimal gallery from $C$ to $\mca^{\theta}$ and of distance $\tdist{C}$.
			
		\end{proposition}
		
		\begin{proof}
			
			The first claim is clear, since such $\bs{p}$ forms an open dense subset of $C$.
			
			Now we prove the second claim. Let $L'$ be another line segment linking $\bs{p}$ and a chamber $F'^{\theta}$ of $\mca^{\theta}$, such that $L'$ does not pass any facet of $\mca$ of codimension greater than one other than $F'$. 
			
			Let $\mfh$ be a wall in $\mca$. We consider the following three cases depending its relation with $\mca^{\theta}$. We say that $\mfh$
			\begin{itemize}
				\item is \emph{parallel to} $\mca^{\theta}$ if the normal vector of $\mfh$ is perpendicular to every vector in $\mca^{\theta}$;
				\item is \emph{perpendicular to} $\mca^{\theta}$ if the normal vector of $\mfh$ is parallel to some non-zero vector in $\mca^{\theta}$;
				\item \emph{obliquely intersects} $\mca^{\theta}$ if the above two cases does not happen, or equivalently, the normal vector of $\mfh$ is not parallel to its image under $\theta$.
			\end{itemize}
			If $\mfh$ obliquely intersects $\mca^{\theta}$, then $\mfh$ is not $\theta$-stable. In particular, every hyperplane is parallel to $\mca^{\theta}$ if it is a point.
			
			Now we consider a wall $\mfh$ that intersects $L$, such that $L\cap \mfh$ is a point. If $\mfh$ is parallel to $\mca^{\theta}$, then $\bs{p}$ and $\mca^{\theta}$ are separated by $\mfh$, meaning that $\mfh$ also intersects $L'$. If $\mfh$ is perpendicular to $\mca^{\theta}$, then $L$ cannot intersect $\mfh$. Finally if $\mfh$ obliquely intersects $\mca^{\theta}$, then the two hyperplanes $\mfh$ and $\theta(\mfh)$ divide $\mca$ into four open domains, where two of them contain points in $\mca^{\theta}$ and the other two do not. Assume $L$ intersects either $\mfh$ or $\theta(\mfh)$. Then it is clear that $\bs{p}$ lies in an open domain that do not contain points in $\mca^{\theta}$, and $L$ cannot intersect both $\mfh$ and $\theta(\mfh)$. But $F'^{\theta}$, as a chamber of $\mca^{\theta}$, must lie in another domain other than that of $\bs{p}$, so $L'$ intersects either $\mfh$ or $\theta(\mfh)$. 
			
			Combining these cases together and ranging over all the walls $\mfh$, we proved that the number of walls intersecting $L$ is no more than the number of walls intersecting $L'$. Using Lemma \ref{lemmaFCdist} we proved the second claim.
			
		\end{proof}
		
		\begin{corollary}\label{corthetaadapt}
			
			Let $F_{1}^{\theta}, F_{2}^{\theta}$ be two adjacent chambers in $\mca^{\theta}$ with $D_{\theta}$ being the panel between them. Let $\mfh$ be the unique hyperplane in $\mca$ containing $D_{\theta}$ and perpendicular to $\mca^{\theta}$. Then the followings are equivalent:
			\begin{itemize}
				
				\item $\mfh$ is not in $\mfH$;
				
				\item The hyperplane $\mfh_{\theta}$ in $\mca^{\theta}$ containing $D_{\theta}$ is not $\theta$-adapted;
				
				\item there exists a chamber $C$ in $\mca$ whose distance towards $F_{1}^{\theta}$ and $F_{2}^{\theta}$ are the same and equal to $\tdist{C}$. 
				
			\end{itemize}  
			
		\end{corollary}
		
		\begin{figure}[htbp]
			\begin{center}
				\tikzstyle{every node}=[scale=1]
				\begin{tikzpicture}[line width=0.4pt,scale=0.65][>=latex]
					\pgfmathsetmacro\ax{2}
					\pgfmathsetmacro\by{2 * sin(60)}
					\pgfmathsetmacro\cy{4 * sin(60)}
					
					\draw[ultra thick] (-2,0) -- (0,0);
					\draw[brown,ultra thick] (0,0) -- (2,0);
					\draw[red,-] (-4,0) -- (-2,0);
					\draw[red,-] (2,0) -- (4,0)node [right] {\(\mathcal{A}^{\theta}\)};
					\draw[blue,-] (0,\cy-2) -- (2,2*\cy-2);
					\draw[blue,-] (0,\cy-2) -- (-2,2*\cy-2);
					\draw[blue,-] (-2,2*\cy-2) -- (2,2*\cy-2);  
					\draw[brown,dashed] (0.7,0)node [below] {\(\bs{p}'_2\)} -- (0.7,4.5) node [right] {\(\bs{p}_2\)};  
					\draw[black,dashed] (-0.7,0)node [below] {\(\bs{p}'_1\)} -- (-0.7,4.5) node [left] {\(\bs{p}_1\)};       
					\draw[orange,-] (0,-0.5) -- (0,5.5);
					\node at (0.3,0.5) [text=orange] {\(\bs{p}'\)};
					\node at (0.3,3) [text=blue] {\(C\)};
					\node at (0:0) [text=orange] {\(\bullet\)};
					\node at (-0.7,4.5) [text=black] {\(\bullet\)};
					\node at (0.7,4.5) [text=brown] {\(\bullet\)};
					\node at (0.7,0) [text=brown] {\(\bullet\)};
					\node at (-0.7,0) [text=black] {\(\bullet\)};
					\node at (0,4) [text=orange] {\(\bullet\)};
					\node at (0.3,4.2) [text=orange] {\(\bs{p}\)};
					\node at (88:5.5) [text=orange] {\(\mfh\)};
					\node at (-1.5,-0.5)  {\(F_1^{\theta}\)};
					\node at (1.5,-0.5) [text=brown] {\(F_2^{\theta}\)};
				\end{tikzpicture}
			\end{center}
			\caption{Corollary \ref{corthetaadapt}, $\mathfrak{h}_{\theta}$ is not $\theta$-adapted}
		\end{figure}

		\begin{proof}
			
			The equivalence of the first two statements are direct. Now we prove the equivalence of the first and the third statements.
			
			First assume that $\mfh$ is not in $\mfH$. We pick a point $\bs{p}'$ in $D_{\theta}$ and another point $\bs{p}$ in $\mfh$, such that the line segment $[\bs{p},\bs{p}']$ is perpendicular to $\mca^{\theta}$. Since $\mfh$ is not in $\mfH$, by changing our choice of $\bs{p}'$ and $\bs{p}$, we may assume that $\bs{p}$ lies in a certain chamber $C$ of $\mca$. Also, $\mfh$ separates $C$ into two open parts, denoted by $C_{1}$ and $C_{2}$. Using Proposition \ref{propdistCAtheta}, we may find $\bs{p}_{1}\in C_{1}$ and $\bs{p}_{2}\in C_{2}$ that are sufficiently close to $\bs{p}$, such that the perpendicular of $\bs{p}_{i}$ towards $\mca^{\theta}$ intersects $\mca^{\theta}$ with a point $\bs{p}_{i}'$ that lies in $F_{i}^{\theta}$, and moreover the gallery passing $[\bs{p}_{i},\bs{p}_{i}']$ is a minimal gallery from $C$ to $\mca^{\theta}$ for $i=1,2$. As a result, the distance of $C$ towards $F_{1}^{\theta}$ and $F_{2}^{\theta}$ are the same.
			
			Conversely, if $\mfh$ lies in $\mfH$, then any chamber $C$ lies on the same side of $\mfh$ with either $F_{1}^{\theta}$ or $F_{2}^{\theta}$. Assuming the former, then the distance between $C$ and $F_{1}^{\theta}$ is less than the distance between $C$ and $F_{2}^{\theta}$. This could be seen by showing that the number of hyperplanes that separate $C$ and $F_{1}^{\theta}$ is less than the number of hyperplanes that separate $C$ and $F_{2}^{\theta}$ as in the proof of Proposition \ref{propdistCAtheta}. In particular, we used the fact that $\mfh$ separates $C$ and $F_{2}^{\theta}$ but does not separate $C$ and $F_{1}^{\theta}$.
			
		\end{proof}

		The following corollary is direct.
		
		\begin{corollary}\label{corthetaadapt'}
			
			Under the condition of Corollary \ref{corthetaadapt} and	assume that $\mfh$ is in $\mfH$. Let $D$ be a panel in $\mfh$ and $C_{1},C_{2}$ two chambers adjacent to $D$ in $\mca$ such that $C_{1}$ is on the same side of $\mfh$ as $F_{1}$. Then the reflection $s_{\mfh}$ maps $C_{1}$ to $C_{2}$ and $F_{1}$ to $F_{2}$, and we may take line segments $[\bs{p}_{i},\bs{p}_{i}']$ from $C_i$ to $F_{i}$ perpendicular to $\mca^{\theta}$ satisfying the condition of
			Proposition \ref{propdistCAtheta} for $i=1,2$, that are also symmetric with respect the reflection $s_{\mfh}$. As a result, $\tdist{C_{1}}=\tdist{C_{2}}$.
			
		\end{corollary}
		\begin{figure}[htbp]
			\begin{center}
				\tikzstyle{every node}=[scale=1]
				\begin{tikzpicture}[line width=0.4pt,scale=1][>=latex]
					\pgfmathsetmacro\ax{cos(120)}
					\pgfmathsetmacro\ay{sin(120)}
					\pgfmathsetmacro\bx{2 * cos(60)}
					\pgfmathsetmacro\by{2 * sin(60)}
					\draw[ultra thick] (-\by,0) -- (0,0);
					\draw[brown,ultra thick] (0,0) -- (\by,0);
					\draw[orange] (0,-0.5) -- (0,1);
					\draw[orange] (0,3) -- (0,3.5);  
					\draw[red,-] (-4,0) -- (-\by,0);
					\draw[red,-] (\by,0) -- (4,0);
					\draw[orange,ultra thick] (0,1) -- (0,3);
					\draw[blue,-] (0,1) -- (0,3);
					\draw[blue,-] (0,1) -- (\by,2);
					\draw[blue,-] (0,1) -- (-\by,2);  
					\draw[blue,-] (0,3) -- (\by,2);
					\draw[blue,-] (0,3) -- (-\by,2);    
					\draw[brown,dashed] (\by-0.5,0) node {\(\bullet\)}-- (\by-0.5,2)node {\(\bullet\)};     
					\draw[black,dashed] (-\by+0.5,0) node {\(\bullet\)} -- (-\by+0.5,2) node {\(\bullet\)};  
					\node at (-10:\by-0.5) [text=brown] {\(\bs{p}'_2\)};
					\node at (-20:0.5*\by)  {\(F_2^{\theta}\)};
					\node at (0:4.5)[text=red] {\(\mca^{\theta}\)};
					\node at (-170:\by-0.5)  {\(\bs{p}'_1\)};
					\node at (-160:0.5*\by)  {\(F_1^{\theta}\)};
					\node at (\by-0.7,2.2)  [text=brown]{\(\bs{p}_2\)};
					\node at (-\by+1.1,2.2) [left] {\(\bs{p}_1\)};
					\node at (-70:\by-1)[text=orange] {\(\mathfrak{h}\)};
					\node at (110:2)[text=blue] {\(C_1\)};
					\node at (70:2)[text=blue] {\(C_2\)};
					\node at (85:2) [text=orange]{\(D\)};
				\end{tikzpicture}
			\end{center}
			\caption{Corollary \ref{corthetaadapt'}, $\mathfrak{h}_{\theta}$ is $\theta$-adapted}
		\end{figure}
		
		\begin{remark}
			
			The discussion in this subsection could be used to the following two special cases: \begin{itemize}
				\item The vectorial apartment $\mca=\mca_v(\ul{G},\ul{S})$, where $\ul{G}$ is a reductive group over $\bs{k}$ and $\ul{S}$ is a $\theta$-stable maximal split torus of $\ul{G}$.
				\item The affine apartment  $\mca=\mca(G,S)$, where $G$ is a reductive group over $K$ and $S$ is a $\theta$-stable maximal split torus of $G$.
			\end{itemize} 
			
		\end{remark}

		\subsection{The $\theta$-rank and $\theta$-distance of a chamber}
		
		Let $G$ be a reductive group over $K$, and $\theta$ an involution of $G$ over $K$, and $H=(G^{\theta})^{\circ}$. We define $\ch{}{}(G)$ the set of chambers in the Bruhat--Tits building of $G$.
		
		Let $G^{\mathrm{sc}}$ be the simply connected cover of the derived subgroup of $G$ endowed with the natural morphism $p:G^{\mathrm{sc}}\rightarrow G^{\mathrm{der}}\rightarrow G$. Using Lemma \ref{lemmascinvolutionext} we get an involution $\theta'$ of $G^{\mathrm{sc}}$ such that $\theta\circ p=p\circ\theta'$. Let $\Hstar=p((G^{\mathrm{sc}})^{\theta'}(K))$\index{$\Hstar$} be a subgroup of $G(K)$.
		
		Given $C\in\ch{}{}(G)$, there exists a unique $H'$-conjugacy class of $\theta$-stable apartment $\mca=\mca(G,S)$, such that $C$ is a chamber in $\mca$ (\emph{cf.} Proposition \ref{propAcontainC}). Here, $S$ is a maximal split torus of $G$ that is also $\theta$-stable. 
		
		Recall that we have defined the $\theta$-rank of $\mca$ and $C$ in Definition \ref{defthetarank}, which we denote by  $\trank{\mca}$\index{$\trank{\mca}$} and $\trank{C}$\index{$\trank{C}$}  respectively. 
		
		We also define the $\theta$-distance of a chamber $C$. We call a chamber $C$ of $\theta$-\emph{distance 0} if the unique maximal $\theta$-stable facet $F$ of $C$ satisfies that $F^{\theta}$ is a chamber of $\mca^{\theta}$. In general, the \emph{$\theta$-distance} of $C$, denoted by $\tdist{C}$\index{$\tdist{C}$}, is defined as the smallest number $d$ such that there exists a gallery $C_{0},C_{1},\dots,C_{d}$ in $\mca$ with $C_{0}$ of $\theta$-distance 0 and $C_{d}=C$. In other words, we define the $\theta$-distance of $C$ as its $\theta$-distance in any of its underlying $\theta$-stable apartment $\mca$. Still such definition is independent of the choice of $\mca$.
		
		We have the following basic proposition about the $\theta$-distance.
		
		\begin{proposition}\label{propdistchamber}
			
			Let $C\in\ch{}{}(G)$ such that $\tdist{C}=d$. Let $\mca$ be a $\theta$-stable apartment containing $C$. Let $C'$ be the other chamber adjacent to $C$, and $D$ the panel between $C$ and $C'$ in $\mca$, and $\mfh_{D}$ the hyperplane in $\mca$ containing $D$. 
			
			\begin{itemize}
				
				\item Then $\tdist{C}=\tdist{C'}$ if and only if $\mfh_{D}$ is $\theta$-stable.
				
				\item Moreover, if $\mfh_{D}$ is not $\theta$-stable and assume $\tdist{C}<\tdist{C'}$, then $\ch{D}{}(G)$ consists of two $\Hstar$-orbits, where the first one is the singleton $\{C\}$, and the second one consists of chambers in $\ch{D}{}(G)$ that are $\Hstar\cap \parah{C\cup\theta(C)}$-conjugate to $C'$.
				
			\end{itemize} 
			
		\end{proposition}
  
       \begin{figure}[htbp]
		\begin{center}
			\tikzstyle{every node}=[scale=1]
			\begin{tikzpicture}[line width=0.4pt,scale=0.8][>=latex]
				\pgfmathsetmacro\ax{cos(120)}
				\pgfmathsetmacro\ay{sin(60)}
				\pgfmathsetmacro\by{2*sin(60)}
				
				\draw[red,-] (-4,0) -- (4,0);
				\draw[blue,-] (0,2) -- (0,4);
				\draw[blue,dashed] (0.5,0) -- (0.5,3); 
				\draw[blue,-] (0,2) -- (\by,1);
				\draw[orange,ultra thick] (0,2) -- (\by,3);
				\draw[orange,-] (\by,3) -- (\by*2,4);
				\draw[orange,-] (-\by,1) -- (0,2);
				\draw[blue,-] (\by,1) -- (\by,3);
				\draw[blue,-] (0,2) -- (\by,3);
				\draw[blue,-] (0,4) -- (\by,3);  
				\node at (0:4.5)[text=red] {\(\mca^{\theta}\)};
				\node at (75:3) {\(\bs{p}\)};
				\node at (75:3.4) [text=blue]{\(C'\)};
				\node at (70:2.3)[text=orange] {\(D\)};
				\node at (60:2.3)[text=blue] {\(C\)};
				\node at (45:5)[text=orange] {\(\mfh_D\)};
				\node at (0.5,3) {\(\bullet\)};
			\end{tikzpicture}
		\end{center}
    \caption{Proposition \ref{propdistchamber}}
       \end{figure}
  
		\begin{proof}

			If $\mfh_{D}$ is $\theta$-stable, meaning that $\mfh_{D}\cap \mca^{\theta}$ is a $\theta$-adapted wall of $\mca^{\theta}$, then using Corollary \ref{corthetaadapt'} $\tdist{C}=\tdist{C'}$.
			
			If $\mfh_{D}$ is not $\theta$-stable, we claim that 
			
			\begin{itemize}
				\item $\tdist{C}\neq\tdist{C'}$;
				\item Assume without loss of generality that $\tdist{C}<\tdist{C'}$, then $C$ and $\theta(C)$ are on the same side of both $\mfh_{D}$ and $\theta(\mfh_{D})$, and $C'$ and $\theta(C')$ are on the different sides of both $\mfh_{D}$ and $\theta(\mfh_{D})$.
			\end{itemize} 
			Indeed, if $\mfh_{D}$ is not $\theta$-stable, the two hyperplanes $\mfh_{D}$ and $\theta(\mfh_{D})$ separate the full $\mca$ into three or four open domains, depending on whether $\mfh_{D}$ and $\theta(\mfh_{D})$ intersect or not. Also, either $C$ or $C'$ is contained in a domain that does not contain points in $\mca^{\theta}$. Assume without loss of generality that this is so for $C'$, then $C'$ and $\theta(C')$ are on the different sides of $\mfh_{D}$ and $\theta(\mfh_{D})$; on the other hand, $C$ and $\theta(C)$ are on the same side of $\mfh_{D}$ and $\theta(\mfh_{D})$. Then, we may pick a point $\bs{p}$ in $C'$ and its perpendicular $L$ towards $\mca^{\theta}$, such that $L$ satisfies the condition of Proposition \ref{propdistCAtheta}, and $L$ also passes $D$ and $C$. Using Proposition \ref{propdistCAtheta}, this perpendicular $L$ gives a minimal gallery $C_{0},C_{1},\dots,C_{d},C_{d+1}=C'$ between $\mca^{\theta}$ and $C'$, and from our construction $C_{d}=C$. Thus $\tdist{C}=\tdist{C'}-1$. 
			
			The rest of the proposition follows from \cite{courtes2017distinction}*{Proposition 5.5}, whose proof could be modified here directly. We need to show that given two chambers $C',C''$ in $\ch{D}{}(G)$ different from $C$ with $\tdist{C}<\tdist{C'}$, there exists an element $h\in \Hstar\cap \parah{C\cup\theta(C)}$ mapping $C'$ to $C''$. Indeed, in \emph{loc. cit.} the author first constructed an element $u$ in the integral model of some rank one unipotent group that fixes $C$, $\theta(C)$ and maps $C'$ to $C''$. Moreover, the commutator $[\theta(u),u]$ is a pro-unipotent element that fixes $C'$ and $\theta(C')$. Since the first cohomology group $H^1(\pairangone{\theta},\parah{C'\cup\theta(C'),+})$ is trivial, there exists $h'$ in $\parah{C'\cup\theta(C'),+}$ such that $\theta(h')h'^{-1}=[\theta(u),u]$. Thus the element $h=\theta(h'^{-1})\theta(u)u$ is $\theta$-stable, fixing $C$ and $\theta(C)$ and mapping $C'$ to $C''$. Moreover, we notice that $u$ and $\theta(u)$ lie in $p(G^{\mathrm{sc}}(K))$, and thus $h'$ lies in $ p(G^{\mathrm{sc}}(K))\cap\parah{C'\cup\theta(C'),+}$ as well. So the element $h$ lies in $\Hstar\cap \parah{C\cup\theta(C)}$.
			
		\end{proof}
		
		\begin{remark}
			
			The above definitions and results can be in parallel generalized to the vectorial building of a finite reductive group. More precisely, let $\ul{G}$ be a reductive group over a finite field $\bs{k}$, and $\theta$ an involution of $\ul{G}$ over $\bs{k}$, and $\ul{H}=(\ul{G}^{\theta})^{\circ}$. Let $\ch{}{}(\ul{G})$ be the set of chambers in the vectorial building of $\ul{G}$. For $C\in\ch{}{}(\ul{G})$, by taking a $\theta$-stable apartment containing $C$, we may define $\trank{C}$ and $\tdist{C}$ as above, and Proposition \ref{propdistchamber} remains valid with a similar proof.
			
		\end{remark}
		
		\subsection{Classification of panel pairs}\label{subsectionpanelpair}
		
		We keep the notation as above. Let $\mca=\mca(G,S)$ be a $\theta$-stable apartment and $D$ a panel contained in $\mca$. Then we call $(\mca,D)$ a \emph{panel pair}. Our goal here is to classify such panel pairs.
		
		We define the group schemes $\mcg_{D}$ and $\mcg_{\theta(D)}$ over $\mfo_{K}$ and reductive groups $\ul{G}_{D}$ and $\ul{G}_{\theta(D)}$ over $\bs{k}$. Let $\mcs$ be the maximal split subtorus of $\mcg_{D}$, whose generic fiber is $S$ and whose special fiber, denoted by $\ul{S}$, is a maximal split torus of $\ul{G}_{D}$. In particular, since $D$ is of codimension one, $\ul{G}_{D}$ is of semi-simple rank one over $\bs{k}$.
		
		Let $\Psi_D(G,S)$ denote the subset of affine roots in $\Psi(G,S)$ that vanish on $D$ and $\Phi_D(G,S)$ the subset of $\Phi(G,S)$ as derivatives of the affine roots in $\Psi_D(G,S)$.  Then, the set of roots $\Phi(\ul{G}_{D},\ul{S})$ of $\ul{G}_{D}$ could be identified with $\Phi_D(G,S)$ (\emph{cf.} Proposition \ref{proppanelDroot}). 
		
		The simply connected cover of the derived subgroup of $\ul{G}_{D}$ for each $D$ in $\mfh_{D}$ is 
		\begin{equation}\label{eqGalpha}\ul{G}_{D}^{\mathrm{sc}}\cong\begin{cases}
				\mrres_{\bs{l}_{D}/\bs{k}}\mrsl_{2},\quad&\text{if}\ \Phi_D(G,S)\ \text{is reduced};\\
				\mrres_{\bs{l}_{D}/\bs{k}}\mrsu_{3},\quad&\text{if}\  \Phi_D(G,S)\ \text{is not reduced},
		\end{cases}\end{equation} 
		where $\bs{l}_{D}/\bs{k}$ is a finite extension.  
		Let $q_{D}$\index{$q_{D},q_{\mfh}$} be the cardinality of $\bs{l}_{D}$\index{$\bs{l}_{D},\bs{l}_{\mfh}$}. Let $Q_{D}=q_{D}$\index{$Q_{D},Q_{\mfh}$} if $\Phi_D(G,S)$ is reduced, and $Q_{D}=q_{D}^{3}$ otherwise. Since $\ch{D}{}(G)$ is in bijection with $\ch{}{}(\ul{G}_{D})$, it is of cardinality $Q_{D}+1$. 
		
		Let $\mfh=\mfh_{D}$\index{$\mfh_{D}$} be the hyperplane in $\mca$ that contains $D$. By definition, $\Psi_D(G,S)$, $\Phi_D(G,S)$, $\ul{G}_D$ and $\ul{G}_D^{\mathrm{sc}}$ depend only on $\mfh$.
		In particular, $\bs{l}_D$, $q_D$ and $Q_{D}$ depend only on $\mfh$, which we also write as $\bs{l}_{\mfh}$, $q_{\mfh}$ and $Q_{\mfh}$ respectively.  
		
		We first discuss the case where $\mfh_{D}$ is not $\theta$-stable. The following proposition follows exactly from Proposition \ref{propdistchamber}.
		
		\begin{proposition}\label{propskewpair}
			
			Assume that $\mfh_{D}$ is not $\theta$-stable. Then there are exactly two $\Hstar$-orbits of $\ch{D}{}(G)$ of $\theta$-distance $d$ and $d+1$ and of cardinality $1$ and $Q_{D}$ respectively for some $d\geq 0$. Moreover, the chambers in the second orbit are $\Hstar\cap P_{C}$-conjugate to each other, where $C$ is the unique chamber in the first $\Hstar$-orbit.
			
		\end{proposition}
		
		Now we assume that $\mfh_{D}$ is $\theta$-stable, which in particular means that both $D$ and $\theta(D)$ are panels in $\mca$. We consider the $\theta$-stable bounded set $D\cup \theta(D)$ in $\mca$, which is in particular contained in $\mfh_{D}$. We also define the group scheme $\mcg_{D\cup \theta(D)}$  over $\mfo_{K}$, and the corresponding reductive group $\ul{G}_{D\cup\theta(D)}$ over the finite field $\bs{k}$, whose semi-simple rank is 1. Since $D\cup\theta(D)$ is $\theta$-stable, $\theta$ induces an involution on $\ul{G}_{D\cup\theta(D)}$ over $\bs{k}$, which is still denoted by $\theta$ by abuse of notation.
		
		Recall that the group $\ul{G}_{D\cup\theta(D)}$ is independent of $D$, but only depends on $\mfh_{D}$. In particular, we have two homomorphisms $\rho_{D}$ and $\rho_{\theta(D)}$
		$$\xymatrix{ \mcg_{D}  & \mcg_{D\cup\theta(D)} \ar[l]_-{\rho_{D}} \ar[r]^-{\rho_{\theta(D)}} & \mcg_{\theta(D)} }$$ induced from the inclusions $D\subset D\cup\theta(D)$ and $\theta(D)\subset D\cup\theta(D)$, which induces isomorphisms:
		$$\xymatrix{ \ul{G}_{D}  & \ul{G}_{D\cup\theta(D)} \ar[l]_-{\rho_{D}} \ar[r]^-{\rho_{\theta(D)}} & \ul{G}_{\theta(D)} }.$$
		We use these maps to regard $\mcs$ as a maximal split torus of $\mcg_{D}$, $\mcg_{\theta(D)}$ and $\mcg_{D\cup \theta(D)}$  over $\mfo_{K}$, and $\ul{S}$ as a maximal split torus of $\ul{G}_{D}$, $\ul{G}_{\theta(D)}$ and $\ul{G}_{D\cup\theta(D)}$ over $\bs{k}$.
		
		We denote by $\ch{D}{}(G)$ (resp. $\ch{\theta(D)}{}(G)$) the set of chambers in $\mcb(G)$ having $D$ (resp. $\theta(D)$) as a panel, and by $\ch{}{}(\ul{G}_{D})$ (resp. $\ch{}{}(\ul{G}_{\theta(D)})$) the set of Borel subgroups of $\ul{G}_{D}$ (resp. $\ul{G}_{\theta(D)}$). Then we have the induced diagram of bijections:
		\begin{equation}\label{eqChGDthetaD}
			\begin{aligned}
				\xymatrix{ \ch{}{}(\ul{G}_{D}) \ar[d]_-{}  & \ar[l]_-{} \ar[ld]^-{} \ch{}{}(\ul{G}_{D\cup\theta(D)}) \ar[r]^-{} \ar[rd]_-{} & \ch{}{}(\ul{G}_{\theta(D)}) \ar[d]^-{} \\
					\ch{D}{}(G)& & \ch{\theta(D)}{}(G)}
			\end{aligned},
		\end{equation}
		where the two horizontal isomorphisms are induced by $\rho_{D}$ and $\rho_{\theta(D)}$, and the two vertical isomorphisms are induced by taking the generic fiber and the maximal reductive quotient of the special fiber of the group scheme $\mcg_{D}$ and $\mcg_{\theta(D)}$ as in \eqref{bijection}.

		The simply connected covers of the derived subgroups of $\ul{G}_{D}$, $\ul{G}_{\theta(D)}$ and $\ul{G}_{D\cup\theta(D)}$ are isomorphic to the group in \eqref{eqGalpha}. Let $p_{D}:\ul{G}_{D}^{\mathrm{sc}}\rightarrow\ul{G}_{D}$ be the natural morphism. So $\ul{G}_{D}^{\mathrm{sc}}\cong \ul{G}_{D\cup\theta(D)}^{\mathrm{sc}}$ is endowed with an involution $\theta_{D}$ from the involution $\theta$ on $\ul{G}_{D\cup\theta(D)}$ via Lemma \ref{lemmascinvolutionext}. Let $\ul{S}_{D}=p_{D}^{-1}(\ul{S})$ be defined as a rank one $\theta_{D}$-stable split torus in $\ul{G}_{D}^{\mathrm{sc}}$.
		
		Let $\ul{H}_{D\cup \theta(D)}=(\ul{G}_{D\cup \theta(D)}^{\theta})^{\circ}$. Then the $\ul{H}_{D\cup \theta(D)}$-orbits in $\ch{}{}(\ul{G}_{D\cup \theta(D)})$ are classified in Proposition \ref{propgeneralclass}. Using \eqref{eqChGDthetaD}, the $\parah{D\cup\theta(D)}{}\cap H$-orbits of $\ch{D}{}(G)$ and $\ch{\theta(D)}{}(G)$ can be classified accordingly, saying that each $\parah{D\cup\theta(D)}{}\cap H$-orbit of $\ch{D}{}(G)$ or $\ch{\theta(D)}{}(G)$ corresponds bijectively to an $\ul{H}_{D\cup \theta(D)}$-orbit of $\ch{}{}(\ul{G}_{D\cup \theta(D)})$ . 
		
		In general, the $H$-orbits in $\ch{D}{}(G)$ and $\ch{\theta(D)}{}(G)$ are also classified, where each $H$-orbit in $\ch{D}{}(G)$ (resp. $\ch{\theta(D)}{}(G)$) corresponds to the union of several $\ul{H}_{D\cup \theta(D)}$-orbits in $\ch{}{}(\ul{G}_{D\cup \theta(D)})$. 
		
		This correspondence in \eqref{eqChGDthetaD} maintains the $\theta$-rank. More precisely, we have the following lemma.
		
		\begin{lemma}\label{lemmathetarankconst}

			The difference between the $\theta$-rank of any chamber $C\in\ch{D}{}(G)$ and the $\theta$-rank of the corresponding chamber $\ul{C}$ in $\ch{}{}(\ul{G}_{D\cup \theta(D)})$ is a constant independent of $C$.
			
		\end{lemma} 
		
		\begin{proof}
			
			Let $\mcs'$ be a maximal split torus in $\mcg_{D\cup\theta(D)}$ that is $\theta$-stable, and $S'$ its generic fiber in $G$ and $\ul{S}'$ its special fiber in $\ul{G}_{D\cup\theta(D)}$, such that the apartments $\mca(G,S')$ and $\mca_v(\ul{G}_{D\cup\theta(D)},\ul{S}')$ contain $C$ and $\ul{C}$ respectively. Notice that the $\theta$-rank of $S'$ equals the $\theta$-rank of $\ul{S}'$. Then the difference between the $\theta$-rank of $C$ and $\ul{C}$ equals the difference between the $\theta$-rank of the center $Z(\ul{G}_{D\cup\theta(D)})$ and the $\theta$-rank of the center $Z(G)$. This difference is a constant independent of the choice of $C$ and $S'$.

		\end{proof}
		
		Now we are able to propose the following important definition.
		
		\begin{definition}\label{defclasspanelpair}
			
			A \emph{panel pair} $(\mca,D)$ consists of a panel $D$ and a $\theta$-stable apartment $\mca=\mca(G,S)$ containing $D$. Let $r$ be the $\theta$-rank of $\mca$. We call $(\mca,D)$
			\begin{itemize}
				\item a \emph{skew} panel pair if $\mfh_{D}$ is not $\theta$-stable;
			\end{itemize}
			And if $\mfh_{D}$ is $\theta$-stable, we call $(\mca,D)$
			\begin{itemize}
				\item a \emph{trivial} panel pair if  there exists only one $H$-orbit in $\ch{D}{}(G)$, which is of $\theta$-rank r;
				\item an \emph{upper} panel pair if $\ul{S}_{D}$ is $\theta_{D}$-split (i.e. of $\theta_{D}$-rank 0) and $\ch{D}{}(G)$ consists of $H$-orbits of either $\theta$-rank $r$ or $r+1$. 
				\item an \emph{even} panel pair if $\ul{S}_{D}$ is $\theta_{D}$-split (i.e. of $\theta_{D}$-rank 0) and $\ch{D}{}(G)$ consists of two $H$-orbits of $\theta$-rank $r$.
				\item a \emph{lower} panel pair if $\ul{S}_{D}$ is $\theta_{D}$-invariant (i.e. of $\theta_{D}$-rank 1) and $\ch{D}{}(G)$ consists of $H$-orbits of either $\theta$-rank $r$ or $r-1$. 
				
			\end{itemize}

		\end{definition}

		The following proposition follows directly from Proposition \ref{propgeneralclass}, Proposition \ref{propskewpair}, Lemma \ref{lemmathetarankconst} and the above discussion.
		
		\begin{proposition}\label{proppanelpair}
			Fix a panel $D$.
			\begin{enumerate}
				\item Let $\mca$ be a $\theta$-stable apartment containing $D$. Then the panel pair $(\mca,D)$ is of one of the types in Definition \ref{defclasspanelpair}.
				\item Consider all the panel pairs $(\mca,D)$ with fixed $D$, then they are 
				\begin{itemize}
					\item either skew panel pairs;
					\item or trivial panel pairs;
					\item or even panel pairs;
					\item or upper or lower pairs.
				\end{itemize}
				\item There are at most two $\Hstar$-orbits of $\ch{D}{}(G)$ of $\theta$-rank $r$ for any non-skew panel pair $(\mca,D)$, where $r=\dim\mca^\theta$. In the upper (resp. lower) case, there are at most two $\Hstar$-orbits of $\ch{D}{}(G)$ of $\theta$-rank $r+1$ (resp. $r-1$).
				\item If two different $H$-orbits (resp. $\Hstar$-orbits) are of the same $\theta$-rank, then they are of the same cardinality.
			\end{enumerate}
		\end{proposition}
		
		Thus, we also have the following definition.
		
		\begin{definition}\label{defclasspanel}
			
			A panel $D$ is called \emph{skew} (resp. \emph{non-skew}) if any panel pair $(\mca,D)$ is skew (resp. non-skew). For a non-skew panel $D$, we call it \emph{trivial/even/upper-lower} if any panel pair $(\mca,D)$ is trivial/even/upper-or-lower.
			
		\end{definition}

		\begin{remark}
			
			However, it makes no sense to say a panel $D$ is upper or lower, since if for some $\mca$ the panel pair $(\mca,D)$ is upper/lower, then we may always find another $\mca'$ such that $(\mca',D)$ is lower/upper.
			
		\end{remark}

		\subsection{Passage to the distance 0 case}
		
		Let $\mcf(G)$\index{$\mcf(G)$} be the set of facets in $\mcb(G)$ and $\Ft(G)$\index{$\Ft(G)$} its subset of $\theta$-stable facets. We define the set of \emph{maximal $\theta$-stable facets}
		$$\Ftmax(G)=\{F\in\Ft(G)\mid F^{\theta}\ \text{is a chamber of}\ \mca^{\theta}\ \text{for some}\ \theta\text{-stable apartment}\ \mca \}.\index{$\Ftmax(G)$}$$  
		
		\begin{remark}
			
			Although $\theta$-stable facets of maximal dimension are indeed maximal $\theta$-stable facets, in general a maximal $\theta$-stable facet need not be of maximal dimension.
			
		\end{remark}

		For any $F\in\Ftmax(G)$, we denote by $\ch{F}{}(G)$ the set of chambers that admit $F$ as a facet and by $\ch{F}{0}(G)$\index{$\ch{F}{0}(G)$} the subset of those chambers of $\theta$-distance 0 that admit $F$ as the  maximal $\theta$-stable facet. Equivalently,  a chamber $C\in \ch{F}{}(G)$ lies in $\ch{F}{0}(G)$ if and only if it is of $\theta$-distance $0$ and $\theta$-rank $\mrdim F^{\theta}$. 
        We denote by $\ch{}{0}(G)$ the set of chambers of $\theta$-distance 0. Then by definition we have
		$$\ch{}{0}(G)=\bigsqcup_{F\in\Ftmax(G)}\ch{F}{0}(G).\index{$\ch{}{0}(G)$}$$
		For any non-negative integer $d_{0}$, we denote by $\ch{F}{d_{0}}(G)$ the set of chambers $C$, such that there exist a $\theta$-stable apartment $\mca$ containing $F$ and $C$, and a minimal gallery $C_{0},C_{1},\dots,C_{d_{0}}$ in $\mca$ with $C_{0}\in\ch{F}{0}(G)$ and $\tdist{C_{d_{0}}}=d_{0}$. Still the definition here is independent of $\mca$.
		
		\begin{remark}
			
			Here we warn the readers that the chambers in $\ch{F}{d_{0}}(G)$ do not necessarily admit $F$ as a facet. Indeed, when $d_{0}$ is large enough the intersection $\ol{C}\cap F$ is empty.
		\end{remark} 
		
		Finally, we define $\ch{F}{\infty}(G)=\bigsqcup_{d_{0}\geq 0}\ch{F}{d_{0}}(G)$. Then by definition we have
		\begin{equation}\label{eqCHGunion}
			\ch{}{}(G)=\bigcup_{F\in\Ftmax(G)}\ch{F}{\infty}(G).
		\end{equation}
		This is because for any $C$ in $\ch{}{}(G)$ with $\tdist{C}=d_{0}$, we may find a $\theta$-stable apartment $\mca$ containing $C$ and a minimal gallery $C_{0},C_{1},\dots,C_{d_{0}}$ in $\mca$, such that $\tdist{C_{0}}=0$ and $C_{d_{0}}=C$. For the maximal $\theta$-stable facet $F$ of $C_{0}$, we have $C\in\ch{F}{d_{0}}(G)$.
		
		\begin{remark}
			The union on the right-hand side of \eqref{eqCHGunion} is not necessarily disjoint. 
			For instance, assume that $F_1,F_2\in\Ftmax(G)$ are contained in some $\theta$-stable apartment $\mca$, such that $F_1^\theta$ and $F_2^\theta$ are adjacent chambers of $\mca^\theta$ sharing a panel $D_\theta$. Moreover, assume that the hyperplane $\mfh_\theta$ passing $D_\theta$ in $\mca^\theta$ is not $\theta$-adapted. Then using Corollary \ref{corthetaadapt}, we may construct a chamber $C$ lying in $\ch{F_{1}}{d_{\theta}(C)}(G)\cap\ch{F_{2}}{d_{\theta}(C)}(G)$. A specific example could be found when $G=\mrres_{L/K}\mrgl_{3}$ with $L/K$ a ramified quadratic extension of non-archimedean local fields, and $\theta$ is the usual Galois involution, and $\mca$ is a $\theta$-stable apartment of $\theta$-rank one (\emph{cf.} \cite{courtes2017distinction}).
			
		\end{remark}
		
		Before moving on, we digress to discuss the following assumption on the character $\chi$ of $H$.
		
		\begin{assumption}\label{assumpchi}
			
			$\chi$ is trivial on the subgroup $\Hstar$ of $H$.
			
		\end{assumption}
		
		We argue that the assumption is reasonable in the following sense. 
		
		First, during our proof below it is necessary to have the $\Hstar$-invariance, otherwise the dimension of $\mch(G)^{(H,\chi)}\cong \mrhom_{H}(\mrst_{G},\chi)$ becomes uncontrollable. This makes the problem of distinction ``bad enough'', since we can easily forge an example (\emph{cf.} Example \ref{examplechi}) without this assumption such that the dimension is out of control. 
		
		Secondly, for the existing case the character $\chi$ is usually a priori given, which indeed satisfies this assumption. For many cases the character $\chi$ is simply trivial; In the Galois case, since we have $(G^{\mathrm{sc}})^{\theta'}=H^{\mathrm{sc}}$, an element in $\Hstar$ equals a commutator element in $H$. So the character $\chi$ is trivial on $\Hstar$.
		
		Now we are able to prove the following claim. 
		
		\begin{proposition}\label{propreddist0}
			
			Let $\chi$ be a character of $H$ satisfying Assumption \ref{assumpchi} and $f\in\mch(G)^{(H,\chi)}$, i.e. $f$ is a harmonic cochain defined on $\ch{}{}(G)$ that is $(H,\chi)$-equivariant. Let $\mca$ be a $\theta$-stable apartment and $C_{0},C_{1},\dots,C_{d}$ a gallery in $\mca$ such that $\tdist{C_{i}}=i$. Then the value of $f$ on $C_{d}$ is determined by that on $C_{0}$. As a result, $f$ is determined by its value on $\ch{}{0}(G)$.
			
		\end{proposition}
		
		\begin{proof}
			
			Let $D_{i}$ be the panel between $C_{i-1}$ and $C_{i}$ for $i=1,\dots,d$. Then each $(\mca,D_{i})$ is a skew panel pair. By Proposition \ref{propskewpair}, there are two $\Hstar\cap \parah{C}$-orbits of $\ch{D_{i}}{}(G)$ of $\theta$-distance $i-1$ and $i$ and of cardinality $1$ and $Q_i:=Q_{D_i}$ respectively. 
			Using the harmonic condition and the $\Hstar$-invariance, we have that 
			$$f(C_{i-1})+Q_{i}f(C_{i})=0.$$
			Thus we have
			$$f(C)=\big(\prod_{i=1}^{d}(-Q_{i})^{-1}\big)f(C_{0}),$$
			which finishes the proof.

		\end{proof}
		
		\subsection{Interlude on finite reductive groups}
		
		In this part, let $\ul{G}$ be a reductive group over $\bs{k}$, and $\theta$ an involution of $\ul{G}$ over $\bs{k}$, and $\ul{H}=(\ul{G}^{\theta})^{\circ}$. We identify  $\ch{}{}(\ul{G})$ with the set of Borel subgroups of $\ul{G}$. A Borel subgroup $\ul{B}$ of $\ul{G}$ is called $\theta$-split if $\theta(\ul{B})\cap \ul{B}$ is a maximal torus of $\ul{G}$. Assume that $\ul{G}$ has a $\theta$-split Borel subgroup, or in other words, the set $\ch{-}{\theta}(\ul{G})$ of $\theta$-split Borel subgroups of $\ul{G}$ is not empty.

		We define the graph $\gamma(\ul{G},\theta)$, whose vertices are elements in the quotient $\ch{-}{\theta}(\ul{G})/\ul{H}$. We add an edge between two $\ul{H}$-conjugacy classes $[\ul{B}_1]_{\ul{H}}$ and $[\ul{B}_2]_{\ul{H}}$, if there exist two related representatives $\ul{B}_1$ and $\ul{B}_2$ in each class that share a common panel $D$. In other words, it means that there exists a parabolic subgroup $\ul{P}=\ul{P}_D$ containing $\ul{B}_1$ and $\ul{B}_2$, whose Levi subgroup is of semi-simple rank 1. We add a loop to an $\ul{H}$-conjugacy class $[\ul{B}]_{\ul{H}}$ if there is a panel $D$ of $\ul{B}$, such that the chambers in $\ch{-}{\theta}(\ul{G})$ having $D$ as a panel lie in the same $\ul{H}$-conjugacy class $[\ul{B}]_{\ul{H}}$. Using Proposition \ref{proppanelpair}.(3), the set of edges are in bijection with $\ul{H}$-conjugacy classes of panels contained in some chamber corresponding to a $\theta$-split Borel subgroup.  
		
		Our next goal is to show that such a graph must be connected. 
		
		\begin{proposition}\label{propBGHconnnected}
			
			The graph $\gamma(\ul{G},\theta)$\index{$\gamma(\ul{G},\theta)$} is connected, or in other words any two $\ul{H}$-conjugacy classes of $\ch{-}{\theta}(\ul{G})$\index{$\ch{-}{\theta}(\ul{G})$} are connected by a gallery. More generally, for any two $\theta$-split Borel subgroups $\ul{B}$, $\ul{B}'$ of $\ul{G}$, we may find a gallery $\ul{B}_{0}, \ul{B}_{1},\dots,\ul{B}_{d}$ with 
			$\ul{B}_{0}=\ul{B}$, and $\ul{B}_{d}$ and $\ul{B}'$ are in the same $\Hbarstar$-orbit, and each $\ul{B}_{i}$ is $\theta$-split.

		\end{proposition}
		
		\begin{proof}
			
			We consider the derived subgroup and the simply connected cover \begin{equation}\label{equniversalcover}
				p:	\ul{G}^{\mathrm{sc}}\rightarrow\ul{G}^{\mathrm{der}}\rightarrow\ul{G}.
			\end{equation}
			The restriction of $\theta$ to $\ul{G}^{\mathrm{der}}$ is also an involution, and using Lemma \ref{lemmascinvolutionext} we get a unique involution $\theta'$ of $\ul{G}^{\mathrm{sc}}$ that is compatible with $\theta$. We have bijections of Borel subgroups
			$$\ch{}{}(\ul{G}^{\mathrm{sc}})\leftrightarrow\ch{}{}(\ul{G}^{\mathrm{der}})\leftrightarrow\ch{}{}(\ul{G})$$
			induced by \eqref{equniversalcover}. Then once we get a legitimate gallery in $\ch{-}{\theta'}(\ul{G}^{\mathrm{sc}})$, we may use these two bijections to transfer it to a legitimate gallery in $\ch{-}{\theta}(\ul{G})$. So we may assume without loss of generality that $\ul{G}$ is semi-simple and simply connected. 
			
			We consider a $\theta$-stable apartment $\mca=\mca_v(\ul{G},\ul{S})$ such that $\mca^{\theta}=\{0\}$, or in other words, $\ul{S}$ is $\theta$-split. Then by definition any Borel subgroup containing $\ul{S}$ is necessarily $\theta$-split. Thus for any two Borel subgroups containing $\ul{S}$, we may find a gallery in $\mca$ connecting them, such that each Borel subgroup in this gallery is $\theta$-split. 
			
			Thus, the proposition reduces to the following claim:
			
			\begin{proposition}\label{propapartmentchain}
				
				Let $r$ be the split rank of $\ul{G}$. Then we may find a sequence of $\theta$-stable apartments $\mca_{0},\mca_{1},\dots,\mca_{r}$, such that
				\begin{itemize}
					\item $\mca_{i}^{\theta}=\{0\}$ for each $i=0,1,\dots,r$;
					\item $\ul{B}$ is a chamber in $\mca_{0}$, and $\ul{B}'$ is $\ul{H}$-conjugate to a chamber in $\mca_{r}$;
					\item $\mca_{i-1}$ and $\mca_{i}$ are connected for each $i=1,2,\dots,r$, saying that there exist two chambers in $\mca_{i-1}$ and $\mca_{i}$ respectively that are adjacent.
				\end{itemize} 
				
			\end{proposition} 
			
			Let $\ul{S}$ be the maximal split $\theta$-stable torus contained in $\ul{B}$. Indeed, $\ul{S}$ is the $\bs{k}$-split part of the maximal torus $\ul{T}=\ul{B}\cap\theta(\ul{B})$. Moreover, $\ul{S}$ is $\theta$-split. Let $g\in\ul{G}$ such that both $\ul{B}'=\ul{B}^{g}$ and $\ul{S}^{g}$ are also $\theta$-split, thus $t:=\theta(g)g^{-1}\in\ul{T}$. 
			
			For any non-divisible root $\alpha\in\Phi(\ul{G},\ul{S})$, we may define a related Levi subgroup $\ul{M}_{\alpha}$ of semi-simple rank one of $\ul{G}$, that is generated by $\ul{T}$, $\ul{U}_{\alpha}$ and $\ul{U}_{-\alpha}$. Let $\ul{T}_{\alpha}=\pairang{\ul{U}_{-\alpha}}{\ul{U}_{\alpha}}\cap \ul{T}$. Since $\ul{G}$ is semi-simple and simply connected, $\ul{T}$ is generated by $\ul{T}_{\alpha}$ with $\alpha$ ranging over simple roots $\Delta=\Delta(\ul{B},\ul{S})$ of $\Phi(\ul{G},\ul{S})$. Then we may write $$t=\prod_{\alpha\in\Delta}t_{\alpha}$$ 
			with $t_{\alpha}\in \ul{T}_{\alpha}$. Since $\theta(t)=t^{-1}$, we also have $\theta(t_{\alpha})=t_{\alpha}^{-1}$ for each $\alpha$. 
			
			Let $\ul{G}_{\alpha}=\pairang{\ul{U}_{-\alpha}}{\ul{U}_{\alpha}}$, which by definition is either $\mrres_{\bs{l}_{\alpha}/\bs{k}}\mrsl_{2}$ or $\mrres_{\bs{l}_{\alpha}/\bs{k}}\mrsu_{3}$ for some finite extension $\bs{l}_{\alpha}/\bs{k}$. Let $q_\alpha$ be the cardinality of $\bs{l}_{\alpha}$. Let $\theta_{\alpha}$ be the restriction of $\theta$ to $\ul{G}_{\alpha}$ and $\ul{S}_{\alpha}$ the maximal $\bs{k}$-split subtorus of $\ul{T}_{\alpha}$. Since $\ul{S}$ is $\theta$-split, the restriction of $\theta_{\alpha}$ to $\ul{S}_{\alpha}$ is the inversion map.
			
			\begin{lemma}\label{lemmagalphatalpha}
				
				Let $\theta_{\alpha}'$ be an involution on $\ul{G}_{\alpha}$ such that $\ul{S}_{\alpha}$ is $\theta_{\alpha}'$-split, and $t_{\alpha}\in\ul{T}_{\alpha}$ such that $\theta_{\alpha}'(t_{\alpha})=t_{\alpha}^{-1}$. Then there exists $g_{\alpha}\in\ul{G}_{\alpha}$ such that $\theta_{\alpha}'(g_{\alpha})g_{\alpha}^{-1}=t_{\alpha}$.
				
			\end{lemma}
			
			\begin{proof}
				
				When $\ul{G}_{\alpha}=\mrres_{\bs{l}_{\alpha}/\bs{k}}\mrsl_{2}$, we are in \textbf{Case (II)} of Proposition \ref{proprank1classSL2}. In \textbf{Case (II).(\romannumeral1)} it follows from the fact that two symmetric matrices of $\mrgl_{2}(\bs{l}_{\alpha})$ of the same determinant are in the same $\mrsl_{2}(\bs{l}_{\alpha})$-orbit with respect to the action $ x\cdot g=\,^{t}gxg$. In \textbf{Case (II).(\romannumeral2)} similarly it follows from the fact that two hermitian matrices of $\mrgl_{2}(\bs{l}_{\alpha})$ of the same determinant are in the same $\mrsl_{2}(\bs{l}_{\alpha})$-orbit with respect to the action $ x\cdot g=\sigma(\,^{t}g)xg$.
				
				When $\ul{G}_{\alpha}=\mrres_{\bs{l}_{\alpha}/\bs{k}}\mrsu_{3}$, we are in \textbf{Case (II)} of Proposition \ref{proprank1classSU3}. Let $t_{\alpha}=\mrdiag(y,y^{-1}\sigma_{\alpha}(y),\sigma_{\alpha}(y)^{-1})$ such that $\theta_{\alpha}'(t_{\alpha})=t_{\alpha}^{-1}$, we claim that we may find 
				$$g_{\alpha}=\begin{pmatrix}a &0 &b\\ 0 & (ad-bc)^{-1} & 0\\ c & 0 & d\end{pmatrix}\in \ul{G}_{\alpha}=\mrsu_{3}(\bs{l}_{\alpha,2}/\bs{l}_{\alpha}),\ a,b,c,d\in\bs{l}_{\alpha,2}$$
				such that $\theta_{\alpha}'(g_{\alpha})g_{\alpha}^{-1}=t_{\alpha}$. Here, $\bs{l}_{\alpha,2}/\bs{l}_{\alpha}$ is the quadratic extension, and we denote by $\sigma_{\alpha}$ the corresponding order 2 automorphism given by the $q_{\alpha}$-th power map. By direct calculation, the condition to guarantee $g_{\alpha}\in\mrsu_{3}(\bs{l}_{\alpha,2}/\bs{l}_{\alpha})$ is $ad-bc\neq 0$ and $\sigma_{\alpha}(a)=ax,\sigma_{\alpha}(b)=bx,\sigma_{\alpha}(c)=cx,\sigma_{\alpha}(d)=dx$ with $x=\sigma_{\alpha}(ad-bc)$.
				
				In \textbf{Case (II).(\romannumeral1)}, we are required to solve the equation
				$$\begin{pmatrix} 0 & 0 & \delta^{-1}\\ 0 & \epsilon\delta^{-1} & 0\\ 1 & 0 & 0
				\end{pmatrix}g_{\alpha}\begin{pmatrix} 0 & 0 & 1\\ 0 & \epsilon^{-1}\delta & 0\\ \delta  & 0 & 0
				\end{pmatrix}=t_{\alpha}g_{\alpha}.$$
				In this case, the condition $\theta_{\alpha}'(t_{\alpha})=t_{\alpha}^{-1}$ shows that $y\in\bs{l}_{\alpha}^{\times}$, and we have $\epsilon\in\bs{l}_{\alpha}^{\times}, \delta=\epsilon^{2}\in\bs{l}_{\alpha}^{\times}$. By direct calculation, the solution is given by $d=ay$ and $c=by\delta$, and we choose $a,b\in\bs{l}_{\alpha}$ such that $a^{2}y-b^{2}\delta y=1$, then we have $a,b,c,d\in\bs{l}_{\alpha}$,
				$g_{\alpha}\in\mrsu_{3}(\bs{l}_{\alpha,2}/\bs{l}_{\alpha})$ and $\theta_{\alpha}'(g_{\alpha})g_{\alpha}^{-1}=t_{\alpha}$. 
				
				In \textbf{Case II.(\romannumeral2)}, we are required to solve the equation
				$$\begin{pmatrix} 0 & 0 & \delta^{-1}\\ 0 & \epsilon\delta^{-1} & 0\\ 1 & 0 & 0
				\end{pmatrix}
				\sigma_{\alpha}(g_{\alpha})\begin{pmatrix} 0 & 0 & 1\\ 0 & \epsilon^{-1}\delta & 0\\ \delta  & 0 & 0
				\end{pmatrix}=t_{\alpha}g_{\alpha}.$$
				In this case, we have $\epsilon\in \bs{l}_{\alpha,2}^{\times}$ and  $\delta=\sigma_{\alpha}(\epsilon)\epsilon\in\bs{l}_{\alpha}^{\times}$. By direct calculation, the solution is given by $d=\sigma_{\alpha}(ay)$ and $c=\sigma_{\alpha}(by)\delta$, and we choose $a,b\in\bs{l}_{2,\alpha}$ such that $\sigma_{\alpha}(a)=y^{(1-q_{\alpha})/2}a$, $\sigma_{\alpha}(b)=y^{(1-q_{\alpha})/2}b$ and $\sigma_{\alpha}(a)ay-\sigma_{\alpha}(b)b\delta y=y^{(1-q_{\alpha})/2}$. Then by direct calculation, we have $\sigma_{\alpha}(c)=y^{(1-q_{\alpha})/2}c$ and $\sigma_{\alpha}(d)=y^{(1-q_{\alpha})/2}d$,  $g_{\alpha}\in\mrsu_{3}(\bs{l}_{\alpha,2}/\bs{l}_{\alpha})$ and $\theta_{\alpha}'(g_{\alpha})g_{\alpha}^{-1}=t_{\alpha}$. 
			\end{proof}
			
			We finish the proof of Proposition \ref{propapartmentchain}. We write $\Delta=\{\alpha_{1},\alpha_{2},\dots,\alpha_{r}\}$ and $t_{i}:=t_{\alpha_{i}}$. Let $\theta_{\alpha_{1}}'=\theta_{\alpha_{1}}$ be an involution on $\ul{G}_{\alpha_{1}}$ and for each $k=2,3,\dots ,r$ let  $\theta_{\alpha_{k}}'$ be an involution on $\ul{G}_{\alpha_{k}}$ given by $$\theta_{\alpha_{k}}'(g)=(\prod_{i=1}^{k-1}t_{i})^{-1}\theta_{\alpha_{k}}(g)\prod_{i=1}^{k-1}t_{i}\quad, g\in\ul{G}_{\alpha_{k}}.$$ Then using Lemma \ref{lemmagalphatalpha}, for each $\alpha_{k}$ we choose $g_{k}:=g_{\alpha_{k}}\in\ul{G}_{\alpha_{k}}$ such that $\theta_{\alpha_{k}}'(g_{k})g_{k}^{-1}=t_{k}$. Let $\mca_{0}=\mca_v(\ul{G},\ul{S})$ and $\mca_{k}=\mca_v(\ul{G},\ul{S}^{\prod_{i=1}^{k}g_{k-i+1}})$ for $k=1,\dots,r$. Then from our construction, since \begin{equation}\label{eqprodtithetagg-1}
				\theta(\prod_{i=1}^{k}g_{k-i+1})(\prod_{i=1}^{k}g_{k-i+1})^{-1}=\prod_{i=1}^{k}t_{i}\in\ul{T},
			\end{equation}
			we deduce that $\ul{S}^{\prod_{i=1}^{k}g_{k-i+1}}$ is $\theta$-split, thus $\mca_{k}^{\theta}=\{0\}$. Also $\mca_{k-1}$ and $\mca_{k}$ are connected for $k=1,\dots,r$. This is because $\mca_v(\ul{G},\ul{S})$ and $\mca_v(\ul{G},\ul{S}^{g_{k}})$ are connected, and then we may take the $\prod_{i=1}^{k-1}g_{k-i}$-conjugation. Finally, let $h=g^{-1}\prod_{i=1}^{r}g_{r-i+1}$, then by definition $\theta(h)=h$ and thus $h\in\ul{H}$. So $\ul{B}'^{h}$ is in $\mca_{r}$. Thus we proved the Proposition \ref{propapartmentchain} as well as the Proposition \ref{propBGHconnnected}.
			
		\end{proof}
		
		Under some conditions, we show that the graph is indeed \emph{bipartite}, meaning that we may divide its set of vertices into two disjoint subsets, such that each edge links vertices in different subsets. 
		
		\begin{proposition}\label{propbipartite}
			
			Assume further that
			
			\begin{itemize}
				
				\item For each panel $D$, the set of $\theta$-split Borel subgroups adjacent to $D$ is either empty or has exactly two $\ul{H}$-orbits.
				
				\item The natural morphism $p:\ul{G}^{\mathrm{sc}}\rightarrow\ul{G}$ induces a bijection
				$$\ch{-}{\theta'}(\ul{G}^{\mathrm{sc}})/(\ul{G}^{\mathrm{sc}})^{\theta'}\rightarrow\ch{-}{\theta}(\ul{G})/\ul{H}$$
				
			\end{itemize}
			
			Then the graph $\gamma(\ul{G},\theta)$ is bipartite.

		\end{proposition}
		
		\begin{remark}
			
			It is possible that the second condition is redundant. Indeed, the second condition may follow from the first condition. For instance, if $\ul{G}$ is of semi-simple rank one, then clearly it is so.	
			
		\end{remark}
		
		\begin{remark}\label{remH=Hscbipart}
			
			The morphism $p$ induces a bijection $\ch{}{}(\ul{G}^{\mathrm{sc}})\rightarrow\ch{}{}(\ul{G})$. In particular if $\ul{H}=(\ul{G}^{\mathrm{sc}})^{\theta'}$, we have bijections
			$$\ch{}{}(\ul{G}^{\mathrm{sc}})/(\ul{G}^{\mathrm{sc}})^{\theta'}\rightarrow\ch{}{}(\ul{G})/\ul{H}\quad\text{and}\quad\ch{-}{\theta'}(\ul{G}^{\mathrm{sc}})/(\ul{G}^{\mathrm{sc}})^{\theta'}\rightarrow\ch{-}{\theta}(\ul{G})/\ul{H},$$
			thus the second condition in Proposition \ref{propbipartite} is satisfied.
			
		\end{remark}
		
		\begin{proof}[Proof of Proposition \ref{propbipartite}]
			
			Using the second condition, we may replace $\ul{G}$ with $\ul{G}^{\mathrm{sc}}$ and assume without loss of generality that $\ul{G}$ is semi-simple and simply connected. In this case, the graph $\gamma(\ul{G},\theta)$ can be fully characterized.  
			
			Let $\ul{S}$ be a maximal $\theta$-split $\bs{k}$-split torus of $\ul{G}$, let $\ul{T}$ be the centralizer of $\ul{S}$ in $\ul{G}$, let $\ul{B}$ be a Borel subgroup containing $\ul{T}$, and let $\Delta=\Delta(\ul{B},\ul{S})$ be the corresponding set of simple roots. Let $r=\car{\Delta}$, which is also the split rank of $\ul{G}$. For each $\alpha\in\Delta$, we define the corresponding subgroups $\ul{G}_{\alpha}$, $\ul{S}_{\alpha}$, $\ul{T}_{\alpha}$ of $\ul{G}$.  In particular, the first condition implies that $\ch{-}{\theta}(\ul{G}_{\alpha})$ has exactly two $\ul{H}_{\alpha}$-orbits, where $\ul{H}_{\alpha}=(\ul{G}_{\alpha}^\theta)^\circ$.
			
			We remark that $\ul{T}$ is $\theta$-split, meaning that for every $t$ in $\ul{T}$, we have $\theta(t)=t^{-1}$. Indeed, we have the decomposition 
			$$\ul{T}=\prod_{\alpha\in\Delta}\ul{T}_{\alpha}\quad\text{and}\quad t=\prod_{\alpha\in\Delta}t_{\alpha}.$$
			Using the classification result, the fact that $\ul{S}$ is $\theta$-split and $\ch{-}{\theta}(\ul{G}_{\alpha})$ has exactly two $\ul{H}_{\alpha}$-orbits, for each $\alpha$ we are confined in either \textbf{Case (II).(\romannumeral1)} of Proposition \ref{proprank1classSL2} or \textbf{Case (II).(\romannumeral2)} of Proposition \ref{proprank1classSU3}. Thus we have that $\theta(t_{\alpha})=t_{\alpha}^{-1}$ for each $\alpha$, and $\theta(t)=t^{-1}$ and $\ul{T}$ is  $\theta$-split.
			
			By definition, the map
			\begin{equation}\label{eqCHG/H}
				\ul{B}\backslash\ul{G}/\ul{H}\rightarrow \ch{}{}(\ul{G})/\ul{H},\quad g\mapsto\ul{B}^{g}
			\end{equation}
			is a bijection. 
			
			\begin{lemma}
				
				Let $g\in\ul{G}$.
				
				\begin{itemize}
					\item We may change $g$ with an element in the coset $\ul{B} g$, such that $\theta(g)g^{-1}$ is in the normalizer $N_{\ul{G}}(\ul{S})$.
					
					\item Moreover, if the corresponding Borel subgroup $\ul{B}^{g}$ is $\theta$-split, then we may change $g$ with an element in the coset $\ul{B} g$, such that $\theta(g)g^{-1}$ is in $\ul{T}$.

				\end{itemize}
				
			\end{lemma}
			
			\begin{proof}
				
				Using \cite{helminck1993rationality}*{Lemma 2.3}, there exists a maximal $\theta$-stable $\bs{k}$-split torus $\ul{S}'$ in $\ul{B}^{g}$. Since both $\ul{S}$ and $\ul{S}'^{g^{-1}}$ are maximal $\bs{k}$-split torus in $\ul{B}$, we may find $b\in\ul{B}$ such that $\ul{S}^{b}=\ul{S}'^{g^{-1}}$. As a result, $\ul{S}^{bg}=\ul{S}'$ is $\theta$-stable, meaning that $\theta(bg)(bg)^{-1}\in N_{\ul{G}}(\ul{S})$. If moreover $\ul{B}^{g}$ is $\theta$-split, $\ul{S}'$ is also $\theta$-split. Then
				$\theta(bg)(bg)^{-1}\in Z_{\ul{G}}(\ul{S})=\ul{T}$.
				
			\end{proof}
			
			Using the lemma, we may further define the map
			$$\ul{B}\backslash\ul{G}/\ul{H}\rightarrow ( N_{\ul{G}}(\ul{S})\cap\{\theta(g)g^{-1}\mid g\in\ul{G}\})/\sim_{\theta,\ul{B}},\quad g\mapsto\theta(g)g^{-1} $$
			which is indeed a bijection. Here, $\sim_{\theta,\ul{B}}$ denotes the equivalence relation that for $g_1,g_2\in\ul{G}$, we have $g_1\sim_{\theta,\ul{B}}g_2$ if and only if $g_2=\theta(b)g_1b^{-1}$ for some $b\in \ul{B}$. Composing it with \eqref{eqCHG/H}, we get a bijection
			$$\ch{}{}(\ul{G})/\ul{H}\leftrightarrow
			( N_{\ul{G}}(\ul{S})\cap\{\theta(g)g^{-1}\mid g\in\ul{G}\})/\sim_{\theta,\ul{B}}.$$
			Restricting to $\ch{-}{\theta}(\ul{G})/\ul{H}$ and using again the above lemma, we have a bijection
			$$\ch{-}{\theta}(\ul{G})/\ul{H}\leftrightarrow(\ul{T}\cap\{\theta(g)g^{-1}\mid g\in\ul{G}\})/\sim_{\theta,\ul{B}}.$$
			We need the following lemma.
			
			\begin{lemma}
				
				We have $ \ul{T}\cap\{\theta(g)g^{-1}\mid g\in\ul{G}\}= \ul{T}$ and $(\ul{T}\cap\{\theta(g)g^{-1}\mid g\in\ul{G}\})/\sim_{\theta,\ul{B}}=\ul{T}/\ul{T}^2$.
				
			\end{lemma}
			
			\begin{proof}
				
				Since $\ul{T}$ is  $\theta$-split,  using the same argument as that of Proposition \ref{propapartmentchain}, we may find $g\in\ul{G}$ such that $\theta(g)g^{-1}=t$. More precisely, the element $g$ could be chosen as the product $\prod_{i=1}^{r}g_{r-i+1}$ in \eqref{eqprodtithetagg-1}. This finishes the first statement. 
				
				Since both $\ul{B}$ and $\ul{T}$ are $\theta$-split, if $t_2=\theta(b)t_1b^{-1}$ for some $t_1,t_2\in\ul{T}$, $b\in\ul{B}$, using the Bruhat decomposition we must have $b\in\ul{T}$ and $t_1=t_2b^2$, finishing the proof of the second statement.
				
			\end{proof}
			
			Using the lemma, we get a bijection
			$$\ch{-}{\theta}(\ul{G})/\ul{H}\leftrightarrow  \ul{T}/\ul{T}^2$$
			given as follows: for each $\ul{B}'\in\ch{-}{\theta}(\ul{G})$, we may find $g\in\ul{G}$ with $\ul{B}'=\ul{B}^{g}$ and $\theta(g)g^{-1}\in \ul{T}$. The class $\theta(g)g^{-1}$ in $\ul{T}/\ul{T}^2$ does not depend on the choice of $g$, which leads to the above bijection.
			
			Since
			$$\ul{T}/\ul{T}^2=\prod_{\alpha\in\Delta}\ul{T}_{\alpha}/\ul{T}_{\alpha}^2=\prod_{\alpha\in\Delta}\{\pm 1\},$$
			the graph $\gamma(\ul{G},\theta)$ has $2^{r}$'s vertices. We claim that two vertices corresponding to $(a_{\alpha})_{\alpha\in\Delta}$ and $(b_{\alpha})_{\alpha\in\Delta}$ in $\prod_{\alpha\in\Delta}\{\pm 1\}$ have a common edge if and only if there exists exactly one $\alpha\in\Delta$ such that $a_{\alpha}\neq b_{\alpha}$. 
			
			On the one hand, two vertices satisfying the related condition must be connected by an edge, since the two corresponding $\theta$-split Borel subgroups representing them are included in a common parabolic subgroup related to the root $\alpha$ whose Levi subgroup is of semi-simple rank 1. This in particular means that each vertex has degree at least $r$. On the other hand, since each chamber in the vectorial building $\mcv(\ul{G})$ has exactly $r$'s panels, so each vertex has degree at most $r$. 
			
			Thus $\gamma(\ul{G},\theta)$ is exactly the graph we described as above, which is in particular connected, $r$-regular and bipartite.
			
		\end{proof}
		
		To make the theory more complete, we also introduce the following twisted version of the graph. Let $\chi$ be a quadratic character of $\ul{H}$ that is trivial on $\Hbarstar$. Let $\ul{H}_{\chi}$\index{$\ul{H}_{\chi}$} be the kernel of $\chi$, thus $\ul{H}_{\chi}$ contains $\Hbarstar$. We similarly define a graph $\gamma(\ul{G},\theta,\chi)$\index{$\gamma(\ul{G},\theta,\chi)$}, whose vertices are $\ch{-}{\theta}(\ul{G})/\ul{H}_{\chi}$, and whose panels are $\ul{H}_{\chi}$-conjugacy classes of panels contained in some chamber corresponding to a $\theta$-split Borel subgroup.  Using Proposition \ref{propgeneralclass} and Proposition \ref{propBGHconnnected}, we have
		
		\begin{proposition}\label{propBGHchiconnnected}
			
			Under Assumption \ref{assumpchi}, $\gamma(\ul{G},\theta,\chi)$ is well-defined and also connected. 
			
		\end{proposition} 
		
		Moreover, a necessary condition for this graph to be bipartite is the following:
		
		\begin{assumption}\label{assumpHchibipartite}
			
			For each panel $D$, the set of $\theta$-split Borel subgroups adjacent to $D$ is either empty or has exactly two $\ul{H}_{\chi}$-orbits.
			
		\end{assumption}  
		
		Otherwise, the set of $\theta$-split Borel subgroups adjacent to $D$ has exactly one $\ul{H}_{\chi}$-orbit, meaning that $[D]_{\ul{H}_{\chi}}$ is a loop in $\gamma(\ul{G},\theta,\chi)$. Thus $\gamma(\ul{G},\theta,\chi)$ cannot be bipartite. 
		
		We conjecture that Assumption \ref{assumpHchibipartite} also gives a sufficient condition of the graph being bipartite.
		
		In particular, if the projection map  $\ch{-}{\theta}(\ul{G})/\ul{H}_{\chi}\rightarrow\ch{-}{\theta}(\ul{G})/\ul{H}$ is a bijection, which happens if the character $\chi$ is trivial, then we have $\gamma(\ul{G},\theta)=\gamma(\ul{G},\theta,\chi)$ and we return to the non-twisted case.
		
		Finally, the following example illustrates the necessity of considering the twisted version.
		
		\begin{example}
			
			Let $L/K$ be a quadratic ramified extension, $G=\mrres_{L/K}\mrgl_2$, $\sigma$ the Galois involution, $\theta(g)=w_0\sigma(g)w_0$ with $w_0=\begin{pmatrix}
				0 & 1\\ -1 & 0
			\end{pmatrix}$, and $x$ the point in $\mcb(G)$ corresponding to the maximal compact group $\mrgl_2(\mfo_L)$. Since $x$ is $\theta$-stable, we may take $\ul{G}=\ul{G}_x$ and $\ul{H}=(\ul{G}^\theta)^\circ$. Then one may verify directly that $\ch{-}{\theta}(\ul{G})$ has one $\ul{H}$-orbit. However, if we consider the Prasad character $\chi$ of $H$ and restrict it to $\ul{H}$, then $\ul{H}\neq\ul{H}_\chi$ and $\ch{-}{\theta}(\ul{G})$ has two $\ul{H}_\chi$-orbits. So to study the distinction of $\mrst_G$ by the Prasad character $\chi$, the graph $\gamma(\ul{G},\theta,\chi)$ is the correct object while $\gamma(\ul{G},\theta)$ is not. See \cite{courtes2017distinction} for the missing details.
			
		\end{example}
		
		\subsection{Existence of a gallery of zero $\theta$-distance chambers}
		
		In this part, we fix a facet $F\in\Ftmax(G)$. We consider the set $\ch{F}{0}(G)$. Our main proposition here is as follows:
		
		\begin{proposition}\label{propJWBgallerylocal}
			
			For any two chambers $C,C'\in\ch{F}{0}(G)$, there exists a gallery $C_{0},C_{1},\dots,C_{d}$ in $\ch{F}{0}(G)$ such that 
			$C_{0}=C$, and $C_{d}$ and $C'$ are in the same $\Hstar$-orbit.

		\end{proposition}
		
		\begin{proof}
			
			We consider the group scheme $\mcg_{F}$ over $\mfo_{K}$, and the related reductive group $\ul{G}_{F}$ over $\bs{k}$. Since $F$ is $\theta$-stable, we may also realize $\theta$ as an involution on $\mcg_{F}$ and $\ul{G}_{F}$. Then then there is a bijection between $\ch{F}{}(G)$ and $\ch{}{}(\ul{G}_{F})$ given by \eqref{bijection}. Under this bijection, the subset $\ch{F}{0}(G)$ corresponds bijectively to $\ch{-}{\theta}(\ul{G}_{F})$. Then proposition follows from Proposition \ref{propBGHconnnected} with $\ul{G}=\ul{G}_F$.

		\end{proof}
		
		\subsection{Conclusion}
		
		Combining with all the discussions above, we may give an upper bound of the distinguished dimension of the Steinberg representation.
		
		\begin{theorem}\label{thmupperbound}
			
			Let $\chi$ be a character of $H$ satisfying Assumption \ref{assumpchi}. Then the dimension of
			$$\mrhom_{H}(\mrst_{G},\chi)\cong \mch(G)^{(H,\chi)}$$ is bounded by $$\car{\Ftmax(G)/H}$$
			
		\end{theorem}
		
		\begin{proof}
			
			Let $f\in\mch(G)^{(H,\chi)}$ and $F_{1},\dots,F_{k}$ representatives in $\Ftmax(G)/H$. We also assume without loss of generality that $\mrdim F_{1}^{\theta}\geq\mrdim F_{2}^{\theta}\geq\dots\geq\mrdim F_{k}^{\theta}$. We claim that for any $n=1,2,\dots,k$ if the value of $f$ on $\ch{F_{i}}{\infty}(G)$ with $i=1,\dots,n-1$ is determined, then the value of $f$ on $\ch{F_{n}}{\infty}(G)$ is determined by its value on any fixed chamber in $\ch{F_{n}}{0}(G)$. 
			
			Let $C$ be a fixed chamber in $\ch{F_{n}}{0}(G)$. We need to show that the value of $f$ on any other chamber $C'$ in $\ch{F_{n}}{\infty}(G)$ is determined. Using Proposition \ref{propreddist0}, the value of $f$ on each $\ch{F_{i}}{\infty}(G)$ is determined by its value on $\ch{F_{i}}{0}(G)$. So essentially we may assume that $C'$ is of $\theta$-distance 0. Using Proposition \ref{propJWBgallerylocal} we may find $h\in \Hstar$ and a gallery in $\ch{F_{n}}{0}(G)$ connecting $C$ and $h\cdot C'$. Moreover, for any panel $D$ in this gallery, $\ch{D}{}(G)$ consists of at most two $\Hstar$-orbits of $\theta$-rank $\mrdim F_{n}^{\theta}$, and other $\Hstar$-orbits of $\theta$-rank $\mrdim F_{n}^{\theta}+1$ (\emph{cf.} Proposition \ref{proppanelpair}). 
			
			For $n=1$, we notice that $r:=\mrdim F_{1}^{\theta}$ is indeed the maximal possible $\theta$-rank of a $\theta$-stable apartment. Thus for any panel $D$ in the gallery connecting $C$ and $h\cdot C'$, the set $\ch{D}{}(G)$ has at most two $\Hstar$-orbits. Then using the $(H,\chi)$-equivariance and the harmonic condition of $f$, the value of $f$ on $C'$ is determined by its value on $C$. 
			
			For general $n$, the value of $f$ on $H$-conjugacy classes of chambers in $\ch{F_{i}}{\infty}(G)$ with $i<n$ are determined. In particular, the value of $f$ on any chamber of $\theta$-rank greater than $\mrdim F_{n}^{\theta}$ is determined, which follows from our ranking of $F_{i}$. Still using the $(H,\chi)$-equivariance of $f$ and the harmonic condition, the value of $f$ on $C'$ is determined by its value on $C$.
			
			As a result, if we fix $C_{i}\in\ch{F_{i}}{0}(G)$ for each $i=1,2,\dots,k$, then any $f$ is determined by its value on $C_{i}$ for each $i$. Thus $\mrdim_{\mbc}\mch(G)^{(H,\chi)}$ is bounded by $\car{\Ftmax(G)/H}$.
			
		\end{proof}
		
		We end this section with an example, showing that our Assumption \ref{assumpchi} on $\chi$ is somehow necessary.
		
		\begin{example}\label{examplechi}
			
			Let $G=\mrsl_2(K)$, $\theta$ the conjugation by $\mrdiag(1,-1)$ and $H$ the diagonal torus of $G$. Let $F$ be the vertex in $\mcb(G)$ such that $\parah{F}=\mrsl_2(\mfo_K)$. Then $\parah{F}/\parah{F,+}\cong \mrsl_2(\bs{k})$ and the set $\ch{F}{}(G)$ is in bijection with $\ch{}{}(\mrsl_2(\bs{k}))$, which corresponds to $\mbp^1(\bs{k})$. Let $\chi$ be a character of $H(K)$ trivial on $H(\mfo_K)_+$, the maximal open compact pro-$p$-subgroup of $H(\mfo_K)$, such that it induces a \textbf{faithful} character of $H(\bs{k})\cong H(\mfo_K)/H(\mfo_K)_+$. Then, the $(H,\chi)$-equivariance on harmonic cochains defined on $\ch{F}{}(G)$ imposes no restriction. Thus locally we have $q$'s linearly independant harmonic cochains. Since in this case $\mcb(G)$ is a tree, it is easily seen that each harmonic cochain defined on $\ch{F}{}(G)$ could be extended to an $(H,\chi)$-equivariant harmonic cochain in $\mch(G)$. This is definitely not what we want.
			
		\end{example}
		
		\section{$H$-orbits of maximal $\theta$-stable facets}\label{sectionHorbitofFmax}
		
		In this section, we further characterize the set $\Ftmax(G)/H$. Let $\chi$ be a character of $H$ satisfying the Assumption \ref{assumpchi}. We define a graph structure $\Gamma(G,\theta)$ on $\Ftmax(G)/H$, with emphasis on some of its connected components called $\chi$-effective connected components. Then we show that  the dimension of $\mch(G)^{(H,\chi)}$ is bounded by the number of effective connected components of $\Gamma(G,\theta)$.
		
		From now on, we admit Assumption \ref{assumpZHZG}. Then $\mcb(G)^{\theta}$ is identified with $\mcb(H)$ via the embedding $\iota$.
		
		\subsection{Equivalence relation on $H$-conjugacy classes of $\theta$-stable apartments}
		
		We consider $H$-orbits of $\theta$-stable apartments of $G$. For a $\theta$-stable apartment, we denote by $[\mca]_{H}$\index{$[\mca]_{H},[F]_{H},[F^{\theta}]_{H}$} its $H$-orbit. We have shown that the set of $H$-orbits $\{[\mca]_{H}\}$ is finite. Similarly, given a $\theta$-stable facet $F\in\Ft(G)$, we denote by $[F]_{H}$ the $H$-conjugacy class of $F$, and by $[F^{\theta}]_{H}$ that of $F^{\theta}$.
		
		We call two $\theta$-stable apartments $\mca$ and $\mca'$ \emph{$\theta$-equivalent}, if their $\theta$-invariant affine subspaces $\mca^{\theta}$ and $\mca'^{\theta}$ are $H$-conjugate. In particular, any two representatives in a certain $H$-orbit $[\mca]_{H}$ are $\theta$-equivalent. Let $[\mca]_{H,\sim}$\index{$[\mca]_{H,\sim}$} be the corresponding $\theta$-equivalence class of a $\theta$-stable apartment $\mca$.
		
		We define a map
		\begin{equation}\label{eqthetamaxfacetbij}
			\Ftmax(G)/H\rightarrow \bigsqcup_{[\mca]_{H,\sim}}\ch{}{}(\mca^{\theta})/H,\quad [F]_{H}\mapsto [F^{\theta}]_{H}
		\end{equation}
		via the following lemma. 
		
		\begin{lemma}\label{lemmauniquethetaequiv}

			Given $F\in\Ftmax(G)$, there exists a $\theta$-stable apartment $\mca$ containing $F$ unique up to $\theta$-equivalence, such that $F^{\theta}$ is a chamber of $\mca^{\theta}$.
			
		\end{lemma}
		
		\begin{proof}
			
			Take two $\theta$-stable apartments $\mca_{1},\mca_{2}$ of $G$, such that $F^{\theta}$ is a chamber of both $\mca_{1}^{\theta}$ and $\mca_{2}^{\theta}$. Since both $\mca_{1}^{\theta}$ and $\mca_{2}^{\theta}$ are affine subspaces of some apartments of $\mcb(H)$, using \cite{bruhat1972groupes}*{Corollaire 7.4.9.(1)}, up to $\para{F^{\theta}}{H}$-conjugacy we may assume that $\mca_{1}^{\theta}$ and $\mca_{2}^{\theta}$ are included in a common apartment $\mca_{H}$ of $\mcb(H)$. Since $F^{\theta}$ generates the affine subspaces $\mca_{1}^{\theta}$ and $\mca_{2}^{\theta}$ in $\mca_{H}$, we must have $\mca_{1}^{\theta}=\mca_{2}^{\theta}$.
			
		\end{proof}

		Then, it is clear that the map \eqref{eqthetamaxfacetbij} is well-defined and bijective.
		
		\begin{remark}\label{remassumZHZG}
			
			Indeed, Lemma \ref{lemmauniquethetaequiv} is the only place where Assumption   \ref{assumpZHZG} is used in Section \ref{sectionHorbitofFmax}. On the other hand, if we already have a classification result of $H$-conjugacy classes of $\theta$-stable apartments of $\mcb(G)$, it is possible to prove Lemma \ref{lemmauniquethetaequiv} directly without consulting Assumption   \ref{assumpZHZG}.
			
		\end{remark}
		
		We also record the following lemma as a special case of \cite{bruhat1972groupes}*{Corollaire 7.4.9.(1)}.
		
		\begin{lemma}\label{lemmathetaequigconj}
			
			Let $\mca_{1},\mca_{2}$ be two $\theta$-stable apartments of $G$ such that $\mca_{1}^{\theta}=\mca_{2}^{
				\theta}$. Then we may find $g\in G$ fixing each point in $\mca_{1}^{\theta}=\mca_{2}^{
				\theta}$, such that $\mca_{2}=\mca_{1}^{g}$.
			
		\end{lemma}
		
		\subsection{Effective maximal $\theta$-stable facets}
		
		A maximal $\theta$-stable facet $F$ is called $\chi$-\emph{effective} (resp. \emph{weakly $\chi$-effective}) if there exists a non-zero $(H,\chi$)-equivariant (resp. $(\parah{F}\cap H,\chi)$-equivariant) harmonic cochain $\phi_{F}$ on $\ch{F}{}(G)$ that is supported on $\ch{F}{0}(G)$. By Proposition \ref{propJWBgallerylocal} such $\phi_{F}$ is unique up to a scalar. Since $\chi$ is usually fixed, we will simply say \emph{effective} or \emph{weakly effective} instead by omitting $\chi$. 
		
		It is clear that effectiveness implies weakly effectiveness. We also have the following obvious sufficient condition.
		
		\begin{proposition}\label{propweakeffective}
			
			A weakly effective maximal $\theta$-stable facet $F$ is effective if the $H$-orbits and the $\parah{F}\cap H$-orbits of $\ch{F}{0}(G)$ coincide.
			
		\end{proposition}
		
		We denote by $\Fteff(G)$\index{$\Fteff(G)$} the subset of $\Ftmax(G)$ of effective maximal $\theta$-stable facets. It is clear that $\Fteff(G)$ is stable under $H$-conjugacy, thus it makes sense to say that a certain $H$-conjugacy classes $[F]_{H}$ is effective or not.

		In order to give a criterion of weakly effectiveness, we temporally assume that the character $\chi$ is trivial on $\parah{F,+}\cap H$. For instance, this happens if $\chi$ is a quadratic character, since $\parah{F,+}$ is a pro-$p$-group with $p\neq 2$. 
		In this case, $\chi$ induces a character of the subgroup $\ul{H}_{F}:=(\ul{G}_F^{\theta})^\circ\cong P_{F}\cap H/P_{F,+}\cap H$\index{$\ul{H}_{F}$}, which we still denote by $\chi$.
		
		Then, the following criterion of weakly effectiveness is clear.
		
		\begin{proposition}
			
			Let $\chi$ be a character of $H$ satisfying Assumption \ref{assumpchi} and trivial on $\parah{F,+}\cap H$ with $F$ a maximal $\theta$-stable facet. Then $F$ is weakly effective if and only if the associated connected graph $\gamma(\ul{G}_{F},\theta,\chi)$ is bipartite.
			
		\end{proposition}

		The following corollary follows from Proposition \ref{propbipartite} and Remark \ref{remH=Hscbipart}.
		
		\begin{corollary}\label{coreffectiveness}
			
			Assume further that the restriction of $\chi$ to  $\ul{H}_{F}$ is trivial. Then a maximal $\theta$-stable facet $F$ is weakly effective if 
			\begin{itemize}
				\item for any panel $D$ admitting $F$ as its maximal $\theta$-stable facet, the set $\ch{D}{0}(G)$ has two $\parah{F}\cap H$-orbits, and
				\item the natural morphism $p:\ul{G}_{F}^{\mathrm{sc}}\rightarrow\ul{G}_{F}$ induces a bijection
				$$\ch{-}{\theta'}(\ul{G}_{F}^{\mathrm{sc}})/(\ul{G}_{F}^{\mathrm{sc}})^{\theta'}\rightarrow\ch{-}{\theta}(\ul{G}_{F})/\ul{H}_{F}.$$ 
			\end{itemize}
		\end{corollary}

		\subsection{A graph structure on $\Ftmax(G)/H$}\label{subsectiongraphGammaGAtheta}
		
		In this part, we define a graph structure on $\Ftmax(G)/H$. We use the bijection \eqref{eqthetamaxfacetbij} to identify it with $\bigsqcup_{[\mca]_{H,\sim}}\ch{}{}(\mca^{\theta})/H$. We realize each element in  $\bigsqcup_{[\mca]_{H,\sim}}\ch{}{}(\mca^{\theta})/H$ as a vertex. Two vertices $[F_{1}^{\theta}]_{H}$ and $[F_{2}^{\theta}]_{H}$ are linked by an edge if 
		\begin{enumerate}
			
			\item there are two related representatives $F_{1}^{\theta}$ and $F_{2}^{\theta}$ that are adjacent as two chambers of $\mca^{\theta}$ for some $\theta$-stable apartment $\mca$, and;
			
			\item Denote by $D_{\theta}$ the panel between $F_{1}^{\theta}$ and $F_{2}^{\theta}$ in $\mca^{\theta}$, by $\mfh_{\theta}$ the wall of $\mca^{\theta}$ containing $D_{\theta}$ and by $\mfh$ the unique hyperplane of $\mca$ containing $D_{\theta}$ and perpendicular to $\mca^{\theta}$. Then, at least one of the following cases happens:
			
			\begin{enumerate}
				
				\item $F_{1}^{\theta}$ and $F_{2}^{\theta}$ are in the same $H$-orbit, in which case $[F_{1}^{\theta}]_{H}$ and $[F_{2}^{\theta}]_{H}$ are the same and the corresponding edge is a loop;
				
				\item The hyperplane $\mfh_{\theta}$ is not $\theta$-adapted, or in other words, $\mfh$ is not a wall of $\mca$;
				
				\item The hyperplane $\mfh_{\theta}$ is $\theta$-adapted, and for some panel $D$ in $\mfh$ the panel pair $(\mca,D)$ is  trivial/even/upper (\emph{cf.} Definition \ref{defclasspanelpair}, Proposition \ref{proppanelpair}). If moreover the set of chambers adjacent to $D$ and of $\theta$-rank $\mrdim(\mca^{\theta})$ has only one $H_{\chi}$-orbit\footnote{Noting that under the Assumption \ref{assumpchi}  $H_{\chi}$ contains $\Hstar$, thus in this case there are at most two $H_{\chi}$-orbits.}, then we mark the corresponding edge as a double edge, where $H_{\chi}=\ker(\chi:H\rightarrow\mbc^{\times})$.
				
			\end{enumerate}
			
		\end{enumerate}
		We denote by $\Gamma(G,\theta)$\index{$\Gamma(G,\theta)$} the associated graph. Finally, we call a vertex $[F^{\theta}]_{H}$ \emph{effective} if the corresponding $[F]_{H}$ is effective. We call a connected component of $\Gamma(G,\theta)$ \emph{effective} if every vertex within is effective.
		
		We notice that if two vertices $[F_{1}^{\theta}]_{H}$ and $[F_{2}^{\theta}]_{H}$ are in the same connected component, then they lie in $\ch{}{}(\mca^{\theta})/H$ for some  $\mca\in[\mca]_{H,\sim}$. In particular, we have $\mrdim(F_{1}^{\theta})=\mrdim(F_{2}^{\theta})$. 
		
		Thus, if we denote by $\Gamma(G,\theta,[\mca]_{H,\sim})$\index{$\Gamma(G,\theta,[\mca]_{H,\sim})$} the subgraph of $\Gamma(G,\theta)$ with vertices in $\ch{}{}(\mca^{\theta})/H$, then $\Gamma(G,\theta)=\bigsqcup_{[\mca]_{H,\sim}}\Gamma(G,\theta,[\mca]_{H,\sim})$.

		\begin{remark}
			
			We notice that condition (2b) and (2c) may depend on the choice of $\mca$. So indeed we add a corresponding edge if one of the conditions is satisfied for at least one such $\mca$. It is interesting to know if these conditions are independent of the choice of $\mca$ or not. 
			
		\end{remark}
		
		\subsection{The refined upper bound}
		
		In this part, we give a refined upper bound of the distinguished dimension.
		
		\begin{theorem}\label{thmupperboundrefine}
			
			Under the Assumption \ref{assumpchi} and Assumption \ref{assumpZHZG}, the dimension of
			$$\mrhom_{H}(\mrst_{G},\chi)\cong \mch(G)^{(H,\chi)}$$ is bounded by the number of effective connected components of $\Gamma(G,\theta)$ that do not have any double edge.
			
		\end{theorem}
		
		\begin{proof}
			
			As in Theorem \ref{thmupperbound}, 	let $f\in\mch(G)^{(H,\chi)}$ and $F\in\Ftmax(G)$. Then in \emph{loc. cit.} we have shown that the value of $f$ on $\ch{F}{\infty}(G)/H$ are determined by its value on any fixed $C\in\ch{F}{0}(G)$ and those $\ch{F''}{\infty}(G)$ with $\mrdim(F''^{\theta})>\mrdim(F^{\theta})$, which is sufficient to finish its proof. 
			
			Here instead, we only need to prove that
			\begin{enumerate}
				\item Let $F'$ be a maximal $\theta$-stable facet, such that $[F^{\theta}]_{H}$ and $[F'^{\theta}]_{H}$ are in the same connected component of $\Gamma(G,\theta)$. Then the value of $f$ on $\ch{F'}{\infty}(G)$ are determined by its value on any fixed chamber $C\in\ch{F}{0}(G)$ and those $\ch{F''}{\infty}(G)$ with $\mrdim(F''^{\theta})>\mrdim(F^{\theta})$.
				\item If moreover the connected component of $[F^{\theta}]_{H}$ in $\Gamma(G,\theta)$ is either not effective or has a double edge, then the value of $f$ on $\ch{F}{\infty}(G)$ is determined by its value on those $\ch{F''}{\infty}(G)$ with $\mrdim(F''^{\theta})>\mrdim(F^{\theta})$. 
			\end{enumerate}
			To prove the above two claims, by subtracting $f$ with another element in $\mch(G)^{(H,\chi)}$, we may without loss of generality assume that $f$ is zero on any $\ch{F''}{\infty}(G)$ with $\mrdim(F''^{\theta})>\mrdim(F^{\theta})$. 
			
			We may also without loss of generality assume that $F^{\theta}$ and $F'^{\theta}$ are adjacent in $\mca^{\theta}$ for some $\theta$-stable apartment $\mca$. Denote by $D_{\theta}$ the panel between $F^{\theta}$ and $F'^{\theta}$ in $\mca^{\theta}$, by $\mfh_{\theta}$ the hyperplane of $\mca^{\theta}$ containing $D_{\theta}$ and by $\mfh$ the unique hyperplane of $\mca$ containing $D_{\theta}$ and perpendicular to $\mca^{\theta}$. 
			
			If $F^{\theta}$ and $F'^{\theta}$ are in the same $H$-orbit, the claim (1) follows from the $(H,\chi)$-equivariance of $f$. 
			
			If $\mfh_{\theta}$ is not $\theta$-adapted, then using Corollary \ref{corthetaadapt} we may find a chamber $C'$ and two minimal galleries from $C'$ towards $C_{1}\in\ch{F}{0}(G)$ and $C_{2}\in\ch{F'}{0}(G)$ respectively. Using Proposition \ref{propreddist0} and Proposition \ref{propJWBgallerylocal}, the value of $f$ on $\ch{F'}{\infty}(G)$ is determined by that on $C_{2}$, and then by that on $C'$, and then by that on $C_{1}$ and finally by that on $C$.
			
			If $\mfh_{\theta}$ is $\theta$-adapted, let $D$ be a panel in $\mfh$ such that the panel pair $(\mca,D)$ is  identity/even/upper. Then we may find two chambers $C_{3},C_{4}$ in $\mca$ adjacent to $D$. Using Proposition \ref{propdistCAtheta} and Corollary \ref{corthetaadapt'}, we may further find two minimal galleries of the same length in $\mca$ from $C_{3}$ (resp. $C_{4}$) towards some chamber $C_{1}\in\ch{F}{0}(G)$ (resp. $C_{2}\in\ch{F'}{0}(G)$) respectively. As above, the value of $f$ on $\ch{F'}{\infty}(G)$ is determined by that on $C_{2}$, and then by that on $C_{4}$, and then by that on $C_{3}$ from our assumption on $(\mca,D)$ and $f$, and then by that on $C_{1}$ and finally by that on $C$. If moreover the edge between $[F]_{H}$ and $[F']_{H}$ is a double edge, then $\ch{D}{}(G)$ has only one $H_\chi$-orbit of $\theta$-rank $\mrdim(F^\theta)$, meaning that the value of $f$ on $C_{3}$ and $C_{4}$ is zero, and thus the value of $f$ on $\ch{F}{\infty}(G)$ is zero.
			
			Finally assume that $[F']_{H}$ is not effective with $[F'^{\theta}]_{H}$ in the same connected component of $[F^{\theta}]_{H}$. Then the restriction of $f$ to $\ch{F'}{\infty}(G)$ is zero, and thus its restriction to $\ch{F}{\infty}(G)$ is also zero.
			
			Thus claim (1) and (2) are proved, so the proof of the theorem can be finished by induction as in Theorem \ref{thmupperbound}.
			
		\end{proof}
		
		\begin{remark}
			
			We believe that the upper bound in Theorem \ref{thmupperboundrefine} is sharp.
			
		\end{remark}
		
		\begin{corollary}
			
			If each subgraph $\Gamma(G,\theta,[\mca]_{H,\sim})$ is connected, then the dimension of $\mch(G)^{(H,\chi)}$ is bounded by the cardinality of $\{[\mca]_{H,\sim}\}$.
			
		\end{corollary}
		
		\begin{example}\label{exGL3galois}
			
		Let $G=\mrgl_{3}(L)$, $L/K$ a quadratic ramified extension, $\theta$ the Galois involution and $\chi$ the Prasad character. In this case, there are exactly two $\theta$-equivalent classes $[\mca_1]_{H,\sim}$ and $[\mca_2]_{H,\sim}$ of $\theta$-rank 1 and 2 respectively. The subgraph $\Gamma(G,\theta,[\mca_1]_{H,\sim})$ has six non-effective vertices and six edges (of case (2b)) forming a cycle of length 6.  The subgraph $\Gamma(G,\theta,[\mca_2]_{H,\sim})$ has four effective vertices with no edges among them. While one point is isolated, the other three points have a double loop. See \cite{courtes2017distinction}*{Proposition 4.5, Corollary 5.3, Proposition 5.10, Proposition 5.11.(4)} for more detail. This explains the multiplicity one in this case.

        \end{example}
   
		\begin{figure}[htbp]
		\begin{minipage}{0.3\textwidth}
		\centering
				\tikzstyle{every node}=[scale=1]
				\begin{tikzpicture}[line width=0.4pt,scale=0.5][>=latex]
				\pgfmathsetmacro\ax{1}
				\pgfmathsetmacro\cy{2*sin(60)}
						
				\draw[black,-] (\ax,\cy)node[black]  {\(\bullet\)} -- (-\ax, \cy)node[black]  {\(\bullet\)} ;
				
                \draw[black,-] (-\ax,\cy) -- (-2*\ax, 0)node[black]{\(\bullet\)} ;
						\draw[black,-] (\ax,\cy) -- (2*\ax, 0)node[black]{\(\bullet\)} ;
						\draw[black,-] (-2*\ax, 0)--(-\ax,-\cy)node[black]{\(\bullet\)} ;
						\draw[black,-] (-\ax,-\cy)--(\ax,-\cy)node[black]{\(\bullet\)} ;
						\draw[black,-] (\ax,-\cy)--(2*\ax,0)node[black]{\(\bullet\)} ;
				\end{tikzpicture}
	
        \end{minipage}
        \begin{minipage}{0.3\textwidth}
		\centering
		\tikzstyle{every node}=[scale=1]
		\begin{tikzpicture}[line width=0.4pt,scale=0.5][>=latex]
					\pgfmathsetmacro\ax{1}
					\pgfmathsetmacro\cy{2*sin(60)}
					
					\foreach \ang\lab\anch in {90/ /north, -30/ /{south east}, 210/ /{south west}}{
						
						\draw ($(0,0)+(\ang:2*\ax)$) .. controls +(\ang+40:2*\ax) and +(\ang-40:2*\ax) .. ($(0,0)+(\ang:2*\ax)$);
						\draw  ($(0,0)+(\ang:2*\ax)$) .. controls +(\ang+40:2.5*\ax) and +(\ang-40:2.5*\ax) .. ($(0,0)+(\ang:2*\ax)$);}
					\node[red] at (0,0)  {\(\bullet\)};
					\node[red]at (-\cy,-\ax) {\(\bullet\)};
					\node[red] at (\cy,-\ax) [] {\(\bullet\)};
					\node[red] at (0,2*\ax) [] {\(\bullet\)} ;
				\end{tikzpicture}
		
        \end{minipage}
        \caption{$\Gamma(G,\theta,[\mca_1]_{H,\sim})\quad \text{and}\quad \Gamma(G,\theta,[\mca_2]_{H,\sim}),\quad \textcolor{black}{\bullet}: \textit{non-effective}, \quad \textcolor{red}{\bullet} : \textit{effective}$}
        \end{figure}

		\begin{example}\label{exGL3unitary}
			
			Let $G=\mrgl_{n}(K)$ ($n\geq 3$), $\theta$ an orthogonal involution and $\chi=1$. In Section \ref{sectionorthogonalgroup} we will show that the each subgraph $\Gamma(G,\theta,[\mca]_{H,\sim})$ consists of a single effective vertex and has no double loop. We will see that the distinguished dimension is the cardinality of $\{[\mca]_{H,\sim}\}$.
			
		\end{example}
		
		\begin{example}\label{exSL2orthogonal}
			
			Let $G=\mrgl_{2}(K)$, and $\theta$ the split orthogonal involution and $\chi=1$. In this case, there are three $\theta$-equivalent classes $[\mca_{0,1}]_{H,\sim}$, $[\mca_{0,2}]_{H,\sim}$ and $[\mca_1]_{H,\sim}$ of $\theta$-rank 0,0,1 respectively. In Section \ref{sectionorthogonalgroup} we will show that the subgraph $\Gamma(G,\theta,[\mca_{0,i}]_{H,\sim})$ consists of an isolated effective vertex for $i=1,2$, while $\Gamma(G,\theta,[\mca_{1}]_{H,\sim})$ consists of two isolated effective vertices.
			
		\end{example}
		
		\begin{remark}
			
			These examples show the necessity of our condition (2a), (2b), (2c) in the definition of the graph $\Gamma(G,\theta)$. In Example \ref{exGL3galois}, in a  $\theta$-stable apartment of $\theta$-rank 2 (resp. $\theta$-rank 1), case (2a) and (2c) (resp. case (2b)) may happen. Indeed, case (2a) and (2c) happen simultaneously. Also we see a double loop in the above cases. \footnote{It is possible that all the double edges are indeed double loops, but we are unable to prove it or find a counter example.} Example \ref{exSL2orthogonal} shows that $\ch{}{}(\mca^{\theta})/H$ may contain more than one effective vertices for some $\theta$-stable apartment $\mca$ (but in this case the assumption $Z(H)\subset Z(G)\cap H$ is not obeyed).
			
		\end{remark}
		
		\section{Constructing $(H,\chi)$-equivariant linear forms}\label{sectionPoincare}

		In this section, under good conditions we construct non-zero linear forms in $\mrhom_{H}(\mrst_{G},\chi)$ using the technique of Poincar\'e series. Our argument largely follows \cite{courtes2017distinction}*{Section 7}. 
		
		Still, we assume $Z(H)\subset Z(G)\cap H$. For a chamber $C$, we denote by $\mco_C$ the  $H$-orbit of $C$.
		
		\subsection{Convergence of an infinite sum over $H$-orbits}
		
		Let $\mco$ be an $H$-orbit in $\ch{}{}(G)$, and $f\in\mch^{\infty}(G)$ a smooth harmonic cochain. We would like to study the infinite sum
		\begin{equation}\label{eqsumfCO}
			\sum_{C\in\mco}\abs{f(C)}.
		\end{equation}
		
		\begin{proposition}\label{propconvabs}
			
			If any matrix coefficient of the Steinberg representation $\mrst_{G}$ is absolutely integrable on $H/Z(G)\cap H$, then the sum \eqref{eqsumfCO} is finite. In particular, this is true if the symmetric pair $(G,H)$ is strongly discrete in the sense of \cite{gurevich2016criterion}*{Definition 5.1}.
			
		\end{proposition}
		
		\begin{proof}
			
			Let $C$ be a fixed chamber. We consider a linear form $\lambda_{C}$ on $\mch^{\infty}(G)$ given by
			$$\lambda_{C}(f)=f(C).$$
			Since the corresponding Iwahori subgroup $\parah{C}$ fixes $\lambda_{C}$, we have that $\lambda_{C}$ lies in the smooth dual of $\mrst_{G}$. Thus $g\mapsto\lambda_{C}(g\cdot f)$ is a matrix coefficient for fixed $C$ and $f\in\mch^{\infty}(G)$. Fix suitable Haar measures on $G,H,H/Z(G)\cap H$,$\parah{C}\cap H/Z(G)\cap\parah{C}\cap H$, we have
			\begin{equation*}
				\begin{aligned}
					\infty&>\int_{H/Z(G)\cap H}\abs{f(h\cdot C)}dh\\
					&=\bigg(\int_{\parah{C}\cap H/Z(G)\cap\parah{C}\cap H}dh\bigg)\cdot\sum_{H/(\parah{C}\cap H)(Z(G)\cap H)}\abs{f(h\cdot C)}\\
					&=c\cdot\sum_{C'\in\mco_{C}}\abs{f(C')}.
				\end{aligned}
			\end{equation*}
			Here, $c$ is a positive real number whose value is immaterial for our purpose. So we proved that the related sum is finite.
		\end{proof}
		
		From now on, we assume that the condition in Proposition \ref{propconvabs} is satisfied. We officially note it down here:
		
		\begin{assumption}\label{assumpabsconverg}
			
			Any matrix coefficient of the Steinberg representation $\mrst_{G}$ is absolutely integrable on $H/Z(G)\cap H$.
			
		\end{assumption}
		
		\subsection{Conditions on the character $\chi$}
		
		As before, let $\chi$ be a character of $H$ satisfying Assumption \ref{assumpchi}. We further impose the following condition on $\chi$.
		
		\begin{assumption}\label{assumpchi'}
			
			The character $\chi$ is quadratic. 
			
		\end{assumption}
		
		As a direct corollary, $\chi$ induces a quadratic character of $\ul{H}_{F}=(\ul{G}_F^{\theta'})^\circ$ for any $\theta$-stable facet $F$.
		
		From now on until the end of this section, we fix such a character $\chi$.  We write $\chi$ as the product $\chi_0\cdot\epsilon_G\rest_H$, where $\chi_0$ is a quadratic character of $H$.  We remark that the character $\chi_0$ is trivial on the stabilizer of $C$ in $H$, where $C$ is any chamber in $\ch{F}{0}(G)$ with $F\in\Fteff(G)$. Indeed, let $\phi_F$ be a related non-zero $(H,\chi)$-equivariant harmonic cochain on $\ch{F}{}(G)$, then for $h\in H$ stabilizing $C$,  $$\chi(h)\phi_F(C)=(h\cdot\phi_F)(C)=\epsilon_G(h)\phi_F(h^{-1}\cdot C)=\epsilon_G(h)\phi_F(C),$$
		so $\chi_0(h)=1$.
		
		\begin{remark}
			
			It is interesting to know in general if there is a ``canonical'' quadratic character $\chi$ associated to $G/H$, such that the problem of distinction is meaningful in the sense of relative local Langlands program. For instance, if $L/K$ is quadratic and $G=\mrres_{L/K}(H_L)$, then $\chi$ is expected to be the Prasad character \cite{prasad2015arelative}. In the case where $G=\mrgl_n(K)$ and $H=\mrso_n(\varepsilon)$ that we are going to discuss later, $\chi$ is expected to be trivial.
			
		\end{remark}
		
		\subsection{$(H,\chi)$-equivariant linear forms and test functions} 
		
		We consider a maximal $\theta$-stable facet $F$. We assume that $F$ is effective, saying that there exists a non-zero $(H,\chi)$-equivariant harmonic cochain $\phi_{F}$\index{$\phi_F$} on $\ch{F}{}(G)$ supported on $\ch{F}{0}(G)$. We will always normalize $\phi_{F}$, such that $\phi_{F}=\pm 1$ on chambers in $\ch{F}{0}(G)$.
		
		We define a corresponding linear form $\lambda_{F}\in\mrhom_{H}(\mrst_{G},\chi)$ as follows:
		$$\lambda_{F}(f)=\sum_{C\in\ch{F}{0}(G)}\phi_{F}(C)\sum_{C'\in\mco_{C}}f(C')\chi_0(C'/C),\quad f\in\mch^\infty(G),\index{$\lambda_{F}$}$$
		where $\chi_0(C'/C)$ denotes the value $\chi_0(h)$ for any $h\in H$ such that $C'=h\cdot C$, which is well-defined from our assumption on $\chi_0$. 
		
		Since the first sum over $\ch{F}{0}(G)$ is finite, using Proposition \ref{propconvabs} $\lambda_{F}(f)$ is absolutely convergent and thus well-defined. Moreover, since
		\begin{align*}
			\lambda_{F}(h\cdot f)&=\sum_{C\in\ch{F}{0}(G)}\phi_{F}(C)\sum_{C'\in\mco_{C}}\epsilon_G(h)f(h^{-1}\cdot C')\chi_0(C'/C)\\
			&=\sum_{C\in\ch{F}{0}(G)}\phi_{F}(C)\sum_{C'\in\mco_{C}}\epsilon_G(h)f(C')\chi_0((h\cdot C')/C)\\
			&=\epsilon_G(h)\chi_0(h)\sum_{C\in\ch{F}{0}(G)}\phi_{F}(C)\sum_{C'\in\mco_{C}}f(C')\chi_0(C'/C)=\chi(h)\lambda_F(h),
		\end{align*} 
		the form $\lambda_F$ is $(H,\chi)$-equivariant.
		
		We also construct a related test function $f_{F}\in\mch^{\infty}(G)$. We have the following proposition due to Court\`es, whose proof can be adapted here directly.
		
		\begin{proposition}[\cite{courtes2017distinction}*{Proposition 7.23, 7.24}]\label{propfFCdef}
			
			Let $C'\in\ch{}{}(G)$, then there exists a unique chamber $C\in\ch{F}{}(G)$ that is in the closure of $C'\cup F$. Indeed, $C$ is the unique chamber in $\ch{F}{}(G)$ such that $d(C,C')=d(F,C')$. Moreover, let $\mca'$ be an apartment containing $C'$ and $F$, and $\mfH_{\mca'}(C',F)$ the set of walls in $\mca'$ that separate $C'$ and $F$, and define $$f_{F}(C'):=\phi_{F}(C)\cdot\prod_{\mfh\in\mfH_{\mca'}(C',F)}(-Q_{\mfh})^{-1},\index{$f_{F}$}$$ 
			then $f_{F}\in\mch^{\infty}(G)$. Here, $Q_{\mfh}$ is a power of $p$ defined as in \S \ref{subsectionpanelpair}.
			
		\end{proposition}
		
		Let $\mfh_{\theta}$ be a wall in $\mca^{\theta}$ for some $\theta$-stable apartment $\mca$. We define $\mfH_{\mca}(\mfh_{\theta})$\index{$\mfH_{\mca}(\mfh_{\theta})$} the set of walls $\mfh$ in $\mca$ such that $\mfh_{\theta}=\mfh\cap\mca^{\theta}$. We further define an integer $$n_{\mfh_{\theta}}=\prod_{\mfh\in\mfH_{\mca}(\mfh_{\theta})}(-Q_{\mfh}).\index{$n_{\mfh_{\theta}}$}$$
		We remark that $n_{\mfh_{\theta}}$ depends on $\mca^{\theta}$ but is independent of the choice of $\mca$. This is because by Lemma \ref{lemmathetaequigconj} for different choices of $\mca$, the corresponding sets $\mfH_{\mca}(\mfh^{\theta})$ are $G$-conjugate. 
		
		\begin{proposition}\label{propfFC'neq0}
			
			Let $F,F'$ be two maximal $\theta$-stable facets and $C\in\ch{F}{0}(G)$. Assume that $F'$ is effective, $f_{F'}(C)\neq 0$ and $\mrdim F^\theta\geq \mrdim F'^\theta$, then 
			\begin{enumerate}
				
				\item Let $C'\in\ch{F'}{}(G)$ be the unique chamber in the closure of $C\cup F'$, then $C'$ is of $\theta$-distance $0$, and $C$ is also in the closure of $C'\cup F$.	
				
				\item There exists a $\theta$-stable apartment $\mca$ containing $C,F,F'$, such that $F^{\theta}$ and $F'^{\theta}$ are chambers of $\mca^{\theta}$.
				
				\item We have $\mfH_{\mca}(F^{\theta},F'^{\theta})=\mfH_{\mca}(F,F')=\mfH_{\mca}(F,C')=\mfH_{\mca}(C,F')=\mfH_{\mca}(C,C')$.
				
				\item  
				We have
				$$[\parah{C}:\parah{C}\cap\parah{C'}]=\prod_{\mfh\in\mfH_{\mca}(C,C')}Q_{\mfh}=\prod_{\mfh\in\mfH_{\mca}(F,F')}Q_{\mfh}=[\parah{F}:\parah{F}\cap \parah{F'}]=[\parah{F^{\theta}}:\parah{F^{\theta}}\cap \parah{F'^{\theta}}].$$
				
				\item We have 
				$$\phi_{F'}(C')=f_{F'}(C)\cdot\prod_{\mfh_{\theta}\in\mfH_{\mca^{\theta}}(F^{\theta},F'^{\theta})}n_{\mfh_{\theta}}.$$

			\end{enumerate} 
			
		\end{proposition}
		
		\begin{proof}
			
			Statement (1) and (2) follow from the argument of \cite{courtes2017distinction}*{Lemma 7.27}. We only explain the last part of statement (1), whose proof was omitted in \emph{loc. cit.}. Let $C'\in\ch{F'}{}(G)$ be the unique chamber in the closure of $C\cup F'$, then we have $d(C,C')=d(C,F')$, or equivalently $\mfH_{\mca}(C,C')=\mfH_{\mca}(C,F')$ for any fixed apartment $\mca$ in $\mcb(G)$ containing $C$, $C'$. We identify $\mcb(H)$ with $\mcb(G)^\theta$ via the map $\iota$. We denote by $\mca'$ the affine subspace spanned by $F'^\theta$ in $\mca$.
			
			We claim that $\mca'$ contains $F^\theta$. Thus it is also the affine subspace generated by $F^\theta$ because $\mrdim F^\theta\geq \mrdim F'^\theta$. If the claim is not true, we may find points $\bs{p}$ in $F^\theta$ and $\bs{p}'$ in $F'^\theta$, such that the line segment $[\bs{p},\bs{p}']$ is not in $\mca'$. The closure of $C\cup F'$ in $\mca$ consists of unions of finitely many closed chambers $\overline{C_i}$, $i=0,1,2,\dots$ with $C_0=C'$, and also contains $[\bs{p},\bs{p}']$ by the convexity. If $\bs{p}'\in\overline{C_i}$ for some $i$, then $C_i$ admits $F'$ as a facet, implying that that $i=0$ from our uniqueness of $C'$. Thus if we pick a point $\bs{p}''$ in $[\bs{p},\bs{p}']$ close enough to $\bs{p'}$, we have $\bs{p}''\in \overline{C'}$. Since $\iota$ is affine, $[\boldsymbol{p},\boldsymbol{p}']$ is $\theta$-invariant and so is $\boldsymbol{p}''$ $\theta$-invariant. It means that the $\theta$-invariant part of $\overline{C'}$, containing $F'^\theta$ and $\bs{p}''$, is of dimension greater than $\mrdim(F'^\theta)$, contradictory! 
			
			Now we explain why $C$ is also in the closure of $C'\cup F$. We only need to show that $\mfH_{\mca}(C,C')=\mfH_{\mca}(F,C')$. Otherwise, there exists a wall $\mfh$ in $\mca$ separating $C,C'$ and containing $F$. Thus $\mfh$ contains $F^\theta$, $\mca'$, $F'^\theta$ and $F'$ as well, contradicting to the fact that $\mfH_{\mca}(C,C')=\mfH_{\mca}(C,F')$.  Thus we may use the argument in \emph{loc. cit.} to finish the proof of statement (1) and (2), and from now on our apartment $\mca$ is chosen to be $\theta$-stable as in the statement (2).
			
			Statement (3) is also clear from the argument above.
			
			To prove statement (4), we first need the following lemma.
			
			\begin{lemma}\label{lemmaPCPCC'index}
				
				Let $C,C'$ be two facets of an apartment $\mca$ that are strongly associated, meaning that they span the same affine subspace $\mca'$ in $\mca$. Then $$[\parah{C}:\parah{C}\cap\parah{C'}]=\prod_{\mfh\in\mfH_{\mca}(C,C')}Q_{\mfh}.$$
				
			\end{lemma}
			
			\begin{proof}
				
				When both $C$ and $C'$ are chambers in $\mca$, it follows from \cite{broussous2014distinction}*{Lemma 4.2} and induction on $d(C,C')$. 
				
				In general, $C$ and $C'$ are chambers of some affine subspace $\mca'$ of $\mca$. By induction on $d(C,C')$ we may still assume that $C$ and $C'$ are adjacent. More precisely, assume that $C''$ is a chamber of $\mca'$ in the closure of $C\cup C'$. By \cite{bruhat1972groupes}*{Corollaire 4.3.13, Remarque 4.3.14} we have 
				$$\parah{C}\parah{C''}\cap\parah{C'}\parah{C''}=\parah{C''}\quad\text{and}\quad(\parah{C''}\cap\parah{C})(\parah{C''}\cap\parah{C'})=\parah{C''}\cap\parah{C}\parah{C'}.$$ 
				Moreover, we have
				\begin{lemma}
					$\parah{C''}\cap\parah{C}\parah{C'}=\parah{C''}$
				\end{lemma}
				
				\begin{proof}
					
					We may without loss of generality assume that $C$ and $C''$ are adjacent in $\mca'$ with a common panel $D$. Let $\mfh_{D}$ be the hyperplane in $\mca'$ containing $D$. Let $\mfh_{D,C'}^{+}$ be the open half space in $\mca'$ containing $C'$ with wall being $\mfh_{D}$, then $C''$ and $C'$ lie in $\mfh_{D,C'}^{+}$. 
					
					We denote by $\parah{D,+}$ (resp. $\parah{C'',+}$) the pro-unipotent radical of $\parah{D}$ (resp. $\parah{C''}$).  Then we have 
					$$\parah{\mfh_{D,C'}^{+}}\parah{D,+}\supset\parah{C'',+}\quad\text{and thus}\quad\parah{C'}\parah{D,+}\supset\parah{C'',+}$$
					This could be shown by taking the quotient $\ul{G}_{D}\cong\parah{D}/\parah{D,+}$, then by definition $\parah{C'',+}/\parah{D,+}$ is the unipotent radical of the Borel subgroup $\parah{\mfh_{D,C'}^{+}}\parah{D,+}/\parah{D,+}$ of $\ul{G}_{D}$.
					Moreover, using the argument in \cite{courtes2017distinction}*{Lemma 7.28} for strongly associated facets $C$ and $C''$, we have
					$$\parah{C}\parah{C'',+}\supset\parah{C''}$$
					Noting that $\parah{D,+}\subset\parah{C,+}$, we get $\parah{C''}\subset\parah{C}\parah{C'}$, which finishes the proof.
					
				\end{proof}
				
				Taking the intersection with $\parah{C}\cap \parah{C'}$ we have
				$$\parah{C}\cap \parah{C'}\cap\parah{C''}=\parah{C}\cap \parah{C'}\cap\parah{C}\parah{C''}\cap\parah{C'}\parah{C''}=\parah{C}\cap\parah{C'}.$$
				As a result, we have 
				\begin{align*}
					[\parah{C}:\parah{C}\cap\parah{C'}]&=[\parah{C}:\parah{C}\cap\parah{C''}][\parah{C}\cap\parah{C''}:\parah{C}\cap \parah{C'}\cap\parah{C''}]\\
					&=[\parah{C}:\parah{C}\cap\parah{C''}][\parah{C''}:\parah{C'}\cap\parah{C''}].
				\end{align*}
				Thus we proved the transitivity of the index $[\parah{C}:\parah{C}\cap\parah{C'}]$, meaning that by induction we may reduce to the adjacent case.
				
				In this case, let $D$ be the panel in $\mca'$ between $C$ and $C'$. Then by considering the finite reductive group $\ul{G}_{D}$ as in \cite{broussous2014distinction}*{Lemma 4.2}, the same argument holds. More precisely, the image of $\parah{C}$ and $\parah{C'}$ in $\ul{G}_{D}$, denoted by $\ul{P}_{C}$ and $\ul{P}_{C'}$, are two maximal opposite parabolic subgroups of $\ul{G}_{D}$ with the intersection $\ul{P}_{C}\cap\ul{P}_{C'}$ the related Levi subgroup of $\ul{G}_{D}$ whose semi-simple $\bs{k}_D$-rank is $1$. 
				
				For the above claim, we used the following lemma which is direct.
				
				\begin{lemma}\label{lemmaopposite}
					
					Let $\ul{G}$ be a reductive group over $\bs{k}$ and $F$, $F'$ two facets that are central symmetric with regard to the origin in an apartment $\mca$ of $\mcv(\ul{G})$, then the corresponding parabolic subgroups $\ul{P}_F$ and $\ul{P}_{F'}$ are opposite with $\ul{P}_F\cap\ul{P}_{F'}$ being the corresponding Levi subgroup of $\ul{G}$.
					
				\end{lemma}
								
				Thus we have $$[\parah{C}:\parah{C}\cap\parah{C'}]=[\ul{P}_{C}:\ul{P}_{C}\cap\ul{P}_{C'}]=\prod_{\mfh\in\mfH_{\mca}(C,C')}Q_{\mfh},$$ where the first equation follows from $ \parah{C}\cap\parah{C'}/\parah{D,+}\cong \ul{P}_{C}\cap\ul{P}_{C'}$ and $\parah{C}/\parah{D,+}\cong\ul{P}_{C}$, and the right-hand side of the second equation is clearly the cardinality of the unipotent radical of $\ul{P}_{C}$.
				
			\end{proof}
			
			Using the lemma and statement (3), we have
			$$[\parah{C}:\parah{C}\cap\parah{C'}]=\prod_{\mfh\in\mfH_{\mca}(C,C')}Q_{\mfh}=\prod_{\mfh\in\mfH_{\mca}(F,F')}Q_{\mfh}=[\parah{F}:\parah{F}\cap \parah{F'}].$$
			We notice that indeed $F$ and $F'$ are strongly associated in $\mca$, since the common affine subspace they span is exactly the minimal affine subspace of $\mca$ as an intersection of walls that contains $\mca^{\theta}$. 
			
			It remains to show that \begin{equation}\label{eqPFF'PFtheta'}
				[\parah{F}:\parah{F}\cap \parah{F'}]=[\parah{F^{\theta}}:\parah{F^{\theta}}\cap \parah{F'^{\theta}}],
			\end{equation}
			which follows from the following two claims
			$$\parah{F^{\theta}}=\parah{F}\cdot(\parah{F^{\theta}}\cap \parah{F'^{\theta}})\quad\text{and}\quad \parah{F}\cap \parah{F'^{\theta}}=\parah{F}\cap \parah{F'}.$$
			Indeed, let $\parah{F^{\theta},+}$ be the pro-unipotent radical of $\parah{F^{\theta}}$, then as in \cite{courtes2017distinction}*{Lemma 7.28} we have $$\parah{F^{\theta}}=\parah{F^{\theta},+}\cdot(\parah{F^{\theta}}\cap \parah{F'^{\theta}})\quad\text{and}\quad\parah{F^{\theta},+}\subset P_{F}.$$
			So we have $\parah{F^{\theta}}=\parah{F}\cdot(\parah{F^{\theta}}\cap \parah{F'^{\theta}})$. On the other hand, we may pick points $\bs{p}'$ in $F'$ of general position\footnote{It means that $\bs{p}'$ does not lie in any affine subspace of $\mca$ whose related reflection stabilizes $F'$. As a result, an element $g$ stabilizing $F'$ and fixing $\bs{p}'$ must fix every point in $F'$.}, $\bs{p}$ in $F$ and $\bs{p}''$ in $F'^{\theta}$, such that $\bs{p}'$ lies in the line segment $[\bs{p},\bs{p}'']$. Since $\parah{F}\cap \parah{F'^{\theta}}$ fixes both $\bs{p}$ and $\bs{p}''$, it also fixes $\bs{p}'$, meaning that $\parah{F}\cap \parah{F'^{\theta}}$ is contained in $\parah{F'}$, showing that $\parah{F}\cap \parah{F'^{\theta}}=\parah{F}\cap \parah{F'}$. So \eqref{eqPFF'PFtheta'} as well as statement (4) are proved.

			Finally, statement (5) follows from statement (3) and (4), Proposition \ref{propfFCdef} and the following obvious equation
			$$\prod_{\mfh\in\mfH_{\mca}(F^{\theta},F'^{\theta})}(-Q_{\mfh})=\prod_{\mfh_{\theta}\in\mfH_{\mca^{\theta}}(F^{\theta},F'^{\theta})}\prod_{\mfh\in\mfH_{\mca}(\mfh_{\theta})}(-Q_{\mfh})=\prod_{\mfh_{\theta}\in\mfH_{\mca^{\theta}}(F^{\theta},F'^{\theta})}n_{\mfh_{\theta}}.$$
			
		\end{proof}
		
		\begin{remark}\label{remFtheta>F'theta}
			
			The condition $\mrdim F^\theta\geq\mrdim F'^\theta$ in the statement is important, which could be illustrated in the following example. Consider $G=\mrgl_3(K)$ and $H$ the split special orthogonal group of $G$. Let $F'$ be a $\theta$-stable facet of $\mcb(G)$ of maximal dimension, which is indeed a chamber of $\mcb(G)$.  Then by definition $f_{F'}(C)\neq 0$ for any chamber $C$ in $\ch{}{}(G)$. However, let $F$ be a $\theta$-stable facet of $F'$ such that $0=\mrdim F^\theta<\mrdim F'^\theta=1$, then $F^\theta$ and $F'^\theta$ cannot be chambers of some $\mca^\theta$. See Section \ref{sectionorthogonalgroup} for more detail.

			\begin{figure}[htbp]
				\begin{center}
					\tikzstyle{every node}=[scale=1]
					\begin{tikzpicture}[line width=0.4pt,scale=0.8][>=latex]
						\pgfmathsetmacro\ax{2}
						\pgfmathsetmacro\ay{0}
						\pgfmathsetmacro\bx{2*cos(120)}
						\pgfmathsetmacro\by{2*sin(120)}
						\pgfmathsetmacro\cx{2*cos(60)}
						\pgfmathsetmacro\cy{2*sin(60)}
						\pgfmathsetmacro\lcy{sin(60)}
						
						\draw[red,-] (0,0) -- (\cy, \cx) ;
      
						\draw[red,-] (-1.5,-\lcy) -- (-2*\cy, -2*\cx) ;
      
						\draw[ultra thick] (0,0) node[text=brown]{\(\bullet\)} node [right][text=brown] {$F=F^{\theta}$}-- (-1.5,-\lcy) ;
						\draw[ultra thick,blue,-] (0,0) -- (-\ax, 0) ;
						\draw[blue,-] (0,0) -- (\bx,\by) ;
						\draw[blue,-] (0,0) -- (-\bx,-\by) ;
						\draw[ultra thick,blue,-] (0,0) -- (-\cx,-\cy) ; 
						\draw[ultra thick,blue,-] (-\cx,-\cy) -- (-\ax,\ay) ;
						\draw[blue,-] (-\cx+\ax,-\cy) -- (-\cx-\ax,-\cy)
						; 
						\draw[blue,-] (\bx,\by) -- (-\cx-\ax,-\cy) ;  
						\node at (30:2.5) {\(\color{red}\mca^{\theta}\)};  
						\node at (-140:1.3) {\({F'}^{\theta}\)};
						\node at (-1, -0.2) [text=blue] {\(F'\)} ;
					\end{tikzpicture}
				\end{center}
            \caption{Remark \ref{remFtheta>F'theta}}
			\end{figure}

		\end{remark}
		
		As a direct corollary, we have:
		
		\begin{corollary}\label{corlambdaFfF'}
			
			Let $F,F'$ be two effective maximal $\theta$-stable facets. We further assume that $\mrdim F^\theta\geq \mrdim F'^\theta$. Then $\lambda_{F}(f_{F'})\neq 0$ implies that $F^{\theta}$ and $F'^{\theta}$ are chambers of $\mca^{\theta}$ for some $\theta$-stable apartment $\mca$. 
			
		\end{corollary}

		\subsection{Conditions on the $\theta$-stable apartment $\mca$}\label{subsectionconditionAtheta}
		
		Let $\mca$ be a $\theta$-stable apartment. Recall that $\mca^{\theta}$ has a poly-simplicial complex structure with the set of facets being
		$$\mcf(\mca^{\theta})=\{F^{\theta}\mid F\in\Ft(\mca)\}.$$
		From now on, we impose the following two assumptions on $\mca$:
		
		\begin{assumption}\label{assumpaffinecoxeter}

		For every wall $\mfh_{\theta}$ of $\mca^{\theta}$, there exists $h\in H$ that stabilizes $\mca^{\theta}$, whose action on $\mca^{\theta}$ is the reflection $s_{\mfh_{\theta}}$.

		\end{assumption} 
		
		Of course, this assumption concerns only the $\theta$-equivalence class $[\mca]_{H,\sim}$ of $\mca$.

		Important consequences of Assumption \ref{assumpaffinecoxeter} are as follows: first $\ch{}{}(\mca^{\theta})$ has only one $H$-orbit; secondly the set of walls $\mfH_{\theta}$ of $\mca^{\theta}$ is stable under every reflection $s_{\mfh_{\theta}}$ with $\mfh_{\theta}\in\mfH_{\theta}$, then the group $W(\mca^{\theta})$\index{$W(\mca^{\theta})$} generated by $s_{\mfh_{\theta}},\mfh_{\theta}\in\mfH_{\theta}$ is an affine Coxeter group; finally $\mca^\theta$ is the Coxeter complex related to $(\mca^\theta,\mfH_\theta)$. Thus, the graph $\Gamma(G,\theta,[\mca]_{H,\sim})$ consists of a single vertex. In particular, Assumption \ref{assumpaffinecoxeter} is satisfied if $\mca^{\theta}$ consists of a single point.

        \begin{assumption}\label{assumpnodoubleedge}

        For any $\theta$-stable apartment  $\mca'$  in $[\mca]_{H,\sim}$ and any non-skew panel $D$ in $\mca'$ such that $\ch{D}{}(G)$ consists of $H$-orbits of $\theta$-rank $r$ and possibly $H$-orbits of $\theta$-rank $r+1$ with $r=\mrdim\mca^\theta$, there are exactly two $H_{\chi}$-orbits of $\ch{D}{}(G)$ of $\theta$-rank $r$. 
            
        \end{assumption}
		
		Assumption \ref{assumpnodoubleedge} indeed means that the graph $\Gamma(G,\theta,[\mca]_{H,\sim})$ doesn't have a double edge (loop).
		
		Under our two assumptions here, we say an equivalence class $[\mca]_{H,\sim}$ is \emph{effective}, if any maximal $\theta$-stable facet $F$ with $F^{\theta}\in\ch{}{}(\mca^{\theta})$ is effective. It means that the only vertex in the graph $\Gamma(G,\theta,[\mca]_{H,\sim})$ is effective. 
		
		\begin{remark}
			
			In the proof of	Theorem \ref{thmupperboundrefine}, we have seen that Assumption \ref{assumpnodoubleedge} is necessary for the existence of a $(H,\chi)$-equivariant harmonic cochain that is ``essentially non-vanishing'' on the chambers in $[\mca]_{H,\sim}$, thus it has to be assumed if we want to constructed a non-zero $\lambda_F$. Later on we will see this assumption guarantees the non-triviality of the involution $\theta$ when restricting to some finite reductive group $\ul{G}_{D}^{\mrsc}$, where $D$ is a certain $\theta$-stable facet in $\mca$ (\emph{cf.} Proof of Lemma \ref{lemmaAthetaFtheta}.(3)).
			
		\end{remark}
		
		\subsection{More technical preparations}
		
		Fix a $\theta$-stable apartment $\mca$ as in \S \ref{subsectionconditionAtheta}. Let $F$ be a maximal $\theta$-stable facet, such that $F^{\theta}$ is a chamber of $\mca^{\theta}$. Let $D_{\theta}$ be a panel of $F^{\theta}$ and $\mfh_{\theta}$ the hyperplane in $\mca^{\theta}$ containing $D_{\theta}$. 
		
		Let $\mca_H$ be an apartment of $\mcb(H)$ containing $\mca^\theta$. First we assume that $\mfh_{\theta}$ is the intersection of some walls of $\mca_{H}$. We denote by $\mfH_{\mca_{H}}(\mfh_{\theta})$ the set of walls in $\mca_{H}$ containing $\mfh_{\theta}$. 
		We define $Q_{\mfh_{H},H}$ as the number $Q_{\mfh}$ in \S \ref{subsectionpanelpair}, but with respect to wall $\mfh_{H}$ and the group $H$. We further define $$Q_{\mfh_{\theta},H}=\prod_{\mfh_{H}\in\mfH_{\mca_{H}}(\mfh_{\theta})}Q_{\mfh_{H},H}.\index{$Q_{\mfh_{\theta},H}$}$$
		If on the other hand $\mfh_{\theta}$ is not the intersection of some walls of $\mca_H$, then we simply define $Q_{\mfh_{\theta},H}=1$.
		
		The definition here is independent of the choice of $\mca_{H}$, since all such $\mca_H$ are $\para{\mca^\theta}{H}$-conjugate to each other.

		\begin{lemma}\label{lemmaAthetaFtheta}
			\begin{enumerate}
			\item Let $F''$ be another $\theta$-stable facet, such that $F^{\theta}$ and $F''^{\theta}$ are chambers of $\mca'^{\theta}$ for some $\theta$-stable apartment $\mca'$. Then $F''^{\theta}$ is $\para{F^{\theta}}{H}$-conjugate to some chamber $F'^{\theta}$ in $\mca^{\theta}$. 
				
				\item If we also fix $F'^{\theta}$, then the number of $F''^{\theta}$  that are $\para{F^{\theta}}{H}$-conjugate to $F'^{\theta}$ is 
				\begin{equation}\label{eqPFthetaF'thetaHindex}
					[\para{F^{\theta}}{H}:\para{F^{\theta}}{H}\cap\para{F'^{\theta}}{H}]=\prod_{\mfh_{\theta}\in\mfH_{\mca^{\theta}}(F^{\theta},F'^{\theta})}Q_{\mfh_{\theta},H}.
				\end{equation}

				\item Let $\mfh_{\theta}$ be a wall of $\mca^{\theta}$ and define $m_{\mfh_{\theta}}:=Q_{\mfh_{\theta,H}}/n_{\mfh_{\theta}}$\index{$m_{\mfh_{\theta}}$}, then $\abs{m_{\mfh_{\theta}}}<1$. More precisely, if $F^{\theta}$ and $F'^{\theta}$ are adjacent and  separated by $\mfh_{\theta}$, then we have $$\abs{m_{\mfh_{\theta}}}=\frac{[\para{F^{\theta}}{H}:\para{F^{\theta}}{H}\cap\para{F'^{\theta}}{H}]}{[P_{F^{\theta}}:P_{F^{\theta}}\cap P_{F'^{\theta}}]}<1.$$ 
				
				\item $m_{\mfh_{\theta}}$ depends only on the $W(\mca^{\theta})$-conjugacy class of $\mfh_{\theta}\in\mfH_{\theta}$.
				
			\end{enumerate} 
			
		\end{lemma}
		
		\begin{proof}
			
			Statement (1) follows from \cite{bruhat1972groupes}*{Corollaire 7.4.9.(1)}. 
			
			For statement (2),  $[\para{F^{\theta}}{H}:\para{F^{\theta}}{H}\cap\para{F'^{\theta}}{H}]$ is exactly the number of $F''^{\theta}$'s that are $\para{F^{\theta}}{H}$-conjugate to $F'^{\theta}$ by definition and statement (1). So it remains to show \eqref{eqPFthetaF'thetaHindex}. Still by induction on the distance of $F^{\theta}$ and $F'^{\theta}$ in $\mca^{\theta}$, we may assume that $F^{\theta}$ and $F'^{\theta}$ are adjacent, where the transitivity of $[\para{F^{\theta}}{H}:\para{F^{\theta}}{H}\cap\para{F'^{\theta}}{H}]$ in a minimal gallery is shown as in Lemma \ref{lemmaPCPCC'index}. In this case, let $D_{\theta}$ be the panel between $F^{\theta}$ and $F'^{\theta}$, and $\mfh_{\theta}$ the hyperplane in $\mca^{\theta}$ containing $D_{\theta}$. 
			
			When $\mfh_{\theta}$ is the intersection of some walls of an apartment $\mca_{H}$ of $\mcb(H)$, let $F_{H}$ of $F_{H}'$ be the facets in $\mcb(H)$ that contain $F^{\theta}$ and $F'^{\theta}$ respectively. Then $F_{H}$ and $F_{H}'$ are adjacent, separated by the wall $\mfh_{\theta}$ in $\mca^{\theta}$. Thus we have $$\para{F^{\theta}}{H}=\para{F_{H}}{H},\quad\para{F'^{\theta}}{H}=\para{F_{H}'}{H}\quad\text{and}\quad \mfH_{\mca_{H}}(F^{\theta},F'^{\theta})=\mfH_{\mca_{H}}(F_{H},F_{H}')=\mfH_{\mca_{H}}(\mfh_{\theta}),$$ 
			where the first and the second equations follow from the fact that $F^\theta$ (resp. $F'^\theta$) is an open subset of $F_H$ (resp. $F_H'$) and thus contains a point of general position, hence $\para{F^{\theta}}{H}$ (resp. $\para{F'^{\theta}}{H}$) stabilizes $F_H$ (resp. $F_H'$). Then using Lemma \ref{lemmaPCPCC'index}, we get \eqref{eqPFthetaF'thetaHindex}. When $\mfh_{\theta}$ is not the intersection of some walls of an apartment of $\mcb(H)$, then $F^\theta$ and $F'^\theta$ are contained in a common facet $F_H$ of $\mca_H$. So we have $\para{F^{\theta}}{H}=\para{F_H}{H}=\para{F'^{\theta}}{H}$ and $Q_{\mfh_{\theta},H}=1$, and both sides of \eqref{eqPFthetaF'thetaHindex} equal 1.
			
			We prove statement (3). By definition, we already have $\abs{m_{\mfh_{\theta}}}\leq 1$. We need to show that the equality could not be obtained. 
			
			If $\mfh_{\theta}$ is not the intersection of some walls of an apartment of $\mcb(H)$, we have $[\para{F^{\theta}}{H}:\para{F^{\theta}}{H}\cap\para{F'^{\theta}}{H}]=1$ and $[P_{F^{\theta}}:P_{F^{\theta}}\cap P_{F'^{\theta}}]>1$, so $\abs{m_{\mfh_{\theta}}}<1$. 
			
			Now we assume that $\mfh_{\theta}$ is the intersection of some walls of an apartment $\mca_{H}$ in $\mcb(H)$. Let $D$ be the maximal $\theta$-stable facet of $\mcb(G)$ contained in both $\ol{F}$ and $\ol{F'}$, or equivalently the $\theta$-stable facet containing $D_{\theta}$.   Let $D_H$ be the facet of $\mcb(H)$ that contains $D_\theta$, and $F_{H}$ and $F_{H}'$ the facets in $\mcb(H)$ that contain $F^{\theta}$ and $F'^{\theta}$ respectively. From our assumption, $\mfh_{\theta}$ is the intersection of the walls in $\mca_H$ passing $D_H$. We consider the related finite reductive group $\ul{G}_{D}$ of the integral model $\mcg_D$ of $G$ and finite reductive group $\ul{H}_{D_{H}}$ of the integral model $\mch_{D_{H}}$ of $H$. Then in our settings we have $\mch_{D_{H}}\cong(\mcg_{D}^\theta)^\circ$ and $\ul{H}_{D_{H}}\cong(\ul{G}_{D}^{\theta})^{\circ}$ (\emph{cf.} \cite{prasad2020finite}*{\S 3.11}). 
			
			Using the bijection $\mcf_D(G)\leftrightarrow\mcf(\ul{G}_{D})$, we may regard $F$, $F'$ as one-dimensional facets in $\mcv(\ul{G}_{D})$, and consider the corresponding parabolic subgroups $\ul{P}_{F}$ and $\ul{P}_{F'}$ of $\ul{G}_{D}$. Then, $\ul{P}_{F}$ and $\ul{P}_{F'}$ are opposite parabolic subgroups of $\ul{G}_{D}$ with $\ul{P}_{F}\cap \ul{P}_{F'}$ the related Levi subgroup (\emph{cf.} Lemma \ref{lemmaopposite}).
			As in the proof of Lemma \ref{lemmaPCPCC'index} taking the quotient by $\parah{D,+}$ we have
			$$[\parah{F^{\theta}}:\parah{F^{\theta}}\cap \parah{F'^{\theta}}]=[\parah{F}:\parah{F}\cap \parah{F'}]=[\ul{P}_{F}:\ul{P}_{F}\cap \ul{P}_{F'}]$$
			Similarly, using the bijection $\mcf_{D_H}(H)\leftrightarrow\mcf(\ul{H}_{D_H})$, we regard $F_H$ and $F'_H$ as one-dimensional facets in $\mcv(\ul{H}_{D_H})$. We consider the corresponding parabolic subgroups $\ul{P}_{F_H,\ul{H}}$ and $\ul{P}_{F'_H,\ul{H}}$ of $\ul{H}=\ul{H}_{D_H}$. 
			Since $\para{F_H}{H}=\para{F^{\theta}}{H}$, $\para{F_H'}{H}=\para{F'^{\theta}}{H}$, as in the proof of Lemma \ref{lemmaPCPCC'index}, taking the quotient by $\para{D_H,+}{H}$ we have
			$$[\para{F^{\theta}}{H}:\para{F^{\theta}}{H}\cap\para{F'^{\theta}}{H}]=[\para{F_H}{H}:\para{F_H}{H}\cap\para{F'_H}{H}]=[\ul{P}_{F_H,\ul{H}}:\ul{P}_{F_H,\ul{H}}\cap\ul{P}_{F'_H,\ul{H}}].$$

			We need the following lemma:
			
			\begin{lemma}
				
				Let $\ul{G}$ be a non-toral reductive group, and $\theta$ a non-trivial involution of $\ul{G}$ over $\bs{k}$ with the related involution $\theta'$ on $\ul{G}^\mrsc$ being non-trivial (\emph{cf.} Lemma \ref{lemmascinvolutionext}). Let $\ul{P}=\ul{M}\ul{U}$ be a $\theta$-stable parabolic subgroup of $\ul{G}$ with $\ul{M}$ being the Levi subgroup and $\ul{U}$ the unipotent radical, and $\ul{S}$ a maximal $\bs{k}$-split $\theta$-stable torus contained in $\ul{P}$. Assume that the $\theta$-invariant part $\ul{S}^{+}$ of $\ul{S}$ is of semi-simple rank $1$. Then $\ul{U}^{\theta}\neq\ul{U}$.
				
			\end{lemma}
			
			\begin{proof}
				
				Without loss of generality, we may assume that $\ul{G}$ is semi-simple and simply connected. If $\ul{G}$ is of $\bs{k}$-rank greater than $1$, then $\ul{U}^{\theta}=\ul{U}$ implies that there are at least two $\theta$-invariant roots in $\Phi(\ul{U},\ul{S})\subset\Phi(\ul{G},\ul{S})$ that are not proportional. Then $\ul{S}^{+}$ is of rank at least $2$, which provides a contradiction. If $\ul{G}$ is of $\bs{k}$-rank $1$, then $\ul{G}$ is the Weil restriction of either $\mrsl(2)$ or $\mrsu(3)$. Then one could verify case-by-case from our classification results (i.e. Proposition \ref{proprank1classSL2} and Proposition \ref{proprank1classSU3}).
				
			\end{proof}
			
			Come back to the proof of statement (3). We first claim that the involution $\theta'$ on $\ul{G}_{D}^{\mrsc}$ is non-trivial. Otherwise, for a panel $D'$ having $D$ as a facet, the set of chambers $\ch{}{}(\ul{G}_{D'})$ has only one $(\ul{H}_{D'})_{\star}^{\mathrm{sc}}$-orbit with $\ul{H}_{D'}=(\ul{G}_{D'}^\theta)^\circ$, or the corresponding set of chambers $\ch{D'}{}(G)$ has only one $\Hstar$-orbit, contradicting to the Assumption \ref{assumpnodoubleedge}. 
			
			\begin{remark}
				
				If Assumption \ref{assumpnodoubleedge} is not satisfied, then the involution $\theta'$ could be trivial on $\ul{G}_{D}^{\mrsc}$. A typical example is that $G=\mrgl_2(K)$ and $\theta$ the symplectic involution. Then we have $H=\mrsp_2(K)=\mrsl_2(K)$ and the related involution $\theta'$ on $H$ is trivial. 
				
			\end{remark}
			
			Using the above lemma for $\ul{G}=\ul{G}_{D}$, $\ul{P}=\ul{P}_{F}$ and $\ul{M}=\ul{P}_{F}\cap\ul{P}_{F'}$, we have 
			$$\ul{P}_{F}=(\ul{P}_{F}\cap\ul{P}_{F'})\ul{U}\supsetneq(\ul{P}_{F}\cap\ul{P}_{F'})\ul{U}^{\theta}\supset(\ul{P}_{F}\cap \ul{P}_{F'})(\ul{P}_{F}\cap\ul{H}),$$
			since $\ul{P}_{F_H,\ul{H}}=\ul{P}_{F}\cap\ul{H}$ and $\ul{P}_{F_H',\ul{H}}=\ul{P}_{F'}\cap\ul{H}$, we have
			$$[\ul{P}_{F}:\ul{P}_{F}\cap \ul{P}_{F'}]>[\ul{P}_{F}\cap\ul{H}:\ul{P}_{F}\cap\ul{P}_{F'}\cap\ul{H}]=[\ul{P}_{F_H,\ul{H}}:\ul{P}_{F_H,\ul{H}}\cap \ul{P}_{F_H',\ul{H}}],$$ finishing the proof of statement (3).
			
			Statement (4) is clear, since $m_{\mfh_{\theta}}$ is invariant under $H$-conjugation, and each element in $W(\mca^{\theta})$ is realized by an $H$-action.
			
		\end{proof}
		
		Under the above setting, we further assume that $F$ is effective. Let $F'$ be another $\theta$-stable facet, such that $F'^\theta$ is also a chamber of $\mca^\theta$.
		
		\begin{lemma}\label{lemmathetaapartmentCF'}
			
			For any $C\in\ch{F}{0}(G)$, we may find a $\theta$-stable apartment $\mca'$ that contains $C$ and $\mca^\theta$.
			
		\end{lemma}
		
		\begin{proof}
			
			We first find a $\theta$-stable apartment $\mca''$ that contains $C$. Since both $\mca^\theta$ and $\mca''^{\theta}$ admit $F^\theta$ as a chamber, using \cite{bruhat1972groupes}*{Corollaire 7.4.9.(1)}, we may find $h_1\in\para{F^\theta}{H}$ such that $(\mca''^{\theta })^{h_1}=\mca^\theta$. In this case, we have $C^{h_1}\in \ch{F}{0}(G)$. Since $\para{F^\theta}{H}=\para{\mca^\theta}{H}\para{F^\theta,+}{H}$, we may find $h_2\in\para{\mca^\theta}{H}$ that maps $C^{h_1}$ to $C$. Then $\mca'=(\mca'')^{h_1h_2}$ is the apartment we want.
			
		\end{proof}
		
		Now we define a bijection $t_{F,F'}:\ch{F}{}(G)\rightarrow \ch{F'}{}(G)$\index{$t_{F,F'}$} by mapping $C\in\ch{F}{}(G)$ to the unique chamber $C'\in\ch{F'}{}(G)$ in the closure of $C\cup F'$. Using the above lemma, it induces a bijection $t_{F,F'}:\ch{F}{0}(G)\rightarrow \ch{F'}{0}(G)$.
		
		\begin{lemma}\label{lemmatFF'color}
			
			$t_{F,F'}$ further induces a bijection $\gamma(\ul{G}_F,\theta,\chi)\rightarrow\gamma(\ul{G}_{F'},\theta,\chi)$ between two bipartite graphs, which maintains or reverses the color.
			
		\end{lemma}
		
		\begin{proof}
			
			Let $C_1,C_2\in\ch{F}{0}(G)$, and $C_1'=t_{F,F'}(C_1),C_2'=t_{F,F'}(C_2)\in\ch{F'}{0}(G)$. 
			
			If $C_2=C_1^{h}$ for some $h\in\para{F^\theta}{H}$, using  $\para{F^\theta}{H}=\para{\mca^\theta}{H}\para{F^\theta,+}{H}$ we may assume that $h\in\para{\mca^\theta}{H}$. Then $h$ fixes $F$ and $F'$. Taking the $h$-conjugation we have $d(C_1,C_1')=d(C_2,C_1'^{h})$. Using the minimality of the distance (\emph{cf.} Proposition \ref{propfFCdef}), we have $C_2'=C_1'^{h}$ as well. Since $C_1,C_2$ are arbitrary, it implies that $t_{F,F'}$ induces a bijection $\gamma(\ul{G}_F,\theta,\chi)\rightarrow\gamma(\ul{G}_{F'},\theta,\chi)$.
			
			Now we consider the basic case, assuming that $C_1,C_2$ share a common panel $D$ admitting $F$ as a facet. First we show that $C_1', C_2'$ share some common panel $D'$ as well. Using Lemma \ref{lemmathetaapartmentCF'}, we may choose a $\theta$-stable apartment $\mca_1$ that contains $C_1$ and $\mca^{\theta}$. By definition $C_1'$ is also in $\mca_1$.
			
			\begin{lemma}\label{lemmapolycone}
				
				Let $\mfH_{\mca_1}(\mca^\theta)$ be the set of walls in $\mca_1$ that contain $\mca^{\theta}$, which separate $\mca_1$ into different poly-cones. Then, each poly-cone contains exactly one chamber in $\ch{F}{0}(G)$ with the walls of the poly-cone being walls of that chamber as well, where $F$ is any $\theta$-stable facet such that $F^\theta$ is a chamber of $\mca^\theta$. Moreover, for two such $\theta$-stable facets $F,F'$ with $F^\theta,F'^\theta$ chambers of $\mca^\theta$, the corresponding chambers in the same poly-cone are related by $t_{F,F'}$. 
				
			\end{lemma}
			
			\begin{proof}
				
				Indeed, for each $C\in\ch{F}{0}(G)\cap\ch{}{}(\mca_1)$ we consider the walls of $C$ that contain $F$ (and thus contain $F^\theta$ and $\mca^\theta$), which spans a poly-cone that contains  only one chamber (i.e. $C$) in $\ch{F}{0}(G)$. This proves the first part. For the second part, if $C\in\ch{F}{}(G)$, $C'\in\ch{F'}{}(G)$ lie in the same poly-cone in $\mca_1$ and $C'\notin t_{F,F'}(C)$, then there exists a wall $\mfh$ separating $C,C'$ and containing $F'$. Thus $\mfh\in\mfH_{\mca_1}(\mca^\theta)$, contradicting to the first part.
				
			\end{proof}
			
			Let $\mfh_D$ be the wall in $\mca_1$ that is generated by $D$. Then $\mfh_D$ contains $F,\mca^{\theta}$ and $F'$. Using the above lemma, $\mfh_D$ is also a wall of $C_1'$. Let $D'$ be the unique panel of $C_1'$ in $\mfh_D$. Now using Lemma \ref{lemmathetaapartmentCF'}, we may find a $\theta$-stable apartment $\mca_2$ that contains $C_2$ and $\mca^{\theta}$.
			
			\begin{lemma}
				The apartment $\mca_2$ contains $D'$ as well. 
			\end{lemma} 
			
			\begin{proof}
				
				To show the lemma, we may without loss of generality assume that $F^\theta$ and $F'^\theta$ are adjacent in $\mca^\theta$. Otherwise, we find a minimal gallery $F^\theta=F_0^\theta,F_1^\theta,\dots,F_k^\theta=F'^\theta$ in $\mca^\theta$, and for each $i=1,2,\dots,k$ pick the corresponding chamber $t_{F,F_i}(C_1)$ and panel $D_i$ of $t_{F,F_i}(C_1)$ in $\mfh_D$. Then by induction on $i$, we may subsequently show that $D_i$ is in $\mca_2$, and reduce to the adjacent case.
				
				In this case, let $\mfh_\theta$ be the hyperplane in $\mca^\theta$ separating $F^\theta$ and $F'^\theta$. We pick $\bs{p}$ in $F^\theta$ and $\bs{p}'$ its reflection alone $\mfh_\theta$ in $F'^\theta$. Then the line passing $[\bs{p},\bs{p}']$ is parallel to $\mca^\theta$. Now we pick a point $\bs{p}_1$ in $C_1$ and a point $\bs{p}_2$ in $D$ that are close to $\bs{p}$. Let $\bs{p}_{1}'$, $\bs{p}_{2}'$ be points close to $\bs{p}'$ and of general position, such that both $[\bs{p}_1,\bs{p}_1']$
				and $[\bs{p}_2,\bs{p}_2']$ are parallel to $[\bs{p},\bs{p}']$.
				Let $C'$ be the chamber containing $\bs{p}_1'$, which must admit $F'$ as a facet. Moreover, we have $\mfH_\mca(F,F')=\mfH_\mca(C_1,C')$, since any wall strictly separating $\bs{p}_1$ and $\bs{p}_1'$ cannot contain neither $F$ or $F'$. Thus we have $C'=t_{F,F'}(C_1)=C_1'$. Since $\bs{p}_2'$ is close to $\bs{p}_1'$, it is contained in a panel of $C_1'$ in $\mfh_D$, which is $D'$.
				
				Since $[\bs{p}_2,\bs{p}_2']$ is parallel to $[\bs{p},\bs{p}']$, we may perturb $\bs{p}_2'$ a little bit to get another point $\bs{p}_2''$ in $D'$, such that the ray in $\mca_1$ starting with $\bs{p}_2$ and containing $[\bs{p}_2,\bs{p}_2'']$ intersects the line in $\mca_1$ containing $[\bs{p},\bs{p}']$. It means that $\bs{p}_2''$ is in the line segment of a point in $D$ and a point in $\mca^\theta$. By convexity, $\bs{p}_2''$ as well as $D'$ are contained in $\mca_2$.
				
			\end{proof}

			Let $\mfh_D'$ be the wall in $\mca_2$ that contains $D$ and $D'$. We claim that $C_2'$ is the unique chamber in $\mca_2$ having $D'$ as a panel and lie on the same side of $\mfh_{D'}$ of $C_2$. Since $\parah{D}=\parah{D\cup D'}\parah{D,+}$, we may find an element $g\in \parah{D\cup D'}$ such that $C_2=C_1^g$. Also, taking $g$-conjugation we have $d(C_1,C_1')=d(C_2,C_1'^g)$, so by the minimality of the distance, we have $C_2'=C_1'^g$ and $C_2$ and $C_2'$ lie on the same side of $\mfh_{D'}$. Also $D'$ is a panel of $C_2'$ since $g$ fixes $D'$.  
			
			Moreover, the color change from $C_1$ to $C_2$ in $\gamma(\ul{G}_F,\theta,\chi)$ is the same as that of $C_1'$ and $C_2'$ in $\gamma(\ul{G}_{F'},\theta,\chi)$. We have found an element $g\in\parah{D\cup D'}$ such that $C_2=C_1^{g}$, $C_2'=C_1'^{g}$. If $g\notin H$, then $C_1$ and $C_2$ (resp. $C_1'$ and $C_2'$) are of different color in $\gamma(\ul{G}_F,\theta,\chi)$ (resp. $\gamma(\ul{G}_{F'},\theta,\chi) $). If $g\in H$, then $C_1$ and $C_2$ (resp. $C_1'$ and $C_2'$) are of the same or different color in $\gamma(\ul{G}_F,\theta,\chi)$ (resp. $\gamma(\ul{G}_{F'},\theta,\chi) $), depending on whether $\chi(g)$ is $1$ or $-1$. 
			
			Finally, since $\gamma(\ul{G}_F,\theta,\chi)$ is connected (\emph{cf.} Proposition \ref{propBGHconnnected}), the general case could be reduced to the above adjacent case.
			
		\end{proof}
		
		Assume that $F,F'\in\Fteff(G)$, such that $F^\theta$ and $F'^\theta$ are chambers of $\mca^\theta$ for some $\theta$-stable apartment $\mca$. We may find an element $h\in H$ such that $F'=F^h$. We define $\epsilon_{F,F'}$\index{$\epsilon_{F,F'},\epsilon_{\mfh_\theta}$} as the sign 
		$$\phi_F(C)\cdot\phi_F(t_{F',F}(h\cdot C))\cdot\chi_0(h)$$
		for any $C\in\ch{F}{0}(G)$. 
		
		\begin{proposition}\label{propepsilonFF'}
			
			\begin{enumerate}
				\item $\epsilon_{F,F'}$ is well-defined, i.e., it does not depend on the choice of $\phi_F$, $h$ and $C$.
				
				\item For $F''^\theta$ another chamber of $\mca^\theta$ lying in the closure of $F^\theta\cup F'^\theta$, we have
				$$\epsilon_{F,F'}=\epsilon_{F,F''}\epsilon_{F'',F'}.$$
				
				\item For $h\in H$, we have $\epsilon_{F^h,F'^h}=\epsilon_{F,F'}$. Notice that here $(F^\theta)^h$ and $(F'^\theta)^h$ are chambers of $(\mca^h)^\theta$ instead.
				
				\item If furthermore $F^\theta$ and $F'^\theta$ are adjacent and separated by a wall $\mfh_\theta$ in $\mca^\theta$, then the value $\epsilon_{F,F'}$ depends only on the wall $\mfh_\theta$, which we denote by $\epsilon_{\mfh_\theta}$.
				
			\end{enumerate}
			
		\end{proposition}
		
		\begin{proof}
			
			It is clear that the sign doesn't depend on $\phi_F$ (recall that we have already normalized $\phi_F$). We show the independence of $C$. It is clear that the $h$-conjugation also induces a bijection $\gamma(\ul{G}_F,\theta,\chi)\rightarrow\gamma(\ul{G}_{F'},\theta,\chi)$ between two bipartite graphs, which maintains or reverses the color. Then when $C$ varies, the color of $C$ and $t_{F',F}(h\cdot C)$ vary simultaneously. So the sign remains unchanged. Now given another element $h'\in H$ such that $F'=F^{h'}$, then we have
			\begin{align*}\phi_F(C)\cdot\phi_F(t_{F',F}(h'\cdot C))\cdot\chi_0(h')&=\phi_F(h'^{-1}h\cdot C)\cdot\phi_F(t_{F',F}(h\cdot C))\cdot\chi_0(h')=\\
				\phi_F( C)\cdot\phi_F(t_{F',F}(h\cdot C))\cdot\chi_0(h^{-1}h')\chi_0(h')&=\phi_F(C)\cdot\phi_F(t_{F',F}(h\cdot C))\cdot\chi_0(h).
			\end{align*}
			Here, we used the fact that
			$$\chi(h^{-1}h')\phi_F(C)=((h^{-1}h')\cdot\phi_F)(C)=\epsilon_G(h^{-1}h')\phi_F(h'^{-1}h\cdot C).$$ 
			So the definition is also independent of $h$. We proved statement (1).
			
			Statement (2) follows from statement (1), Lemma \ref{lemmatFF'color} and the fact that $t_{F',F}=t_{F',F''}\circ t_{F'',F}$, which follows from the definition of $t_{F',F}$ and  $$\mfH_{\mca^{\theta}}(F^{\theta},F'^{\theta})=\mfH_{\mca^{\theta}}(F^{\theta},F''^{\theta})\sqcup \mfH_{\mca^{\theta}}(F''^{\theta},F'^{\theta})\quad\text{and}\quad \mfH_{\mca}(F,F')=\bigsqcup_{\mfh_\theta\in\mfH_{\mca^{\theta}}(F^{\theta},F'^{\theta})}\mfH_\mca(\mfh_\theta).$$
			
			Statement (3) follows from  $h\cdot t_{F,F'}(C)=t_{h\cdot F,h\cdot F'}(h\cdot C)$ and the fact that $\phi_{h\cdot F}(h\cdot C)/\phi_F(C)$ is independent of $C\in\ch{F}{0}(G)$.
			
			Statement (4) follows from statements (2)(3) and the fact that $\mca^\theta$ is a Coxeter complex. More precisely, assume $D_\theta$ and $D_\theta'$ to be two chambers of $\mfh_\theta$. We pick a chamber $F_1^\theta$ (resp. $F_2^\theta$) of $\mca^\theta$ having $D_\theta$ (resp. $D_\theta'$) as a panel, such that $F_1^\theta$ and $F_2^\theta$ are separated by $\mfh_\theta$. Taking the reflection $s_{\mfh_\theta}$, we get $\epsilon_{F_1,F_2}=\epsilon_{s_{\mfh_\theta}(F_1),s_{\mfh_\theta}(F_2)}$. Consider a minimal gallery linking $F_1^\theta$ and $F_2^\theta$, such that $D_\theta$ is the first panel and $D_\theta'$ is the last panel in that gallery, and then take the reflection $s_{\mfh_\theta}$ to get another minimal gallery linking $s_{\mfh_\theta}(F_1^\theta)$ and $s_{\mfh_\theta}(F_2^\theta)$. Since the panels in the two galleries except $D_\theta$ and $D_{\theta'}$ reflect pairwise to each other, thus the related signs cancel out. So by $\epsilon_{F_1,F_2}=\epsilon_{s_{\mfh_\theta}(F_1),s_{\mfh_\theta}(F_2)}$ we have that the two signs related to $D_\theta$ and $D'_{\theta}$ are equal, which finishes the proof.
			
		\end{proof}
		
		\begin{remark}
			
			It is interesting to see if $\epsilon_{F,F'}$ is always one or not.
			
		\end{remark}
		
		\subsection{Poincar\'e series}
		
		Let $F\in\Fteff(G)$ and $\mca$ a $\theta$-stable apartment satisfying Assumption \ref{assumpaffinecoxeter}, \ref{assumpnodoubleedge} such that $F^\theta$ is a chamber of $\mca^\theta$. We calculate $\lambda_{F}(f_{F})$.

		Consider some $C\in\ch{F}{0}(G)$ and $C'\in\mco_C$. Let $F'$ be the maximal $\theta$-stable facet of $C'$. Assume that $f_F(C')\neq 0$. Then using Proposition \ref{propfFC'neq0}, $C'$, $F$ and $F'$ are included in a $\theta$-stable apartment $\mca'$, such that $F^\theta$ and $F'^\theta$ are chambers of $\mca'^\theta$. So $F^\theta$ and $F'^\theta$, as well as $F$ and $F'$ are $H$-conjugate. Let $C'=h\cdot C$ for some $h\in H$. We have
		\begin{equation*}
			\begin{aligned}
				\phi_F(C)f_F(C')\chi_0(C'/C)&=\phi_F(C)\phi_F(t_{F',F}(h\cdot C))\chi_0(h)\prod_{\mfh_{\theta}\in\mfH_{\mca'^{\theta}}(F^{\theta},F'^{\theta})}n_{\mfh_{\theta}}^{-1}\\
				&=\epsilon_{F,F'}\prod_{\mfh_{\theta}\in\mfH_{\mca'^{\theta}}(F^{\theta},F'^{\theta})}n_{\mfh_{\theta}}^{-1}
			\end{aligned}
		\end{equation*}	
		Let $o_C$\index{$o_C$} be the cardinality of $\ch{F'}{0}(G)\cap \mco_C$, where $F'$ is any facet that is $H$-conjugate to $F$. It is a number independent of $F'$. Then we further have
		\begin{equation*}
			\begin{aligned}
				\phi_F(C)\sum_{C'\in\ch{F'}{0}(G)\cap \mco_C}f_F(C')\chi_0(C'/C)=o_C\epsilon_{F,F'}\prod_{\mfh_{\theta}\in\mfH_{\mca'^{\theta}}(F^{\theta},F'^{\theta})}n_{\mfh_{\theta}}^{-1}
			\end{aligned}
		\end{equation*}	
		Let $\Ft_{F}(G)$\index{$\Ft_{F}(G)$} be the set of $\theta$-stable facets $F'$, such that $F^{\theta}$ and $F'^{\theta}$ are chambers of $\mca'^{\theta}$ for some $\theta$-stable apartment $\mca'$. Then we have
		\begin{equation*}
			\begin{aligned}
				\phi_{F}(C)\sum_{C'\in\mco_{C}}f_{F}(C')\chi_0(C'/C)&=o_C\sum_{F'\in\Ft_{F}(G)}\epsilon_{F,F'}\prod_{\mfh_{\theta}\in\mfH_{\mca'^{\theta}}(F^{\theta},F'^{\theta})}n_{\mfh_{\theta}}^{-1}\\
				&=o_C\sum_{F', F'^{\theta}\in\ch{}{}(\mca^{\theta})}\epsilon_{F,F'}\prod_{\mfh_{\theta}\in\mfH_{\mca^{\theta}}(F^{\theta},F'^{\theta})}(Q_{\mfh_{\theta},H}/n_{\mfh_{\theta}})\\
				&=o_C\sum_{F', F'^{\theta}\in\ch{}{}(\mca^{\theta})}\prod_{\mfh_{\theta}\in\mfH_{\mca^{\theta}}(F^{\theta},F'^{\theta})}(\epsilon_{\mfh_\theta}Q_{\mfh_{\theta},H}/n_{\mfh_{\theta}})\\
				&=o_C\sum_{F', F'^{\theta}\in\ch{}{}(\mca^{\theta})}\prod_{\mfh_{\theta}\in\mfH_{\mca^{\theta}}(F^{\theta},F'^{\theta})}\epsilon_{\mfh_\theta}m_{\mfh_{\theta}}\\
				&=o_C\sum_{w=s_{1}\dots s_{k}\in W(\mca^{\theta})}\prod_{i=1}^{k}\epsilon_{s_{\mfh_{s_{i}}}}m_{\mfh_{s_{i}}}.
			\end{aligned}
		\end{equation*}	
		Here, we fix a set of generators $S(\mca^{\theta})$ for the affine Coxeter group $W(\mca^{\theta})$, and for $s_i\in S(\mca^{\theta})$ we denote by $\mfh_{s_{i}}$ the corresponding wall in $\mca^{\theta}$. We used Lemma \ref{lemmaAthetaFtheta}.(1)(2)(4) and Proposition \ref{propepsilonFF'}.(4) for the second and third row, and Proposition \ref{propepsilonFF'}.(3) and Lemma \ref{lemmaAthetaFtheta}.(4) for the fifth row.
		
		\begin{remark}
			
			When $F^\theta$ is a single point, we notice that $\Ft_{F}(G)=\{F\}$. So the above sum degenerates and its value is $o_C$.
			
		\end{remark}
		
		Summing over $C\in\ch{F}{0}(G)$, we have
		\begin{equation}\label{eqcalfFphiF}
			\begin{aligned}
				\lambda_F(f_F)&=\sum_{C\in\ch{F}{0}(G)}	\phi_{F}(C)\sum_{C'\in\mco_{C}}f_{F}(C')\chi_0(C'/C)
				\\&=\bigg(\sum_{C\in\ch{F}{0}(G)}o_C\bigg)\sum_{w=s_{1}\dots s_{k}\in W(\mca^{\theta})}\prod_{i=1}^{k}\epsilon_{\mfh_{s_{i}}}m_{\mfh_{s_{i}}}.
			\end{aligned}
		\end{equation}	
		
		\begin{remark}
			
			From the argument we notice that to define $\lambda_F$, it is not necessary to sum over $C\in\ch{F}{0}(G)$. We could have fix any $C$ and only sum over $\mco_C$ instead. 
			
		\end{remark}
		
		The term 
		$$\sum_{w=s_{1}\dots s_{k}\in W(\mca^{\theta})}\prod_{i=1}^{k}\epsilon_{\mfh_{s_{i}}}m_{\mfh_{s_{i}}}$$
		in \eqref{eqcalfFphiF} is a Poincar\'e series in the sense of \cite{macdonald1972poincare} with respect to the affine Weyl group $W(\mca^{\theta})$ and variables $\epsilon_{\mfh_{s}}m_{\mfh_{s}},$  $s\in S(\mca^{\theta})$. Using \emph{ibid.}, Theorem 3.3, it is positive since $0<\abs{\epsilon_{\mfh_{s}}m_{\mfh_{s}}}<1$ for each $s$. 
		
		Thus we have proved the following proposition:
		
		\begin{proposition}\label{proplambdaFfF>0}
			
			We have $\lambda_{F}(f_{F})>0$.
			
		\end{proposition}
		
		As a result, we have the following theorem:
		
		\begin{theorem}\label{thmdistdimcal}
			
			Assume Assumption \ref{assumpZHZG}, \ref{assumpabsconverg}, \ref{assumpchi}, \ref{assumpchi'}. Assume Assumption  \ref{assumpaffinecoxeter}, \ref{assumpnodoubleedge} for each $\theta$-stable apartment $\mca$ with the related subgraph $\Gamma(G,\theta,[\mca]_{H,\sim})$ having an effective connected component, which guarantees $\Gamma(G,\theta,[\mca]_{H,\sim})$ has a single effective vertex and no double edge. Then the dimension of
			$$\mrhom_{H}(\mrst_{G},\chi)\cong \mch(G)^{(H,\chi)}$$
			equals the number of $\theta$-equivalence classes $[\mca]_{H,\sim}$ that are effective.
			
		\end{theorem}
		
		\begin{proof}
			
			Let $[\mca_{1}]_{H,\sim},\dots,[\mca_{k}]_{H,\sim}$ be all the effective $\theta$-equivalence classes of $\theta$-stable apartments, such that $\mrdim\mca_1^\theta\geq\mrdim\mca_2^\theta\geq\dots\geq\mrdim\mca_k^\theta$. Assumption  \ref{assumpaffinecoxeter} means that each $\Gamma(G,\theta,[\mca_{i}]_{H,\sim})$ consists of a single effective vertex, and for any $[\mca']_{H,\sim}$ different from $[\mca_{i}]_{H,\sim},i=1,\dots,k$, the graph $\Gamma(G,\theta,[\mca']_{H,\sim})$ has no effective connected component. For each $i$, let $F_{i}$ be a maximal $\theta$-stable facet in $\mca_{i}$. We construct $\phi_{F_{i}}$, $f_{F_{i}}$ and $\lambda_{F_{i}}$ as before. Using Corollary \ref{corlambdaFfF'} and Proposition \ref{proplambdaFfF>0}, the matrix $(\lambda_{F_{i}}(f_{F_{j}}))_{1\leq i,j\leq k}$ is invertible and lower triangular. On the other hand, using Theorem \ref{thmupperboundrefine} the dimension of $\mrhom_{H}(\mrst_{G},\chi)$ is bounded by $k$. So $\lambda_{1},\dots,\lambda_{k}$ is a basis of $\mrhom_{H}(\mrst_{G},\chi)$, which finishes the proof.
			
		\end{proof}
		
		\begin{example}\label{exampleGL3O3}
			
			Let $G=\mrgl_{3}(K)$, $\theta(g)=J_{3}\,^{t}g^{-1}J_{3}$, $H=\mrso_3(J_3)$ and $\mca=\mca(G,S)$ with $S$ the diagonal torus. Then $W(\mca^{\theta})=\pairangone{s,t\mid s^2=t^2=1}$ is an affine Weyl group with generators $s,t$ such that $Q_{\mfh_s,H}=q$, $Q_{\mfh_t,H}=1$ (since $\mfh_t$ is not a wall of an apartment of $H$), $n_{\mfh_s}=-1/q^3$ and $n_{\mfh_t}=-1/q$, $m_{\mfh_s}=-q/q^3=-1/q^2$ and $m_{\mfh_t}=-1/q$, $\epsilon_{\mfh_s}=\epsilon_{\mfh_t}=1$. Let $F$ be a $\theta$-stable facet such that $F^{\theta}$ is a chamber of $\mca^\theta$. Then $F$ is effective and $\car{\ch{F}{0}(G)}=1$. Moreover,
			\begin{equation*}
				\begin{aligned}
					\lambda_{F}(f_{F})=\sum_{w=s_{1}\dots s_{k}\in W(\mca^{\theta})}\prod_{i=1}^{k}m_{\mfh_{s_{i}}}&=\sum_{i=0}^{\infty}(m_{\mfh_s}^im_{\mfh_t}^i+m_{\mfh_s}^im_{\mfh_t}^{i+1}+m_{\mfh_s}^{i+1}m_{\mfh_t}^{i}+m_{\mfh_s}^{i+1}m_{\mfh_t}^{i+1})	\\
					&=(1-q^{-1})(1-q^{-2})/(1-q^{-3})>0.
				\end{aligned}
			\end{equation*}	
			
			\begin{figure}[htbp]
				\begin{center}
					\tikzstyle{every node}=[scale=1]
					\begin{tikzpicture}[line width=0.4pt,scale=0.8][>=latex]
						\pgfmathsetmacro\ax{1}
						\pgfmathsetmacro\ay{1.5}
						\pgfmathsetmacro\by{sin(60)}
						\pgfmathsetmacro\cy{2*sin(60)}
						\pgfmathsetmacro\dy{3*sin(60)}
						\foreach \k in {1,...,6} {
							\draw[blue,-] (0,0) -- +(\k * 60 + 60:2); 
						}
						\draw[blue,-] (\ax,\cy) -- (-\ax, \cy) ;
						\draw[blue,-] (-\ax,\cy) -- (-2*\ax, 0) ;
						\draw[blue,-] (\ax,\cy) -- (2*\ax, 0);
						\draw[blue,-] (2*\ax, 0) -- (\ax,-\cy) ;
						\draw[blue,-] (\ax,-\cy) -- (-\ax,-\cy) ;
						\draw[blue,-] (-\ax,-\cy) -- (-2*\ax, 0) ;
						\draw[red,-] (2.5*\ay,2.5*\by) -- (-2*\ay,-2*\by)  ;
						\draw[blue,-] (3,\cy) -- (\ax,\cy);
						\draw[blue,-] (3,\cy) -- (2*\ax, 0);
						\node at (\ay,\by) [blue] {\(\bullet\)};
						\node at (0,0) [thick] {\(\bullet\)};
                       \node at (0,0) [thick] {\(\bullet\)};
						\node at (30:5)[text=red] {\(\mca^{\theta}\)};
						\node at (25:2.2)[text=blue] {\(\mfh_t\)};
						\node at (-18:0.5) {\(\mfh_s\)};
					\end{tikzpicture}
				\end{center}
            \caption{Example \ref{exampleGL3O3}}
			\end{figure}
		\end{example}
		
		\section{Distinction of the Steinberg representation of $\mrgl_{n}(K)$ by an orthogonal subgroup}\label{sectionorthogonalgroup}
		
		In this section, we study the example where $G=\mrgl_{n}(K)$ and $H'$ (resp. $H$) is an orthogonal (resp. special orthogonal) subgroup of $G$. We consider the distinction of the Steinberg representation $\mrst_G$ by the trivial character of $H'$ (resp. $H$).
		
		\subsection{Classification of symmetric matrices and orthogonal groups}\label{subsectionsymmetricmatrices}
		
		We recall the known classification result for orthogonal groups and we refer to \cite{zou2022supercuspidal}*{Section 3} for more details. First we consider the space of invertible symmetric $n\times n$ matrices
		$$X=\{\varepsilon\in\mrgl_n(K)\mid\,^{t}\varepsilon=\varepsilon\}$$
		endowed with an $G$-action given by
		$$\varepsilon\cdot g=\,^{t}g\varepsilon g.$$
		Then the $G$-orbits $X/G$ are finite and classified exactly by the related discriminant and Hasse invariant. Here we recall that for $\varepsilon\in X$, its discriminant is the class of the determinant $\mrdet(\varepsilon)$ in $K^{\times}/K^{\times2}$, and its Hasse invariant $\mathrm{Hasse}(\varepsilon)$ is the sign 
		$$\prod_{1\leq i<j\leq n}(a_{i},a_{j})_{\mathrm{Hilb}}\in\{\pm1\},$$
		where $(\cdot,\cdot)_{\mathrm{Hilb}}$ denotes the Hilbert symbol, and $\mrdiag(a_{1},\dots,a_{n})$ is any diagonal matrix that is in the same $G$-orbit of $\varepsilon$. Thus depending on whether $n$ is one, two or greater than two, there are respectively four, seven or eight $G$-orbits of $X$. More precisely, 
		\begin{itemize}
			\item When $n=1$, there are four $G$-orbits corresponding to $K^{\times}/K^{\times2}$;
			\item When $n\geq 3$, there are eight combinations of discriminant and Hasse invariant;
			\item When $n=2$, there are seven possible combinations of discriminant and Hasse invariant, since in this case if the discriminant is $-1$, the related Hasse invariant must be $1$.
		\end{itemize}  
		
		Now we discuss the isomorphism classes of orthogonal groups. Given a symmetric matrix $\varepsilon$, we may define the corresponding orthogonal involution
		$$\theta_{\varepsilon}(x)=\varepsilon^{-1}\,^{t}x^{-1}\varepsilon,\quad x\in G\index{$\theta_{\varepsilon}$}$$
		and the related orthogonal subgroup
		$$\mro_{n}(\varepsilon)=G^{\theta_{\varepsilon}}:=\{x\in G\mid \theta_{\varepsilon}(x)=x\}.\index{$\mro_{n}(\varepsilon)$}$$
		We also consider the related special orthogonal group
		$$\mrso_{n}(\varepsilon)=\{x\in G\mid \theta_{\varepsilon}(x)=x,\ \mrdet(x)=1\}\index{$\mrso_{n}(\varepsilon)$}$$
		as an index $2$ subgroup of $\mro_{n}(\varepsilon)$. 
		
		Indeed the related $G$-orbits of the quotient $X/K^{\times}$ correspond bijectively to the isomorphism classes of orthogonal subgroups of $G$. More precisely, 
		\begin{itemize}
			\item When $n=1$ there is only one orthogonal group, which is $\{\pm 1\}$;
			\item When $n\geq 3$ is an odd integer, there are two classes of orthogonal groups that are split and non-quasi-split respectively;
			\item When $n\geq 4$ is an even integer, there are five classes of orthogonal groups, where one class is split, three classes are quasi-split and one class is non-quasi-split;
			\item Finally when $n=2$, there are four classes of orthogonal groups, where one class is split and three classes are quasi-split.
		\end{itemize}   In particular, the orthogonal subgroup corresponding to the matrix
		\begin{equation}\label{Jn}
			J_{n}=\begin{pmatrix}0 & 0 & \ldots & 0 & 1\\ 0 & \iddots & \iddots & 1 & 0 \\ \vdots & \iddots & \iddots & \iddots & \vdots \\ 0 & 1 & \iddots & \iddots & 0 \\ 1 & 0 & \ldots & 0 & 0 \end{pmatrix}\in\mrgl_{n}(K)\index{$J_{n}$}
		\end{equation}
		is split. It has discriminant $(-1)^{(n-1)n/2}$ and Hasse invariant $1$.

		\subsection{Classification of $\theta$-stable apartments}
		
		From now on, we fix a symmetric matrix $\varepsilon$ and the corresponding orthogonal involution $\theta=\theta_{\varepsilon}$. Up to $G$-conjugacy, without loss of generality we may and will assume $\varepsilon\in N_{G}(S)$. Let $H'=\mro_{n}(\varepsilon)$ and $H=\mrso_{n}(\varepsilon)$. In particular, we have $\epsilon_{G}\rest_{H'}=1$. 
		
		In this subsection, we study $H$-conjugacy classes of $\theta$-stable apartments of $G$. 
		
		Write $\mca=\mca(G,S)$. We first classify all the $H'$-conjugacy classes of $\theta$-stable apartments $\mca^{g}=\mca(G,S^{g})$, or maximal split tori $S^{g}$, $g\in G$, such that $S^{g}$ is $\theta$-stable. Equivalently, we need to classify all the double cosets 
		\begin{equation}\label{eqorthothetaap}
			\{N_{G}(S)gH'\mid g\in G,\ \theta(g)g^{-1}\in N_{G}(S)\}.
		\end{equation}
		We identify the symmetric group $\mfS_{n}$ with the subgroup of permutation matrices of $G$. Then it is clear that $N_{G}(S)=\mfS_{n}\ltimes S$.
		
		By definition $\theta(g)g^{-1}=\varepsilon^{-1}\,^{t}g^{-1}\varepsilon g^{-1}$. We fix a set of representatives $\{1,\epsilon_{0},\varpi_{K},\epsilon_{0}\varpi_{K}\}$ of $K^{\times}/K^{\times2}$. Here $\epsilon_{0}\in\mfo_{K}^{\times}-\mfo_{K}^{\times 2}$ and $\varpi_{K}$ is a uniformizer of $K$. An elementary calculation shows that, up to replacing $g$ with another element in $N_{G}(S)$ we may assume that
		\begin{equation}\label{eqorthothetaapcond}
			\begin{aligned}
				^{t}g^{-1}\varepsilon g=\delta(\bs{n},r):=
				\mrdiag(J_{2r},\underbrace{1,\dots,1}_{n_{11}\text{-copies}},\underbrace{\epsilon_{0},\dots,\epsilon_{0}}_{n_{12}\text{-copies}},\underbrace{\varpi_{K},\dots,\varpi_{K}}_{n_{21}\text{-copies}},\underbrace{\epsilon_{0}\varpi_{K},\dots,\epsilon_{0}\varpi_{K}}_{n_{22}\text{-copies}}).
			\end{aligned}
		\end{equation}
		Here $n_{11},n_{12},n_{21},n_{22},r$ are non-negative integers such that $n=n_{11}+n_{12}+n_{21}+n_{22}+2r$ and $\bs{n}=(n_{11},n_{12},n_{21},n_{22})$. We also write $n_{1}=n_{11}+n_{12}$ and $n_{2}=n_{21}+n_{22}$. On the other hand, it is clear that $(\bs{n},r)$ are invariants that are independent of the choice of $g\in N_{G}(S)$, since for $g_{1},g_{2}\in N_{G}(S)$ corresponding to the same $(\bs{n},r)$, by definition we have $g_{1}g_{2}^{-1}\in H'$. Finally, to find such $g$, the necessary and sufficient condition is that $\varepsilon$ and $\delta(\bs{n},r)$ have the same discriminant and Hasse invariant. 
		
		We also study $H$-conjugacy classes of $\theta$-stable apartments. It suffices to prove the following lemma.
		
		\begin{lemma}\label{lemmaH'-Hconj}
			
			Let $S^{g}$ be a $\theta$-stable maximal split torus, then $H'\cap N_{G}(S^{g})\neq H\cap N_{G}(S^{g})$ and thus we may choose $h\in H'-H$ that normalizes $S^{g}$. As a result, the $H$-conjugacy class of $\mca^{g}$ coincides with the $H'$-conjugacy class of $\mca^{g}$.
			
		\end{lemma}
		
		\begin{proof}
			
			We may replace $g$ by another representative in $N_{G}(S)gH'$, thus without loss of generality we assume that \eqref{eqorthothetaapcond} is satisfied for some $(\bs{n},r)$. It is clear that we may pick $x\in N_{G}(S)$ such that $$\delta(\bs{n},r)^{-1}\,^{t}x^{-1}\delta(\bs{n},r)=x\quad\text{and}\quad\mrdet(x)=-1.$$
			Thus $x^{g}\in H'\cap N_{G}(S^{g})-H\cap N_{G}(S^{g})$, which finishes the proof.
			
		\end{proof}
		
		To sum up, we have the following proposition.
		
		\begin{proposition}\label{propclassthetaap}
			
			The set of $H$-conjugacy (or $H'$-conjugacy) classes of $\theta$-stable apartments of $\mcb(G)$, or equivalently the set of double cosets \eqref{eqorthothetaap} are in bijection with   
			\begin{equation}\label{eqn'rcondition}
				\bigg\{
				(n_{11},n_{12},n_{21},n_{22},r)\ \bigg|\begin{aligned}& \ n=n_{11}+n_{12}+n_{21}+n_{22}+2r,\ \delta(\bs{n},r)\ \text{and}\ \varepsilon\\ & \ \text{have the same discriminant and Hasse invariant.}
				\end{aligned}	\bigg\}
			\end{equation}
			This bijection is determined by the relation \eqref{eqorthothetaapcond}.
			
		\end{proposition}
		
		This condition could be further explained as follows:
		
		\begin{lemma}\label{lemmadischasse}
			
			\begin{enumerate}
				\item For a fixed $r$, the pair of discriminant and Hasse invariant of $\delta(\bs{n},r)$  is determined by the parity of $n_{11},n_{12},n_{21},n_{22}$.
				
				\item If the symmetric matrices $\varepsilon$ and
				$\delta(\bs{n},r)$
				have the same discriminant and Hasse invariant, then $n_{2}$ and $v_{K}(\mrdet(\varepsilon))$ has the same parity, and moreover 
				$$ \begin{cases} 0\leq r\leq k,\ 2\mid n_{2},\ 0\leq n_{2}\leq 2(k-r) &\text{if}\ n=2k\ \text{and}\ H\ \text{is split};\\
					0\leq r\leq k-1,\ \text{and}&\\ 
					2\mid n_{2},\ 0\leq n_{2}\leq 2(k-1-r)&	\\
					\text{or}\ 2\mid n_{2},\ 2\leq n_{2}\leq 2(k-r)&\\
					\text{or}\ 2\nmid n_{2},\ 1\leq n_{2}\leq 2(k-r)-1& \text{if}\ n=2k\ \text{and}\ H\ \text{is quasi-split but not split};\\
					0\leq r\leq k-2,\ 2\mid n_{2},\ 2\leq n_{2}\leq 2(k-1-r)  & \text{if}\ n=2k\ \text{and}\ H\ \text{is not quasi-split};\\
					0\leq r\leq k,\ 0\leq n_{2}\leq 2(k-r)+1 & \text{if}\ n=2k+1\ \text{and}\ H\ \text{is split};\\
					0\leq r\leq k-1,\ 1\leq n_{2}\leq 2(k-r) & \text{if}\ n=2k+1\ \text{and}\ H\ \text{is not quasi-split}.
				\end{cases}$$
				
				\item For $n_{1},n_{2},r$ satisfying the condition listed in (2), we may find a pair $(\bs{n},r)$ and the related $\delta(\bs{n},r)$, such that $\varepsilon$ and
				$\delta(\bs{n},r)$
				have the same discriminant and Hasse invariant. 
				
			\end{enumerate}
			
		\end{lemma}
		
		\begin{proof}
			
			By direct calculation, we have
			$$\mrdet(\delta(\bs{n},r))=(-1)^{2r(2r-1)/2}\epsilon_{0}^{n_{12}+n_{22}}\varpi_{F}^{n_{2}}$$
			and
			$$\mathrm{Hasse}(\delta(\bs{n},r))=(-1,\varpi_{K})_{\mathrm{Hilb}}^{2r(2r-1)n_{2}/2}\cdot(\epsilon_{0},\varpi_{K})_{\mathrm{Hilb}}^{(n_{12}+n_{22})n_{2}-n_{22}}\cdot(\varpi_{K},\varpi_{K})_{\mathrm{Hilb}}^{n_{2}(n_{2}-1)/2}.$$
			Here we use the fact that for two elements in $\mfo_{K}^{\times}$, their Hilbert symbol is $+1$.
			So the discriminant of $\delta(\bs{n},r)$ is determined by the parity of $n_{12}+n_{22}$ and $n_{2}$. Once the parity of $n_{12}+n_{22}$ and $n_{2}$ are fixed, the parity of $n_{22}$ determines the Hasse invariant of $\delta(\bs{n},r)$ using the fact that $(\epsilon_{0},\varpi_{K})_{\mathrm{Hilb}}=-1$. Thus we have proved (1).
			
			Now we prove (2) and (3). First we remark that if both $n_{1}$ and $n_{2}$ are non-zero, then by changing the parity of $n_{12}$ and $n_{22}$ we may change the Hasse invariant of $\delta(\bs{n},r)$  without changing its discriminant. Thus in this case if the correct discriminant is obtained, the correct Hasse invariant can always be obtained.
			
			Now we may discuss different cases.
			
			\begin{itemize}
				\item Assume $n=2k$ and $H$ is either split or non-quasi-split. Then $v_{K}(\mrdet(\varepsilon))$ is even, meaning that $n_{2}$ is even. If $H$ is split, it is possible that $n_{1}=0$ or $n_{2}=0$, since in both cases by direct calculation the related Hasse invariant is $1$. Conversely, if $H$ is non-quasi-split, it is impossible that $n_{1}=0$ and $n_{2}=0$. In particular, if $H$ is non-quasi-split we must have $r\leq k-2$, otherwise either $n_{1}$ or $n_{2}$ is zero. 
				
				\item Assume $n=2k$ and $H$ is quasi-split but not split. Then $r$ cannot be $k$, otherwise $\varepsilon=J_{2r}$ and $H$ is split. So $r\leq k-1$. If $v_{K}(\mrdet(\varepsilon))$ is even, then $n_{2}$ is even. We consider two cases $n_{1}=0$ and $n_{2}=0$ and we adjust the parity of $n_{12}$ and $n_{22}$ such that the related discriminant of $\delta(\bs{n},r)$ is $(-1)^{2r(2r-1)/2}\epsilon_{0}$, then the Hasse invariants of $\delta(\bs{n},r)$ in these two cases are different. We leave the details to the readers. It means that exactly one of the cases of $n_{1}=0$ or $n_{2}=0$ can be fulfilled. Thus we have either $0\leq n_{2}\leq 2(k-1-r)$ or $2\leq n_{2}\leq 2(k-r)$, and in each case all values of $n_{2}$ can be fulfilled. If $v_{K}(\mrdet(\varepsilon))$ is odd, then $n_{2}$ is also odd. In this case, we have $1\leq n_{2}\leq 2(k-1-r)+1$ and each of them can be fulfilled since $n_{1}$, $n_{2}$ are non-zero.
				
				\item Assume $n=2k+1$. If $n_{1}=0$ or $n_{2}=0$, by calculating the discriminant and Hasse invariant, we must have that $H$ is split. We leave the detail to the readers. Thus if $H$ is split each case of $ 0\leq r\leq k$, $0\leq n_{2}\leq 2(k-r)+1$ can be fulfilled. If $H$ is non-quasi-split, both $n_{1}$ and $n_{2}$ cannot be $0$, and in particular $r$ cannot be $k$. Meanwhile each case in $0\leq r\leq k-1$, $1\leq n_{2}\leq 2(k-r)$ can be fulfilled.
				
			\end{itemize}

		\end{proof}

		We further study the $\theta$-rank of $\mca^{g}=\mca(G,S^{g})$. For $x\in S$, the equation $\theta(x^{g})=x^{g}$ is equivalent to $$\delta(\bs{n},r)^{-1}\,^{t}x^{-1}\delta(\bs{n},r)=x.$$ If we write $x=\mrdiag(b_{1},b_{2},\dots,b_{n})$ with $b_{i}\in K^{\times}$, then it is clear that the solutions are exactly 
		\begin{equation}\label{eqdiagH'}
			b_{i}^{2}=1\ \text{for}\ i=2r+1,2r+2,\dots ,n\quad \text{and}\quad b_{j}=b_{2r+1-j}^{-1}\ \text{for}\ j=1,2,\dots r.
		\end{equation}
		As a result, $((S^{g})^{\theta})^{\circ}$ is a subtorus of $G$ of semi-simple rank $r$, meaning that $\mca^{g}$ is of $\theta$-rank $r$. 
		
		Finally we construct a $\theta$-stable facet $F_{0}^{g}$ of $\mca^{g}$, such that $(F_{0}^{g})^{\theta}$ is a chamber of $(\mca^{g})^{\theta}$. We consider the $\mfo_{K}$-order $\mfa_{F_{0}}=\{(a_{ij})_{1\leq i,j\leq n}\}$ in $\mrm_{n}(K)$ such that

		\begin{itemize}
			\item $a_{ij}\in\mfo_{K}$ if
			\begin{itemize}
				\item $1\leq j\leq i\leq 2r$;
				\item or $1\leq j\leq r,\ 2r+1\leq i\leq n$;
				\item or $r+1\leq j\leq 2r$, $2r+n_{1}+1\leq i\leq n$; 
				\item or $2r+1\leq j\leq 2r+n_{1}$, $r+1\leq i\leq n$;
				\item or $2r+n_{1}+1\leq j \leq n$, $2r+n_{1}+1\leq i \leq n$.
			\end{itemize}
			\item $a_{ij}\in\mfp_{K}$ if 
			\begin{itemize}
				\item $1\leq i<j\leq 2r$;
				\item or $r+1\leq j\leq 2r$, $2r+1\leq i\leq 2r+n_{1}$;
				\item or $2r+1\leq j\leq 2r+n_{1}, \ 1\leq i\leq r$;
				\item or $2r+n_{1}+1\leq j\leq n, \ 1\leq i\leq 2r+n_{1}$.
			\end{itemize}
		\end{itemize}
		Then the invertible elements in $\mfa_{F_{0}}$ form a parahoric subgroup of $G$, which is denoted by $\parah{F_{0}}$. Indeed,
		we let $F_{0}$ be the unique facet in $\mca$ such that $\mcg_{F_{0}}(K)=\parah{F_{0}}$, which justifies our notation. From our construction, it is clear that
		$$\delta(\bs{n},r)^{-1}\,^{t}\parah{F_{0}}^{-1}\delta(\bs{n},r)=\parah{F_{0}}.$$
		Thus $\parah{F_{0}}^{g}$ is $\theta$-stable, meaning that $F_{0}^{g}$ is a $\theta$-stable facet in $\mca^{g}$.

		For $i=0,1,\dots,r$, let $$g_{i}=\mrdiag(\underbrace{1,\dots,1}_{i\text{-copies}},\underbrace{\varpi_{K},\dots,\varpi_{K}}_{(n-i)\text{-copies}})$$
		and
		$$g_{i}'=\mrdiag(\underbrace{1,1,\dots,1}_{(2r-i)\text{-copies}},\underbrace{\varpi_{K},\dots,\varpi_{K}}_{i\text{-copies}},\underbrace{1,1,\dots,1}_{n_{1}\text{-copies}},\underbrace{\varpi_{K},\dots,\varpi_{K}}_{n_{2}\text{-copies}}).$$
		For each $i$, consider the maximal $\mfo_{K}$-order $$\mfa_{v_{i}}:=g_{i}^{-1}\mrm_{n}(\mfo_{K})g_{i}\quad\mathrm{and}\quad\mfa_{v_{i}'}:=g_{i}'^{-1}\mrm_{n}(\mfo_{K})g_{i}'$$ 
		in $\mrm_{n}(K)$. Then from our construction, we have
		$$\mfa_{v_{i}'}=\delta(\bs{n},r)^{-1}\,^{t}\mfa_{v_{i}}^{-1}\delta(\bs{n},r),\quad i=0,1,\dots,r.$$
		The invertible elements in $\mfa_{v_{i}}$ (resp. $\mfa_{v_{i}'}$) form a maximal parahoric subgroup of $G$, which is denoted by $\parah{v_{i}}$ (resp. $\parah{v_{i}'}$). Let $v_{i}$ (resp. $v_{i}'$) be the unique vertex in $\mca$ such that $\mcg_{v_{i}}(K)=\parah{v_{i}}$ (resp. $\mcg_{v_{i}'}(K)=\parah{v_{i}'}$). Let $v_{i}^{\flat}$ be the midpoint of $v_{i}$ and $v_{i}'$ for $i=0,1,\dots,r$, 
		
		In particular, if $i=n_{2}=0$ we have $v_{0}=v_{0}'$, and if $r-i=n_{1}=0$ we have $v_{r}=v_{r}'$. In other cases we have $v_{i}\neq v_{i}'$, by direct calculation there does not exist $g\in G$ such that $gv_{i}=v_{i}'$ (equiv. $\parah{v_{i}}^{g}=\parah{v_{i}'}$) and $gv_{i}'=v_{i}$ (equiv. $\parah{v_{i}'}^{g}=\parah{v_{i}}$). 
		
		We consider the parahoric subgroup $\parah{\{v_{i},v_{i}'\}}$ as the intersection of $\parah{v_{i}}$ and $\parah{v_{i}'}$. The following lemma follows from direct calculation.
		
		\begin{lemma}\label{lemmacalculationPvivi'}
		Let
		\begin{equation}\label{eq}
			\varsigma_i=\begin{pmatrix}
				\varpi_K I_i & & & & \\ &  & & & I_{2r-2i}\\
				& & I_i & & \\
				&  & & I_{n_1} &\\
				& I_{n_2} & &  &  
			\end{pmatrix}\in \mrgl_n(K).
		\end{equation}
		Then the conjugation $\parah{\{v_{i},v_{i}'\}}^{\varsigma_i}$ of $\parah{\{v_{i},v_{i}'\}}$ by $\varsigma_i$ is exactly the stardard parahoric subgroup of $G$ corresponding to the composition $(2i+n_{2},2r-2i+n_{1})$ denoted by $P_{(2i+n_{2},2r-2i+n_{1})}$. 
		
		\end{lemma}
		 Let $\parah{v_{i}^{\flat}}$ be the parahoric subgroup of $G$ that fixes $v_{i}^{\flat}$. Then we have $\parah{v_{i}^{\flat}}=\parah{\{v_{i},v_{i}'\}}$.
		
		Since by definition $\mfa_{F_{0}}=\bigcap_{i=0}^{r}(\mfa_{v_{i}}\cap\mfa_{v_{i}'})$, the set 
		$$\mcv_{F_{0}}:=\{v_{0},v_{1},\dots,v_{r},v_{0}',v_{1}',\dots,v_{r}'\}$$
		is exactly the set of vertices of $F_{0}$. We have,
		$$\car{\mcv_{F_{0}}}=\begin{cases} 2r+2 &\text{if}\ n_{1},n_{2}\neq 0;\\
			2r+1 &\text{if either}\ n_{1}=0\ \text{or}\ n_{2}=0;\\
			2r &\text{if}\ n_{1}=n_{2}=0.
		\end{cases}$$
		Indeed, if $n_{2}=0$ we have $v_{0}=v_{0}'$, and if $n_{1}=0$, we have $v_{r}=v_{r}'$. Except that, the other vertices are pairwise different.
		
		From our construction we have $\theta(v_{i}^{g})=v_{i}'^{g}$ and $\theta(v_{i}'^{g})=v_{i}^{g}$ for $i=0,1,\dots,r$. Thus in different cases $(F_{0}^{g})^{\theta}$ is a simplex of $r+1$ vertices, meaning that it is of dimension $r$. Then
		$$\mcv_{(F_{0}^{g})^{\theta}}=\{(v_{0}^{\flat})^{g},(v_{1}^{\flat})^{g},\dots,(v_{r}^{\flat})^{g}\}$$
		is the set of vertices of $(F_{0}^{g})^{\theta}$. Since the $\theta$-rank of $\mca^{g}$ is $r$, we deduce that $(F_{0}^{g})^{\theta}$ is indeed a chamber of $(\mca^{g})^{\theta}$.

		\subsection{Classification of $\Ft(G)/H$ and $\Ftmax(G)/H$}
		
		For convenience, up to $G$-conjugacy we may and will assume $\varepsilon=\delta(\bs{n}^{0},r_{0})$ for some $\bs{n}^{0}=(n_{11}^{0},n_{12}^{0},n_{21}^{0},n_{22}^{0})$ (\emph{cf.} \eqref{eqorthothetaapcond} and Lemma \ref{lemmadischasse}), where in different cases we have
		$$\begin{cases}  n=2r_{0}\ \text{and}\ H\ \text{is split};\\
			n=2r_{0}+2\ \text{and}\ H\ \text{is quasi-split but not split};\\
			n=2r_{0}+4\ \text{and}\ H\ \text{is not quasi-split};\\
			n=2r_{0}+1\ \text{and}\ H\ \text{is split};\\
			n=2r_{0}+3\ \text{and}\ H\ \text{is not quasi-split}.
		\end{cases}$$
		Write $n_{1}^{0}=n_{11}^{0}+n_{12}^{0}$ and $n_{2}^{0}=n_{21}^{0}+n_{22}^{0}$.
		
		Fix $\mca=\mca(G,S)$ with $S$ the diagonal torus. Then in this case, $\mca$ is of maximal $\theta$-rank $r_{0}$.
		Let $F_{0}$ be the $\theta$-stable facet in $\mca$ defined as before having the set of vertices $$\mcv_{F_{0}}:=\{v_{0},v_{1},\dots,v_{r_{0}},v_{0}',v_{1}',\dots,v_{r_{0}}'\},$$ then $F_{0}^{\theta}$ is a chamber of $\mca^{\theta}$ whose set of vertices is $$\mcv_{F_{0}^{\theta}}=\{v_{0}^{\flat},v_{1}^{\flat},\dots,v_{r_{0}}^{\flat}\}.$$
		
		From now on until \S \ref{subsectionexceptionalcase}, we exclude the case where $n=2$ to make sure that the condition $Z(H)\subset Z(G)\cap H$ is satisfied. 
		Then we have the following commutative diagram
		$$\xymatrix{
			\mcb^{\mathrm{ext}}(H) \ar@{^{(}->}[r]^{\iota^\mathrm{ext}} \ar@{->>}[d]_{\pi} & \mcb^{\mathrm{ext}}(G) \ar@{->>}[d]^{\pi} \\
			\mcb(H) \ar@{^{(}->}[r]^{\iota} & \mcb(G)
		}$$
		where $\iota$ is affine and $H$-equivariant. Moreover, we have $\mcb^{\mathrm{ext}}(G)^{\theta}=\iota(\mcb^{\mathrm{ext}}(H))$ and $\mcb(G)^{\theta}=\iota(\mcb(H))$. Then $\mca^{\theta}$ is an apartment of $\mcb(G)^{\theta}$, and $\mca_{H}:=\iota^{-1}(\mca^{\theta})$ is an apartment of $\mcb(H)$. Let $F_{H}$ be the chamber in $\mca_{H}$ that contains $\iota^{-1}(F_{0}^{\theta})$.
		
		The following known result characterizes the simplicial complex structure of $\mcb(H)$ and $\mcb(G)^{\theta}$.
		
		\begin{proposition}[\cite{tits1979reductive}*{\S 4.4}, \cite{abramenko2002lattice}*{Section 6, Section 8}]\label{propchamberBGthetaBH}
			
			\begin{enumerate}
				\item $\mca^{\theta}$ is an affine Coxeter complex of type $CB_{r_{0}}$.
				
				\item $\mca_{H}$ is an affine Coxeter complex of type
				$$\begin{cases}
					D_{r_{0}} &\text{if}\ n=2r_{0},\ H\ \text{is split};\\
					B_{r_{0}}	& \text{if}\ n=2r_{0}+2,\  H\ \text{is quasi-split but not split}\ \text{and}\ n_{2}^{0}\ \text{is even};\\
					CB_{r_{0}}	& \text{if}\ n=2r_{0}+2,\  H\ \text{is quasi-split but not split}\ \text{and}\ n_{2}^{0}\ \text{is odd};\\
					CB_{r_{0}}	& \text{if}\ n=2r_{0}+4,\  H\ \text{is non-quasi-split};\\
					B_{r_{0}}	& \text{if}\ n=2r_{0}+1,\  H\ \text{is split};\\
					CB_{r_{0}}  & \text{if}\ n=2r_{0}+3,\  H\ \text{is non-quasi-split},
				\end{cases}$$
				
				\item The facet $\iota^{-1}(F_{0}^{\theta})$ equals $F_{H}$ if $\mca_{H}$ is of type $CB_{r_{0}}$, or a half of $F_{H}$ if $\mca_{H}$ is of type $B_{r_{0}}$, or a quarter of $F_{H}$ if $\mca_{H}$ is of type $D_{r_{0}}$. More precisely, let $$\mcv_{F_{H}}=\{w_{0},w_{1},\dots,w_{r_{0}}\}$$ be the set of vertices of the numbering in accordance with the numbering of affine Dynkin diagram \footnote{But we are allowed to interchange $w_{k-1}$ with $w_{k}$ (resp. $w_{k-1}$ with $w_{k}$ and $w_{0}$ with $w_{1}$) in type $B_{r_{0}}$ (resp. type $D_{r_{0}}$) if necessary.}. Let 
				$$\mcv_{F_{H}}^{\flat}=\{w_{0}^{\flat},w_{1}^{\flat},\dots,w_{r_{0}}^{\flat}\}$$
				with $w_{i}^{\flat}=w_{i}$ for $i=0,2,3\dots,r_{0}-2,r_0$, and
				$w_{1}^{\flat}$ (resp. $w_{r_{0}-1}^{\flat}$) denotes the midpoint of $w_{0}$ and $w_{1}$ (resp. $w_{r_{0}-1}$ and $w_{r_{0}}$) if $\mca_{H}$ is of type $B_{r_{0}}$ (resp. of type $B_{r_{0}}$ and $D_{r_{0}}$), or $w_{1}^{\flat}=w_{1}$ and $w_{r_{0}-1}^{\flat}=w_{r_{0}-1}$ otherwise.
				Then $v_{i}^{\flat}=\iota(w_{i}^{\flat})$ for $i=0,1\dots,r_0$ and thus $\mcv_{F_{0}^{\theta}}=\iota(\mcv_{F_{H}}^{\flat})$.
				
				\item In particular, 
				$\mcv_{F_{0}^{\theta}}=\{v_{0}^{\flat},v_{1}^{\flat},\dots,v_{r_{0}}^{\flat}\}$
				forms a representative of an affine Dynkin diagram of type $CB_{r_{0}}$ of the right numbering.
				
				\item $H$ acts transitively on both $\ch{}{}(\mcb(G)^{\theta})$ and $\ch{}{}(\mcb(H))$.
				
			\end{enumerate}
			
		\end{proposition}
	
	    \begin{figure}[htbp]
		\centering
		\begin{minipage}{0.3\textwidth}
			\centering
			\begin{tikzpicture}[thick,scale=3.3]
				\coordinate (A1) at (0, 0);
				\coordinate (A2) at (0, 1);
				\coordinate (A3) at (1, 1);
				\coordinate (A4) at (1, 0);
				\coordinate (B1) at (0.3, 0.3);
				\coordinate (B2) at (0.3, 1.3);
				\coordinate (B3) at (1.3, 1.3);
				\coordinate (B4) at (1.3, 0.3);
				\coordinate (C) at (0.65, 0.15);
				\coordinate (A) at (0.5, 0);
				\coordinate (B) at (0.65, 0.7);
				
								\draw[fill=gray,opacity=0.6] (A1) -- (A) -- (C) -- (B);
				
				\draw[very thick] (A1) -- (A2);
				\draw[very thick] (A2) -- (A3);
				\draw[very thick] (A3) -- (A4);
				\draw[very thick] (A4) -- (A1);
				\draw[very thick] (A1) -- (B1);
				\draw[very thick] (B1) -- (B2);
				\draw[very thick] (A2) -- (B2);
				\draw[very thick] (B2) -- (B3);
				\draw[very thick] (A3) -- (B3);
				\draw[very thick] (A4) -- (B4);
				\draw[very thick] (B4) -- (B3);
				\draw[very thick] (B1) -- (B4);
				
				\draw[blue] [ultra thick] (A1) -- (C);
				\draw[blue] [ultra thick] (A) -- (C);
				\draw[blue] [ultra thick] (A) -- (A1);
				\draw[blue] [ultra thick](B) -- (A1);
				\draw[blue] [ultra thick] (B) -- (A);
				\draw[blue] [ultra thick] (B) -- (C);

				\node at (A1) [text=blue] {$\bullet$};
				\node at (A) [text=blue] {$\bullet$};
				\node at (B) [text=blue] {$\bullet$};
				\node at (C) [text=blue] {$\bullet$};
				\node at (A1) [text=blue] [below] {\tiny$w_0=\color{red}{w_0^\flat}$};
				\node at (A) [text=blue] [below] {\tiny$w_1=\color{red}{w_1^\flat}$};
				\node at (B) [text=blue] [above] {\tiny$w_3=\color{red}{w_3^\flat}$};
				\node at (C) [text=blue] [right] {\tiny$w_2=\color{red}{w_2^\flat}$};
				
			\end{tikzpicture}
		
		\end{minipage}
		\begin{minipage}{0.3\textwidth}
			\centering

			\begin{tikzpicture}[thick,scale=3.3]
				\coordinate (A1) at (0, 0);
				\coordinate (A2) at (0, 1);
				\coordinate (A3) at (1, 1);
				\coordinate (A4) at (1, 0);
				\coordinate (A5) at (0.5, 0);
				\coordinate (B1) at (0.3, 0.3);
				\coordinate (B2) at (0.3, 1.3);
				\coordinate (B3) at (1.3, 1.3);
				\coordinate (B4) at (1.3, 0.3);
				\coordinate (C) at (0.65, 0.15);
				\coordinate (B) at (0.65, 0.7);
				
				\draw[fill=gray,opacity=0.6] (A1) -- (A5) -- (C)--(B);
				
				\draw[very thick] (A1) -- (A2);
				\draw[very thick] (A2) -- (A3);
				\draw[very thick] (A3) -- (A4);
				\draw[very thick] (A4) -- (A1);
				\draw[very thick] (A1) -- (B1);
				\draw[very thick] (B1) -- (B2);
				\draw[very thick] (A2) -- (B2);
				\draw[very thick] (B2) -- (B3);
				\draw[very thick] (A3) -- (B3);
				\draw[very thick] (A4) -- (B4);
				\draw[very thick] (B4) -- (B3);
				\draw[very thick] (B1) -- (B4);
				
				\draw[blue] [ultra thick] (A1) -- (C);
				\draw[blue] [ultra thick] (A4) -- (C);
				\draw[blue] [ultra thick] (A4) -- (A1);
				\draw[blue] [ultra thick](B) -- (A1);
				\draw[blue] [ultra thick] (B) -- (A4);
				\draw[red] [ultra thick] (B) -- (C);
				\draw[blue] [ultra thick] (B) -- (C);
				\draw[red] [ultra thick] (A5) -- (C);
				\draw[red] [ultra thick] (A5) -- (B);
				\draw[red] [ultra thick] (A5) -- (A1);
			
				\node at (A1) [text=blue] {$\bullet$};
				\node at (A4) [text=blue] {$\bullet$};
				\node at (B) [text=blue] {$\bullet$};
				\node at (A5) [text=red] {$\bullet$};
				\node at (C) [text=blue] {$\bullet$};
				\node at (A1) [text=blue] [below] {\tiny$w_0=\color{red}{w_0^\flat}$};
				\node at (A4) [text=blue] [below] {\tiny$w_1$};
				\node at (A5) [text=red] [below] {\tiny$w_1^\flat$};
				\node at (B) [text=blue] [right] {\tiny$w_3=\color{red}{w_3^\flat}$};
				\node at (C) [text=blue] [right] {\tiny$w_2=\color{red}{w_2^\flat}$};
				
			\end{tikzpicture}
		
		\end{minipage}
		\begin{minipage}{0.3\textwidth}
			\centering
			\begin{tikzpicture}[thick,scale=3.3]
				\coordinate (A1) at (0, 0);
				\coordinate (A2) at (0, 1);
				\coordinate (A3) at (1, 1);
				\coordinate (A4) at (1, 0);
				\coordinate (B1) at (0.3, 0.3);
				\coordinate (B2) at (0.3, 1.3);
				\coordinate (B3) at (1.3, 1.3);
				\coordinate (B4) at (1.3, 0.3);
				\coordinate (A) at (0,0.5);
				\coordinate (B) at (0.65, 1.15);
				\coordinate (C) at (0.65, 0.15);
				\coordinate (D) at (1, 0.5);
				\coordinate (A5) at (0.65, 0.65);
				\coordinate (A6) at (0.5, 0.5);
				
				\draw[fill=gray,opacity=0.6] (A) -- (A6) -- (A5) -- (B);
				
				\draw[very thick] (A1) -- (A2);
				\draw[very thick] (A2) -- (A3);
				\draw[very thick] (A3) -- (A4);
				\draw[very thick] (A4) -- (A1);
				\draw[very thick] (A1) -- (B1);
				\draw[very thick] (B1) -- (B2);
				\draw[very thick] (A2) -- (B2);
				\draw[very thick] (B2) -- (B3);
				\draw[very thick] (A3) -- (B3);
				\draw[very thick] (A4) -- (B4);
				\draw[very thick] (B4) -- (B3);
				\draw[very thick] (B1) -- (B4);
				
				\draw[blue] [ultra thick] (A) -- (C);
				\draw[blue] [ultra thick] (A) -- (B);
				\draw[blue] [ultra thick](B) -- (D);
				\draw[blue] [ultra thick] (D) -- (C);
				\draw[blue] [ultra thick] (A) -- (D);
				\draw[blue] [ultra thick] (B) -- (C);
				\draw[red] [ultra thick] (A) -- (A5);
				\draw[red] [ultra thick] (B) -- (A6);
				\draw[red] [ultra thick] (A5) -- (A6);
			
				\node at (D) [text=blue] {$\bullet$};
				\node at (A) [text=blue] {$\bullet$};
				\node at (B) [text=blue] {$\bullet$};
				\node at (C) [text=blue] {$\bullet$};
				\node at (A5) [text=red] {$\bullet$};
				\node at (A6) [text=red] {$\bullet$};
				\node at (A) [text=blue] [below] {\tiny$w_0=\color{red}{w_0^\flat}$};
				\node at (D) [text=blue] [right] {\tiny$w_1$};
				\node at (B) [text=blue] [right] {\tiny$w_3=\color{red}{w_3^\flat}$};
				\node at (C) [text=blue] [right] {\tiny$w_2$};
				\node at (A6) [text=red] [below] {\tiny$w_1^\flat$};
				\node at (A5) [text=red] [right] {\tiny$w_2^\flat$};
				\node at (A1) [text=white] [below] {\tiny$w_1=w_1^\flat$};
				
			\end{tikzpicture}
			
		\end{minipage}
    \caption{Proposition \ref{propchamberBGthetaBH}.(3), type $CB_3$, $B_3$ and $D_3$}

	\end{figure}
		
	\begin{remark}
			
	In Example \ref{exampleGL3O3}, the Coxeter complex $\mca^\theta$ is of type $CB_1$, while the Coxeter complex $\mca_H$ is of type $B_1$, where  $\mca_H$ denotes the apartment of $\mcb(H)$ having the underlying space $\mca^\theta$. That indeed explains why in \emph{loc. cit.} $\mfh_s$ is a wall, but $\mfh_t$ is not a wall of $\mca_H$.
			
	\end{remark}

		Now we study $\Ft(G)/H$.
		
		\begin{proposition}\label{propFthetaHchamber}
			
			Let $F$ be a $\theta$-stable facet of $\mcb(G)$ of maximal dimension, then 
			the $H$-conjugacy classes of the vertices of $F^{\theta}$ are pairwise different.

		\end{proposition}
		
		\begin{proof}
			
			Since the $H$-action on $\ch{}{}(\mcb(G)^{\theta})$ is transitive, we may without loss of generality assume that $F=F_{0}$.
			
			We need to prove that vertices in $\mcv_{F_{0}^{\theta}}=\{v_{0}^{\flat},v_{1}^{\flat},\dots,v_{r_0}^{\flat}\}$ have different $H$-conjugacy classes. Let $\parah{v_{0}^{\flat}},\parah{v_{1}^{\flat}},\dots,\parah{v_{r_{0}}^{\flat}}$ be the corresponding parahoric subgroups.

			We need the following elementary lemma, which is a special case of \cite{bushnell1987hereditary}*{(1.7)} in the language of hereditary orders (noting that in our settings, a parahoric subgroup is the subset of an hereditary order  consisting of invertible elements).
			
			\begin{lemma}\label{lemmasn-sparahoric}
				
	Two pairs of $(s,n-s)$ and $(s',n-s')$ correspond to the same $G$-conjugacy class of standard parahoric subgroup of $G$ if and only if $s$ equals $s'$ or $n-s'$.

		\end{lemma}

		 Assume that there exists $h\in H$ such that  $\parah{v_{i}^{\flat}}^h=\parah{v_{j}^{\flat}}$ for $i\neq j$. Using Lemma \ref{lemmacalculationPvivi'}, we have $P_{(2i+n_{2}^0,2r_0-2i+n_{1}^0)}^{\varsigma_i^{-1}h\varsigma_j}=P_{(2j+n_{2}^0,2r_0-2j+n_{1}^0)}$. Using the above lemma and since $i\neq j$, we have $2i+n_{2}^0=2r_0-2j+n_{1}^0$ and $P_{(2j+n_{2}^0,2r_0-2j+n_{1}^0)}=P_{(2r_0-2i+n_{1}^0,2i+n_{2}^0)}$. 
		Take $$g'=\begin{pmatrix}
			0 & I_{2r_0-2i+n_{1}^0}  \\ \varpi_K I_{2i+n_{2}^0} & 0
		\end{pmatrix}\in G,$$ then by direct calculation  $P_{(2i+n_{2}^0,2r_0-2i+n_{1}^0)}^{g'}=P_{(2r_0-2i+n_{1}^0,2i+n_{2}^0)}$. Thus $\varsigma_i^{-1}h\varsigma_jg'^{-1}$ normalizes $P_{(2i+n_{2}^0,2r_0-2i+n_{1}^0)}$. 
	   By direct calculation, we have $$\mrdet(\varsigma_i^{-1}h\varsigma_jg'^{-1})=\varpi_K^{-i+j+2i+n_2^0}=\varpi_K^{n/2}.$$ 
	   Since we have assumed that $i\neq j$, then we have $$2i+n_{2}^0\neq 2j+n_2^0=2r_0-2i+n_{1}^0.$$
	   In this case, we may also calculate the normalizer of $P_{(2i+n_{2}^0,2r_0-2i+n_{1}^0)}$ in $G$ (\emph{cf.}  \cite{bushnell1987hereditary}), which is $K^\times P_{(2i+n_{2}^0,2r_0-2i+n_{1}^0)}$. Taking the determinant, we have 
	   $$\varpi_K^{n/2}\in \mrdet(K^\times P_{(2i+n_{2}^0,2r_0-2i+n_{1}^0)})=\varpi_K^{n\mbz}\mfo_K^\times,$$
	   which is impossible. Thus there does not exist such $h$, meaning that the vertices in $\mcv_{F_{0}^{\theta}}$ have different $H$-conjugacy classes.
   \end{proof}

		\begin{corollary}\label{corFthetaGHconj}
			
			Let $F$ be a $\theta$-stable facet of $\mcb(G)$ of maximal dimension. Then the set of $H$-conjugacy classes of $\theta$-stable facets of $F$, that are pairwise different for different facets, corresponds bijectively to $\Ft(G)/H$.
			
		\end{corollary}
		
		\begin{proof}
			
			It is clear that each $H$-conjugacy class of a $\theta$-stable facet of $F$ contributes to $\Ft(G)/H$. 
			What remains to show is that given a $F'\in \Ft(G)$ of maximal dimension, it is $H$-conjugate to $F$. This is because both $F^{\theta}$ and $F'^{\theta}$ are chambers of $\mcb(G)^{\theta}$, thus $F^{\theta}$ and $F'^{\theta}$ are conjugate by $H$, and so are $F$ and $F'$.
			
		\end{proof}

		Now we study $\Ftmax(G)/H$. 
		
		\begin{proposition}\label{propA'chamberHconj}
			
			For a $\theta$-stable apartment $\mca'$, the chambers of $\mca'^{\theta}$ are $H$-conjugate. 
			
		\end{proposition}
		
		\begin{proof}
			
			Let $F'$ be a maximal $\theta$-stable facets of $\mca'$, and $r=\mrdim\mca'^{\theta}$, and $$\mcv_{F'^{\theta}}=\{v_{0}'^{\flat},v_{1}'^{\flat},\dots,v_{r}'^{\flat}\}$$ the set of vertices of $F'^{\theta}$. Indeed, $F'^{\theta}$ is a chamber of $\mca'^{\theta}$. 
			
			Let $(n_{11},n_{12},n_{21},n_{12},r)$ be the quintuple of integers corresponding to $\mca'$ via Proposition \ref{propclassthetaap}. From our construction, the parahoric subgroup $P_{v_{i}'^{\flat}}$ is conjugate to the standard parahoric subgroup of $G$ with respect to the composition $(2i+n_{2},2r-2i+n_{1})$ for $i=0,1,\dots,r$. Using Lemma \ref{lemmasn-sparahoric} and Corollary \ref{corFthetaGHconj}, $F'^{\theta}$ is $H$-conjugate to a facet of $\mcb(G)^{\theta}$ of set of vertices $$\mcv_{n_{2},r}:=\{v_{n_{2}}^{\flat},v_{n_{2}+1}^{\flat},\dots,v_{n_{2}+r}^{\flat}\}.$$
			It remains to show that $\mcv_{n_{2},r}$ forms a simplex of an affine Coxeter complex of type $CB_{r}$. This follows from Proposition \ref{propchamberBGthetaBH}.(4) and repetition of the following elementary lemma:
			
			\begin{lemma}
				
				Let $\{v_{0}^{\flat},v_{1}^{\flat},\dots,v_{k}^{\flat}\}$
				be a representative of an affine Dynkin diagram of type $CB_{k}$ of the right numbering. Then both 
				$\{v_{0}^{\flat},v_{1}^{\flat},\dots,v_{k-1}^{\flat}\}$ and $\{v_{1}^{\flat},v_{2}^{\flat},\dots,v_{k}^{\flat}\}$ are representatives of an affine Dynkin diagram of type $CB_{k-1}$.
				
			\end{lemma}

		\end{proof}
		
		Thus we have a surjection
		\begin{equation}\label{eqFthetamax}
			\{\theta\text{-stable apartments}\ \mca'\}/H\twoheadrightarrow\Ftmax(G)/H,\quad [\mca']_{H}\mapsto [F]_{H},
		\end{equation}
		where $F$ is a $\theta$-stable facet such that $F^{\theta}$ is a chamber of $\mca'^{\theta}$. Composing the bijection between $\{\theta\text{-stable apartments}\ \mca'\}/H$ and non-negative integers $(n_{11},n_{12},n_{21},n_{22},r)$ in Proposition \ref{propclassthetaap} with \eqref{eqFthetamax}, we get a surjection
		\begin{equation}\label{eqn'rFthetamax}
			\bigg\{
			(n_{11},n_{12},n_{21},n_{22},r)\ \bigg|\begin{aligned}&\  n=n_{11}+n_{12}+n_{21}+n_{22}+2r,\ \delta(\bs{n},r)\ \text{and}\ \varepsilon\\ & \ \text{have the same discriminant and Hasse invariant.}
			\end{aligned}	\bigg\}\twoheadrightarrow\Ftmax(G)/H.
		\end{equation}

		\begin{proposition}\label{propFmaxortho}
			
			Each fiber of the map \eqref{eqn'rFthetamax} consists exactly of $(n_{11},n_{12},n_{21},n_{12},r)$ with fixed $(n_{1},n_{2},r)$, where $(n_{1},n_{2},r)$ satisfies Lemma \ref{lemmadischasse}.(2). In other words, we have a bijection
			$$\{(n_{1},n_{2},r)\ \text{satisfying Lemma \ref{lemmadischasse}.(2)}\}\leftrightarrow\{[\mca]_{H,\sim}\}.$$ 
			As a result, we obtain $$\car{\{[\mca]_{H,\sim}\}}=\car{\Ftmax(G)/H}=\begin{cases} (k+1)(k+2)/2 &\text{if}\ n=2k\ \text{and}\ H\ \text{is split};\\k(k+1)/2	& \text{if}\ n=2k\ \text{and}\ H\ \text{is quasi-split but not split};\\(k-1)k/2  & \text{if}\ n=2k\ \text{and}\ H\ \text{is not quasi-split};\\(k+1)(k+2)/2 & \text{if}\ n=2k+1\ \text{and}\ H\ \text{is split};\\k(k+1)/2 & \text{if}\ n=2k+1\ \text{and}\ H\ \text{is not quasi-split}.\end{cases}$$
		\end{proposition}
		
		\begin{proof}
			
			As in the argument of Proposition \ref{propA'chamberHconj}, the $H$-conjugacy class of a chamber of $\mca'^{\theta}$ is represented by the set of vertices 
			$$\mcv_{n_{2},r}:=\{v_{n_{2}}^{\flat},v_{n_{2}+1}^{\flat},\dots,v_{n_{2}+r}^{\flat}\}.$$
			Using Corollary \ref{corFthetaGHconj}, the proof is finished.

		\end{proof}
		
		\subsection{Effectiveness}\label{subsectionpanels}

		\begin{proposition}\label{proporthgeff}
			
			Let $D$ be a panel of $\mcb(G)$. Assume that $D$ is not a skew panel and $\ch{D}{}(G)$ consists of $H$-orbits of $\theta$-rank $r$ and possibly $H$-orbits of $\theta$-rank $r+1$. Then there are exactly two $H$-orbits of $\ch{D}{}(G)$ of $\theta$-rank $r$. 
			
		\end{proposition}
		
		\begin{proof}
			
			Consider the $\theta$-stable set $D\cup\theta(D)$. Let $\mca$ be a $\theta$-stable apartment containing $D$. 
			Then $\ul{G}_{D\cup\theta(D)}$ is a $\theta$-stable reductive group. Indeed, in our special case there is a $\theta$-stable reductive subgroup $\ul{G}_{D\cup\theta(D)}^{\flat}$ of $\ul{G}_{D\cup\theta(D)}$ that is isomorphic to $\mrgl_{2}(\bs{k})$, such that $\ul{G}_{D\cup\theta(D)}=Z(\ul{G}_{D\cup\theta(D)})\ul{G}_{D\cup\theta(D)}^{\flat}$
			Also, the restriction of $\theta$ to $\ul{G}_{D\cup\theta(D)}^{\flat}$ is an orthogonal involution. We write  $\ul{H}_{D\cup\theta(D)}=(\ul{G}_{D\cup\theta(D)}^{\theta})^{\circ}$ and $\ul{H}_{D\cup\theta(D)}^{\flat}=((\ul{G}_{D\cup\theta(D)}^{\flat})^{\theta})^{\circ}$, where the latter one is a special orthogonal group. In particular we have $\ul{H}_{D\cup\theta(D)}\subset Z(\ul{G}_{D\cup\theta(D)})\ul{H}_{D\cup\theta(D)}^{\flat}$.
			
			Since there are two $\ul{H}_{D\cup\theta(D)}^{\flat}$-orbits of $\ch{}{}(\ul{G}_{D\cup\theta(D)}^{\flat})=\ch{}{}(\ul{G}_{D\cup\theta(D)})$ of $\theta$-rank 0, there are two $\ul{H}_{D\cup\theta(D)}$-orbits of $\ch{}{}(\ul{G}_{D\cup\theta(D)})$ of $\theta$-rank $0$. Also the set of $\parah{D}\cap H$-orbits of $\ch{D}{}(G)$ is in bijection with the set of $\ul{H}_{D\cup\theta(D)}$-orbits of $\ch{}{}(\ul{G}_{D\cup\theta(D)})$. Thus there are two $\parah{D}\cap H$-orbits of $\ch{D}{}(G)$ of $\theta$-rank $r$. 
			
			Finally, we only need to show that two chambers $C,C'$ in $\ch{D}{}(G)$ are $H$-conjugate if and only if they are $\parah{D}\cap H$-conjugate. Assume $C'=C^{h}$ for some $h\in H$ and choose $x\in \parah{D}$ such that $C'=C^{x}$. Then $xh^{-1}$ normalizes $C$, thus it lies in the normalizer of the Iwahori subgroup $\parah{C}$ in $G$. However, since $H$ is an orthogonal subgroup, we have $\mrdet(xh^{-1})\in\mfo_{K}^{\times}$. So we must have $xh^{-1}\in \parah{C}\subset \parah{D}$, meaning that $h\in \parah{D}\cap H$.
			
		\end{proof}
		
		The following lemma follows from the same proof of the last part of Proposition \ref{proporthgeff} with $D$ replaced by $F$.
		
		\begin{lemma}\label{lemmaHPFcapHcoin}
			
			Let $F$ be a $\theta$-stable facet of $\mcb(G)$. Then the $H$-orbits and the $\parah{F}\cap H$-orbits of $\ch{F}{}(G)$ coincide.	
			
		\end{lemma}
		
		As a corollary of Proposition \ref{propweakeffective}, Corollary \ref{coreffectiveness}, Proposition \ref{proporthgeff} and Lemma \ref{lemmaHPFcapHcoin}, we have
		
		\begin{corollary}\label{cormax=eff}
			
			In this case, every maximal $\theta$-stable facet is ($1_H$-)effective. 
			
		\end{corollary}
		
		Let $F$ be a maximal $\theta$-stable facet, which is also effective. We have constructed $\phi_{F}$ and $$\lambda_{F}(f)=\sum_{C\in\ch{F}{0}(G)}\phi_{F}(C)\sum_{C'\in\mco_{C}}f(C'),\quad f\in\mch^\infty(G).$$ 
		In particular, Assumption \ref{assumpabsconverg} is satisfied since $H\backslash G$ is strongly discrete (\emph{cf.} \cite{gurevich2016criterion}*{\S 5.3}). 
		
		\begin{lemma}\label{lemmalambdaFH'inv}
			
			$\lambda_{F}$ is also $H'$-invariant.
			
		\end{lemma}

		\begin{proof}
			
			We would like to show that for each $C\in\ch{F}{0}(G)$, the $H$-orbit $\mco_{C}$ is also an $H'$-orbit. 
			
			Let $\mca^g=\mca(G,S^g)$ be a $\theta$-stable apartment containing $C$, such that the equation \eqref{eqorthothetaapcond} is satisfied for some $(\bs{n},r)$.
			
			If $n=2r$, then $F$ is a chamber of $\mcb(G)$ such that $F^\theta$ is a chamber of $\mcb(G)^\theta$. Moreover, $\ch{F}{0}(G)=\{F\}$. Using Lemma \ref{lemmaH'-Hconj} and Proposition \ref{propA'chamberHconj}, the $H$-orbit and the $H'$-orbit of $F$ are the same. Thus in this case $\mco_{C}=\mco_{F}$ is an $H'$-orbit as well.
			
			If $n>2r$, using \eqref{eqdiagH'} we may find an element $t^{g}\in S^{g}\cap H'-S^{g}\cap H$ that fixes each point of $\mca^{g}$. Thus $H'\cdot C=H\cdot C\cup Ht^{g}\cdot C=H\cdot C$, meaning that $\mco_{C}$ is also an $H'$-orbit.
			
		\end{proof}
		
		Finally, by Proposition \ref{propA'chamberHconj} and Proposition \ref{proporthgeff}, Assumption \ref{assumpaffinecoxeter}, \ref{assumpnodoubleedge} is also satisfied for each $\theta$-stable apartment $\mca$, implying that each graph $\Gamma(G,\theta,[\mca]_{H,\sim})$ has a single vertex which is effective and has no double edge.

		\subsection{The exceptional case}\label{subsectionexceptionalcase}
		
		In this part, we discuss the exceptional case, where $G=\mrgl_{2}(K)$ and $H$ is an orthogonal subgroup of $G$.
		
		First we assume that $H$ is split. Taking the $G$-conjugation, we may assume that $$\theta(g)=J_{2}\,^{t}g^{-1}J_{2},\quad g\in G.$$ 
		Then $H'=\mfS_{2}\ltimes S$ and $H=S$ with $\mfS_{2}$ being identified with permutation matrices and $S$ being the diagonal torus. In particular, $Z(H)\nsubseteq Z(G)\cap H$, which is the reason that we discuss this case separately.
		
		By direct calculation, there are five $H$-conjugacy classes of apartments $\mca_{i}=\mca(G,S^{g_{i}})$, $i=0,1,2,3,4$, such that
		$$g_{0}=1,\quad \,^{t}g_{1}^{-1}J_{2}g_{1}^{-1}=\mrdiag(1,-1),\quad\,^{t}g_{2}^{-1}J_{2}g_{2}^{-1}=\mrdiag(\varpi_{K},-\varpi_{K}), $$
		$$ \,^{t}g_{3}^{-1}J_{2}g_{3}^{-1}=\mrdiag(\epsilon_{0},-\epsilon_{0}),\quad\,^{t}g_{4}^{-1}J_{2}g_{4}^{-1}=\mrdiag(\epsilon_{0}\varpi_{K},-\epsilon_{0}\varpi_{K}).$$
		In particular, $\mca_{0}$ is of $\theta$-rank $1$, and $\mca_{1},\mca_{2},\mca_{3},\mca_{4}$ are of $\theta$-rank $0$. 
		
		\begin{lemma}
			
			$H'$ acts transitively on $\ch{}{}(\mca_{0}^{\theta})$. On the other hand, $\ch{}{}(\mca_{0}^{\theta})$ has two $H$-orbits.
			
		\end{lemma}
		
		\begin{proof}
			
			Using the lattice model (\emph{cf.} \cite{KP23}*{Chapter 15}), each chamber in $\mca_{0}^{\theta}$ is represented by the hereditary order of the form
			$$\mfa_{k}:=\begin{pmatrix} \mfo_{K} &\mfp_{K}^{k}\\
				\mfp_{K}^{1-k} & \mfo_{K}
			\end{pmatrix},$$
			where $k\in\mbz$. Then, it is clear that $H'$ acts transitively on $\{\mfa_{k}\}_{k\in\mbz}$, and there are two $H$-orbits $\{\mfa_{2k}\}_{k\in\mbz}$ and $\{\mfa_{2k+1}\}_{k\in\mbz}$.
			
		\end{proof}
		
		Now we consider $F_{0}, F_{1}, F_{2}\in\Ftmax(G)$, such that $F_{i}^{\theta}$ is a chamber of $\mca_{i}^{\theta}$.  Indeed, $F_{0}=F_{0}^{\theta}$ is a chamber as a line segment of $\mca_{0}=\mca_{0}^{\theta}$, and $F_{1}=F_{1}^{\theta}$ is the unique point in both $\mca_{1}^{\theta}$ and $\mca_{3}^{\theta}$, and  $F_{2}=F_{2}^{\theta}$ is the unique point in both $\mca_{2}^{\theta}$ and $\mca_{4}^{\theta}$. Let $F_{0}'$ be another chamber of $\mca_{0}$ that is adjacent to $F_{0}$. As in \S \ref{subsectionpanels}, $F_{0},F_{0}',F_{1},F_{2}$ are effective. Thus we have
		$$\abs{\Ftmax(G)/H'}=\abs{\Fteff(G)/H'}=3\quad\text{and}\quad\abs{\Ftmax(G)/H}=\abs{\Fteff(G)/H}=4$$
		Indeed, Assumption \ref{assumpaffinecoxeter}, \ref{assumpnodoubleedge} for $\mca_{0}$, $\mca_{1}$ and $\mca_{2}$ are satisfied for the group $H'$, and we have three different equivalence classes $[\mca_{0}]_{H',\sim}$, $[\mca_{1}]_{H',\sim}$ and $[\mca_{2}]_{H',\sim}$. All of them are effective.
		Using Theorem \ref{thmdistdimcal} (or more precisely its proof),
		$$\mrhom_{H'}(\mrst_{G},\mbc)\cong \mch(G)^{H'}$$
		is of dimension 3. 
		
		On the other hand, the subgraph $\Gamma(G,\theta,[\mca_{0}]_{H,\sim})$ has two isolated effective vertices represented by $F_{0}^{\theta}$ and $F_{0}'^{\theta}$ and has no double edge, while the subgraph $\Gamma(G,\theta,[\mca_{1}]_{H,\sim})$ (resp. $\Gamma(G,\theta,[\mca_{2}]_{H,\sim})$) consists of a single effective vertex represented by $F_{1}^{\theta}$ (resp. $F_{2}^{\theta}$). On the other hand, it is easy to calculate directly that $\lambda_{F_{0}}$, $\lambda_{F_{0}'}$, $\lambda_{F_{1}}$, $\lambda_{F_{2}}$ are linearly independent \footnote{We only need to show that $\lambda_{F_{0}}$ and $\lambda_{F_{0}'}$ are linearly independent by showing that $\begin{pmatrix}
				\lambda_{F_{0}}(f_{F_{0}}) & \lambda_{F_{0}'}(f_{F_{0}})\\
				\lambda_{F_{0}}(f_{F_{0}'}) & \lambda_{F_{0}'}(f_{F_{0}'})
			\end{pmatrix}$ is invertible.}. Thus
		$$\mrhom_{H}(\mrst_{G},\mbc)\cong \mch(G)^{H}$$
		is of dimension 4. 
		
		Now we assume that $H$ is quasi-split but not split.  By Proposition \ref{propclassthetaap} and direct calculation, it is easy to see that all the $\theta$-stable apartments are of $\theta$-rank $0$ and $\theta$-equivalent to each other. So there exists a single $H$-conjugacy (resp. $H'$-conjugacy) class of maximal $\theta$-stable facets. Moreover, by Proposition \ref{proporthgeff} this class is effective. Let $F$ be a maximal $\theta$-stable facet. We may construct $\phi_F$, $f_F$ and $\lambda_F$ as before and show that $\lambda_F(f_F)=\sum_{C\in\ch{F}{0}(G)}o_C\neq 0$ (indeed $F^\theta$ is a single point). Thus in this case
		$$\mrhom_{H}(\mrst_{G},\mbc)\cong \mch(G)^{H}$$
		is of dimension 1. 
		
		\subsection{Summary}
		
		Using Theorem \ref{thmdistdimcal}, whose condition is guaranteed by Proposition \ref{propA'chamberHconj}, Proposition \ref{propFmaxortho},  Corollary \ref{cormax=eff}, Lemma \ref{lemmalambdaFH'inv}, and also discussion in \S \ref{subsectionexceptionalcase}, Theorem \ref{thmmainfour} is proved.
		
		\newpage
		\printindex
		\newpage
		
		\thispagestyle{plain}
		\bibliographystyle{alpha}
		\bibliography{reference.bib}
		\thispagestyle{plain}
		
	\end{document}